\documentclass{article}

\newif\ifshownavigationpage
\newif\ifshowreminders
\newif\ifshownotationindex
\newif\ifshowtheoremlinks
\newif\ifshowtheoremtree
\newif\ifshowtheoremlist
\newif\ifshowequationlist

\newif\ifshowcomments

\newif\ifshowaddressedcomments
\newif\ifshowrvin
\newif\ifshowrvout

\showtheoremtreetrue

\usepackage{amsthm, amsmath, amssymb}
\usepackage{thmtools}
\usepackage{soul}
\usepackage{authblk}
\usepackage{framed}
\usepackage{mathrsfs}
\usepackage{hyperref}
\hypersetup{
    colorlinks,
    linkcolor={red!50!black},
    citecolor={blue!50!black},
    urlcolor={blue!80!black}
}

\usepackage[dvipsnames]{xcolor}
\usepackage{graphicx, multicol}
\usepackage{marvosym, wasysym}
\usepackage{fancyhdr}
\usepackage{stackengine}
\usepackage{algorithm}
\usepackage{bm}
\usepackage[noend]{algpseudocode}
\usepackage{bbm}
\usepackage{verbatim}
\usepackage{mdframed}
\usepackage{needspace}
\makeatletter
\renewcommand{\ALG@beginalgorithmic}{\scriptsize}
\makeatother
\usepackage{xcolor}
\allowdisplaybreaks

\DeclareFontFamily{U}{mathx}{}
\DeclareFontShape{U}{mathx}{m}{n}{ <-> mathx10 }{}
\DeclareSymbolFont{mathx}{U}{mathx}{m}{n}
\DeclareFontSubstitution{U}{mathx}{m}{n}

\newcommand{\bcdot}{\boldsymbol{\cdot}}

\ifshowcomments
    \newcommand{\summ}[1]{{\color{blue}[summary: #1]} } 
    \newcommand{\bz}[1]{{\color{PineGreen}[BZ: #1]} } 
    \newcommand{\rl}[1]{{\color{Periwinkle}[RL: #1]} } 
    \newcommand{\jb}[1]{{\color{Tan}[JB: #1]} } 
    \newcommand{\xw}[1]{{\color{RoyalBlue}[XW: #1]} } 
    \ifshowaddressedcomments
        \newcommand{\bza}[1]{{\color{PineGreen}\sout{[BZ: #1]}} } 
        \newcommand{\rla}[1]{{\color{Periwinkle}\sout{[RL: #1]}} } 
        \newcommand{\jba}[1]{{\color{Tan}\sout{[JB: #1]}} } 
        \newcommand{\xwa}[1]{{\color{RoyalBlue}\sout{[XW: #1]}} } 
    \else
        \newcommand{\bza}[1]{} 
        \newcommand{\rla}[1]{} 
        \newcommand{\jba}[1]{} 
        \newcommand{\xwa}[1]{} 
    \fi
\else
    \newcommand{\summ}[1]{} 
    \newcommand{\bz}[1]{} 
    \newcommand{\rl}[1]{} 
    \newcommand{\jb}[1]{} 
    \newcommand{\bza}[1]{} 
    \newcommand{\rla}[1]{} 
    \newcommand{\jba}[1]{} 
    \newcommand{\xw}[1]{} 
    \newcommand{\xwa}[1]{} 
\fi

\usepackage[shortlabels]{enumitem}

\newlist{thmdependence}{itemize}{10}
\setlist[thmdependence]{nosep,label=-}

\newcommand{\thmtreenode}[4]{\item[#1] {#2}~\ref{#3} {#4}}

\newcommand{\thmtreeref}[2]{\item {{{#1}}}~\ref{#2}}

\ifshowtheoremlinks
    \newcommand{\linksinthm}[1]{\emph{\linkdest{location, #1}\linktopf{#1} \linktothmtree{location, thm tree #1} }}
    \newcommand{\linksinthmwopf}[1]{\emph{\linkdest{location, #1} \linktothmtree{location, thm tree #1} }}
    \newcommand{\linksinpf}[1]{\linkdest{location, proof of #1}\linktothm{#1} \linktothmtree{location, thm tree #1} }
\else
    \newcommand{\linksinthm}[1]{}
    \newcommand{\linksinthmwopf}[1]{}
    \newcommand{\linksinpf}[1]{}
\fi

\ifshownotationindex
    \newcommand{\notationdef}[2]{\linkdest{location, notation definition of #1}\hyperlink{location, notation index of #1}{#2}}
    
\else
    \newcommand{\notationdef}[2]{#2}
\fi
    
\newcommand{\linktopf}[1]{\hyperlink{location, proof of #1}{\pflinksymbol}}
\newcommand{\linktothm}[1]{\hyperlink{location, #1}{\thmlinksymbol}}
\newcommand{\linktothmtree}[1]{\hyperlink{#1}{\thmtreelinksymbol}}
\newcommand{\thmlinksymbol}{{\tiny [Theorem]}}
\newcommand{\pflinksymbol}{{\tiny [Proof]}}
\newcommand{\thmtreelinksymbol}{{\tiny [ThmTree]}}

\makeatletter
\newcommand{\linkdest}[1]{\Hy@raisedlink{\hypertarget{#1}{}}}
\makeatother

\usepackage[margin=1.2in]{geometry}

\newtheorem{theorem}{Theorem}
\newtheorem{lemma}[theorem]{Lemma}

\newtheorem{proposition}[theorem]{Proposition}

\newtheorem{definition}[theorem]{Definition}

\newtheorem{assumption}{Assumption}
\newtheorem{remark}{Remark}

\newtheorem*{theorem-nonumber}{Theorem}
\newtheorem*{condition-nonumber}{Condition}
\newtheorem*{proposition-nonumber}{Proposition}

\usepackage{mathtools}
\DeclarePairedDelimiter{\ceil}{\lceil}{\rceil}
\DeclarePairedDelimiter\floor{\lfloor}{\rfloor}

\newcommand{\D}{\mathbb D}

\newcommand{\cmt}[1]{#1} 
\renewcommand{\cmt}[1]{} 
\renewcommand{\P}{\mathbf{P}}

\newcommand{\E}{\mathbf{E}}
\newcommand{\RV}{\mathcal{RV}}
\newcommand{\MRV}{\mathcal{MHRV}}
\newcommand{\R}{\mathbb{R}}
\newcommand{\Z}{\mathbb{Z}}
\renewcommand{\S}{\mathbb{S}}
\newcommand{\C}{\mathbb{C}}
\newcommand{\M}{\mathbb{M}}

\renewcommand{\complement}{c}

\newcommand{\powerset}[1]{ \widetilde{\mathcal P}_{#1}}
\newcommand{\powersetTilde}[1]{ {\mathcal P}_{#1} }

\newcommand{\stleq}{\underset{ \text{s.t.} }{\leq}}

\newcommand{\lo}{\mathit{o}}
\newcommand{\bo}{\mathcal{O}}

\usepackage[normalem]{ulem}

\usepackage{multirow}
\usepackage{longtable}

\usepackage{array}
\def\delequal{\mathrel{\ensurestackMath{\stackon[1pt]{=}{\scriptscriptstyle\text{def}}}}}
\def\distequal{\mathrel{\ensurestackMath{\stackon[1pt]{=}{\scriptstyle\mathcal{D}}}}}
\newcommand{\norm}[1]{\left\lVert#1\right\rVert}
\usepackage{mathtools}

\usepackage{subfigure}
\usepackage{float}
\usepackage{subfiles}
\title{Tail Asymptotics of Cluster Sizes in \\ Multivariate Heavy-Tailed Hawkes Processes}

\author[1]{Jose Blanchet} 
\author[2]{Roger J.~A. Laeven}
\author[2]{Xingyu Wang}
\author[3]{Bert Zwart\footnote{Corresponding Author}}
\affil[1]{Department of Management Science and Engineering, Stanford University}
\affil[2]{Department of Quantitative Economics, University of Amsterdam}
\affil[3]{Centrum Wiskunde \& Informatica (CWI)}

\begin{document}

\maketitle 

\begin{abstract}
We examine a distributional fixed-point equation related to a multi-type branching process that is key in the cluster sizes analysis of multivariate heavy-tailed Hawkes processes. 
Specifically, we explore the tail behavior of its solution and demonstrate the emergence of a form of multivariate hidden regular variation. 
Large values of the cluster size vector result from one or several significant jumps. 
A discrete optimization problem involving any given rare event set of interest determines the exact configuration of these large jumps and the degree of hidden regular variation. 
Our proofs rely on a detailed probabilistic analysis of the spatiotemporal structure of multiple large jumps in multi-type branching processes.
\end{abstract}

\counterwithin{equation}{section}
\counterwithin{lemma}{section}
\counterwithin{corollary}{section}
\counterwithin{theorem}{section}
\counterwithin{definition}{section}
\counterwithin{proposition}{section}
\counterwithin{figure}{section}
\counterwithin{table}{section}

\tableofcontents

\section{Introduction}

Understanding and managing the interplay of risks and uncertainties is central to many scientific, engineering, and business endeavors. 
In particular, the amplification of risks and uncertainties through feedback across space and time presents modeling challenges in contexts such as pandemics, clustering of financial shocks, earthquake aftershocks, and cascades of information. 
Mutually exciting processes, or multivariate Hawkes processes (\cite{496a8ed0-5e76-38ab-8d67-eaa3d8cd00d7}),
provide a natural formalism to address such challenges by capturing dependencies and clustering effects. Hawkes processes have found applications spanning across finance \cite{AITSAHALIA2015585,Bacry01072014,Hawkes01022018},
neuroscience \cite{LAMBERT20189,10.1214/10-AOS806}, 
seismology \cite{Ikefuji02012022,Ogata01031988}, 
biology \cite{Xu2005},
epidemiology \cite{chiang2022hawkes}, 
criminology \cite{9378017}, 
social science \cite{doi:10.1073/pnas.0803685105,10.1145/2808797.2814178,rizoiu2017tutorialhawkesprocessesevents},
queueing systems \cite{doi:10.1287/stsy.2021.0070,doi:10.1287/stsy.2018.0014,Koops_Saxena_Boxma_Mandjes_2018,selvamuthu2022infinite},
and cyber security \cite{baldwin2017contagion,Bessy-Roland_Boumezoued_Hillairet_2021}.
Lately,
the estimation and inference of Hawkes processes have also become active topics in machine learning \cite{JOSEPH2024115889,doi:10.1137/21M1396927,pmlr-v119-zhang20q,pmlr-v119-zuo20a}.

The cluster representation of Hawkes processes introduced in \cite{209288c5-6a29-3263-8490-c192a9031603}
reveals the branching (i.e., Bienayme-Galton-Watson) processes structure embedded in clusters induced by immigrant events of Hawkes processes.
The analysis of such branching processes 
plays a foundational role in many of the aforementioned works on Hawkes processes,
and is the focus of this paper.
More precisely, we examine a class of fixed-point equations that represents multi-type branching processes in general, and captures the size of Hawkes process clusters in particular.
Let $(\bm S_j)_{j \in [d]}$ be a set of non-negative random vectors that solves (with $[d] = \{1,2,\ldots,d\}$)
\begin{align}
    \bm S_j \distequal
    \bm e_j + \sum_{i \in [d]}\sum_{ m = 1 }^{ B_{i \leftarrow j}  }\bm S_i^{(m)},
    \qquad j \in [d],
     \label{def: fixed point equation for cluster S i}
\end{align}
where 
$\bm e_j$ is the $j^\text{th}$ unit vector in $\R^d$ (i.e., with the $j^\text{th}$ entry equal to 1 and all other entries equal to 0),
$(\bm S_i^{(m)})_{i \in [d],\ m\geq 1}$ are independent across $i$ and $m$ with each $\bm S_i^{(m)}$ being an independent copy of $\bm S_i$,
and the random vector $\bm B_{\bcdot \leftarrow j} = (B_{i \leftarrow j})_{i \in [d]}$ is independent of the $\bm S_i^{(m)}$'s.
The canonical representation of $\bm S_j$ in \eqref{def: fixed point equation for cluster S i} describes the total progeny of
a branching process across the $d$ dimensions, with $B_{i \leftarrow j}$ being the count of a type-$i$ child in one generation from a type-$j$ parent.
Throughout this paper, we consider the sub-critical case regarding the offspring distributions $(B_{i\leftarrow j})_{i,j \in [d]}$,
which ensures the existence, uniqueness, and (almost sure) finiteness of the $\bm S_j$'s;
see, e.g., \cite{JOFFE1967409}.
Variations of Equation~\eqref{def: fixed point equation for cluster S i} have also been studied under the name of multivariate smoothing transforms and are closely related to weighted branching processes;
see, e.g., \cite{BURACZEWSKI20131947,Mentemeier2016}.
In the specific context of Hawkes processes,
$\bm S_j$ represents the size of a cluster induced by a type-$j$ immigrant event, with the law of $\bm B_{\bcdot \leftarrow j}$ admitting a specific (conditional) Poissonian form;
see Remark~\ref{remark, applied to hawkes process clusters} and \cite{daley2003introduction,20.500.11850/151886} for more details.

In this paper, we study the tail asymptotics of $\bm S_j$ under the presence of power-law heavy tails in the  distribution of the offsprings $B_{i \leftarrow j}$.
This research problem: (i) is motivated by the firm relevance and prevalent use of heavy-tailed branching processes and Hawkes processes in
queueing systems \cite{Asmussen_Foss_2018,ernst2018stability},
network evolution \cite{markovich2024extremal},
PageRank algorithms \cite{Jelenković_Olvera-Cravioto_2010,10.1214/20-AAP1623},
and finance \cite{Bacry02082016,hardiman2013critical,10.1214/15-AAP1164};
(ii) fits in the vibrant research area of limit theorems for Hawkes processes \cite{BACRY20132475,baeriswyl2023tail,karim2021exact,karim2023compound,Zhu_2013,ZHU2013885,10.3150/20-BEJ1235,10.1214/14-AAP1005,10.1214/15-AAP1141,10.1214/14-AAP1003,10.1214/22-AAP1796} and branching processes \cite{Basrak02102013,Asmussen_Foss_2018,foss2018tails};
and, more importantly, (iii) addresses significant gaps in the existing literature on the heavy-tailed setting
(see, e.g., \cite{karim2021exact,baeriswyl2023tail,Basrak02102013,Asmussen_Foss_2018,Jelenković_Olvera-Cravioto_2010,BURACZEWSKI20131947}). 


More specifically,
existing asymptotic analyses of heavy-tailed branching processes (possibly with immigration) and Hawkes process clusters \cite{Basrak02102013,Asmussen_Foss_2018,foss2018tails,baeriswyl2023tail,karim2021exact,guo2025precise}
feature manifestations of the \emph{principle of a single big jump}.
In the context of heavy-tailed branching processes,
this well-known phenomenon states that rare events 
are typically caused by a large value of a single component within the system,
such as a specific node 
giving birth to a disproportionately large number of offspring in one generation.
The limitation of this perspective becomes apparent in the multivariate setting, as it addresses only a special class of rare events and ignores 
the hidden regular variation (see, e.g., \cite{resnick2002hidden,10.1214/14-PS231}) in $\bm S_j$.
In our setting, we show that hidden regular variation {emerges} if $\P(\norm{\bm S_j} > n) \sim b(n)$  for some regularly varying $b(\cdot)$, whereas, for some set $A$, $\P(n^{-1}\bm S_j \in A) \sim a(n)$ exhibits a significantly faster (and also regularly varying) rate of decay $a(n) = \lo(b(n))$.
For such $A$, the results in \cite{Asmussen_Foss_2018} verify only $\P(n^{-1}\bm S_j \in A) = \lo\big(\P(\norm{\bm S_j} > n)\big)$ and do not provide a further characterization of the precise rate of decay $a(n)$ or the leading coefficient under the $a(n)$-asymptotic regime.
Likewise, \cite{karim2021exact}  addresses tail asymptotics for Hawkes processes and the induced population processes (i.e., with departure) by focusing on target sets of the form $A = \{ \bm x \in \R^d : \bm c^\top \bm x > 1 \}$.
The corresponding rare events for such $A$ are also driven by the dominating large jump in the clusters of Hawkes processes.
In the context of marked Hawkes processes, Proposition~7.1 of \cite{baeriswyl2023tail} characterizes the extremal behavior of the sum functional in clusters driven by either one particularly large mark, or by observing a large amount of marks in one cluster.

The prior results are not able to describe the hidden regular variation in the distribution of $\bm{S}_j$ due to the limitations of existing approaches, as we review next.
\begin{itemize}
    \item 
        The \emph{Tauberian theorem} approach (see, e.g., \cite{karim2021exact,baeriswyl2023tail}) exploits differentiation and inversion techniques for Laplace transforms.
        For our purpose of characterizing the hidden regular variation of $\bm S_j = (S_{j,1},\ldots,S_{j,d})^\top$ over an arbitrary sub-cone $C \subseteq \R^d_+$,
        the strategy in \cite{karim2021exact} could theoretically be adapted using a multivariate version of the Tauberian theorem (e.g., \cite{resnick2007heavy,resnick2015tauberian}).
        However, this is possible only if one has access to a (semi-)closed form expression for the probability generating function of $\bm S_j\mathbbm{I}\{ \bm S_j \in C \}$---rather than $\psi(\bm z) := \E\big[ \prod_{i \in [d]}{z_i}^{ S_{j,i}  }  \big]$, the probability generating function for $\bm S_j$ itself---in order to apply differentiation techniques and verify the conditions of the Tauberian theorem for the measure $\P( \bm S_j \in \ \cdot \ \cap C  )$.
        While useful expressions for the generating function of $\bm S_j$, as well as the joint transform for multivariate Hawkes processes and the conditional intensity functions (as demonstrated in \cite{karim2021exact}), can be derived by exploiting the fact that the process is branching,
        extending this to the transforms of $\bm S_j\mathbbm{I}\{ \bm S_j \in C \}$ is highly non-trivial as it has to be built upon a detailed understanding of how $\bm S_j$ stays within the cone $C$.
        We note that similar issues arise when studying weighted branching processes and smoothing transforms (see, e.g., \cite{Mentemeier2016,Liu_1998,Volkovich_Litvak_2010}).\footnote{Indeed, taking  $\bm R \distequal \sum_{m \geq 1} W_m \bm R^{(m)}$ as an example,
        where $W_m$ are i.i.d.\ scalar variables and $\bm R^{(m)}$ are i.i.d.\ copies of $\bm R$,
        while $\psi(\bm t) = \E\big[ \prod_{m \geq 1} \psi( W_m \bm t ) \big]$ follows directly with $\psi$ being the Laplace transform of $\bm R$, such equality does not hold for the Laplace transform of $\bm R\mathbbm{I}\{\bm R \in C\}$ given a general cone $C$.}
        \smallskip

    \item 
        Another approach takes a more probabilistic route by
        establishing or exploiting \emph{asymptotics for randomly stopped/weighted sums of regularly varying variables};
        see, e.g., 
        \cite{Basrak02102013,guo2025precise,Asmussen_Foss_2018,10.1214/20-AAP1623,fay2006modeling,Olvera-Cravioto_2012,10.3150/10-BEJ251}.
        See also \cite{xu2015convolution,markovich2020maxima,markovich2022weightedmaximasumsnonstationary,axioms12070641,FOSS2024104422} for recent progress in this area.
        However, existing multivariate results do not allow for the characterization of hidden regular variation in random sums of heavy-tailed vectors (see, e.g., \cite{10.3150/08-BEJ125}) or hinge on the light-tailedness of the random count in the sums (see, e.g., Theorem~4.2 of \cite{konstantinides2024randomvectorspresencesingle} and Theorem~4.3 of \cite{das2023aggregatingheavytailedrandomvectors}),
        making them largely incompatible with our setting
        and the goal of understanding the mechanism by which {$n^{-1}\bm S_j$ stays within a general set $A$.}
        Similarly, in the literature on weighted branching processes and smoothing transforms, 
        Theorem~5.1 in
        \cite{Jelenković_Olvera-Cravioto_2010} makes use of large deviations results for weighted recursions on trees, and 
        \cite{10.3150/13-BEJ576,buraczewski2016precisetailasymptoticsattracting} rely on large deviations for the product of i.i.d.\ random variables or matrices.
        However, these technical tools essentially characterize the probability of observing a large norm of the underlying processes  and do not reveal the hidden regular variation therein.
        \smallskip

    \item 
        On a related note, 
        \emph{renewal-theoretic} tools have been useful when studying tail asymptotics of weighted branching processes and smoothing transforms
        (e.g., \cite{Jelenković_Olvera-Cravioto_2012,JELENKOVIC2015217,BURACZEWSKI20131947}
        and Theorem~4.2 in \cite{Jelenković_Olvera-Cravioto_2010}).
        For instance, in multivariate smoothing transforms
        $
        \bm R \distequal \sum_{i = 1}^B \textbf{W}^{(m)} \bm R^{(m)} + \bm Q,
        $
        with weights $\textbf{W}^{(m)}$ being i.i.d.\ matrices,
        $\bm Q$ being a random vector,
        the random variable $B$ taking values in $\mathbb Z_+$,
        and $\bm R^{(m)}$ being i.i.d.\ copies of $\bm R$,
        the tail asymptotics in $\bm R$ can be established by verifying integrability conditions regarding $\bm R$ and $\textbf{W} \bm R$.
        This method is well-suited to analyze 
        random fluctuations from multiplying the weights $\textbf{W}^{(m)}$ (in the spirit of the classical Kesten-Goldie Theorem \cite{Kesten1973, Goldie1991}),
        but seems less natural in our setting 
        \eqref{def: fixed point equation for cluster S i}, 
        where weights are deterministic (i.e., $\textbf{W}^{(m)} \equiv \textbf{I}$) and the offspring counts are heavy-tailed and sub-critical. 

\end{itemize}

To resolve the technical challenges in the asymptotic analysis of $\P( n^{-1}\bm S_j \in A )$ for sufficiently general $A \subseteq \R^d_+ \delequal [0,\infty)^d$, we develop an approach that reveals \emph{the spatio-temporal structure of multiple big jumps in branching processes}.
Specifically, through another set of distributional fixed-point equations (given $M > 0$),
\begin{align}
    \bm S_j^{\leqslant}(M) \distequal
    \bm e_j + \sum_{i \in [d]}\sum_{ m = 1 }^{ B_{i \leftarrow j}\mathbbm{I}\{ B_{i \leftarrow j} \leq M  \}  }\bm S_i^{(m),\leqslant}(M),
    \qquad j \in [d],
     \label{def: fixed point equation for pruned cluster S i leq M}
\end{align}
with the $\bm S_i^{(m),\leqslant}(M)$'s being i.i.d.\ copies of $\bm S_i^{\leqslant}(M)$,
we construct a ``pruned'' version of $\bm S_j$ in \eqref{def: fixed point equation for cluster S i}
by identifying nodes in the underlying branching process that give birth to more than $M$ children along the same dimension and then removing these children.
Our analysis hinges on an intuitive yet crucial connection between $\bm S_j$ and $\bm S_j^{\leqslant}(M)$:
\begin{align}
    \bm S_j \distequal \bm S^{\leqslant}_j(M) + \sum_{i \in [d]}\sum_{m = 1}^{ W^>_{j;i}(M) }\bm S^{(m)}_i,
    \qquad j \in [d],
    \label{proof, equality, from S i truncated to S i}
\end{align}
where
$W^>_{j;i}(M)$ counts the pruned children along the $i^\text{th}$ dimension under threshold $M$, and
the  $\bm S^{(m)}_i$'s are independent copies of the $\bm S_i$'s.
That is, a branching process can be generated by:
(i) halting the reproduction of a node if it plans to give birth to a large number (more precisely, more than $M$) of children along the same dimension, which yields $\bm S_j^\leqslant(M)$, and then
(ii)
resuming the reproduction of child nodes that were previously on hold (and their offspring), which recovers the law of the original branching process and yields $\bm S_j$.
Furthermore, by recursively applying this argument onto the i.i.d.\ copies $\bm S^{(m)}_i$ in the RHS of \eqref{proof, equality, from S i truncated to S i},
we decompose $\bm S_j$ into a nested tree of independent samples of the pruned clusters $\bm S^\leqslant_i(M)$.
We formalize this decomposition by proposing the notion of ``types'',
which characterizes the spatio-temporal relationship of nodes giving birth to a large number of children (i.e., big jumps) in a branching process.
In Sections~\ref{subsec: proof methodology, cluster size} and \ref{subsec: proof, theorem: main result, cluster size}, we provide details of the proof strategy and the definitions involved,
highlighting that under this framework,
the problem largely reduces to establishing concentration inequalities for $\bm S^\leqslant_j(M)$
and deriving the probability of observing each type of structure (as in Definitions~\ref{def: cluster, type} and \ref{def: cluster, generalized type}) in $\bm S_j$.

Building upon this framework, Theorem~\ref{theorem: main result, cluster size} characterizes the hidden regular variation in $\bm S_i$.
Specifically,
given a non-empty index set $\bm j \subseteq \{1,2,\ldots,d\}$
and a set $A \subseteq \R^d_+$ that is bounded away from the origin and ``roughly contained within'' $\R^d(\bm j) \delequal \{ \sum_{i \in \bm j}w_i\cdot \E\bm S_i:\ w_i \geq 0\ \forall i \in \bm  j \}$, which is the cone generated by $(\E\bm S_i)_{i \in \bm j}$,
Theorem~\ref{theorem: main result, cluster size} indicates that
\begin{align}
    \mathbf C^{\bm j}_i(A^\circ)
    \leq 
    \liminf_{n \to \infty}
    \frac{
        \P(n^{-1}\bm S_i \in A)
    }{ \lambda_{ \bm j }(n) }
    \leq 
    \limsup_{n \to \infty}
     \frac{
        \P(n^{-1}\bm S_i \in A)
    }{ \lambda_{ \bm j}(n)}
    \leq 
    \mathbf C^{\bm j}_i(A^-),
    \label{statement, intro, tail asymptotics, 1}
\end{align}
where 
$A^\circ$ and $A^-$ are the interior and closure of $A$, respectively,
$\mathbf C^{\bm j}_i(\cdot)$ is a Borel measure supported on $\R^d(\bm j)$,
and
$\lambda_{\bm j}(n) \in \RV_{-\alpha(\bm j)}(n)$ is some regularly varying function dictated by the law of the $B_{j \leftarrow i}$'s.
That is, over each cone $\R^d(\bm j)$,
$\bm S_i$ exhibits hidden regular variation with rate function $\lambda_{\bm j}(\cdot)$, power-law index $-\alpha(\bm j)$, and limiting measure $\mathbf C^{\bm j}_i(\cdot)$.
Furthermore, given a general set $A \subseteq \R^d_+ \setminus\{\bm 0\}$, which may span multiple cones $\R^d(\bm j)$,
Theorem~\ref{theorem: main result, cluster size} establishes asymptotics of the form
\begin{align}
    \mathbf C^{\bm j(A)}_i(A^\circ)
    \leq 
    \liminf_{n \to \infty}
    \frac{
        \P(n^{-1}\bm S_i \in A)
    }{ \lambda_{ \bm j(A) }(n) }
    \leq 
    \limsup_{n \to \infty}
     \frac{
        \P(n^{-1}\bm S_i \in A)
    }{ \lambda_{ \bm j(A)}(n)}
    \leq 
    \mathbf C^{\bm j(A)}_i(A^-),
     \label{statement, intro, tail asymptotics, 2}
\end{align}
where $\bm j(A) \delequal \underset{ \bm j \subseteq \{1,2,\ldots,d\}:\ \R^d(\bm j)\cap A \neq \emptyset }{\arg\min}\alpha(\bm j)$.
In other words, for a general set $A$,
the asymptotics $\P(n^{-1}\bm S_i \in A)$ are determined by a discrete optimization problem identifying,
among all cones $\R^d(\bm j)$ that intersect the set $A$,
which one has the heaviest tail in terms of $\alpha(\bm j)$, and hence the highest probability of observing a large $\bm S_i$ over this cone.
Besides, the limiting measures $\mathbf C^{\bm j}_i(\cdot)$ are amenable to straightforward computation using Monte Carlo simulation;
see
Section~\ref{subsec: tail of cluster S, statement of results} and remarks therein for the precise statement of Theorem~\ref{theorem: main result, cluster size} and the rigorous definitions of the notions involved.
Here, we note that Theorem~\ref{theorem: main result, cluster size} is stated in terms of $\MRV$,
a notion of multivariate hidden regular variation we propose in Section~\ref{subsec: MRV}.
Compared to existing formalisms (e.g., \cite{Resnick_2004, hult2006regular,das2023aggregatingheavytailedrandomvectors}),
$\MRV$ offers a richer characterization of tail asymptotics and 
provides a more adequate framework for describing heavy tails in branching processes and Hawkes processes: 
as demonstrated in Remark~\ref{remark: bounded away condition under polar transform, tail asymptotics of cluster sizes}, 
asymptotics \eqref{statement, intro, tail asymptotics, 1} and \eqref{statement, intro, tail asymptotics, 2}
would fail
under existing formalisms of multivariate hidden regular variation.

In a companion paper \cite{wangHawkesLDWorkingPaper},
we apply the tail asymptotics of $\bm S_j$ to characterize the sample path large deviations for a multivariate heavy-tailed Hawkes process $\bm N(t)$.
Specifically, under heavy-tailed offspring distributions and proper tail conditions on the fertility functions of $\bm N(t)$,
we establish asymptotics of the form
\begin{align}
    \breve{\mathbf C}_{ \bm k(E) }(E^\circ)
    \leq 
    \liminf_{n \to \infty}
    \frac{
        \P(\bar{\bm N}_n \in E)
        }{
            \breve\lambda_{\bm k(E)}(n)
        }
    \leq 
    \limsup_{n \to \infty}
    \frac{
        \P(\bar{\bm N}_n \in E)
        }{
            \breve\lambda_{\bm k(E)}(n)
        }
    \leq 
    \breve{\mathbf C}_{ \bm k(E) }(E^-),
    \label{intro, LD for Hawkes}
\end{align}
for a collection of sets 
%
$E \subseteq \D[0,\infty)$
general enough to capture scenarios involving multiple big jumps. 
Here, 
 $\bar{\bm N}_n = \{ \bm N(nt)/n:\ t \geq 0  \}$
is the scaled sample path of $\bm N(t)$ embedded in $\D[0,\infty)$,
the $\breve \lambda_{\bm k}(\cdot)$'s are regularly varying functions,
the limiting measures $\breve{\mathbf C}_{ \bm k}(\cdot)$'s are supported on $\D[0,\infty)$,
and the vector $\bm k(E)$ plays a role analogous to rate functions in the classical large deviation principle framework.
Specifically, $\bm k(E)$
represents the most likely configuration of jumps required for a linear path with slope $\mu_{\bm N}$ to enter the set $E$;
here, $\mu_{\bm N}$ is the expectation of increments in $\bm N(t)$ under stationarity,
and the linear function with slope $\mu_{\bm N}$ represents the nominal behavior of the Hawkes process.
Furthermore, as established in \eqref{statement, intro, tail asymptotics, 1}--\eqref{statement, intro, tail asymptotics, 2},
the probability of observing a large cluster over the cone $\R^d(\bm j)$---and thus the ``cost'' of adding a jump aligned in  $\R^d(\bm j)$ to the nominal path---is dictated by tail indices $\alpha(\bm j)$.
Therefore, the characterization of hidden regular variation in Equations \eqref{statement, intro, tail asymptotics, 1}--\eqref{statement, intro, tail asymptotics, 2} in this paper allows us to determine the rate function $\bm k(E)$,
revealing the most likely configuration of big jumps that push $\bar{\bm N}_n$ into $E$,
and develop sample path large deviations for Hawkes processes in  \eqref{intro, LD for Hawkes} that go well beyond the domain of a single big jump.
These results bridge the gaps in the existing literature, provide detailed qualitative insights, and can serve as stepping stones towards efficient rare-event simulation of risks in
practical systems with clustering or mutual-excitation effects.

The rest of the paper is structured as follows.
Section~\ref{sec: M convergence and hidden RV} reviews the notion of $\mathbb M$-convergence and proposes $\MRV$, a new notion of multivariate hidden regular variation.
Section~\ref{sec: tail of cluster size} presents Theorem~\ref{theorem: main result, cluster size}---the main result of this paper---that characterizes the hidden regular variation of $\bm S_j$ in
\eqref{def: fixed point equation for cluster S i},
and describes the proof strategy.
Section~\ref{sec: proof, cluster size asymptotics} provides the proofs.
In the Appendix,
Section~\ref{sec: appendix, lemmas} collects useful auxiliary results,
Section~\ref{sec: counterexamples} provides the details of the counterexample in Remark~\ref{remark: bounded away condition under polar transform, tail asymptotics of cluster sizes},
Section~\ref{subsec: proof, M convergence and asymptotic equivalence} collects the proofs of technical tools regarding $\mathbb M$-convergence and asymptotic equivalence,
Section~\ref{subsec: proof, technical lemmas, cluster size} contains the proofs of technical lemmas applied in Section~\ref{sec: proof, cluster size asymptotics},
and Section~\ref{sec: appendix, theorem tree} provides the theorem tree.

\section{$\mathbb M(\mathbb S\setminus\mathbb C)$-Convergence and Multivariate Hidden Regular Variation}
\label{sec: M convergence and hidden RV}

We review the notion of $\mathbb M$-convergence in Section~\ref{subsec: M convergence},
and then develop the $\MRV$ formalism in Section~\ref{subsec: MRV} to  generalize the classical notion of multivariate regular variation.
This framework supports the formulation and proof of our main results in Section~\ref{sec: tail of cluster size},
capturing the phenomenon of varying power-law index across different directions in Euclidean spaces.

We first introduce notations that will be used frequently throughout the paper.
 Let 
 $\mathbb Z$ be the set of integers,
 $\notationdef{notation-non-negative-numbers-and-zero}{\mathbb{Z}_+} = \{0,1,2,\cdots\}$ 
be the set of non-negative integers,
and
 $\notationdef{notation-non-negative-numbers}{\mathbb{N}} = \{1,2,\cdots\}$ be the set of strictly positive integers.
Let $\notationdef{set-for-integers-below-n}{[n]} \delequal \{1,2,\cdots,n\}$ for any positive integer $n$.
As a convention, we set $[0] = \emptyset$.
For each positive integer $m$, 
let $\notationdef{notation-power-set-for-[d]}{\powerset{m}}$ be the power set of $[m]$, i.e., the collection of all subsets of $\{1,2,\ldots,m\}$,
and  let $\notationdef{notation-power-set-for-[d]-without-empty-set}{\powersetTilde{m}} \delequal \widetilde{\mathcal P}_m \setminus \{\emptyset\}$ be the collection of all \emph{non-empty} subsets of $[m]$.
Let $\notationdef{notation-R}{\R}$ be the set of reals. 
For any $x \in \R$,
let
$\notationdef{floor-operator}{\floor{x}}\delequal \max\{n \in \mathbb{Z}:\ n \leq x\}$
and
$\notationdef{ceil-operator}{\ceil{x}} \delequal \min\{n \in \Z:\ n \geq x\}$.
Let $\notationdef{notation-R-d-+}{\R^d_+} = [0,\infty)^d$. 
Given some metric space $(\mathbb S,\bm d)$ and a set $E \subseteq \mathbb S$,
let
$\notationdef{notation-interior-of-set-E}{E^\circ}$ and $\notationdef{notation-closure-of-set-E}{E^-}$ be the interior and closure of $E$, respectively.
For any $r > 0$, 
let
$\notationdef{notation-epsilon-enlargement-of-set-E}{E^r} \delequal 
\{ y \in \mathbb{S}:\ \bm{d}(E,y)\leq r\}$ be the $r$-enlargement of the set $E$,
and
$
\notationdef{notation-epsilon-shrinkage-of-set-E}{E_{r}} \delequal
((E^c)^r)^\complement
=
\{
y \in \mathbb S:\ \bm d(E^c,y) > r
\}
$
be the $r$-shrinkage
of $E$.
Note that $E^r$ is closed and $E_r$ is open for any $r > 0$.
Throughout,
we adopt the $L_1$ norm $\notationdef{notation-L1-norm}{\norm{\bm x}} = \sum_{i \in [d]}|x_i|$ for any real vector $\bm x \in \R^d$.
We use $\notationdef{notation-R-d-+-unit-sphere}{\mathfrak N^d_+} \delequal \{\bm x \in \R_+^d:\ \norm{\bm x} = 1\}$ to denote the unit sphere under the $L_1$ norm, restricted to the positive quadrant.

\subsection{$\mathbb M(\mathbb S \setminus \mathbb C)$-Convergence}
\label{subsec: M convergence}
We recall the notion of  $\mathbb M(\mathbb S \setminus \mathbb C)$-convergence (\cite{10.1214/14-PS231}),
which has recently emerged as a suitable  foundation for large deviations analyses of heavy-tailed stochastic systems (\cite{10.1214/14-PS231,10.1214/18-AOP1319,10.1214/24-EJP1115}).
Consider a complete and separable metric space $(\mathbb{S},\bm{d})$.
Given Borel measurable sets $A,B \subseteq \mathbb S$,
we say that $A$ is bounded away from $B$ (under $\bm d$) if 
$
\bm d(A,B) \delequal \inf_{ x \in A, y \in B }\bm d(x,y) > 0.
$
Given a Borel set $\mathbb{C} \subseteq \mathbb{S}$,
let $\mathbb{S}\setminus \mathbb{C}$
be the metric subspace of $\mathbb{S}$ in the relative topology,
which induces the 
$\sigma$-algebra
$
\mathscr{S}_{ \mathbb{S}\setminus \mathbb{C}}
\delequal 
\{ A \in \mathscr{S}_{\mathbb{S}}:\ A \subseteq \mathbb{S}\setminus \mathbb{C}\}.
$
Here, we use $\notationdef{notation-borel-sigma-algebra-of-S}{\mathscr S_{\mathbb S}}$ to denote the Borel $\sigma$-algebra of $\mathbb S$.
Let
\begin{align*}
    \notationdef{notation-M-S-exclude-C}{\mathbb{M}(\S\setminus \mathbb{C})}
    \delequal 
    \{
    \nu(\cdot)\text{ is a Borel measure on }\mathbb{S}\setminus \mathbb{C} :\ \nu(\mathbb{S}\setminus \mathbb{C}^r) < \infty\ \forall r > 0
    \}.
\end{align*}
We topologize
$\mathbb{M}(\mathbb{S}\setminus \mathbb{C})$\linkdest{notation-M-convergence}
by the sub-basis generated by sets of the form
$
\{
\nu \in \mathbb{M}(\mathbb{S}\setminus \mathbb{C}):\ \nu(f) \in G
\},
$
where $G \subseteq [0,\infty)$ is open, $f \in \mathcal{C}({ \mathbb{S}\setminus \mathbb{C} })$,
and
$
\notationdef{notation-mathcal-C-S-exclude-C}{\mathcal{C}({ \mathbb{S}\setminus \mathbb{C} })}
$
is the set of all real-valued, non-negative, bounded and continuous functions with support bounded away from $\mathbb{C}$ (i.e., $f(x) = 0\ \forall x \in \mathbb{C}^r$ for some $r > 0$).
We now state the definition of $\mathbb M(\mathbb S \setminus \mathbb C)$-convergence.

\begin{definition}[$\mathbb M(\mathbb S \setminus \mathbb C)$-convergence]
\label{def: M convergence}
Given $\mu_n, \mu \in \M(\S\setminus\C)$,
we say that \textbf{$\mu_n$ converges to $\mu$ in $\M(\S\setminus\C)$} as $n \to \infty$
if
\begin{align*}
    \lim_{n \to \infty}|\mu_n(f) - \mu(f)| = 0,
    \qquad
    \forall f \in \mathcal C(\mathbb S \setminus \mathbb C).
\end{align*}
\end{definition}
When there is no ambiguity about $\mathbb S$ and $\mathbb C$, we refer to Definition~\ref{def: M convergence} as $\mathbb M$-convergence.
Next, we recall the Portmanteau Theorem for $\mathbb M$-convergence.

\begin{theorem}[Theorem 2.1 of \cite{10.1214/14-PS231}]
\label{portmanteau theorem M convergence}
 Let  $\mu_n, \mu \in \M(\S\setminus\C)$.
 We have $\mu_n \to \mu$ in $\M(\S\setminus\C)$ as $n \to \infty$ if and only if
\begin{align*}
    \limsup_{n \to \infty}\mu_n(F) \leq \mu(F),
    \qquad
    \liminf_{n \to \infty}\mu_n(G) \geq \mu(G),
\end{align*}
for any closed set $F$ and open set $G$ that are bounded away from $\mathbb C$.
    
\end{theorem}

\subsection{Multivariate Hidden Regular Variation}
\label{subsec: MRV}

Recall that a measurable function $\phi:(0,\infty) \to (0,\infty)$ is said to be regularly varying as $x \rightarrow\infty$ with index $\beta \in \R$ (denoted as $\phi(x) \in \notationdef{notation-regular-variation}{\RV_\beta}(x)$ as $x \to \infty$) if $\lim_{x \rightarrow \infty}\phi(tx)/\phi(x) = t^\beta$ for all $t>0$. 
See, e.g., \cite{bingham1989regular,resnick2007heavy,foss2011introduction} for properties of regularly varying functions.


The goal of this subsection is to generalize the classical notion of multivariate regular variation (e.g., \cite{Resnick_2004, hult2006regular})
and propose $\MRV$, 
a framework suitable for describing the multivariate hidden regular variation in branching processes and Hawkes processes.
To formally present the definition and encode the geometry and the degree of hidden regular variation over arbitrarily positioned cones in $\R^d_+$,
we introduce the following key elements.
\begin{itemize}
    \item 
        Recall that
  ${\powersetTilde{m}} $ is the collection of all \emph{non-empty} subsets of $[m] = \{1,2,\ldots,m\}$,
  and note that $|{\mathcal P}_m| = 2^m - 1$.
Given $\bar{\textbf S} = \{\bar{\bm s}_j \in [0,\infty)^d: j \in [k]\}$
and some $\bm j \in \powersetTilde{k}$ (i.e., $\bm j \subseteq [k]$ with $\bm j \neq \emptyset$), let
\begin{align}
    \notationdef{notation-convex-cone-R-d-i-cluster-size}{\R^d(\bm j;\bar{\textbf S})}
    \delequal 
    \Bigg\{ \sum_{i \in \bm j}  w_i\bm{\bar s}_i:\ w_i \geq 0\ \forall i \in \bm j \Bigg\}
    \label{def: cone R d index i}
\end{align}
be the convex cone in $\R^d_+$ generated by the vectors $\{ \bar{\bm s}_i:\ i \in \bm j \}$.
The purpose of the $\MRV$ formalism is to describe the hidden regular variation of a measure $\nu$ over the collection of cones $\big(\R^d(\bm j;\bar{\textbf S})\big)_{ \bm j \in \powersetTilde{k} }$ generated \emph{under the basis} $\bar{\textbf S}$.

    \item 
        Next,
        consider the collection of \emph{tail indices}
$
        \bm \alpha = \big\{  \alpha(\bm j) \in [0,\infty):\ \bm j \subseteq [k] \big\}
        $
        that is strictly monotone w.r.t.\ $\bm j$: that is,
        $\alpha(\bm j) < \alpha(\bm j^\prime)$ holds for any $ \bm j \subsetneq \bm j^\prime \subseteq [k]$.
        We adopt the convention that $\alpha(\emptyset) =0$.
        Each $\alpha(\bm j)$ denotes the power-law tail index of the hidden regular variation over the cone ${\R^d(\bm j;\bar{\textbf S})}$.

    \item 
        More precisely, for each $\bm j \in \powersetTilde{k}$,
        the hidden regular variation over the cone $\R^d(\bm j;\bar{\textbf S})$ is characterized by a \emph{rate function}
        $\lambda_{\bm j}: (0,\infty) \to (0,\infty)$ such that $\lambda_{\bm j}(x) \in \RV_{ -\alpha(\bm j) }(x)$ as $x \to \infty$.

    \item
        Meanwhile, under the limiting regime with $\lambda_{\bm j}(n)$-scaling,
        the tail behavior of the measure $\nu$ over the cone  $\R^d(\bm j;\bar{\textbf S})$
        is captured by the \emph{limiting measure} $\mathbf C_{\bm j}$.
        Specifically, recall that we use ${\mathfrak N^d_+}$ to denote the unit sphere restricted in $\R^d_+$.
        For each $\epsilon \geq 0$ and $\bm j \in \powersetTilde{k}$, let
\begin{align}
    \notationdef{notation-enlarged-convex-cone-bar-R-d-i-epsilon}{ \bar \R^d(\bm j, \epsilon;\bar{\textbf S}) }
    \delequal
    \bigg\{
        w\bm s:\ w \geq 0,\ \bm s \in \mathfrak N^d_+,\ 
            \inf_{  \bm x\in \R^d(\bm j;\bar{\textbf S}) \cap \mathfrak N^d_+   }\norm{ \bm s - \bm x  } \leq \epsilon
    \bigg\}
    \label{def: enlarged cone R d index i epsilon}
    \end{align}
be an enlarged version of the cone $\R^d(\bm j;\bar{\textbf S})$ by considering the polar coordinates of its elements under $\epsilon$-perturbation to their angles.
Note that 
$
\bar \R^d(\bm j, 0;\bar{\textbf S}) = \R^d(\bm j;\bar{\textbf S}).
$
We also adopt the convention that $\bar\R^d(\emptyset,\epsilon;\bar{\textbf S}) = \{\bm 0\}$.
        We say that $A \subseteq \R^d_+$ is bounded away from $B \subseteq \R^d_+$ if $\inf_{\bm x \in A,\ \bm y \in B}\norm{\bm x - \bm y} > 0$.
        For each $\bm j \in \powersetTilde{k}$,
        the limiting measure
        $\mathbf C_{\bm j}(\cdot)$ is a Borel measure supported on $\R^d(\bm j;\bar{\bm S})$ such that 
        $
        \mathbf C_{\bm j}(A) < \infty
        $
        holds for any Borel set $A \subseteq \R^d_+$ that is bounded away from
        \begin{align}
            \notationdef{notation-cone-R-d-leq-i-basis-S-index-alpha}{\bar\R^{d}_\leqslant(\bm j,\epsilon; \bar{\textbf S},\bm \alpha)}
            \delequal
            \bigcup_{
                \bm j^\prime \subseteq [k]:\ 
                \bm j^\prime \neq \bm j,\ \alpha(\bm j^\prime) \leq \alpha(\bm j)
            } \bar\R^d(\bm j^\prime,\epsilon;\bar{\textbf S})
            \label{def: cone R d i basis S index alpha}
        \end{align}
        under some (and hence all) $\epsilon > 0$ small enough.
        Note that by the convention $\bar{\R}^d(\emptyset,\epsilon;\bar{\textbf S}) = \{\bm 0\}$,
        we either have ${\bar\R^{d}_\leqslant(\bm j,\epsilon; \bar{\textbf S},\bm \alpha)} = \{\bm 0\}$,
        or that
        $
        {\bar\R^{d}_\leqslant(\bm j,\epsilon; \bar{\textbf S},\bm \alpha)}
        $
        is the union of all ${\bar\R^{d}(\bm j^\prime,\epsilon; \bar{\textbf S})}$
        such that $\bm j^\prime \in \powersetTilde{k}$, $\bm j^\prime \neq \bm j$, and $\alpha(\bm j^\prime) \leq \alpha(\bm j)$.
        
\end{itemize}


We are now ready to state the definition of $\MRV$.

\begin{definition}[$\MRV$]
\label{def: MRV}
Let $\nu(\cdot)$ be a Borel measure on $\R^d_+$, and let $\nu_n(\cdot) \delequal \nu( n\ \cdot\ )$ (i.e., $\nu_n(A) = \nu(nA) = \nu\big\{ n\bm x:\ \bm x \in A \big\}$).
The measure
$\nu(\cdot)$ is said to be \textbf{multivariate regularly varying} on $\R^d_+$
with basis $\bar{\textbf S} = \{ \bar s_j:\ j \in [k] \}$,
tail indices $\bm \alpha$,
rate functions $\lambda_{\bm j}(\cdot)$,
and limiting measures $\mathbf C_{\bm j}(\cdot)$,
which we denote by
$
\nu \in \notationdef{notation-MRV}{\MRV\Big(\bar{\textbf S},\bm \alpha, (\lambda_{\bm j})_{ \bm j \in \powersetTilde{k}}, (\mathbf C_{\bm j})_{  \bm j \in \powersetTilde{k} }  \Big)},
$
if 
    \begin{align}
            \mathbf C_{\bm j}(A^\circ)
            & \leq 
            \liminf_{n \to \infty}\frac{
                \nu_n(A)
            }{
                \lambda_{\bm j}(n)
            }
            \leq 
            \limsup_{n \to \infty}\frac{
                \nu_n(A)
            }{
                \lambda_{\bm j}(n)
            }
            \leq 
            \mathbf C_{\bm j}(A^-) < \infty
            \label{cond, finite index, def: MRV}
        \end{align}
    holds for any $ \bm j \in \powersetTilde{k}$ and any Borel set $A \subseteq \R^d_+$ that is
        bounded away from $\bar \R^{d}_\leqslant(\bm j,\epsilon; \bar{\textbf S},\bm \alpha)$ under some (and hence all) $\epsilon > 0$ small enough.
Additionally, if for any Borel set $A \subseteq \R^d_+$ that is
        bounded away from $\bar \R^{d}([k],\epsilon; \bar{\textbf S},\bm \alpha)$ under some (and hence all) $\epsilon > 0$ small enough,
we have 
\begin{align}
    \nu_n(A) = \lo( n^{-\gamma})\ \text{ as }n \to \infty,
    \qquad \forall \gamma > 0,
    \label{cond, outside of the full cone, finite index, def: MRV}
\end{align}
then we write 
$
\nu \in \notationdef{notation-MRV}{\MRV^*\Big(\bar{\textbf S},\bm \alpha, (\lambda_{\bm j})_{ \bm j \in \powersetTilde{k}}, (\mathbf C_{\bm j})_{  \bm j \in \powersetTilde{k} }  \Big)}.
$



\end{definition}

In \eqref{def: cone R d i basis S index alpha}, we write
$
\notationdef{notation-cone-R-d-leq-i-basis-S-index-alpha-no-epsilon}{\R^{d}_\leqslant(\bm j; \bar{\textbf S},\bm \alpha)} \delequal {\bar\R^{d}_\leqslant(\bm j,0; \bar{\textbf S},\bm \alpha)}.
$
Besides,
when there is no ambiguity about the basis $\bar{\textbf S}$ and the tail indices $\bm \alpha$,
we adopt simpler notations
$
\notationdef{notation-convex-cone-R-d-i-cluster-size-short}{\R^d(\bm j)} \delequal {\R^d(\bm j;\bar{\textbf S})},
$
$
\notationdef{notation-cone-R-d-leq-i-basis-S-index-alpha-short}{\R^{d}_\leqslant(\bm j)} \delequal {\R^{d}_\leqslant(\bm j; \bar{\textbf S},\bm \alpha)},
$
$
\notationdef{notation-enlarged-convex-cone-bar-R-d-i-epsilon-short}{ \bar \R^d(\bm j, \epsilon) } \delequal \bar \R^d(\bm j, \epsilon;\bar{\textbf S}),
$
and
$
\notationdef{notation-cone-R-d-leq-i-basis-S-index-alpha-short}{\bar\R^{d}_\leqslant(\bm j,\epsilon)} \delequal
{\bar\R^{d}_\leqslant(\bm j,\epsilon; \bar{\textbf S},\bm \alpha)}.
$
Notably,
the conditions \eqref{cond, finite index, def: MRV} and \eqref{cond, outside of the full cone, finite index, def: MRV} in Definition~\ref{def: MRV} 
are equivalent to a characterization of heavy tails through polar coordinates. 
Specifically,
we endow the space $[0,\infty) \times \R^{d}$ with the metric
\begin{align}
    \notationdef{notation-uniform-metric-under-polar-transform}{\bm d_\textbf{U}}\big((r_1,\bm w_1),(r_2,\bm w_2)\big)
    =
    |r_1 - r_2| \vee \norm{\bm w_1 - \bm w_2},
    \qquad
    \forall r_i \geq 0,\ \bm w_i \in \R^d,
    \label{def: metric for polar coordinates}
\end{align}
which is the metric
induced by the uniform norm.
Note that $\big( [0,\infty) \times \mathfrak N^d_+, \bm d_\textbf{U} \big)$ is a complete and separable metric space.
Next,
we define the mapping $\Phi: \R^d_+ \to [0,\infty) \times \mathfrak N^d_+$ by
\begin{align}
    \notationdef{notation-set-mapping-Phi-polar-transform}{\Phi(\bm x)}
    \delequal
    \begin{cases}
         \Big(\norm{\bm x},\frac{\bm x}{\norm{\bm x}}\Big) &\text{ if }\bm x \neq 0,
         \\
         \big( 0, (1,0,0,\cdots,0)\big) & \text{ otherwise.}
    \end{cases}
    \label{def: Phi, polar transform}
\end{align}
Since the value of $\Phi(\bm x)$ at $\bm x = \bm 0$ is of no consequence to our subsequent analysis,
$
\Phi
$
can be interpreted as the polar transform with domain extended to $\bm 0$.
Given a Borel measure $\mu(\cdot)$ on $\R^d_+\setminus\{\bm 0\}$,
we define the measure $\mu \circ \Phi^{-1}$ on $(0,\infty) \times \mathfrak N^d_+$ by 
\begin{align}
    \notationdef{notation-measure-composition-Phi-inverse}{\mu \circ \Phi^{-1}}(A) \delequal \mu\Big( \Phi^{-1}(A) \Big),
    \qquad \forall\text{ Borel set }A\subseteq \R^d_+\setminus\{\bm 0\}.
    \label{def: mu composition Phi inverse measure}
\end{align}
As shown in Lemma~\ref{lemma: M convergence for MRV},
$\MRV$ is equivalent to a characterization of hidden regular variation in terms of the $\mathbb M(\mathbb S\setminus\C)$-convergence of polar coordinates,
i.e., under the choice of $\mathbb S=[0,\infty) \times \mathfrak N^d_+$ with metric $\bm d_\textbf{U}$.

\begin{lemma}\label{lemma: M convergence for MRV}
\linksinthm{lemma: M convergence for MRV}
Let $\mathbb C$ be a closed cone in $\R^d_+$ and
$
\mathbb C_\Phi \delequal
    \big\{
        (r,\bm \theta) \in [0,\infty) \times \mathfrak N^d_+ :\ r\bm \theta \in \mathbb C
    \big\}.
$
Let $\mu \in \mathbb M(\R^d_+ \setminus \mathbb C)$.
Let $X_n$ be a sequence of random vectors taking values in $\R^d_+$, and
$
(R_n,\Theta_n) = \Phi(X_n).
$
Let
$\epsilon_n$ be a sequence of positive real numbers with $\lim_{n\to\infty}\epsilon_n = 0$.
Endow the space  $[0,\infty) \times \mathfrak N^d_+$ with metric $\bm d_\textbf{U}$ in \eqref{def: metric for polar coordinates}.
The following two conditions are equivalent:
\begin{enumerate}[(i)]
    \item 
        as $n \to \infty$,
        \begin{align}
    \epsilon_n^{-1}\P\big( (R_n,\Theta_n) \in\ \cdot\ \big) \to \mu \circ \Phi^{-1}(\cdot)
    \quad
    \text{in $\mathbb M\Big( \big([0,\infty) \times \mathfrak N^d_+ \big)\setminus \mathbb C_\Phi\Big)$;}
    \label{condition: M convergence for polar coordinates, lemma: M convergence for MRV}
\end{align}

    \item 
        for any $\epsilon > 0$ and any Borel set $A \subseteq \R^d_+$ that is bounded away from $\bar{\mathbb C}(\epsilon)$,
\begin{align}
    \mu(A^\circ) 
    \leq 
    \liminf_{n \to \infty}\epsilon_n^{-1}\P(X_n \in A)
    \leq 
    \limsup_{n \to \infty}\epsilon_n^{-1}\P(X_n \in A)
    \leq 
    \mu(A^-) < \infty,
    \label{claim, lemma: M convergence for MRV}
\end{align}
where
\begin{align}
    \bar{\mathbb C}(\epsilon) \delequal
\big\{
    w\bm s:\ w \geq 0,\ 
    \bm s \in \mathfrak N^d_+,\ 
            \inf_{  \bm x\in \C \cap \mathfrak N^d_+   }\norm{ \bm s - \bm x  } \leq \epsilon
\big\}.
\label{def: cone C enlarged by angles}
\end{align}
\end{enumerate}
\end{lemma}

The proof of Lemma~\ref{lemma: M convergence for MRV} is relatively straightforward and is presented in Section~\ref{subsec: proof, M convergence and asymptotic equivalence} of the Appendix.
We add a concluding remark about the key differences between our Definition~\ref{def: MRV} and existing formalisms for multivariate regular variation (MRV).

\begin{remark}[Comparison to Existing Notions of MRV]
\label{remark: comparison to AMRV}
Classical formalisms of MRV (e.g., \cite{Resnick_2004, hult2006regular}) characterize the dominating power-law tail over the entirety of $\R^d$ or $\R^d_+$.
In the language of Definition~\ref{def: MRV}, this generally corresponds to a $\MRV$ condition with a single vector $\bar{\bm s}_{j^*}$ in the basis, where $j^* = \arg\min_{j\in[k]}\alpha(\{j\})$.
In comparison, $\MRV$ enables richer characterizations of tail asymptotics by revealing hidden regular variation beyond the direction $\bar{\bm s}_{j^*}$ of the dominating power-law tail.
It is worth noting that the Adapted-MRV in \cite{das2023aggregatingheavytailedrandomvectors} also aims to characterize hidden regular variation across different directions.
However, the following key differences make our study of $\MRV$ more suitable for the purpose of this paper and more flexible in many cases.
\begin{enumerate}[(i)]
    \item 
        The definition of $\MRV$ allows for arbitrary choices of the $\bar{\bm s}_j$'s beyond the standard basis $(\bm e_j)_{j \in [d]}$ used in \cite{das2023aggregatingheavytailedrandomvectors}.
        While straightforward, such generalizations are required for studying heavy-tailed systems in which the contributions of large jumps are not aligned with mutually orthogonal directions.

    \item 
        
        For each $m = 1,2,\ldots, k$,
        Adapted-MRV in \cite{das2023aggregatingheavytailedrandomvectors}
        investigates the most likely $m$-jump cases:
        that is, given the basis $\{\bar{\bm s}_j\}_{j \in [k]}$
        and among all cones $\R^d(\bm j)$ with $|\bm j| = m$,
        it essentially captures the hidden regular variation over the cone with the smallest tail index $\alpha(\bm j)$.
        This strict hierarchy of $k$ scenarios covers only a subset of the $2^k - 1$ scenarios characterized by $\MRV$.

    \item 
        Adapted-MRV can be interpreted as a stricter version of $\MRV$, in the sense that it requires condition \eqref{cond, finite index, def: MRV} to hold for any $A$ bounded away from $\bar\R_{\leqslant}^d(\bm j, 0)$ (i.e., by forcing $\epsilon = 0$).
        However, as demonstrated in Theorem~\ref{theorem: main result, cluster size} and Remark~\ref{remark: bounded away condition under polar transform, tail asymptotics of cluster sizes}, 
        tail asymptotics of the form \eqref{cond, finite index, def: MRV} 
        would
        not hold for $\bm S_j$ in \eqref{def: fixed point equation for cluster S i} if we set $\epsilon = 0$, thus hindering the use of Adapted-MRV in contexts such as branching processes and Hawkes processes.

        
\end{enumerate}
\end{remark}

\section{Tail Asymptotics of $S_j$}
\label{sec: tail of cluster size}

In this section, we study the tail asymptotics of $\bm S_j$ in
\eqref{def: fixed point equation for cluster S i},
which represents the total progeny of a multi-type branching process in general and the cluster size in multivariate Hawkes processes in particular,
under the presence of power-law heavy tails in the offspring distributions.
That is, while our prime interest is in the tail asymptotics of cluster sizes in multivariate heavy-tailed Hawkes processes, our results apply more generally to multi-type branching processes solving \eqref{def: fixed point equation for cluster S i}; see also the definitions in \eqref{def, proof strategy, collection of B i j copies}--\eqref{def: branching process X n} below.
Section~\ref{subsec: tail of cluster S, statement of results} states the main result.
Section~\ref{subsec: proof methodology, cluster size} gives an overview of the proof strategy.
We defer detailed proofs to Section~\ref{sec: proof, cluster size asymptotics}.

\subsection{Main Result}
\label{subsec: tail of cluster S, statement of results}

We fix some $d \geq 1$ and focus on the $d$-dimensional setting in \eqref{def: fixed point equation for cluster S i}.
We first state the assumptions we will work with.
Let
\begin{align}
    \notationdef{notation-b-i-j-cluster-size}{\bar b_{j\leftarrow i}} \delequal \E B_{j\leftarrow i},
    \label{def: bar b i j, mean of B i j}
\end{align}
which represents the expected number of type-$j$ children of a type-$i$ individual in one generation.
Below, we impose a sub-criticality condition on the $\bar b_{j\leftarrow i}$'s.
Under this assumption, Proposition 1 of \cite{Asmussen_Foss_2018} verifies existence and uniqueness of solutions to Equation \eqref{def: fixed point equation for cluster S i} such that $\E \norm{\bm S_{j}} < \infty$ for all $j \in [d]$.

\begin{assumption}[Sub-Criticality]
\label{assumption: subcriticality}
     The spectral radius of the mean offspring matrix $\notationdef{notation-bar-B-matrix-cluster-size}{\bar{\textbf B}} = (\bar b_{j\leftarrow i})_{j,i \in [d]}$
    \xwa{have $\bar{\bm B}$ transposed?}%
    is strictly less than $1$.
\end{assumption}

Next, we specify the regularly varying heavy tails in the offspring distribution.

\begin{assumption}[Heavy Tails in Progeny]
\label{assumption: heavy tails in B i j}
For any $(i,j) \in [d]^2$,
there exists $\notationdef{notation-alpha-i-j-cluster-size}{\alpha_{j\leftarrow i}} \in (1,\infty)$ such that
\begin{align*}
    \P(B_{j\leftarrow i} > x) \in \RV_{-{\alpha_{j\leftarrow i}} }(x),\qquad \text{as }x \to \infty.
\end{align*}
Furthermore, given $i \in [d]$,
the random vector
$\bm B_{ \bcdot \leftarrow i } =  (B_{j \leftarrow i})_{j \in [d]}$ has independent coordinates across $j \in [d]$.
\end{assumption}

Let
\begin{align}
    \notationdef{notation-vector-bar-s-i-cluster-size}{\bm{\bar s}_i} = (\bar s_{i,1},\ \bar s_{i,2},\ldots,\bar s_{i,d})^\top,
    \qquad 
    \text{ where }\notationdef{notation-bar-s-i-j-cluster-size}{\bar s_{i,j}} \delequal \E S_{i,j}.
    \label{def: bar s i, ray, expectation of cluster size}
\end{align}
That is, $\bar{\bm s}_i = \E\bm S_i$.
As discussed in Remark~\ref{remark: assumptions in theorem, tail asymptotics} below,
the following two assumptions are imposed for convenience of the analysis and can be relaxed at the cost of more involved bookkeeping in Theorem~\ref{theorem: main result, cluster size}.
\begin{assumption}[Full Connectivity]
\label{assumption: regularity condition 1, cluster size, July 2024}
For any $i,j \in [d]$, $\bar s_{i,j} = \E S_{i,j} > 0$.
\end{assumption}

\begin{assumption}[Exclusion of Critical Cases]
\label{assumption: regularity condition 2, cluster size, July 2024}
In Assumption~\ref{assumption: heavy tails in B i j},
$\alpha_{j\leftarrow i} \neq \alpha_{j^\prime \leftarrow i^\prime }$ for any $(i,j),(i^\prime,j^\prime) \in [d]^2$ with $(i,j) \neq (i^\prime,j^\prime)$.
\end{assumption}

To present our main result in terms of $\MRV$ in Definition~\ref{def: MRV},
we specify the basis, tail indices, rate functions, and limiting measures involved.
In particular, we consider the basis
$\bar{\textbf S} = \{ \bar{\bm s}_j:\ j \in [d] \}$ with $\bar{\bm s}_{j}$ defined in \eqref{def: bar s i, ray, expectation of cluster size}.
Next, let
\begin{align}
    \notationdef{notation-alpha-*-j-cluster-size}{\alpha^*(j)} \delequal 
        \min_{l \in [d]}\alpha_{j\leftarrow l},
    \qquad 
    \notationdef{notation-l-*-j-cluster-size}{l^*(j)}  \delequal 
        \arg\min_{l \in [d]}\alpha_{j\leftarrow l}.
    \label{def: cluster size, alpha * l * j}
\end{align}
By Assumption~\ref{assumption: regularity condition 2, cluster size, July 2024},
the argument minimum in the definition of $l^*(j)$ uniquely exists for each $j \in [d]$.
Besides, Assumption~\ref{assumption: heavy tails in B i j} ensures that $\alpha^*(j) > 1\ \forall j \in [d]$.
Recall that $\powersetTilde{d}$
is the collection of all \emph{non-empty} subsets of $[d]$.
Let
\begin{align}
    \notationdef{notation-cost-function-bm-j-cluster-size}{\alpha(\bm j)} \delequal 1 + \sum_{i \in \bm j} \big(\alpha^*(i) - 1\big),
    \qquad 
        \forall \bm j \in \powersetTilde{d}.
    \label{def: cost function, cone, cluster}
\end{align}
As in Section~\ref{subsec: MRV},
we adopt the convention $\alpha(\emptyset) = 0$.
The collection $\bm \alpha = \{ \alpha(\bm j):\ \bm j \subseteq [d] \}$ plays the role of the tail indices for the $\MRV$ description of the $\bm S_j$'s.
As for the rate functions, 
given $\bm j \in \powersetTilde{d}$,
we define
\begin{align}
    \notationdef{notation-lambda-j-n-cluster-size}{\lambda_{\bm j}(n)} \delequal
    n^{-1}\prod_{ i \in \bm j  }n\P(B_{i \leftarrow l^*(i)} > n),
    \qquad
    \forall n \geq 1.
    \label{def: rate function lambda j n, cluster size}
\end{align}
Note that $\lambda_{\bm j}(n) \in \RV_{ -\alpha(\bm j)  }(n)$.
For the limiting measures, we introduce a few definitions.


\begin{definition}[Type]
\label{def: cluster, type}
    ${\bm I} = (I_{k,j})_{k \geq 1,\ j \in [d]}$ is a \textbf{type} if 
\begin{itemize}

    \item $I_{k,j}\in\{0,1\}$ for each $k \geq 1$ and $j \in [d]$;
 
    \item There exists $\notationdef{notation-K-I-for-type-I-cluster-size}{\mathcal K^{\bm I}} \in \mathbb Z_+$ such that $\sum_{j \in [d]}I_{k,j} = 0\ \forall k > \mathcal K^{\bm I}$ and $\sum_{j \in [d]}I_{k,j} \geq 1 \ \forall 1 \leq k \leq \mathcal K^{\bm I}$;

    \item $\sum_{k \geq 1}I_{k,j} \leq 1$ holds for each $j \in [d]$;

      \item For $k = 1$, the set $\{ j \in [d]:\ I_{1,j} = 1 \}$ is either empty or contains exactly one element.
\end{itemize}
We use $\notationdef{notation-mathscr-I-set-of-all-types-cluster-size}{\mathscr I}$ to denote the set containing all types.
For each $\bm I \in \mathscr I$, we say that
\begin{align*}
    \notationdef{notation-active-indices-of-type-I-cluster-size}{\bm j^{\bm I}} \delequal \bigg\{ j \in [d]:\ \sum_{k \geq 1}I_{k,j} = 1  \bigg\}
\end{align*}
is the set of \textbf{active indices} of type $\bm I$,
and $\mathcal K^{\bm I}$ is the \textbf{depth} of type $\bm I$.
Besides, by defining
\begin{align*}
    \notationdef{notation-active-indices-at-depth-k-of-type-I-cluster-size}{\bm j^{\bm I}_k} \delequal 
    \big\{
        j \in [d]:\ I_{k,j} = 1
    \big\},
    \qquad
    \forall k \geq 1,
\end{align*}
we say that $\bm j^{\bm I}_k$ is the set of \textbf{active indices at depth $k$} in type $\bm I$.
\end{definition}
\begin{remark} \label{remark: def of type}
    Note that 
    \begin{enumerate}[(i)]
        \item 
            the only type with $\bm j^{\bm I} = \emptyset$ (and hence $\mathcal K^{\bm I} = 0$) is $I_{k,j} \equiv 0$ for all $k$ and $j$;

        \item 
            if $\mathcal K^{\bm I} \geq 1$, there uniquely exists some $j^{\bm I}_1 \in [d]$ such that $\bm j^{\bm I}_1 = \{j^{\bm I}_1\}$;

        \item 
            for any type $\bm I \in \mathscr I$ with $\bm j^{\bm I} \neq \emptyset$,
            by \eqref{def: rate function lambda j n, cluster size} we have
        \begin{align}
            \lambda_{\bm j^{\bm I}}(n)
            =
            n^{-1}\prod_{ k = 1 }^{ \mathcal K^{\bm I} }\prod_{ j \in \bm j^{\bm I} }
            n\P(B_{j \leftarrow l^*(j)} > n).
            \label{property: rate function for type I, cluster size}
        \end{align}
    \end{enumerate}
\end{remark}



Next,
we adopt the definitions of $\bar\R^d(\bm j,\epsilon)$, $\R^d(\bm j)$ and $\bar\R^d_\leqslant(\bm j,\epsilon)$, $\R^d_\leqslant(\bm j)$ given in Section~\ref{subsec: MRV}
under the basis $\{ \bar{\bm s}_j:\ j \in [d] \}$ and tail indices $(\alpha(\bm j))_{\bm j \in \powersetTilde{d}}$.
Meanwhile,
given $\beta > 0$, define the Borel measure on $(0,\infty)$ by
\begin{align}
    \notationdef{notation-power-law-measure-nu-beta}{\nu_\beta(dw)} \delequal \frac{ \beta dw}{ w^{\beta + 1} }\mathbbm{I}\{w > 0\}.
    \label{def, power law measure nu beta}
\end{align}
Given non-empty index sets $\mathcal I \subseteq [d]$ and $\mathcal J \subseteq [d]$,
we say that 
$
\{\mathcal J(i):\ i \in \mathcal I\}
$
is an \notationdef{notation-assignment-from-J-to-I}{\textbf{assignment of $\mathcal J$ to $\mathcal I$}} if 
\begin{align}
    \mathcal J(i) \subseteq \mathcal J\quad \forall i \in \mathcal I;
    \qquad 
    \bigcup_{i \in \mathcal I}\mathcal{J}(i) = \mathcal J;
    \qquad 
    \mathcal J(i) \cap \mathcal J(i^\prime) = \emptyset\quad \forall i \neq i^\prime.
    \label{def: assignment from mathcal J to mathcal I}
\end{align}
We use $\notationdef{notation-assignment-from-mathcal-J-to-mathcal-I}{\mathbb T_{ \mathcal I \leftarrow \mathcal J }}$ to denote the set of all assignments of $\mathcal{J}$ to $\mathcal I$.
Given non-empty $\mathcal I \subseteq [d]$ and $\mathcal J \subseteq [d]$,
define the mapping
\begin{align}
    \notationdef{notation-function-g-mathcal-I-mathcal-J-cluster-size}{g_{\mathcal I \leftarrow \mathcal J}(\bm w)}
    \delequal 
    \sum_{
        \{\mathcal J(i):\ i \in \mathcal I\} \in \mathbb T_{ \mathcal I \leftarrow \mathcal J }
    }
    \prod_{i \in\mathcal I}
        \prod_{j \in \mathcal J(i)}
        w_i\bar s_{i,l^*(j)},
    \quad
    \forall \bm w = (w_i)_{i \in \mathcal I} \in [0,\infty)^{|\mathcal I|}.
    \label{def: function g mathcal I mathcal J, for measure C i bm I, cluster size}
\end{align}
Given a type $\bm I \in \mathscr I$ with non-empty active index set $\bm j^{\bm I}$,
recall the definitions of $\mathcal K^{\bm I}$ and $\bm j^{\bm I}_k$ in Definition~\ref{def: cluster, type},
and that (when $\bm j^{\bm I} \neq \emptyset$)
there uniquely exists some $j^{\bm I}_1 \in [d]$ such that $\bm j^{\bm I}_1 = \{j^{\bm I}_1\}$.
Let
\begin{align}
    \notationdef{notation-measure-nu-for-type-I-cluster-size}{\nu^{\bm I}(d \bm w)}
    & \delequal
    \bigtimes_{k = 1}^{ \mathcal K^{\bm I} }
    \bigg(
        \bigtimes_{ j \in \bm j^{\bm I}_k }\nu_{\alpha^*(j)  }(d w_{k,j})
    \bigg),
    \label{def, measure nu type I, cluster size}
    \\
    \notationdef{notation-measure-C-type-I-cluster-size}{\mathbf C^{\bm I}(\cdot)}
        & \delequal
    \int \mathbbm{I}
    \Bigg\{
        \sum_{ k = 1 }^{ \mathcal K^{\bm I} }\sum_{ j \in \bm j^{\bm I}_k  }w_{k,j}\bar{\bm s}_j \in \ \cdot\ 
    \Bigg\}
     \Bigg(
    \prod_{ k = 1 }^{ \mathcal K^{\bm I} - 1}
        g_{ \bm j^{\bm I}_k \leftarrow \bm j^{\bm I}_{k+1} }(\bm w_k)
    \Bigg)
    \nu^{\bm I}(d \bm w),
    \label{def: measure C I, cluster}
    \\
    \notationdef{notation-measure-C-i-type-I-cluster-size}{\mathbf C_i^{\bm I}(\cdot)}
        & \delequal
    \int \mathbbm{I}
    \Bigg\{
        \sum_{ k = 1 }^{ \mathcal K^{\bm I} }\sum_{ j \in \bm j^{\bm I}_k  }w_{k,j}\bar{\bm s}_j \in \ \cdot\ 
    \Bigg\}
     \Bigg(
     \bar s_{i,l^*(j^{\bm I}_1)}
    \prod_{ k = 1 }^{ \mathcal K^{\bm I} - 1}
        g_{ \bm j^{\bm I}_k \leftarrow \bm j^{\bm I}_{k+1} }(\bm w_k)
    \Bigg)
    \nu^{\bm I}(d \bm w)
    \label{def: measure C i I, cluster}
    \\ 
    & =  \bar s_{i,l^*(j^{\bm I}_1)}  \mathbf C^{\bm I}(\cdot),
    \nonumber
\end{align}
where we write $\bm w_k = (w_{k,j})_{j \in \bm j^{\bm I}_k}$
and $\bm w = (\bm w_k)_{k \in [\mathcal K^I]}$.
Besides, note that $\mathbf C^{\bm I}(\cdot)$ is supported on the cone $\R^d(\bm j^{\bm I})$.
We are now ready to state the main result of this paper.

\begin{theorem}
\label{theorem: main result, cluster size}
\linksinthm{theorem: main result, cluster size}
Under Assumptions~\ref{assumption: subcriticality}--\ref{assumption: regularity condition 2, cluster size, July 2024},
it holds for any $i \in [d]$ that
\begin{align}
    \P(\bm S_i \in \ \cdot\ )
    \in \MRV^*
    \Bigg(
        (\bar{\bm s}_j)_{j \in [d]},\ 
        \big(\alpha(\bm j)\big)_{ \bm j \subseteq [d] },\ 
        (\lambda_{\bm j})_{\bm j \in \powersetTilde{d}},\ 
        \bigg(  \sum_{\bm I \in \mathscr I:\ \bm j^{\bm I} = \bm j}\mathbf C^{\bm I}_i\bigg)_{\bm j \in \powersetTilde{d}}
    \Bigg).
    \nonumber
\end{align}
That is,
given $i \in [d]$ and  $\bm j \subseteq [d]$ with $\bm j \neq \emptyset$, if a Borel measurable set $A \subseteq \R^d_+$ is bounded away from  $\bar{\R}^d_\leqslant(\bm j,\epsilon)$ under some (and hence all) $\epsilon > 0$ small enough,
then
 \begin{equation}\label{claim, theorem: main result, cluster size}
     \begin{aligned}
         \sum_{\bm I \in \mathscr I:\ \bm j^{\bm I} = \bm j}\mathbf C_i^{\bm I}(A^\circ) 
    & \leq 
        \liminf_{n \to \infty}
        \frac{
        \P(n^{-1}\bm S_i \in A)
        }{
            \lambda_{\bm j}(n)
        }
    \\ 
    & \leq 
        \limsup_{n \to \infty}
        \frac{
        \P(n^{-1}\bm S_i \in A)
        }{
            \lambda_{\bm j}(n)
        }
    \leq 
    \sum_{\bm I \in \mathscr I:\ \bm j^{\bm I} = \bm j}\mathbf C_i^{\bm I}(A^-) < \infty.
     \end{aligned}
 \end{equation}
Here, $\bar{\R}^d_\leqslant(\bm j,\epsilon)$ is defined in \eqref{def: enlarged cone R d index i epsilon},
$\bm j^{\bm I}$ is the set of active indices of type $\bm I$ in Definition~\ref{def: cluster, type},
the rate functions $\lambda_{\bm j}(\cdot)$ are defined in \eqref{def: rate function lambda j n, cluster size},
and the measures $\mathbf C_i^{\bm I}(\cdot)$ are defined in \eqref{def: measure C i I, cluster}.
Furthermore, if the Borel measurable set $A \subseteq \R^d_+$  is bounded away from  $\bar{\R}^d(\{1,2,\ldots,d\},\epsilon)$ for some (and hence all) $\epsilon > 0$ small enough, then
\begin{align}
    \lim_{n \to \infty}n^{\gamma}\cdot\P(n^{-1}\bm S_i \in A) = 0,\qquad\forall \gamma > 0.
    \label{claim, 2, theorem: main result, cluster size}
\end{align}
\end{theorem}

In Section~\ref{subsec: proof methodology, cluster size},
we provide an overview of the proof strategy for Theorem~\ref{theorem: main result, cluster size}.
To conclude this subsection, we state a few remarks about the interpretation of \eqref{claim, theorem: main result, cluster size},
the evaluation of the limiting measures in \eqref{claim, theorem: main result, cluster size},
the application to Hawkes process clusters,
potential relaxations of the assumptions,
and the necessity of the bounded-away from $\bar\R^d_\leqslant(\bm j, \epsilon)$ condition (and hence the $\MRV$ characterization for hidden regular variation) in Theorem~\ref{theorem: main result, cluster size}.

\begin{remark}[Interpreting Asymptotics \eqref{claim, theorem: main result, cluster size}]
\label{remark: interpreation of tail asymptotics of cluster sizes}
Given $A\subseteq\R^d_+$, the asymptotics \eqref{claim, theorem: main result, cluster size} hold for any $\bm j \in \powersetTilde{d}$ such that $A$ is bounded away from $\bar\R^d_\leqslant(\bm j, \epsilon)$ under some $\epsilon > 0$.
However, the index set $\bm j$ that leads to non-trivial bounds in  \eqref{claim, theorem: main result, cluster size} agrees with
\begin{align}
    \bm j(A) \delequal \underset{ \bm j \in \powersetTilde{d}:\ \R^d(\bm j)\cap A \neq \emptyset }{\arg\min}\alpha(\bm j),
    \label{def: rate function j A}
\end{align}
provided that the argument minimum exists uniquely.
Indeed, for any $\bm I \in \mathscr I$ with $\bm j^{\bm I} = \bm j$, note that the measures $\mathbf C^{\bm I}(\cdot)$ are supported on the cone $\R^d(\bm j)$.
As a result, in \eqref{claim, theorem: main result, cluster size} we need to have at least
$
\R^d(\bm j) \cap A \neq \emptyset
$
for the lower bounds to be non-trivial.
In other words, 
\eqref{claim, theorem: main result, cluster size} shows that
given a Borel set $A \subseteq \R^d_+$,
if $A \cap \R^d([d]) \neq \emptyset$, the argument minimum $\bm j(A)$ is unique, and 
 $A$ is bounded away from $\bar{\R}^d_\leqslant(\bm j(A),\epsilon)$ under some $\epsilon > 0$, then
\begin{align*}
    \mathbf C^{\bm j(A)}_i(A^\circ)
    \leq 
    \liminf_{n \to \infty}
    \frac{
        \P(n^{-1}\bm S_i \in A)
    }{ \lambda_{ \bm j(A) }(n) }
    \leq 
    \limsup_{n \to \infty}
     \frac{
        \P(n^{-1}\bm S_i \in A)
    }{ \lambda_{ \bm j(A)}(n)}
    \leq 
    \mathbf C^{\bm j(A)}_i(A^-),
\end{align*}
with limiting measure  $\mathbf C^{\bm j}_i =  \sum_{\bm I \in \mathscr I:\ \bm j^{\bm I} = \bm j}\mathbf C_i^{\bm I}$.
From this perspective, given the rare event set $A$,
the solution $\bm j(A)$ of the discrete optimization problem in \eqref{def: rate function j A} determines the most likely configuration of big jumps triggering the event (i.e., through big jumps aligned with $\bar s_j$ for each $j \in \bm j(A)$) and the degree of hidden regular variation (i.e., with power-law rate $\lambda_{\bm j(A)}(n) \in \RV_{ -\alpha(\bm j(A))  }(n)$).
In particular,
$\bm j(A)$ plays the role of the rate functions in the classical large deviation principle (LDP) framework, dictating the power-law rate of decay for the rare-event probability $\P(n^{-1}\bm S_i \in A)$,
and the limiting measure
$
 \sum_{\bm I \in \mathscr I:\ \bm j^{\bm I} = \bm j}\mathbf C_i^{\bm I}(\cdot)
$
allows the characterization of exact asymptotics beyond the log asymptotics typically available in classical LDPs.
We note that these results also lay the foundation for sample path large deviations of heavy-tailed Hawkes processes in our companion paper \cite{wangHawkesLDWorkingPaper}.
\end{remark}

\begin{remark}[Evaluation of Limiting Measures]
Continuing the discussion in Remark~\ref{remark: interpreation of tail asymptotics of cluster sizes},
we note that $\mathbf C^{\bm I}(\cdot)$ can be readily computed by Monte Carlo simulation.
In particular, given some type $\bm I \in \mathscr I$ with $\bm j^{\bm I}= \bm j$ and some $A \subseteq \R^d_+$ that is bounded away from $\bar\R^d_\leqslant(\bm j,\epsilon)$ under some $\epsilon > 0$,
Lemma~\ref{lemma: choice of bar epsilon bar delta, cluster size} shows that: (i) $\mathbf C^{\bm I}(A)<\infty$, and 
(ii) there exists $\bar\delta > 0$ such that, in \eqref{def: measure C I, cluster}, 
$
 \sum_{ k = 1 }^{ \mathcal K^{\bm I} }\sum_{ j \in \bm j^{\bm I}_k  }w_{k,j}\bar{\bm s}_j \in A
 \ \Longrightarrow\ 
 w_{k,j} > \bar\delta\ \forall k,j.
$
Therefore, 
given $\delta > 0$ small enough,
$\mathbf C^{\bm I}(A)$ can be evaluated by simulating, for each $k \in [\mathcal K^{\bm I}]$ and $j \in \bm j^{\bm I}_k$, a Pareto random variable $W^{ (\delta) }_{k,j}$ with lower bound $\delta$ and power-law index $\alpha^*(j)$, and then estimating
\begin{align*}
    \bigg( \prod_{ k = 1 }^{ \mathcal K^{\bm I} }\prod_{ j \in \bm j^{\bm I}_k  }
    \delta^{ -\alpha^*(j)  }
    \bigg) \cdot 
    \E\Bigg[
    \mathbbm{I}
    \Bigg\{
        \sum_{ k = 1 }^{ \mathcal K^{\bm I} }\sum_{ j \in \bm j^{\bm I}_k  }W^{ (\delta) }_{k,j}\bar{\bm s}_j \in A
    \Bigg\}
    \cdot \Bigg(
    \prod_{ k = 1 }^{ \mathcal K^{\bm I} - 1}
        g_{ \bm j^{\bm I}_k \leftarrow \bm j^{\bm I}_{k+1} }\Big( \big(W^{ (\delta) }_{k,j}\big)_{ j \in \bm j^{\bm I}_k }\Big)
    \Bigg)
    \Bigg].
\end{align*}
Here, note that it is easy to compute $\bar{\bm s}_j = \E \bm S_j$ (and hence the mapping $g_{\mathcal I \leftarrow \mathcal J}$)
as long as the mean offspring matrix 
$
{\bar{\textbf B}} = (\E B_{j\leftarrow i})_{j \in [d],i \in [d]}
$
is available; see \cite{Asmussen_Foss_2018}.
\end{remark}

\begin{remark}[Hawkes Process Clusters]
\label{remark, applied to hawkes process clusters}
Theorem~\ref{theorem: main result, cluster size} establishes tail asymptotics of multi-type branching processes solving \eqref{def: fixed point equation for cluster S i} applying in particular to
cluster sizes in a multivariate Hawkes process,
i.e., a point process
$\bm N(t) = \big(N_1(t),\ldots,N_d(t)\big)^\top$ with 
initial value $\bm N(0) = \bm 0$ and 
conditional intensity
$
h_{i}(t) \delequal  c_{i} + \sum_{j \in [d]}\int_0^t \tilde B_{i \leftarrow j}(s)f_{i \leftarrow j}(s)dN_j(s)
$
for each dimension $i \in [d]$.
Here, the positive constants $c_i$ are the arrival rates of immigrants along each dimension,
the deterministic functions $f_{i \leftarrow j}(\cdot)$ are such that $\norm{f_{i \leftarrow j}}_1 \delequal \int_0^\infty f_{i \leftarrow j}(t)dt < \infty$,
and the excitation rates $\big(\tilde B_{i \leftarrow j}(s)\big)_{s > 0}$ are
i.i.d.\ copies of $\tilde B_{i \leftarrow j}$.
The size of a cluster induced by a type-$j$ immigrant admits the law of $\bm S_j$ solving \eqref{def: fixed point equation for cluster S i} under the offspring distribution
$
\P(B_{i \leftarrow j} > x) = \int_0^\infty \P\big(  \text{Poisson}( w \norm{f_{i\leftarrow j}}_1  ) > x \big)    \P(\tilde B_{i \leftarrow j} \in dw),
$
implying that $B_{i\leftarrow j}$ and $\tilde B_{i \leftarrow j}$ share the same regularly varying index in this context.
Therefore, in heavy-tailed Hawkes processes, the cluster size vectors $\bm S_i$ exhibit the $\MRV^*$ tails characterized in Theorem~\ref{theorem: main result, cluster size} (i.e., under the tail indices, rate functions, and limiting measures defined in \eqref{def: cost function, cone, cluster}, \eqref{def: rate function lambda j n, cluster size}, and \eqref{def, measure nu type I, cluster size}--\eqref{def: measure C i I, cluster}, respectively), with $B_{i \leftarrow j}$ as specified above and $\alpha_{i \leftarrow j}$ as the regular variation index of $\tilde B_{i \leftarrow j}$.
\end{remark}

\begin{remark}[Relaxing Assumptions]
\label{remark: assumptions in theorem, tail asymptotics}
Although not pursued in this paper, Assumptions~\ref{assumption: regularity condition 1, cluster size, July 2024} and \ref{assumption: regularity condition 2, cluster size, July 2024} could be relaxed, albeit at the cost of more involved bookkeeping in Theorem~\ref{theorem: main result, cluster size}:
\begin{itemize}
    \item 
        The full-connectivity condition in Assumption~\ref{assumption: regularity condition 1, cluster size, July 2024}
        can be relaxed by adapting the notion of a Hawkes graph in \cite{karim2021exact}.
        The key idea is to modify $\alpha^*(j)$ and $l^*(j)$ in \eqref{def: cluster size, alpha * l * j} and only consider the subset of $[d]$ corresponding to the ``essential dimensions'' related to $j$:
        for instance, in \eqref{def: cluster size, alpha * l * j} one can safely disregard any $l \in [d]$ with $\E S_{l,j} = 0$, as an ancestor along the $l^\text{th}$ dimension will almost surely have no offspring along the $j^\text{th}$ dimension.

    \item 
        Suppose that Assumption~\ref{assumption: regularity condition 2, cluster size, July 2024} is dropped and there are some $j \in [d]$ and $i,i^\prime \in [d]$ with $i \neq i^\prime$ such that 
         $\alpha_{j \leftarrow i^\prime} = \alpha_{j \leftarrow i}$.
        That is, by only comparing the tail indices, it is unclear whether $\P(B_{ j \leftarrow i^\prime} > x)$ or $\P(B_{j \leftarrow i} > x)$ has a heavier tail, thus preventing us to determine the most likely cause for a large jump along the direction $\bar{\bm s}_j$.
        In such cases, one can either impose extra assumptions about the tail CDFs of the $B_{j \leftarrow i}$'s to break the ties, or work with the non-uniqueness of the argument minimum in \eqref{def: cluster size, alpha * l * j}.
        The latter could result in rougher asymptotics of a more involved form, due to the need to keep track of all possible scenarios in the arguments minimum;
        see for instance the comparison between Theorem 3.4 and Theorem 3.5 in \cite{10.1214/18-AOP1319}.        
\end{itemize}
\end{remark}

\begin{remark}[Bounded-Away from $\bar\R^d_\leqslant(\bm j,\epsilon)$ Condition]
\label{remark: bounded away condition under polar transform, tail asymptotics of cluster sizes}
The characterization in Theorem~\ref{theorem: main result, cluster size} is, in some sense, the tightest one can hope for, as asymptotics of the form \eqref{claim, theorem: main result, cluster size} do not hold under the weaker condition that $A$ is only bounded away from $\R^d_\leqslant(\bm j)$ (i.e., by forcing $\epsilon = 0$ in the statement of Theorem~\ref{theorem: main result, cluster size}).
In Section~\ref{sec: counterexamples}  of the Appendix,
we show that \eqref{claim, theorem: main result, cluster size} could fail for choices of the set $A$ such as
$A = \{ \bm x \in \R^d_+:\ \inf_{ \bm y \in \R^d(\bm j)  }\norm{ \bm x - \bm y } > c  \}$,
which is bounded away from $\R^d(\bm j)$ but not from any $\bar \R^d(\bm j, \epsilon)$ with $\epsilon > 0$.
The gist of the counterexample in Section~\ref{sec: counterexamples} is that the underlying structure of branching processes in $\bm S_j$ leads to a multiplicative effect,
and big jumps in previous generations may amplify a CLT-scale perturbation in subsequent generations to the large-deviation scale.
As shown in Lemma~\ref{lemma: M convergence for MRV}, $\MRV$ is a characterization of hidden regular variation through $\mathbb M(\mathbb S \setminus \mathbb C)$-convergence of polar coordinates.
On the other hand, forcing $\epsilon = 0$ in the bounded-away condition in Theorem~\ref{theorem: main result, cluster size}
is equivalent to considering a Cartesian-coordinates-based characterization (see, e.g., Adapted-MRV in \cite{das2023aggregatingheavytailedrandomvectors}).
Therefore, the counterexample confirms that $\MRV$ provides a more adequate framework for characterizing hidden regular variation in the contexts such as branching processes and Hawkes processes.
\end{remark}

\subsection{Proof Strategy}
\label{subsec: proof methodology, cluster size}



As has been noted in the Introduction, our proof of Theorem~\ref{theorem: main result, cluster size} relies on 
a recursive application of Equation \eqref{proof, equality, from S i truncated to S i}.
To make sense of the terms involved in \eqref{proof, equality, from S i truncated to S i}, we consider a natural coupling between $\bm S_j$ in \eqref{def: fixed point equation for cluster S i}
and $\bm S_j^\leqslant(M)$ in \eqref{def: fixed point equation for pruned cluster S i leq M}.
More precisely, consider a probability space supporting a collection of independent random vectors 
\begin{align}
    \Big\{ \notationdef{notation-B-j-n-m-cluster-size}{\bm B^{(t,m)}_{ \bcdot \leftarrow j }}:\ j \in [d],\ t \geq 1,\ m \geq 1 \Big\},
    \label{def, proof strategy, collection of B i j copies}
\end{align}
where each 
$
\bm B^{(t,m)}_{ \bcdot \leftarrow j }
=
(
    B^{(t,m)}_{ 1 \leftarrow j }, B^{(t,m)}_{ 2 \leftarrow j }, \ldots ,B^{(t,m)}_{ d \leftarrow j }
)^\top
$
is an i.i.d.\ copy of the random vector
$
\bm B_{ \bcdot \leftarrow j }
=
(
    B_{ 1 \leftarrow j }, B_{ 2 \leftarrow j }, \ldots ,B_{ d \leftarrow j }
)^\top.
$
Define a multivariate branching process $\bm X_j(t) = (X_{j,i}(t))_{i \in [d]}$ by 
\begin{align}
    \notationdef{notation-X-n-i-Galton-Watson-process}{\bm X_j(t)} \delequal \sum_{i \in [d]}\sum_{m = 1}^{ X_{j,i}(t-1) }{\bm B}^{(t,m)}_{\bcdot \leftarrow i},
    \qquad \forall t \geq 1,
    \label{def: branching process X n}
\end{align}
under initial value $\bm X_j(0) = \bm e_j$ (i.e., the unit vector $(0,0,1,0,\ldots,0)$ with the $j^\text{th}$ coordinate being 1).
The sub-criticality condition in Assumption~\ref{assumption: subcriticality} ensures that 
the summation $\sum_{t \geq 0} \bm X_j(t)$ converges almost surely and
$\sum_{t \geq 0} \bm X_j(t) \distequal \bm S_j$,
thus solving the fixed-point equation in \eqref{def: fixed point equation for cluster S i}.
Likewise, let
\begin{align}
    B^{\leqslant, (t,m)}_{i \leftarrow j}(M) \delequal  B^{(t,m)}_{i \leftarrow j}\mathbbm{I}\{ B^{(t,m)}_{i \leftarrow j} \leq M  \},
    \quad
    \notationdef{notation-B-leq-M-n-m-j-cluster-size}{{\bm B}^{\leqslant, (t,m)}_{\bcdot \leftarrow j}(M)}
    \delequal
    \big(B^{\leqslant, (t,m)}_{i \leftarrow j}(M)\big)_{i \in [d]}.
    \label{def: truncated offspring count,B leq M j k}
\end{align}
be the truncated version of the $\bm B^{(t.m)}_{\bcdot \leftarrow j}$'s under threshold $M$,
and define the multivariate branching process $\bm X_j^\leqslant(t;M) = \big( X_{j,i}^\leqslant(t;M) \big)_{i \in [d]}$
by 
\begin{align}
    \notationdef{notation-X-leqM-n-i-cluster-size}{\bm X^{\leqslant}_j(t;M)} = \sum_{i \in [d]} \sum_{m = 1}^{ X^{\leqslant}_{j,i}(t - 1;M) }{\bm B}^{\leqslant, (t,m)}_{ \bcdot \leftarrow i}(M),
    \qquad \forall  t \geq 1,
    \label{def: branching process X n leq M}
\end{align}
under initial value $\bm X_j^\leqslant(0;M) = \bm e_j$.
Note that
\begin{align}
    \sum_{t \geq 0}\bm X^{\leqslant}_j(t;M) \distequal
    \notationdef{notation-pruned-cluster-S-leqM-i}{\bm S^{\leqslant}_j(M)},
    \qquad
    \forall M \in (0,\infty),\ j \in [d],
    \label{def: pruned tree S i leq M}
\end{align}
thus solving Equation \eqref{def: fixed point equation for pruned cluster S i leq M}.

Furthermore, the coupling between $\bm X_j(t)$ in \eqref{def: branching process X n}  and $\bm X_j^\leqslant(t;M)$ in \eqref{def: branching process X n leq M}
allows us to count the nodes pruned under $M$ due to large $B^{(t,m)}_{l \leftarrow i}$'s (i.e., big jumps in the branching processes).
Specifically, by defining
\begin{align}
        \notationdef{notation-W-i-M-l-j-cluster-size}{W^{>}_{j;i \leftarrow l}(M)}
    & \delequal 
        \sum_{t \geq 1}\sum_{m = 1}^{ X^{\leqslant}_{j,l}(t-1;M) }
            B^{(t,m)}_{i \leftarrow l}\mathbbm{I}\big\{
                B^{(t,m)}_{i \leftarrow l} > M
            \big\},
    \label{def: W i M j, pruned cluster, 1, cluster size}
    \\
        \notationdef{notation-W-i-M-cdot-j-cluster-size}{W^{>}_{j;i}(M)}
    & \delequal
        \sum_{l \in [d]}W^{>}_{j;i\leftarrow l}(M)
    =
    \sum_{t \geq 1}
    \sum_{l \in [d]}\sum_{m = 1}^{ X^{\leqslant}_{j,l}(t-1;M) }
            B^{(t,m)}_{i \leftarrow l}\mathbbm{I}\big\{
                B^{(t,m)}_{i \leftarrow l} > M
            \big\},
    \label{def: W i M j, pruned cluster, 2, cluster size}
\end{align}
we can use ${W^{>}_{j;i}(M)}$ to count descendants along the $i^\text{th}$ dimension pruned in the branching process $\bm{X}^{\leqslant}_j(t;M)$ (due to their parent node giving birth to more than $M$ children along the $i^\text{th}$ dimension in one generation),
and use ${W^{>}_{j;i \leftarrow l}(M)}$ to specifically count pruned nodes along the $i^\text{th}$ dimension with parent along the $l^\text{th}$ dimension.
Similarly, by defining
\begin{align}
        \notationdef{notation-N-i-M-l-j}{N^{>}_{j;i \leftarrow l}(M)}
    & \delequal 
        \sum_{t \geq 1}\sum_{m = 1}^{ X^{\leqslant}_{j,l}(t-1;M) }
            \mathbbm{I}\big\{
                B^{(t,m)}_{i \leftarrow l} > M
            \big\},
            \label{def, N i | M l j, cluster size}
    \\ 
        \notationdef{notation-N-i-M-cdot-j}{N^{>}_{j;i}(M)}
    & \delequal
   \sum_{l \in [d]}N^{>}_{j;i\leftarrow l}(M)
    =
    \sum_{t \geq 1}
    \sum_{l \in [d]}\sum_{m = 1}^{ X^{\leqslant}_{j,l}(t-1;M) }
            \mathbbm{I}\big\{
                B^{(t,m)}_{i \leftarrow l} > M
            \big\},
    \label{def: cluster size, N i | M cdot j}
\end{align}
we can employ $N^{>}_{j;i}(M)$ to count the times pruning occurs in $\bm X^\leqslant_j(t;M)$ for nodes along the $i^\text{th}$ dimension,
and 
employ
$N^{>}_{j;i \leftarrow l}(M)$ to specifically count the times of pruning for nodes along the $i^\text{th}$ dimension with parents along the $l^\text{th}$ dimension.

In summary, the probability space specified above allows us to consider a coupling between $\bm S_j$ and $\bm S^\leqslant_j(M)$,
where
$
\bm S_j = \sum_{t \geq 0}\bm X_j(t)
$
and
$
\bm S^\leqslant_j(M) = \sum_{t \geq 0}\bm X^\leqslant_j(t;M).
$
This gives a clear construction for the $\bm S^\leqslant_j(M)$ and $W_{j;i}^>(M)$'s in \eqref{proof, equality, from S i truncated to S i} on the same probability space,
where
 ${N^{>}_{j;i}(M)}$ counts the big jumps along the $i^\text{th}$ dimension that are removed from the underlying branching process $\bm X_j^\leqslant(t;M)$,
and ${W^{>}_{j;i}(M)}$ represents the accumulated size of these big jumps.
Furthermore, 
the equality \eqref{proof, equality, from S i truncated to S i} indicates a two-step procedure that generates $\bm S_j$.
At the first step, whenever a node plans to give birth to more than $M$ children along the same dimension, we skip the birth of these children as if their births are put ``on hold'';
in doing so, we obtain a branching process under the truncated offspring distribution \eqref{def: truncated offspring count,B leq M j k} yielding $\bm S^\leqslant_j(M)$.
At the second step, we resume the births of these previously on-hold nodes;
more precisely, for each dimension $i \in [d]$ there are $W^{>}_{j;i}(M)$ nodes whose birth were skipped in step one;
by generating these nodes and the sub-trees induced by them (corresponding to the i.i.d.\ copies $\bm S^{(m)}_i$ on the RHS of \eqref{proof, equality, from S i truncated to S i}), we recover the law of $\bm S_j$.\footnote{For the sake of completeness,
we collect the rigorous proof of \eqref{proof, equality, from S i truncated to S i}
in Section~\ref{sec: appendix, lemmas} of the Appendix.}

We then apply \eqref{proof, equality, from S i truncated to S i} recursively.
For instance,  two iterations of \eqref{proof, equality, from S i truncated to S i} lead to
\begin{align}
    \bm S_j \distequal 
    \bm S^\leqslant_j(M)
    +
    \sum_{i \in [d]}\sum_{ m = 1 }^{ W^{>}_{j;i}(M)  }
    \Bigg[
        \bm S_i^{\leqslant,(m)}(M)
        +
        \sum_{l \in [d]}\sum_{ q = 1 }^{ W^{>,(m)}_{i;l}(M)  }\bm S^{(m,q)}_l
    \Bigg],
    \label{proof strategy, two iteration version of the pruned cluster decomp}
\end{align}
where 
$\bm S_l^{(m,q)}$'s are i.i.d.\ copies of $\bm S_l$,
and
$\big(\bm S^{\leqslant,(m)}_i(M),  W^{>,(m)}_{i;1}(M), \ldots,  W^{>,(m)}_{i;d}(M) \big)$'s
are i.i.d.\ copies of 
$\big(\bm S^{\leqslant}_i(M),  W^{>}_{i;1}(M), \ldots,  W^{>}_{i;d}(M) \big)$.
The proof of the main result in Section~\ref{subsec: proof, theorem: main result, cluster size} is built upon a suitable recursive application of the equality \eqref{proof, equality, from S i truncated to S i}
that further extends \eqref{proof strategy, two iteration version of the pruned cluster decomp} and decomposes $\bm S_j$ as a (random) sum of i.i.d.\ copies of the $\bm S_i^\leqslant(M)$'s.
Roughly speaking, given $n \geq 1$ and $\delta > 0$, under the truncation threshold $M = n\delta$ we get
\begin{align}
    \bm S_j \distequal 
    \bm S^{\leqslant}_j(n\delta) + 
    \sum_{ k \geq 1 }\sum_{i \in [d]}\sum_{ m = 1 }^{  \tau^{n|\delta}_{j;i}(k) } \bm S^{ \leqslant, (k,m) }_i(n\delta),
    \label{proof strategy, decomposition of S j}
\end{align}
where the
$\bm S_i^{ \leqslant, (k,m) }(n\delta)$'s are i.i.d.\ copies of $\bm S_i^\leqslant(n\delta)$,
and $\tau^{n|\delta}_{j;i}(k)$ denotes the number of pruned nodes along the $i^\text{th}$ dimension during the $k^\text{th}$ iteration in the recursive application of \eqref{proof, equality, from S i truncated to S i}:
for instance, $\tau^{n|\delta}_{j;i}(1)$ agrees with $W^{>}_{j;i}(n\delta)$ defined in \eqref{def: W i M j, pruned cluster, 2, cluster size}.
We provide the detailed construction of the $\tau^{n|\delta}_{j;i}(k)$'s in Section~\ref{subsec: proof, theorem: main result, cluster size}, and note here that:
(i) the decomposition used in our analysis is slightly more involved than \eqref{proof strategy, decomposition of S j} and specifies a different truncation threshold for each iteration (instead of fixing $M = n\delta$);
and (ii) the procedure almost surely terminates after finitely many steps (i.e., $\tau^{n|\delta}_{j;i}(k) = 0$ eventually for all $k$ large enough) due to $\norm{\bm S_j} < \infty$ almost surely.

Now, the proof of Theorem~\ref{theorem: main result, cluster size} reduces to studying in detail the concentration inequalities of  $\bm S^{\leqslant}_i(M)$ and the law of  $\big(\tau^{n|\delta}_{j;i}(k)\big)_{k \geq 1, i \in [d]}$.
First,
Proposition~\ref{proposition: asymptotic equivalence, tail asymptotics for cluster size} shows that, for the asymptotic analysis of $\P(n^{-1}\bm S_i \in A)$ in Theorem~\ref{theorem: main result, cluster size},
it is (asymptotically) equivalent to study
\begin{align}
    \hat{\bm S}^{n|\delta}_j
    \delequal 
    \sum_{ k \geq 1 }\sum_{i \in [d]} n^{-1} \tau^{n|\delta}_{j;i}(k) \cdot \bar{\bm s}_j.
    \label{def: hat S n delta i, cluster size, proof sketch}
\end{align}
Specifically,
Lemmas~\ref{lemma: tail bound, pruned cluster size S i leq n delta} and \ref{lemma: concentration ineq for pruned cluster S i}, which support the proof of Proposition~\ref{proposition: asymptotic equivalence, tail asymptotics for cluster size}, establish
tail asymptotics and concentration inequalities for $\bm S^\leqslant_i(M)$:
under the proper choice of $\delta$, the running average for i.i.d.\ copies of $\bm S^\leqslant_i(n\delta)$ 
concentrates around $\bar{\bm s}_i = \E \bm S_i$ at
arbitrarily fast power-law rates, justifying the approximation of $n^{-1}\bm S_j$ by $\hat{\bm S}^{n|\delta}_j$ in light of the decomposition \eqref{proof strategy, decomposition of S j}.
Then, the problem amounts to analyzing the joint asymptotics of $\big(n^{-1}\tau^{n|\delta}_{j;i}(k)\big)_{ i \in [d],\ k \geq 1 }$.
This is the content of Proposition~\ref{proposition, M convergence for hat S, tail asymptotics for cluster size}.
In particular, note that
$
{\bm I^{n|\delta}_j} \delequal \big(I^{n|\delta}_{j;i}(k)\big)_{k \geq 1,\ i \in [d]}
$
with ${I^{n|\delta}_{i;j}(k)} \delequal \mathbbm{I}\big\{ \tau^{n|\delta}_{i;j}(k) > 0 \big\}$
indicates the existence of big jumps across different dimensions $i$ and depths $k$ in the decomposition \eqref{proof strategy, decomposition of S j} for $\bm S_j$.
For $\hat{\bm S}^{n|\delta}_j$ defined in \eqref{def: hat S n delta i, cluster size, proof sketch} to fall into a given set $A \subseteq \R^d_+$,
${\bm I^{n|\delta}_j}$ must take specific values.
That is,
for the rare event
$
\{ n^{-1}\bm S_j \in A  \}
$
to occur,
the big jumps in the branching process will almost always exhibit specific \emph{types}  of  spatio-temporal structures (as in Definitions~\ref{def: cluster, type} and \ref{def: cluster, generalized type}).
Proposition~\ref{proposition, M convergence for hat S, tail asymptotics for cluster size} then characterizes the asymptotic law of $\big(n^{-1}\tau^{n|\delta}_{j;i}(k)\big)_{ i \in [d],\ k \geq 1 }$ when conditioned on the \emph{type} of $\bm S_j$ (i.e., the value of ${\bm I^{n|\delta}_j}$).
The proof of Proposition~\ref{proposition, M convergence for hat S, tail asymptotics for cluster size} relies on the asymptotics of ${W^{>}_{j;i}(M)}$ and ${N^{>}_{j;i}(M)}$ from Lemma~\ref{lemma: cluster size, asymptotics, N i | n delta, cdot j, crude estimate}, which reduce to analyzing
the sums of regularly varying variables truncated from below, conditioned on the sum being large.
We provide the detailed proofs in Section~\ref{sec: proof, cluster size asymptotics},
and include the theorem tree in Section~\ref{sec: appendix, theorem tree} of the Appendix to aid readability.

\section{Proofs of the Main Result, Two Key Propositions and a Lemma}
\label{sec: proof, cluster size asymptotics}

\subsection{Proof of the Main Result}
\label{subsec: proof, theorem: main result, cluster size}

We start by highlighting several properties of $\bm S_j^\leqslant(M)$ in \eqref{def: fixed point equation for pruned cluster S i leq M} and the quantities ${W^{>}_{j;i \leftarrow l}(M)}$, ${W^{>}_{j;i}(M)}$, ${N^{>}_{j;i \leftarrow l}(M)}$, ${N^{>}_{j;i}(M)}$
defined in \eqref{def: W i M j, pruned cluster, 1, cluster size}--\eqref{def: cluster size, N i | M cdot j}.
First, by definitions,
\begin{equation}\label{property, W and N i M l j when positive, cluster size}
    \begin{aligned}
        W^{>}_{i;j\leftarrow l}(M) > 0
        \quad
        \Longleftrightarrow
        \quad
        W^{>}_{i;j\leftarrow l}(M) > M
        \quad
        \Longleftrightarrow
        \quad
        N^{>}_{i;j\leftarrow l}(M) \geq 1,
        \\ 
        W^{>}_{i;j}(M) > 0
        \quad
        \Longleftrightarrow
        \quad
        W^{>}_{i;j}(M) > M
        \quad
        \Longleftrightarrow
        \quad
        N^{>}_{i;j}(M) \geq 1.
    \end{aligned}
\end{equation}
Next, we consider a useful stochastic comparison relation between branching processes. 
Here, for any random vectors $\bm V$ and ${\bm V}^\prime$ in $\R^d$,
we use $\notationdef{notation-stochastic-dominance}{\bm V\stleq {\bm V}^\prime}$ to denote stochastic comparison between $\bm V$ and $\bm V^\prime$,
in the sense that
$
\P(\bm V > \bm x) \leq \P({\bm V}^\prime > \bm x)
$
holds
for any real vector $\bm x \in \R^d$.
Let $\textbf W = (W_{i,j})_{i,j \in [d]}$ and $\textbf V = (V_{i,j})_{i,j \in [d]}$ be two random matrices in $(\Z_+)^{d \times d}$, and $\bm W_j$, $\bm V_j$ be the $j$-th row vector of $\textbf W$ and $\textbf V$.
Let the $\textbf W^{(t,m)}$'s be i.i.d.\ copies of $\textbf W$, and we adopt similar notations for $\textbf V$.
Consider $d$-dimensional branching processes $(\bm X^{\bf W}(t))_{t \geq 0}$ and $(\bm X^{\bf V}(t))_{t \geq 0}$ defined by
\begin{align*}
     \bm X^{\bf W}(t) = \sum_{j \in [d]} \sum_{m = 1}^{ X^{\bf W}_{j}(t-1) }{\bm W}^{(t,m)}_{j},
     \qquad
     \bm X^{\bf V}(t) = \sum_{j \in [d]} \sum_{m = 1}^{ X^{\bf V}_{j}(t-1) }{\bm V}^{(t,m)}_{j},
     \qquad \forall t \geq 1,
\end{align*}
initialized by $\bm X^{\bf W}(0) = \bm X^{\bf V}(0) = \bm X^\prime$
using some random vector $\bm X^\prime$ taking values in $\mathbb Z_+^d$.
Under the condition $\bf W \stleq \bf V$,
one can see that
\begin{align}
    \bm X^{\bf W}(t) \stleq \bm X^{\bf V}(t),\qquad \forall t \geq 0.
    \label{property: stochastic comparison, general, glaton watson trees}
\end{align}
In fact, using
the coupling argument in Section~\ref{subsec: proof methodology, cluster size},
one can construct a probability space that supports both $\big(\bm X^{\bf W}(t)\big)_{t \geq 0}$ and 
$\big(\bm X^{\bf V}(t)\big)_{t \geq 0}$, with $\bm X^{\bf W}(t) \leq \bm X^{\bf V}(t)$ (almost surely) for each $t \geq 0$.
Similarly, by considering the coupling of
$
\bm S_j = \sum_{t \geq 0}\bm X_j(t)
$
and
$
\bm S^\leqslant_j(M) = \sum_{t \geq 0}\bm X^\leqslant_j(t;M)
$
(see definitions in \eqref{def: branching process X n} and \eqref{def: branching process X n leq M}),
we have
\begin{align}
    \bm S^{\leqslant}_i(M) \leq \bm S^\leqslant_i(M^\prime) \leq \bm S_i,
    \qquad \forall i \in [d],\ 0 < M < M^\prime < \infty.
    \label{property: stochastic comparison, S and pruned S}
\end{align}

As described in Section~\ref{subsec: proof methodology, cluster size},
our proof of Theorem~\ref{theorem: main result, cluster size} hinges on a recursive application of the equality~\eqref{proof, equality, from S i truncated to S i}
that decomposes $\bm S_i$ into a nested collection of the pruned $\bm S_j^\leqslant(M)$'s.
Now, we describe this recursive procedure in full detail.
Consider a probability space supporting (for each $j \in [d]$)
\begin{align}
    \bm B^{(k,m,t,q)}_{ \bcdot \leftarrow j } \stackrel{i.i.d.}{\sim} \bm B_{ \bcdot \leftarrow j },
    \qquad
    \forall k,m,t,q \geq 1,
    \label{def: probability space with i.i.d. copies B k m t q j, cluster size}
\end{align}
with 
$
\bm B_{ \bcdot \leftarrow j }
=
(
    B_{ 1 \leftarrow j }, B_{ 2 \leftarrow j }, \ldots ,B_{ d \leftarrow j }
)^\top
$
being the offspring distribution
in \eqref{def: fixed point equation for cluster S i}.
Let
\begin{align}
\bm S^{(k,m)}_j \delequal \sum_{t \geq 0}\bm X_j^{(k,m)}(t),\quad \text{where }
    {\bm X_j^{(k,m)}(t)} \delequal \sum_{i \in [d]} \sum_{q = 1}^{ X_{j,i}^{(k,m)}(t-1) }{\bm B}^{(k,m,t,q)}_{ \bcdot \leftarrow i},
    \ \ \forall t \geq 1,
    \label{def: probability space with i.i.d. copies B k m t q j, cluster size, 2}
\end{align}
under initial value $\bm X_{j}^{(k,m)}(0) = \bm e_j$.
Besides, we adopt the notation in \eqref{def: truncated offspring count,B leq M j k} and use $\bm B^{\leqslant, (k,m,t,q)}_{ \bcdot \leftarrow j}(M)$
to denote the truncated version of $\bm B^{(k,m,t,q)}_{ \bcdot \leftarrow j}$ under threshold $M$.
Let
\begin{equation}\label{def: probability space with i.i.d. copies B k m t q j, cluster size, 3}
    \begin{aligned}
        {\bm X^{\leqslant,(k,m)}_j(t;M)} & \delequal 
    \sum_{i \in [d]} \sum_{q = 1}^{ X^{\leqslant,(k,m)}_{j,i}(t - 1;M) }
    {\bm B}^{\leqslant, (k,m,t,q)}_{ \bcdot \leftarrow i}(M),
    \quad \forall t \geq 1,
    \\
        \bm S_j^{\leqslant, (k,m)}(M) & \delequal \sum_{t \geq 0}\bm X_j^{\leqslant,(k,m)}(t;M),
    \end{aligned}
\end{equation}
with initial value ${\bm X^{\leqslant,(k,m)}_j(0;M)} = \bm e_j$.
Analogous to \eqref{def: W i M j, pruned cluster, 2, cluster size}, \eqref{def: cluster size, N i | M cdot j},
we define 
\begin{equation}
    \begin{aligned}
        W^{>,(k,m)}_{j;i}(M) & \delequal 
        \sum_{t \geq 1}
    \sum_{l \in [d]}\sum_{q = 1}^{ X^{\leqslant,(k,m)}_{j,l}(t-1;M) }
            B^{(k,m,t,q)}_{i \leftarrow l}\mathbbm{I}\big\{
                B^{(k,m,t,q)}_{i \leftarrow l} > M
            \big\},
        \\ 
        N^{>,(k,m)}_{j;i}(M) & \delequal 
        \sum_{t \geq 1}
    \sum_{l \in [d]}\sum_{q = 1}^{ X^{\leqslant,(k,m)}_{j,l}(t-1;M) }
            \mathbbm{I}\big\{
                B^{(k,m,t,q)}_{i \leftarrow l} > M
            \big\}.
    \end{aligned}
\end{equation}
Given $M > 0$ and $i \in [d]$,
note that
the collection of vectors
\begin{align}
    \Big(\bm S^{\leqslant,(k,m)}_j(M),
    W^{>,(k,m)}_{j;1}(M),\ldots, W^{>,(k,m)}_{j;d}(M),
     N^{>,(k,m)}_{j;1}(M),\ldots, N^{>,(k,m)}_{j;d}(M)
    \Big)_{ k, m \geq 1 }
\label{proof, indp samples, pruned clusters}
\end{align}
are i.i.d.\ copies of
$
 \big(\bm S_j^\leqslant(M),\  W^{>}_{j;1}(M), \ldots, W^{>}_{j;d}(M),\
     N^{>}_{j;1}(M), \ldots, N^{>}_{j;d}(M)
    \big).
$

Now, given $j \in [d]$, $n \in \mathbb N$, and $\delta > 0$, we consider the following procedure,
where $k$ denotes the iteration in the recursive application of \eqref{proof, equality, from S i truncated to S i}, and 
$
M^{n|\delta}_{\bcdot \leftarrow j}(k)
$
denotes the truncation threshold employed in the $k^\text{th}$ iteration
for sub-trees induced by type-$j$ nodes.
\begin{enumerate}[(i)]
    \item 
        Set
        \begin{align}
            \bm{\tau}^{n|\delta}_{j}(0) = \bm e_j.
            \label{def: layer zero, from pruned cluster to full cluster}
        \end{align}
        In addition, set
        \begin{align}
            M^{n|\delta}_{\bcdot \leftarrow i}(1) = n\delta,
            \qquad \forall i \in [d].
            \label{def: layer zero, truncation threshold, from pruned cluster to full cluster}
        \end{align}

    \item Starting from $k \geq 1$, do the following inductively.
    If there is some $i \in [d]$ such that 
    $
    \tau^{n|\delta}_{j;i}(k-1) > 0,
    $
    let
        \begin{align}
                \notationdef{notation-tau-k-M-j}{\tau^{n|\delta}_{j;l}(k)}
            & \delequal
            \sum_{i \in [d]}\sum_{m = 1}^{ \tau^{n|\delta}_{j;i}(k- 1) }
                W^{>,(k,m)}_{i;l}\Big(M^{n|\delta}_{\bcdot \leftarrow i}(k)\Big),\qquad \forall l \in [d],
            \label{def: from pruned cluster to full cluster, tau and S at step k + 1}
            \\
            \bm S^{n|\delta}_{j;\bcdot \leftarrow i}(k)
            & \delequal 
            \sum_{m = 1}^{ \tau^{n|\delta}_{j;i}(k-1) }\bm S^{\leqslant,(k,m)}_i
                \Big(M^{n|\delta}_{\bcdot \leftarrow i}(k)\Big),
            \qquad \forall i \in [d],
            \label{def: from pruned cluster to full cluster, tau and S at step k + 1, 2}
            \\
            \bm S^{n|\delta}_{j}(k)
            & \delequal 
            \sum_{i \in [d]}
            \bm S^{n|\delta}_{j;\bcdot \leftarrow i}(k)
            =
            \sum_{i \in [d]}
            \sum_{m = 1}^{ \tau^{n|\delta}_{j;i}(k-1) }\bm S^{\leqslant,(k,m)}_i
                \Big(M^{n|\delta}_{\bcdot \leftarrow i}(k)\Big),
            \label{def: from pruned cluster to full cluster, tau and S at step k + 1, 3}
        \end{align}
        and set
        \begin{align}
            M_{\bcdot \leftarrow i}^{n|\delta}(k+1) =  \delta \cdot \tau^{n|\delta}_{j;i}(k),
            \qquad \forall i \in [d].
            \label{def: from pruned cluster to full cluster, tau and S at step k + 1, 4}
        \end{align}
        Otherwise, move onto step (iii).

    \item 
        Now,
        let
        \begin{align}
            \mathcal K^{n|\delta}_j \delequal 
            \max\Big\{ k \geq 0:\ \tau^{n|\delta}_{j;i}(k) > 0\text{ for some }i \in [d]  \Big\}.
            \label{def: from pruned cluster to full cluster, tau and S at step k + 1, 5}
        \end{align}
        By step (ii) and the definition of $\mathcal K^{n|\delta}_j$,
        we have $\tau^{n|\delta}_{j;i}(k) = 0\ \forall i \in [d]$ under $k = \mathcal K^{n|\delta}_j + 1$. 
        For all $k >  \mathcal K^{n|\delta}_j + 1$, we also set
        \begin{align}
            \tau^{n|\delta}_{j;i}(k) = 0,\qquad 
            \forall i \in [d].
            \label{def: from pruned cluster to full cluster, tau and S at step k + 1, 6}
        \end{align}
        By \eqref{proof, equality, from S i truncated to S i},
        \begin{align}
            \bm S_j \distequal 
            \sum_{k = 1 }^{\mathcal K^{n|\delta}_j + 1}\bm S_j^{n|\delta}(k) 
            =
            \sum_{k = 1 }^{\mathcal K^{n|\delta}_j + 1}
            \sum_{i \in [d]}
            \sum_{m = 1}^{ \tau^{n|\delta}_{j;i}(k-1) }\bm S^{\leqslant,(k,m)}_i
                \Big(M^{n|\delta}_{\bcdot \leftarrow i}(k)\Big).
            \label{def: from pruned cluster to full cluster, tau and S at step k + 1, 7}
        \end{align}
\end{enumerate}
In particular, 
step (ii) is a recursive application of \eqref{proof, equality, from S i truncated to S i}.
For each $k \geq 1$,
we use $\tau^{n|\delta}_{j;i}(k-1)$
to count the number of copies of $\bm S_i$ that remain to be generated after the  $(k-1)^\text{th}$ iteration of step (ii).
At the $k^\text{th}$ iteration,
the independent copies of $\bm S_i$ are generated via \eqref{proof, equality, from S i truncated to S i}
under the truncation threshold
$
M^{n|\delta}_{\bcdot \leftarrow i}(k),
$
which is determined by the rule \eqref{def: from pruned cluster to full cluster, tau and S at step k + 1, 4}
using the values of $\tau^{n|\delta}_{j;i}(k-1)$ in the previous iteration.
We add a few remarks:
\begin{itemize}

    \item 
        Under the sub-criticality condition in Assumption~\ref{assumption: subcriticality},
        step (ii) will almost surely terminate after finitely many steps 
        (meaning that $\mathcal K^{n|\delta}_j < \infty$ almost surely).
        This is because $\norm{\bm S_j} < \infty$ almost surely, and each copy $\bm S^{\leqslant,(k,m)}_i(M)$ will add a least one node---the ancestor along the $i^\text{th}$ dimension that induces this sub-tree.

    \item 
        To prove Theorem~\ref{theorem: main result, cluster size}, 
        it suffices to consider a finite-iteration version of step (ii).
        Indeed, Lemma~\ref{lemma: crude estimate, type of cluster} confirms that, given $\gamma > 0$,
        it holds for all $K$ large enough that
        the probability of step (ii) running beyond $K$ iterations is of order $\lo (n^{-\gamma})$.
        Therefore, 
        by picking a constant $K$ large enough to ensure an $\lo (\lambda_{\bm j}(n))$ bound for such pathological cases,
        we can prove \eqref{claim, theorem: main result, cluster size} (given $A$ and $\bm j$) or \eqref{claim, 2, theorem: main result, cluster size}
        by only
        applying equality \eqref{proof, equality, from S i truncated to S i} for $K$ times (instead of stopping randomly) in step (ii).
        Switching to this alternative approach has no real consequences for our subsequent analysis, and we will not explore it in detail.

    \item 
        Lastly, by \eqref{property, W and N i M l j when positive, cluster size} and the choices of $M^{n|\delta}_{\bcdot \leftarrow i}(k)$ above, we have 
        \begin{align}
            \tau^{n|\delta}_{j;i}(k) > 0
            \quad
            \iff
            \quad
            \tau^{n|\delta}_{j;i}(k) > n\delta^k,
            \qquad
            \forall k \geq 1,\ i \in [d].
            \label{property: lower bound for tau n delta i j k, cluster size}
        \end{align}

\end{itemize}

\medskip
To proceed with our proof of Theorem~\ref{theorem: main result, cluster size},
let
\begin{align}
    \bar{\bm S}^n_j & \delequal n^{-1}\sum_{k  = 1}^{\mathcal K^{n|\delta}_j + 1}\bm S_j^{n|\delta}(k) 
    \distequal
    n^{-1}\bm S_j,
    \label{def: scaled cluster bar S n i}
    \\
    \hat{\bm S}^{n|\delta}_j
    & \delequal 
    \sum_{ k \geq 1 }\sum_{i \in [d]} n^{-1} \tau^{n|\delta}_{j;i}(k) \cdot \bar{\bm s}_j.
    \label{def: hat S n delta i, cluster size}
\end{align}
The last equality in \eqref{def: scaled cluster bar S n i} follows from 
\eqref{def: from pruned cluster to full cluster, tau and S at step k + 1, 7}.
Next, recall the definition of the mapping $\Phi$ in \eqref{def: Phi, polar transform},
and define
\begin{align}
    \notationdef{notation-polar-coordiates-for-bar-S}{(\bar R^{n}_j, \bar \Theta^n_j)} \delequal \Phi(\bar{\bm S}^n_j),
    \qquad
    \notationdef{notation-polar-coordinates-for-hat-S}{(\hat R^{n|\delta}_j, \hat \Theta^{n|\delta}_j)}
    \delequal 
    \Phi(\hat{\bm S}^{n|\delta}_j),
    \label{def: polar coordinates, bar S and hat S, cluster size}
\end{align}
which can be interpreted as the polar coordinates of $\bar{\bm S}^n_j$ and $\hat{\bm S}^{n|\delta}_j$.
Note that $\bar R^n_j = \norm{\bar{\bm S}^n_j}$ and $\hat{R}^{n|\delta}_j = \norm{\hat{\bm S}^{n|\delta}_j}$.
Meanwhile,
based on the definitions of $\tau^{n|\delta}_{j;i}(k)$ in \eqref{def: layer zero, from pruned cluster to full cluster} and \eqref{def: from pruned cluster to full cluster, tau and S at step k + 1},
we define
\begin{align}
    \notationdef{notation-I-k-M-j}{I^{n|\delta}_{j;i}(k)} \delequal \mathbbm{I}\big\{ \tau^{n|\delta}_{j;i}(k) > 0 \big\},
    \qquad \forall i \in [d],\ k \geq 1.
    \label{def, I k M j for M type, cluster size}
\end{align}
By the definition of $\mathcal K^{n|\delta}_{j}$ in \eqref{def: from pruned cluster to full cluster, tau and S at step k + 1, 5}, we have
\begin{itemize}
    \item 
        $I^{n|\delta}_{j;i}(k) = 0\ \ \forall k \geq \mathcal K^{n|\delta}_{j} + 1,\  i \in [d]$,

    \item 
        For any $k= 1,\ldots,\mathcal K^{n|\delta}_{j}$, there exists some $i \in [d]$ such that $I^{n|\delta}_{j;i}(k) = 1$.
\end{itemize}
We say that
$\notationdef{notation-M-type-I-|-M-cluster-size}{\bm I^{n|\delta}_j} = \big(I^{n|\delta}_{j;i}(k)\big)_{k \geq 1,\ i \in [d]}$ is the \textbf{$(n,\delta)$-type of} $\bar{\bm S}^{n}_j$.
Note that $\bm I^{n|\delta}_j$ can take values outside of $\mathscr I$,  the collection of all types in Definition~\ref{def: cluster, type}.
We thus consider the following generalization,
where $\bm I^{n|\delta}_j \in \widetilde{\mathscr I}$ a.s.\ due to $\mathcal K^{n|\delta}_j < \infty$ a.s.

\begin{definition}[Generalized Type]
    \label{def: cluster, generalized type}
$\bm I = (I_{k,i})_{k \geq 1,\ j \in [d]}$ is a \textbf{generalized type} if it satisfies the following conditions:
\begin{itemize}

    \item
        $I_{k,j}\in\{0,1\}$ for all $k \geq 1$ and $j \in [d]$;

 
    \item
        There exists ${\mathcal K^{\bm I}} \in \mathbb Z_+$ such that $\sum_{i \in [d]}I_{k,j} = 0\ \forall k > \mathcal K^{\bm I}$ and $\sum_{j \in [d]}I_{k,j} \geq 1 \ \forall 1 \leq k \leq \mathcal K^{\bm I}$.

\end{itemize}
We use $\notationdef{notation-tilde-mathscr-I-generalized-type-cluster-size}{\widetilde{\mathscr I}}$ to denote the set containing all generalized types.
For each $\bm I \in \widetilde{\mathscr I}$, we say that
$
\bm j^{\bm I} \delequal \big\{ j \in [d]:\ \sum_{k \geq 1}I_{k,j} = 1  \big\}
$
is the set of \textbf{active indices} of the generalized type $\bm I$,
and $\mathcal K^{\bm I}$ is the \textbf{depth} of $\bm I$.
For each $k \geq 1$, we say that
$
{\bm j^{\bm I}_k} \delequal 
    \big\{
        j \in [d]:\ I_{k,j} = 1
    \big\}
$
is the set of \textbf{active indices at depth $k$} in $\bm I$.
\end{definition}

We prove Theorem~\ref{theorem: main result, cluster size}
by establishing the asymptotic equivalence between 
$
(\bar R^{n}_j, \bar \Theta^n_j)
$
and
$
(\hat R^{n|\delta}_j, \hat \Theta^{n|\delta}_j)
$
in terms of  $\mathbb M$-convergence (see Definition~\ref{def: M convergence}),
and we view ${\bm I^{n|\delta}_j} = \big(I^{n|\delta}_{j;i}(k)\big)_{k \geq 1,\ i \in [d]}$ as a mark of $(\hat R^{n|\delta}_j, \hat \Theta^{n|\delta}_j)$ that encapsulates the spatio-temporal information of the big jumps in the underlying branching process.
To this end, we prepare 
Lemma~\ref{lemma: asymptotic equivalence, MRV in Rd}.
This result can be seen as a version of Lemma~2.4 in \cite{wang2023large} tailored for the space of polar coordinates $\mathbb S=[0,\infty) \times \mathfrak N^d_+$ under the metric $\bm d_\textbf{U}$ defined in \eqref{def: metric for polar coordinates},
and there are only two key differences.
First, 
Condition \eqref{condition: C is a cone, lemma: asymptotic equivalence, MRV in Rd} explicitly requires that, under polar transform, the pre-image of $\mathbb C$ is a cone in $\R^d_+$.
Second, we augment the approximations $Y^\delta_n$ with random marks $V^\delta_n$; in this regard, Lemma~2.4 in \cite{wang2023large} can be seen as a simplified version of our Lemma~\ref{lemma: asymptotic equivalence, MRV in Rd} featuring ``dummy'' marks (e.g., $V^\delta_n \equiv 1$).
The proof  is similar to that of Lemma~2.4 in \cite{wang2023large} and is collected in Section~\ref{subsec: proof, M convergence and asymptotic equivalence}  of the Appendix for the sake of completeness.

\begin{lemma}\label{lemma: asymptotic equivalence, MRV in Rd}
\linksinthm{lemma: asymptotic equivalence, MRV in Rd}
Let $(R_n,\Theta_n)$ and $(\hat R^\delta_n,\hat \Theta^\delta_n)$ be random elements taking values in $\mathbb S = [0,\infty) \times \mathfrak N^d_+$ with metric $\bm d_\textbf{U}$ in \eqref{def: metric for polar coordinates}.
Let $V^\delta_n$ be random elements taking values in a countable set $\mathbb V$.
Let $\mathbb C \subseteq \mathbb S$ be such that $(0,\bm w) \in \mathbb C$ for any $\bm w \in \mathfrak N^d_+$, and 
\begin{align}
    (r,\bm w) \in \mathbb C,\ r > 0,\ \bm w \in \mathfrak N^d_+
      \quad \Longrightarrow \quad 
    (t, \bm w) \in \mathbb C\ \  \forall t \geq 0.
    \label{condition: C is a cone, lemma: asymptotic equivalence, MRV in Rd}
\end{align}
Let
$\mathcal V \subset \mathbb V$ be a set containing only finitely many elements (i.e., $|\mathcal V| < \infty$),
and  let $\mu_v \in \mathbb M(\mathbb S \setminus \mathbb C)$ for each $v \in \mathcal V$.
Let $\epsilon_n$ be a sequence of strictly positive real numbers with $\lim_{n \to \infty}\epsilon_n = 0$.
Suppose that
\begin{enumerate}[(i)]
    \item 
        (\textbf{Asymptotic equivalence})
        Given $\Delta \in (0,1)$ ,
        it holds for any $\delta > 0$ small enough that 
           \begin{align*}
           \lim_{n \to \infty}\epsilon^{-1}_n
        & \P\bigg(
            \big\{
                 R_n \vee \hat R^{\delta}_n > \Delta 
            \big\}
            \cap 
            \Big\{
            \frac{\hat R^\delta_n}{R_n}\notin [ 1-\Delta, 1+\Delta]
            \text{ or }
            \norm{\Theta_n - \hat \Theta^\delta_n} > \Delta 
            \Big\}
        \bigg) = 0;
        \end{align*}
        
    \item 
         (\textbf{Convergence given the marks $V^\delta_n$})
         Let the Borel set $B \subseteq \mathbb S$ be  bounded away from $\mathbb{C}$ under $\bm d_{\textbf U}$,
         and let
        $\Delta \in \big(0,\bm d_{\textbf U}(B,\mathbb C)\big)$;
        under any $\delta > 0$ small enough,
        the claim
        \begin{align*}
        \mu_v(B_\Delta) - \Delta
        & \leq 
         \liminf_{n \to \infty}\epsilon^{-1}_n\P\big( (\hat R^\delta_n,\hat\Theta^\delta_n)  \in B,\ V^\delta_n = v\big)
         \\
       & \leq
       \limsup_{n \to \infty}\epsilon^{-1}_n\P\big( (\hat R^\delta_n,\hat\Theta^\delta_n)  \in B,\ V^\delta_n = v\big) \leq \mu_v(B^{\Delta}) + \Delta
    \end{align*}
    holds for each $v \in \mathcal V$, and we also have
    \begin{align}
        \lim_{n \to \infty}\epsilon^{-1}_n\P\big( (\hat R^\delta_n,\hat\Theta^\delta_n)  \in B,\ V^\delta_n \notin \mathcal V\big) = 0.
        \nonumber
    \end{align}
\end{enumerate}
Then,
$
\epsilon^{-1}_n\P\big( (R_n,\Theta_n) \in \ \cdot\ \big)
\to \sum_{v \in \mathcal V}\mu_v(\cdot)
$
in $\mathbb M(\mathbb S \setminus \mathbb C)$.
\end{lemma}

Recall the definitions of $\R^d(\bm j)$ and $\R^d_{\leqslant}(\bm j) = \bar\R^d_\leqslant(\bm j, 0)$ in 
\eqref{def: cone R d index i} and \eqref{def: cone R d i basis S index alpha}, respectively.
For any $\bm j \subseteq \{1,2,\ldots,d\}$
that is non-empty, let
\begin{equation}\label{def: cone C d leq j, cluster size}
    \begin{aligned}
        \notationdef{notation-cone-C-bm-j}{ \mathbb C^d(\bm j)}
    \delequal 
    \big\{
        (r,\bm w) \in [0,\infty) \times \mathfrak N^d_+:\ r\bm w \in \R^d(\bm j)
    \big\},
    \\ 
    \notationdef{notation-cone-C-leq-bm-j}{ \mathbb C^d_\leqslant(\bm j)}
    \delequal 
    \big\{
        (r,\bm w) \in [0,\infty) \times \mathfrak N^d_+:\ r\bm w \in \R^d_\leqslant(\bm j)
    \big\}.
    \end{aligned}
\end{equation}
Recall the definitions of 
$
{(\bar R^{n}_j, \bar \Theta^n_j)}
$
and
$
 {(\hat R^{n|\delta}_j, \hat \Theta^{n|\delta}_j)}
$
in \eqref{def: polar coordinates, bar S and hat S, cluster size}.
The next two key propositions allow us to apply Lemma~\ref{lemma: asymptotic equivalence, MRV in Rd} and establish Theorem~\ref{theorem: main result, cluster size}.

\begin{proposition}\label{proposition: asymptotic equivalence, tail asymptotics for cluster size}
\linksinthm{proposition: asymptotic equivalence, tail asymptotics for cluster size}
Let Assumptions~\ref{assumption: subcriticality}--\ref{assumption: regularity condition 2, cluster size, July 2024} hold.
Given $i \in [d]$, $\gamma > 0,\ \Delta \in (0,1)$, and
non-empty $\bm j \subseteq \{1,2,\ldots,d\}$,
it holds for any $\delta > 0$ small enough that
 \begin{align}
     \lim_{n \to \infty}n^{\gamma} \cdot 
     \P\Bigg(
        \Big\{
           \bar R^n_i \vee \hat R^{n|\delta}_i > \Delta
        \Big\}
        \cap
        \bigg\{
            \frac{ \hat R^{n|\delta}_i }{\bar R^n_i} \notin [ 1-\Delta,1+\Delta]
            \text{ or }
            \norm{ \bar\Theta^n_i - \hat \Theta^{n|\delta}_i } > \Delta
        \bigg\}
     \Bigg)
     & = 0.
     \label{claim, proposition: asymptotic equivalence, tail asymptotics for cluster size}
 \end{align}
\end{proposition}

\begin{proposition}\label{proposition, M convergence for hat S, tail asymptotics for cluster size}
\linksinthm{proposition, M convergence for hat S, tail asymptotics for cluster size}
Let Assumptions~\ref{assumption: subcriticality}--\ref{assumption: regularity condition 2, cluster size, July 2024} hold.
Let $i \in [d]$, and let $\bm j \subseteq \{1,2,\ldots,d\}$ be non-empty.
Let
$
\mathscr I(\bm j)
    \delequal
    \big\{ \bm I \in \mathscr I:\ \bm j^{\bm I} = \bm j \big\},
$
where $\bm j^{\bm I}$ is the set of active indices of type $\bm I$, and $\mathscr I$ is the collection of all types (see Definition~\ref{def: cluster, type}).
Given $\Delta > 0$ and a Borel set
$B \subseteq [0,\infty) \times \mathfrak N^d_+$ that is  bounded away from
 $\C^d_\leqslant(\bm j)$ under $\bm d_\textbf{U}$,
 the following claims hold.
 \begin{enumerate}[(i)]

     \item 
         Under any $\delta > 0$ small enough, it holds for each $\bm I \in \mathscr I(\bm j)$ that
         \begin{equation}\label{claim, proposition, M convergence for hat S, tail asymptotics for cluster size}
            \begin{aligned}
                \limsup_{n \to \infty}
                \big(\lambda_{\bm j}(n)\big)^{-1} \cdot 
                \P\Big(
                    (\hat R^{n|\delta}_i, \hat \Theta^{n|\delta}_i) \in B,\ 
                    \bm I^{n|\delta}_i = \bm I
                \Big)
                & \leq 
                \mathbf C^{\bm I}_i \circ \Phi^{-1}(B^\Delta) + \Delta,
                \\
                \liminf_{n \to \infty}
                \big(\lambda_{\bm j}(n)\big)^{-1} \cdot 
                \P\Big(
                    (\hat R^{n|\delta}_i, \hat \Theta^{n|\delta}_i) \in B,\ 
                    \bm I^{n|\delta}_i = \bm I
                \Big)
                & \geq 
                \mathbf C^{\bm I}_i \circ \Phi^{-1}(B_\Delta) - \Delta,
            \end{aligned}
         \end{equation}
         where  
         the measure $\mathbf C_i^{\bm I} \circ \Phi^{-1}(\cdot)$ is defined using
         \eqref{def: mu composition Phi inverse measure}
with the $\mathbf C^{\bm I}_i(\cdot)$'s in \eqref{def: measure C i I, cluster}.

    \item 
        Under any $\delta > 0$ small enough
        \begin{align}
            \lim_{n \to \infty}
                \big(\lambda_{\bm j}(n)\big)^{-1} \cdot 
                \P\Big(
                    (\hat R^{n|\delta}_i, \hat \Theta^{n|\delta}_i) \in B,\ 
                    \bm I^{n|\delta}_i \notin \mathscr I(\bm j)
                \Big) = 0.
                \label{claim, under wrong marks, no  proposition, M convergence for hat S, tail asymptotics for cluster size}
        \end{align}

    \item 
        $
        \sum_{\bm I \in \mathscr I(\bm j)}\mathbf C_i^{\bm I} \circ \Phi^{-1} (B)
     < \infty.
        $
 \end{enumerate}
\end{proposition}
To conclude Section~\ref{subsec: proof, theorem: main result, cluster size},
we 
provide the proof of Theorem~\ref{theorem: main result, cluster size} using Propositions~\ref{proposition: asymptotic equivalence, tail asymptotics for cluster size} and \ref{proposition, M convergence for hat S, tail asymptotics for cluster size}.
The remainder of Section~\ref{sec: proof, cluster size asymptotics} is devoted to establishing Propositions~\ref{proposition: asymptotic equivalence, tail asymptotics for cluster size} and \ref{proposition, M convergence for hat S, tail asymptotics for cluster size}.

\begin{proof}[Proof of Theorem \ref{theorem: main result, cluster size}]
\linksinpf{theorem: main result, cluster size}
We first prove Claim~\eqref{claim, theorem: main result, cluster size}.
Under the choice of 
$(R_n,\Theta_n) = ( \bar R^n_i, \bar \Theta^n_i )$,
$
(\hat R^\delta_n, \hat \Theta^\delta_n )=( \hat R^{n|\delta}_i, \hat \Theta^{n|\delta}_i ),
$
$
\mathbb S = [0,\infty)\times \mathfrak N^d_+,
$
$
\mathbb C =
   \C^d_\leqslant(\bm j),
$
$
\epsilon_n = \lambda_{\bm j}(n),
$
$
V^\delta_n = \bm I^{n|\delta}_i,
$
and
$
\mathcal V = 
\mathscr I(\bm j)
    =
    \big\{ \bm I \in \mathscr I:\ \bm j^{\bm I} = \bm j \big\},
$
Propositions~\ref{proposition: asymptotic equivalence, tail asymptotics for cluster size} and \ref{proposition, M convergence for hat S, tail asymptotics for cluster size}
verify the conditions in 
Lemma~\ref{lemma: asymptotic equivalence, MRV in Rd}.
In particular,
part (iii) of Proposition~\ref{proposition, M convergence for hat S, tail asymptotics for cluster size}
confirms that $\mathbf C^{\bm I}_{ i  } \circ \Phi^{-1} \in \mathbb M\big(\mathbb S\setminus \mathbb C_\leqslant^d(\bm j)\big)$
for each $\bm I \in \mathscr I(\bm j)$.
Next, Proposition~\ref{proposition: asymptotic equivalence, tail asymptotics for cluster size} verifies condition (i) of Lemma~\ref{lemma: asymptotic equivalence, MRV in Rd},
and 
parts (i) and (ii) of Proposition~\ref{proposition, M convergence for hat S, tail asymptotics for cluster size} verify condition (ii) of Lemma~\ref{lemma: asymptotic equivalence, MRV in Rd}.
This allows us to apply Lemma~\ref{lemma: asymptotic equivalence, MRV in Rd} and obtain
\begin{align}
    \big(\lambda_{\bm j}(n)\big)^{-1}
    \P\Big( (\bar R^n_i, \bar \Theta^n_i) \in\ \cdot \ \Big)
    \to 
    \sum_{\bm I \in \mathscr I(\bm j)}\mathbf C_i^{\bm I} \circ \Phi^{-1}(\cdot)
    \quad
    \text{ in }\mathbb M\big( \mathbb S \setminus \mathbb C^d_\leqslant(\bm j) \big).
    \nonumber
\end{align}
Lastly, applying Lemma~\ref{lemma: M convergence for MRV}
under the choice of $X_n = \bar{\bm S}^n_i$, $(R_n,\Theta_n) = (\bar R^{n}_i, \bar \Theta^n_i)$, 
and $\mu
    =
    \sum_{\bm I \in \mathscr I(\bm j)}\mathbf C_i^{\bm I}$,
we conclude the proof of Claim~\eqref{claim, theorem: main result, cluster size}.

The proof of Claim~\eqref{claim, 2, theorem: main result, cluster size} is almost identical,
and
the plan is to apply Lemma~\ref{lemma: asymptotic equivalence, MRV in Rd} under the choices of
$(R_n,\Theta_n) = ( \bar R^n_i, \bar \Theta^n_i )$,
$
(\hat R^\delta_n, \hat \Theta^\delta_n )=( \hat R^{n|\delta}_i, \hat \Theta^{n|\delta}_i ),
$
$
\mathbb S = [0,\infty)\times \mathfrak N^d_+,
$
$
\mathbb C =
   \C^d([d]),
$
$
\epsilon_n = n^{-\gamma},
$
and with dummy marks (i.e., $V^\delta_n \equiv 0$, $\mathcal V = \mathbb V = \{0\}$).
 Again, Proposition~\ref{proposition: asymptotic equivalence, tail asymptotics for cluster size} verifies condition (i) of Lemma~\ref{lemma: asymptotic equivalence, MRV in Rd}.
Meanwhile, note that
$\hat{\bm S}^{n|\delta}_i \in \R^d([d])$ and $( \hat R^{n|\delta}_i, \hat \Theta^{n|\delta}_i ) \in \mathbb C^d([d])$
 by definitions in \eqref{def: hat S n delta i, cluster size} and \eqref{def: polar coordinates, bar S and hat S, cluster size}.
Then, for any $B \in \mathbb S$ that is bounded away from $\mathbb C^d([d])$, it holds trivially that
$
\P\big(  ( \hat R^{n|\delta}_i, \hat \Theta^{n|\delta}_i ) \in B  \big) = 0,
$
thus verifying condition (ii) of Lemma~\ref{lemma: asymptotic equivalence, MRV in Rd} with $\mu_v \equiv 0$.
By Lemma~\ref{lemma: asymptotic equivalence, MRV in Rd}, 
we get
$
n^{\gamma}\cdot\P\big( (\bar R^n_i, \bar \Theta^n_i) \in\ \cdot \  \big)
\rightarrow 0
$
in 
$
\mathbb M\big( \mathbb S \setminus \mathbb C^d([d]) \big)
$
for any $\gamma > 0$.
Applying Lemma~\ref{lemma: M convergence for MRV} again,
we conclude the proof.
\end{proof}

\subsection{Proof of Proposition~\ref{proposition: asymptotic equivalence, tail asymptotics for cluster size}}
\label{subsec: proof, proposition: asymptotic equivalence, tail asymptotics for cluster size}

We first state two lemmas to characterize tail asymptotics and provide concentration inequalities for $\bm S^{\leqslant}_i(M)$ defined in \eqref{proof, equality, from S i truncated to S i}. 

\begin{lemma}\label{lemma: tail bound, pruned cluster size S i leq n delta}
\linksinthm{lemma: tail bound, pruned cluster size S i leq n delta}
Let Assumptions~\ref{assumption: subcriticality}--\ref{assumption: regularity condition 2, cluster size, July 2024} hold.
Given any $\Delta,\ \gamma \in (0,\infty)$,
there exists $\delta_0 = \delta_0(\Delta,\gamma) > 0$ such that
\begin{align}
    \lim_{n \to \infty}
    n^\gamma \cdot 
    \P\Big(
        \norm{ \bm S^{\leqslant}_i(n\delta) } > n\Delta
    \Big) = 0,
    \qquad 
    \forall \delta \in (0,\delta_0),\ i \in [d].
    \label{claim, lemma: tail bound, pruned cluster size S i leq n delta}
\end{align}
\end{lemma}

\begin{lemma}\label{lemma: concentration ineq for pruned cluster S i}
\linksinthm{lemma: concentration ineq for pruned cluster S i}
Let Assumptions~\ref{assumption: subcriticality}--\ref{assumption: regularity condition 2, cluster size, July 2024} hold.
Given $\epsilon,\ \gamma > 0$, there exists $\delta_0 = \delta_0(\epsilon,\gamma) > 0$ such that
\begin{align*}
    \lim_{n \to \infty} 
        n^\gamma \cdot 
        \P\Bigg(
        \norm{
            \frac{1}{n}\sum_{m = 1}^n \bm S^{\leqslant, (m)}_i(n\delta) - \bar{\bm s}_i
        }
        > \epsilon
        \Bigg) = 0,
        \qquad \forall \delta \in (0,\delta_0),\ i \in [d],
\end{align*}
where 
    $\bar{\bm s}_i  = \E \bm S_i$ (see \eqref{def: bar s i, ray, expectation of cluster size}),
    and the
    $\bm S^{\leqslant,(m)}_i(M)$'s
    are 
    independent copies of $\bm S^{\leqslant}_i(M)$.
\end{lemma}

These results follow from concentration inequalities for truncated heavy-tailed random vectors.
Indeed, by definitions
in \eqref{def: truncated offspring count,B leq M j k}--\eqref{def: branching process X n leq M},
$\bm S^\leqslant_i(n\delta)$ can be expressed as a randomly stopped sum of i.i.d.\ copies of heavy-tailed random vectors $\bm B_{\bcdot \leftarrow j}$ truncated under threshold $n\delta$.
Furthermore, one can establish useful bounds on the random count in the summation,
as the number of individuals born in each generation of the branching process $\big(\bm X^\leqslant_j(t;M)\big)_{t \geq 0}$  is expected to contract geometrically fast 
if
$
\norm{\bar{\textbf B}} \delequal \sup_{ \norm{\bm x} = 1  }\norm{\bar{\textbf B}\bm x} < 1.
$
Therefore, we are able to suitably apply Lemma~3.1 of \cite{wang2023large} and prove Lemma~\ref{lemma: tail bound, pruned cluster size S i leq n delta}.
As an implication of Lemma~\ref{lemma: tail bound, pruned cluster size S i leq n delta},
$
\bm S_i^\leqslant(n\delta)
$
is almost always bounded by $\bm S_i \mathbbm{I}\{ \norm{\bm S_i} \leq n\Delta \}$ (under smaller enough $\delta$), thus
allowing us to apply Lemma~3.1 of \cite{wang2023large} again and verify Lemma~\ref{lemma: concentration ineq for pruned cluster S i}.
We collect their proofs in Section~\ref{subsec: proof, technical lemmas, cluster size}  of the Appendix.
Here, we note that the proof of Lemma~\ref{lemma: tail bound, pruned cluster size S i leq n delta} becomes more involved if
$
\norm{\bar{\textbf B}} \geq 1:
$
in that case, inspired by the Gelfand's formula based approach in  \cite{KEVEI2021109067},
we identify some $r$ with $\norm{\bar{\textbf B}^r} < 1$, 
and apply the same arguments to the sub-trees constructed by sampling the original branching tree every $r$ generations.

Next, we discuss properties of ${\bm I^{n|\delta}_j} = \big(I^{n|\delta}_{j;i}(k)\big)_{k \geq 1,\ i \in [d]}$ defined in \eqref{def, I k M j for M type, cluster size},
as well as the notion of generalized types in Definition~\ref{def: cluster, generalized type}.
Analogous to $\alpha(\cdot)$ defined in \eqref{def: cost function, cone, cluster},
we let
\begin{align}
    \notationdef{notation-tilde-c-cost-function-cluster-size}{\tilde{\alpha}(\bm I)} \delequal
        1 + \sum_{k \geq 1}\sum_{j \in [d]}I_{k,j}\cdot(\alpha^*(j) - 1)
    \label{def: tilde c I, cost for generalized type, cluster size}
\end{align}
for any generalized type $\bm I = (I_{k,j})_{k \geq 1, j \in [d]} \in \widetilde{\mathscr I}$
with $\mathcal K^{\bm I} \neq 0$ 
(i.e., there are some $k \geq 1$ and $j \in [d]$ such that $I_{k,j} \neq 0$).
If $\mathcal K^{\bm I} = 0$ (i.e., $I_{k,j} \equiv 0$),
we set $\tilde{\alpha}(\bm I) = 0$.
We stress again that Definition~\ref{def: cluster, generalized type} generalizes Definition~\ref{def: cluster, type} as $\widetilde{\mathscr I} \supseteq \mathscr I$.
In particular, given a generalized type $\bm I \in \widetilde{\mathscr I}$,
for any  $j \in \bm j^{\bm I}$ there could be multiple $k \geq 1$ such that $I_{k,j} = 1$.
However, this would not be the case for a type $\bm I \in \mathscr I$.
Besides, recall that we work with Assumption~\ref{assumption: heavy tails in B i j} in this paper,
which ensures that $\alpha^*(j) > 1\ \forall j \in [d]$ in \eqref{def: cluster size, alpha * l * j}.
Therefore, for $\alpha(\cdot)$ defined in \eqref{def: cost function, cone, cluster},
\begin{align}
    \alpha(\bm j^{\bm I}) & \leq \tilde{\alpha}(\bm I),
    \qquad
    \forall \bm I \in \widetilde{\mathscr I}\text{ with }\mathcal K^{\bm I}\geq 1,
    \label{property: cost of j I and type I, 1, cluster size}
    \\ 
    \alpha(\bm j^{\bm I}) & = \tilde{\alpha}(\bm I),
    \qquad
    \forall \bm I \in {\mathscr I}\text{ with }\mathcal K^{\bm I}\geq 1.
    \label{property: cost of j I and type I, 2, cluster size}
\end{align}
Similarly,
if there exists $j \in [d]$ such that $\#\{k \geq 1:\ I_{k,j} = 1\} \geq 2$, 
then, by definitions in \eqref{def: cost function, cone, cluster} and \eqref{def: tilde c I, cost for generalized type, cluster size},
we must have $\alpha(\bm j^{\bm I}) < \tilde \alpha(\bm I)$.
As a result, for any generalized type $\bm I = (I_{k,j})_{k \geq 1, j \in [d]} \in \widetilde{\mathscr I}$,
\begin{align}
    \alpha(\bm j^{\bm I}) = \tilde \alpha(\bm I)
    \quad \Longrightarrow \quad 
    \#\{k \geq 1:\ I_{k,j} = 1\} \leq 1\ \forall j \in [d].
    \label{property: cost of generalized type when alpha agrees with tilde alpha}
\end{align}
Besides, note that
the events $\{\bm I^{n|\delta}_{i} = \bm I  \}$
are mutually exclusive across different generalized types $\bm I \in \widetilde{\mathscr I}$.
Then, due to $\mathcal K^{n|\delta}_j < \infty$ almost surely,
it holds for any $n \geq 1$, $\delta > 0$, $i \in [d]$ that
\begin{align}
    \P\Big( \big( \bar{\bm S}^n_i, \hat{\bm S}^{n|\delta}_i\big) \in\ \cdot\ \Big)
    = 
    \sum_{ \bm I \in \widetilde{\mathscr I} }\P\Big(
        \big( \bar{\bm S}^n_i, \hat{\bm S}^{n|\delta}_i\big)\in \ \cdot\ ,\ \bm I^{n|\delta}_i = \bm I
    \Big),
    \label{proof: decomposition via generalized type, cluster size}
\end{align}
where  $\bar{\bm S}^n_i \distequal n^{-1}\bm S_i$ (see \eqref{def: scaled cluster bar S n i})
and $\hat{\bm S}^{n|\delta}_i$ is defined in \eqref{def: hat S n delta i, cluster size}.

We also highlight the Markov property embedded in $\bm I^{n|\delta}_i$.
Recall the probability space considered in Section~\ref{subsec: proof, theorem: main result, cluster size}
that supports the
$\bm B^{(k,m,t,q)}_{\bcdot \leftarrow j}$'s in \eqref{def: probability space with i.i.d. copies B k m t q j, cluster size},
and define the $\sigma$-algebra
\begin{align*}
    \notationdef{notation-mathcal-F-k-sigma-algebra}{\mathcal F_k}
    \delequal
    \sigma\Big\{
        \bm B^{(k^\prime,m,t,q)}_{\bcdot \leftarrow j}:\  m,t,q \geq 1,\ j \in [d],\ k^\prime \in [k]
    \Big\},
    \qquad
    \forall k \geq 1.
\end{align*}
Let $\mathcal F_0 \delequal \{ \emptyset,\Omega \}$.
By \eqref{def: probability space with i.i.d. copies B k m t q j, cluster size, 2} and \eqref{def: probability space with i.i.d. copies B k m t q j, cluster size, 3},
$
(\bm S^{(k,m)}_i)_{m \geq 1, i \in [d] }
$
and
the random vectors in \eqref{proof, indp samples, pruned clusters}
are measurable w.r.t.\ $\mathcal F_k$.
Then, in the procedure  \eqref{def: layer zero, from pruned cluster to full cluster}--\eqref{def: from pruned cluster to full cluster, tau and S at step k + 1, 7},
$
\tau^{n|\delta}_{j;i}(k),\ 
    \bm S^{n|\delta}_{j;\bcdot \leftarrow i}(k),\text{ and }
    \bm S^{n|\delta}_{j}(k)
$
are measurable w.r.t.\ $\mathcal F_k$.
Besides,
 $M_{\bcdot \leftarrow i}^{n|\delta}(k)$ is determined by $\tau^{n|\delta}_{j;i}(k-1)$ (see \eqref{def: from pruned cluster to full cluster, tau and S at step k + 1, 4}),
and hence measurable w.r.t.\ $\mathcal F_{k -1}$.
Furthermore,
conditioned on the value of  $\big(\tau^{n|\delta}_{j;i}(k-1)\big)_{i \in [d]}$,
the random vector 
$
\big( \tau^{n|\delta}_{j;i}(k),\ \bm S^{n|\delta}_{j;\bcdot \leftarrow i}(k)\big)_{i \in [d]}
$
is independent from $\mathcal F_{k-1}$. (i.e., the history).
As a result, for each $k \geq 2$,
\begin{align}
    & \P\Bigg(
        \Big( \tau^{n|\delta}_{j;i}(k),\ \bm S^{n|\delta}_{j;\bcdot \leftarrow i}(k) \Big)_{i \in [d]} \in \ \bcdot
        \ \Bigg|\ 
        \mathcal F_{k-1} 
    \Bigg)
    \label{proof, cluster size, markov property in pruned clusters}
    \\ 
    & = 
    \P\Bigg(
         \Big( \tau^{n|\delta}_{j;i}(k),\ \bm S^{n|\delta}_{j;\bcdot \leftarrow i}(k) \Big)_{i \in [d]} \in \ \bcdot
        \ \Bigg|\ 
        \Big( \tau^{n|\delta}_{j;l}(k-1)\Big)_{l \in [d]}
    \Bigg)
    \nonumber
    \\ 
    & = 
    \P\Bigg(
    \Bigg(
        \sum_{l \in [d]}\sum_{m = 1}^{ \tau^{n|\delta}_{j;l}(k- 1) }
                W^{>,(k,m)}_{l;i}\Big(\delta \cdot \tau^{n|\delta}_{j;l}(k-1)\Big),
     \nonumber
     \\ 
     &\qquad\qquad\quad\quad\quad
     \sum_{m = 1}^{ \tau^{n|\delta}_{j;i}(k-1) }\bm S^{\leqslant,(k,m)}_i
                \Big(\delta \cdot \tau^{n|\delta}_{j;i}(k-1)\Big)
     \Bigg)_{i \in [d]} \in\ \bcdot\ 
     \Bigg|\ \Big( \tau^{n|\delta}_{j;l}(k-1)\Big)_{l \in [d]}
    \Bigg).
    \nonumber
\end{align}
An immediate consequence of \eqref{proof, cluster size, markov property in pruned clusters} is that,
given any generalized type $\bm I  \in \widetilde{\mathscr I}$,
\begin{align}
    &
    \P( \bm I^{n|\delta}_i = \bm I)
    \label{property: decompo of event type I using Markov property, cluster size}
    \\ 
    & = \P\Big( I^{n|\delta}_{i;j}(1) = 1\text{ iff }j \in \bm j^{\bm I}_1\Big)
    \cdot 
    \prod_{ k = 2 }^{ \mathcal K^I + 1}
    \P\Big(
        I^{n|\delta}_{i;j}(k) = 1\text{ iff }j \in \bm j^{\bm I}_k\ \Big|\ 
        I^{n|\delta}_{i;j}(k-1) = 1\text{ iff }j \in \bm j^{\bm I}_{k-1}
    \Big)
    \nonumber
    \\ 
    & = 
    \P\Big( \tau^{n|\delta}_{i;j}(1) > 0\text{ iff }j \in \bm j^{\bm I}_1\Big)
    \cdot 
    \prod_{ k = 2 }^{ \mathcal K^I + 1}
    \P\Big(
        \tau^{n|\delta}_{i;j}(k) > 0\text{ iff }j \in \bm j^{\bm I}_k\ \Big|\ 
        \tau^{n|\delta}_{i;j}(k-1) > 0\text{ iff }j \in \bm j^{\bm I}_{k-1}
    \Big),
    \nonumber
\end{align}
where $\bm j^{\bm I}_k$ is the set of active indices at depth $k$ of $\bm I$ (see Definition~\ref{def: cluster, generalized type}),
and the display above follows from \eqref{proof, cluster size, markov property in pruned clusters}
as well as the definition of $I^{n|\delta}_{i;j}(k)$ in \eqref{def, I k M j for M type, cluster size}.

In light of \eqref{property: decompo of event type I using Markov property, cluster size},
as well as the definitions of the $\tau^{n|\delta}_{i;j}(k)$'s in \eqref{def: from pruned cluster to full cluster, tau and S at step k + 1},
the asymptotic analysis of events $\{\bm I^{n|\delta}_i = \bm I\}$ 
boils down to characterizing the asymptotic law of (the sums of) $W^{>}_{i;j}(M)$ and $N^{>}_{i;j}(M)$
in \eqref{def: W i M j, pruned cluster, 1, cluster size}--\eqref{def: cluster size, N i | M cdot j}.
This is the content of Lemma~\ref{lemma: cluster size, asymptotics, N i | n delta, cdot j, crude estimate},
which will be a key tool in our analysis.
In particular, independently for each $i \in [d]$, let
\begin{align}
    \bigg\{
    \Big(
    W^{>,(m)}_{i;1}(M),\ldots, W^{>,(m)}_{i;d}(M),
     N^{>,(m)}_{i;1}(M),\ldots, N^{>,(m)}_{i;d}(M)
    \Big):\ 
   m \geq 1
\bigg\}
   \label{proof: def copies of W > and N > vectors}
\end{align}
be independent copies of
$
\big(
    W^{>}_{i;1}(M),\ldots, W^{>}_{i;d}(M),
     N^{>}_{i;1}(M),\ldots, N^{>}_{i;d}(M)
    \big).
$
Given any non-empty $\mathcal I \subseteq [d]$ and any $\bm t(\mathcal I) = (t_i)_{i \in \mathcal I}$ with $t_i \in \mathbb Z_+$ for each $i \in \mathcal I$,
we write
\begin{align}
    \notationdef{notation-N-bm-t-mathcal-I-cdot-j}{N^{>|\delta}_{\bm t(\mathcal I);j}} \delequal \sum_{i \in \mathcal I} \sum_{m = 1}^{ t_i }N^{>,(m)}_{i;j}(\delta t_i),
    \qquad 
    \notationdef{notation-W-bm-t-mathcal-I-cdot-j}{W^{>|\delta}_{\bm t(\mathcal I);j}} \delequal \sum_{i \in \mathcal I} \sum_{m = 1}^{ t_i }W^{>,(m)}_{i;j}(\delta t_i).
    \label{def, proof cluster size, N W mathcal I mathcal J bcdot j}
\end{align}
By \eqref{property, W and N i M l j when positive, cluster size}, note that 
\begin{align}
    {N^{>|\delta}_{\bm t(\mathcal I);j}} > 0
    \qquad\iff \qquad
     {N^{>|\delta}_{\bm t(\mathcal I);j}} \geq 1
     \qquad\iff \qquad
    {W^{>|\delta}_{\bm t(\mathcal I);j}} > 0.
    \label{property, sum of W and N i M l j when positive, cluster size}
\end{align}

\begin{lemma}
\label{lemma: cluster size, asymptotics, N i | n delta, cdot j, crude estimate}
\linksinthm{lemma: cluster size, asymptotics, N i | n delta, cdot j, crude estimate}
Let Assumptions~\ref{assumption: subcriticality}--\ref{assumption: regularity condition 2, cluster size, July 2024} hold.

\begin{enumerate}[$(i)$]
    \item 
        Under any $\delta > 0$ small enough, we have (as $n \to \infty$),
        \begin{align}
            \P\Big(N^{>}_{i;j}(n\delta) = 1\Big) & \sim \bar s_{i,l^*(j)} \P(B_{j \leftarrow l^*(j)} > n\delta),
        \label{claim 1, lemma: cluster size, asymptotics, N i | n delta, cdot j, crude estimate}
            \\
            \P\Big(N^{>}_{i;j}(n\delta) \geq 2\Big) & = \lo\big(\P(B_{j \leftarrow l^*(j)} > n\delta)\big),
        \label{claim 2, lemma: cluster size, asymptotics, N i | n delta, cdot j, crude estimate}
        \end{align}
        where
        $l^*(\cdot)$ and $\alpha^*(\cdot)$ are defined in \eqref{def: cluster size, alpha * l * j}.
        Besides, given $0<c<C<\infty$ and $i \in [d],\ j \in [d]$, 
it holds for any $\delta > 0$ small enough that
        \begin{align}
            \lim_{n \to \infty}
            \P\Big(
                N^{>}_{i;j\leftarrow l^*(j)}(n\delta) = 1;\ 
                N^{>}_{i;j\leftarrow l}(n\delta) = 0\ \forall l \neq l^*(j)\ 
                \Big|\ 
                N^{>}_{i;j}(n\delta)  \geq 1
            \Big)
            & = 1,
            \label{claim 1, part iii, lemma: cluster size, asymptotics, N i | n delta, cdot j, refined estimates}
            \\
            \lim_{n \to \infty}
            \sup_{ x \in [c,C] }
            \Bigg| 
                \frac{ 
                    \P\big(
                W^{>}_{i;j}(n\delta) > nx\ 
                \big|\ 
                N^{>}_{i;j}(n\delta)  \geq 1
            \big)
                }{
                    ({\delta}/{x})^{\alpha^*(j)}
                }
                - 1
            \Bigg|
            & = 0,
            \label{claim 2, part iii, lemma: cluster size, asymptotics, N i | n delta, cdot j, refined estimates}
        \end{align}
        where $N^{>}_{i;j\leftarrow l}(M)$, $N^{>}_{i;j}(M)$ are defined in \eqref{def, N i | M l j, cluster size}--\eqref{def: cluster size, N i | M cdot j}.

    \item 
        There exists $\delta_0 > 0$ such that the following holds for any $\delta \in (0,\delta_0)$:
        for each $\mathcal J \subseteq [d]$ with $|\mathcal J| \geq 2$,
        \begin{align*}
            \P\Big(
                N^{>}_{i;j}(n\delta) \geq 1\ \forall j \in \mathcal J
            \Big)
            =
            \lo 
                \bigg(
                    n^{|\mathcal J| - 1}
                    \prod_{j \in \mathcal J}
                    \P(B_{j \leftarrow l^*(j)} > n\delta)
                \bigg),
            \qquad \text{ as }n\to\infty.
        \end{align*}

    \item 
        Let 
        $\mathcal I \subseteq \{1,2,\ldots,d\}$ and
        $\mathcal J \subseteq \{1,2,\ldots,d\}$ be non-empty.
        There exists $\delta_0 > 0$ such that
    \begin{align}
    \limsup_{ n \to \infty}
    \sup_{ 
        t_i \geq nc\ \forall i \in \mathcal I
    }
    \frac{
        \P\big(
            {N^{>|\delta}_{\bm t(\mathcal I);j}} \geq 1\text{ iff }j \in \mathcal J
        \big)
    }{
    \prod_{j \in \mathcal J}n\P\big( B_{j \leftarrow l^*(j)} > n\delta\big)
    }
    < \infty,
    \qquad
    \forall \delta \in (0,\delta_0),\ c > 0,
    \label{claim, part iii, lemma: cluster size, asymptotics, N i | n delta, cdot j, crude estimate}
    \end{align}
    where ${N^{>|\delta}_{\bm t(\mathcal I);j}}$ is defined in \eqref{def, proof cluster size, N W mathcal I mathcal J bcdot j}.
\end{enumerate}

\end{lemma}

We defer the proof of Lemma~\ref{lemma: cluster size, asymptotics, N i | n delta, cdot j, crude estimate} to Section~\ref{subsec: proof, lemma: cluster size, asymptotics, N i | n delta, cdot j, crude estimate}.
In this section, we focus on applying Lemma~\ref{lemma: cluster size, asymptotics, N i | n delta, cdot j, crude estimate} and establishing
Proposition~\ref{proposition: asymptotic equivalence, tail asymptotics for cluster size}.
First, given $\bm I, \bm I^\prime \in \widetilde{\mathscr I}$,
we define the following (partial) ordering:
\begin{align}
    \bm I \subseteq \bm I^\prime 
    \quad
    \Longleftrightarrow
    \quad
    I_{k,j} = I^\prime_{k,j},\qquad \forall j \in [d],\ k \in [\mathcal K^{\bm I}].
    \label{proof: def ordering of generalized types, cluster size}
\end{align}
That is, $\bm I \subseteq \bm I^\prime$ if they match with each other up to the depth of $\bm I$. 
Lemma~\ref{lemma: crude estimate, type of cluster} provides bounds for events of the form $\{ \bm I \subseteq \bm I^{n|\delta}_i \}$.

\begin{lemma}\label{lemma: crude estimate, type of cluster}
\linksinthm{lemma: crude estimate, type of cluster}
Let Assumptions~\ref{assumption: subcriticality}--\ref{assumption: regularity condition 2, cluster size, July 2024} hold.
Given $\bm I = (I_{k,j})_{k \geq 1, j \in [d]} \in \widetilde{\mathscr I}$
with $\mathcal K^{\bm I} \geq 1$,
it holds for any $\delta > 0$ small enough that
\begin{align}
    \P\big( \bm I \subseteq \bm I^{n|\delta}_i\big)
    =
    \bo
    \Bigg(
        n^{-1}
        \prod_{k = 1}^{ \mathcal K^{\bm I} }\prod_{ j \in \bm j^{\bm I}_k }
        n\P\big(B_{j \leftarrow l^*(j)} > n\delta\big)
    \Bigg),
    \ \text{ as }n \to \infty.
    \label{claim: lemma: crude estimate, type of cluster}
\end{align}
Furthermore, if $|\bm j^{\bm I}_1| \geq 2$, it holds for any $\delta > 0$ small enough that
\begin{align}
    \P\big( \bm I \subseteq \bm I^{n|\delta}_i\big)
    =
    \lo 
    \Bigg(
        n^{-1}
        \prod_{k = 1}^{ \mathcal K^{\bm I} }\prod_{ j \in \bm j^{\bm I}_k }
        n\P\big(B_{j \leftarrow l^*(j)} > n\delta\big)
    \Bigg),
    \ \text{ as }n \to \infty.
    \label{claim: lemma: crude estimate, |j 1 type I| geq 2, type of cluster}
\end{align}
\end{lemma}

\begin{proof}
\linksinpf{lemma: crude estimate, type of cluster}
By the definition  in \eqref{def, I k M j for M type, cluster size},
$
\{ \bm I \subseteq \bm I^{n|\delta}_i \}
     = \bigcap_{k = 1}^{ \mathcal K^{\bm I} } \big\{ 
        I^{n|\delta}_{i;j}(k) = 1\text{ iff }j \in \bm j^{\bm I}_k
    \big\}
    =
    \bigcap_{k = 1}^{ \mathcal K^{\bm I} } \big\{ 
        \tau^{n|\delta}_{i;j}(k) > 0 \text{ iff }j \in \bm j^{\bm I}_k
    \big\}.
$
Then, analogous to \eqref{property: decompo of event type I using Markov property, cluster size}, we have
\begin{align}
    & \P\big(\bm I \subseteq \bm I^{n|\delta}_i\big)
    \label{proof: applying markov property, decomp of events, lemma: crude estimate, type of cluster}
    \\ 
    & = 
    \P\big(
        \tau^{n|\delta}_{i;j}(1) > 0 \text{ iff }j \in \bm j^{\bm I}_1
    \big)
    \cdot 
    \prod_{ k = 2 }^{\mathcal K^{\bm I}}
    \P\Big(
        \tau^{n|\delta}_{i;j}(k) > 0 \text{ iff }j \in \bm j^{\bm I}_k
        \ \Big|\ 
        \tau^{n|\delta}_{i;j}(k-1) >0 \text{ iff }j \in \bm j^{\bm I}_{k-1}
    \Big)
    \nonumber
    \\ 
    & = 
    \P\big(
        \tau^{n|\delta}_{i;j}(1) > n\delta \text{ iff }j \in \bm j^{\bm I}_1
    \big)
    \nonumber
    \\ 
    &\qquad
    \cdot 
    \prod_{ k = 2 }^{\mathcal K^{\bm I}}
    \P\Big(
        \tau^{n|\delta}_{i;j}(k) > 0 \text{ iff }j \in \bm j^{\bm I}_k
        \ \Big|\ 
        \tau^{n|\delta}_{i;j}(k-1) >n\delta^{k-1} \text{ iff }j \in \bm j^{\bm I}_{k-1}
    \Big)
    \quad 
    \text{by \eqref{property: lower bound for tau n delta i j k, cluster size}}
    \nonumber
    \\ 
    & \leq 
    \P\big(
        \tau^{n|\delta}_{i;j}(1) > n\delta \ \forall j \in \bm j^{\bm I}_1
    \big)
    \nonumber
    \\
    &\qquad
    \cdot 
    \prod_{ k = 2 }^{\mathcal K^{\bm I}}
    \P\Big(
        \tau^{n|\delta}_{i;j}(k) > 0\ \forall j \in \bm j^{\bm I}_k
        \ \Big|\ 
        \tau^{n|\delta}_{i;j}(k-1) >n\delta^{k-1} \text{ iff }j \in \bm j^{\bm I}_{k-1}
    \Big)
    \nonumber
    \\ 
    & \leq 
   \P\Big(
        W^{>}_{i;j}(n\delta) > 0\ \forall j \in \bm j^{\bm I}_1
    \Big)
    \nonumber
    \\ 
    &\quad 
    \cdot 
    \prod_{k = 2}^{\mathcal K^{\bm I}}\ 
    \sup_{  t_l \geq n\delta^{k-1}\ \forall l \in \bm j^{\bm I}_{k-1}  }
    \P\Bigg(
        \sum_{ l \in \bm j^{\bm I}_{k-1} }\sum_{m = 1}^{t_l}
        W^{>, (m) }_{ l; j  }(\delta \cdot t_l) > 0\ \forall j \in \bm j^{\bm I}_k
    \Bigg)
    \quad
    \text{by \eqref{def: layer zero, from pruned cluster to full cluster} and \eqref{def: from pruned cluster to full cluster, tau and S at step k + 1}}
    \nonumber
    \\ 
    & = 
    \P\Big(
        W^{>}_{i;j}(n\delta) > 0\ \forall j \in \bm j^{\bm I}_1
    \Big)
    \cdot 
    \prod_{k = 2}^{\mathcal K^{\bm I}}\ 
    \sup_{  t_l \geq n\delta^{k-1}\ \forall l \in \bm j^{\bm I}_{k-1}  }
    \P\Big(
        W^{>|\delta}_{ \bm t(\bm j^{\bm I}_{k-1}) ; j  } > 0\ \forall j \in \bm j^{\bm I}_k
    \Big)
    \nonumber
    \\ 
    &\qquad\qquad\qquad\qquad\qquad\qquad
    \text{using notations in \eqref{def, proof cluster size, N W mathcal I mathcal J bcdot j}, where we write $ \bm t(\bm j^{\bm I}_{k-1}) = (t_l)_{l \in \bm j^{\bm I}_{k-1} }$ }
    \nonumber
    \\ 
    & = 
    \underbrace{ \P\Big(
        N^{>}_{i;j}(n\delta) \geq 1\ \forall j \in \bm j^{\bm I}_1
    \Big) }_{ \delequal p_1(n,\delta) }
    \cdot 
    \prod_{k = 2}^{\mathcal K^{\bm I}}\ 
    \underbrace{ \sup_{  t_l \geq n\delta^{k-1}\ \forall l \in \bm j^{\bm I}_{k-1}  }
    \P\Big(
        N^{>|\delta}_{ \bm t(\bm j^{\bm I}_{k-1}) ; j  } \geq 1\ \forall j \in \bm j^{\bm I}_k
    \Big) }_{ \delequal p_k(n,\delta)  }
    \nonumber
    \\ 
    &\qquad\qquad\qquad\qquad\qquad\qquad
    \text{due to \eqref{property, W and N i M l j when positive, cluster size} and \eqref{property, sum of W and N i M l j when positive, cluster size}}.
    \nonumber
\end{align}
We first analyze $p_1(n,\delta)$.
If $|\bm j^{\bm I}_1| = 1$ (i.e., the set $\bm j^{\bm I}_1$ contains only one element),
we write $\bm j^{\bm I}_1 =  \{j_1\}$.
Using \eqref{claim 1, lemma: cluster size, asymptotics, N i | n delta, cdot j, crude estimate} and \eqref{claim 2, lemma: cluster size, asymptotics, N i | n delta, cdot j, crude estimate} in part (i), Lemma~\ref{lemma: cluster size, asymptotics, N i | n delta, cdot j, crude estimate},
for any $\delta > 0$ small enough,
\begin{align}
    p_1(n,\delta) = \bo \Big(\P\big( B_{j_1 \leftarrow l^*(j_1)} > n\delta\big)\Big)
    \text{ as }n\to \infty,
    \qquad
    \text{if }\bm j^{\bm I}_1 =  \{j_1\}.
    \label{proof: bound 1, one index at depth 1, lemma: crude estimate, type of cluster}
\end{align}
When $|\bm j^{\bm I}_1| \geq 2$, 
by part (ii) of Lemma~\ref{lemma: cluster size, asymptotics, N i | n delta, cdot j, crude estimate},
it holds for any $\delta > 0$ small enough that
\begin{align}
    p_1(n,\delta) = \lo \Bigg( n^{-1}\prod_{j \in \bm j^{\bm I}_1}n\P\big( B_{ j \leftarrow l^*(j)} > n\delta\big) \Bigg)
    \text{ as }n \to \infty,
    \qquad
    \text{if }|\bm j^{\bm I}_1| \geq 2.
    \label{proof: bound 1, lemma: crude estimate, type of cluster}
\end{align}
On the other hand, by part (iii) of Lemma~\ref{lemma: cluster size, asymptotics, N i | n delta, cdot j, crude estimate} under the choice of $c = \delta^{ \mathcal K^{\bm I} }$, it holds for any $\delta > 0$ small enough that 
\begin{align}
    p_k(n,\delta) = 
    \bo \Bigg( \prod_{j \in \bm j^{\bm I}_k}n\P\big( B_{j \leftarrow l^*(j)} > n\delta\big) \Bigg)
    \ \text{ as }n \to \infty,
    \qquad\forall k = 2,3,\ldots,\mathcal K^{\bm I}.
    \label{proof: bound 2, lemma: crude estimate, type of cluster}
\end{align}
Combining \eqref{proof: bound 1, one index at depth 1, lemma: crude estimate, type of cluster} (resp.\ \eqref{proof: bound 1, lemma: crude estimate, type of cluster}) with \eqref{proof: bound 2, lemma: crude estimate, type of cluster},
we establish \eqref{claim: lemma: crude estimate, type of cluster} (resp.\ \eqref{claim: lemma: crude estimate, |j 1 type I| geq 2, type of cluster}).
\end{proof}

Next, recall the definitions of $(\bar R^{n}_j, \bar \Theta^n_j)$ and $(\hat R^{n|\delta}_j, \hat \Theta^{n|\delta}_j)$ in \eqref{def: polar coordinates, bar S and hat S, cluster size}.
We prepare
Lemma~\ref{lemma: distance between bar s and hat S given type, cluster size} to 
bound the probability that 
$(\bar R^{n}_j, \bar \Theta^n_j)$ and $(\hat R^{n|\delta}_j, \hat \Theta^{n|\delta}_j)$  are not close to each other.

\begin{lemma}\label{lemma: distance between bar s and hat S given type, cluster size}
\linksinthm{lemma: distance between bar s and hat S given type, cluster size}
Let Assumptions~\ref{assumption: subcriticality}--\ref{assumption: regularity condition 2, cluster size, July 2024} hold.
Given $i \in [d]$, $\gamma,\epsilon> 0$, $\Delta \in (0,\epsilon]$, and $\bm I \in \widetilde{\mathscr I}$,
\begin{align*}
    \lim_{n \to \infty}
    n^\gamma\cdot 
    \P\Big(
        B^n_i(\delta,\bm I,\epsilon,\Delta)
    \Big) = 0,
    \qquad
    \forall \delta > 0\text{ small enough},
\end{align*}
with
\begin{align}
    & B_{i}^n(\delta,\bm I,\epsilon,\Delta)
    \label{proof: def event B n i, lemma: distance between bar s and hat S given type, cluster size}
    \\
    & \delequal
    \big\{
        \bm I^{n|\delta}_i = \bm I,\ 
        \bar R^n_i
        \vee \hat R^{n|\delta}_i
        > \epsilon
    \big\}
    \cap 
    \Bigg\{
        \frac{ \bar R^n_i  }{ \hat R^{n|\delta}_i } \notin [1-\Delta, 1+ \Delta]\text{ or }
        \norm{ \bar \Theta^n_i - \hat \Theta^{n|\delta}_i } > \Delta
    \Bigg\}.
    \nonumber
\end{align}
\end{lemma}

\begin{proof}
\linksinpf{lemma: distance between bar s and hat S given type, cluster size}
By Definition~\ref{def: cluster, generalized type},
the only generalized type $\bm I = (I_{k,j})_{k \geq 1,\ j \in [d]} \in \widetilde{\mathscr I}$
with depth $\mathcal K^{\bm I} = 0$ is $I_{k,j} \equiv 0\ \forall k,j$.
We first discuss the case where $\mathcal K^{\bm I} \geq 1$.
At the end of this proof, we address the case where $\mathcal K^{\bm I} = 0$.

We start by fixing some constants.
First,
recall the definition of $\bar{\bm s}_j = \E \bm S_j$ in \eqref{def: bar s i, ray, expectation of cluster size}.
Since the branching process for $\bm S_j$ contains at least the ancestor along the $j^\text{th}$ dimension, 
we have $\norm{\bar{\bm s}_j} \geq 1$ for each $j \in [d]$. 
Therefore, for
\begin{align}
    \rho = \min_{ j \in [d] }\norm{\bar{\bm s}_j} \Big/ \max_{j \in [d]}\norm{\bar{\bm s}_j},
    \label{proof: constant tilde Delta, 1, lemma: distance between bar s and hat S given type, cluster size}
\end{align}
we have $\rho \in (0,1)$.
Next, given $\gamma,\epsilon,\Delta > 0$,
we fix $\tilde\Delta > 0$ small enough such that
\begin{align}
    \tilde \Delta < \frac{\epsilon}{2},
    \qquad
    \frac{ \tilde \Delta }{
        \epsilon \cdot \rho/4
    }
    < \frac{\Delta}{4},
    \qquad 
    \frac{
        \tilde \Delta 
    }{
        \min_{j \in [d]}\norm{\bm s_j}
    }
    < \frac{\Delta}{4},
    \qquad
    \tilde \Delta < \frac{1}{2}\min_{j \in [d]}\norm{\bar{\bm s}_j}.
    \label{proof: constant tilde Delta, 2, lemma: distance between bar s and hat S given type, cluster size}
\end{align}
To proceed, recall that in Definition~\ref{def: cluster, generalized type}, we use $\mathcal K^{\bm I}$ to denote the depth of the generalized type $\bm I$,
and
$\bm j^{\bm I}_k$ for the set of active indices at depth $k$.
By \eqref{def: scaled cluster bar S n i}--\eqref{def: hat S n delta i, cluster size},
 on the event $\{ \bm I^{n|\delta}_i = \bm I \}$, it holds that
(using notations in \eqref{def: from pruned cluster to full cluster, tau and S at step k + 1}--\eqref{def: from pruned cluster to full cluster, tau and S at step k + 1, 4})
\begin{equation} \label{proof: def hat S n delta i, lemma: distance between bar s and hat S given type, cluster size}
    \begin{aligned}
        \hat{\bm S}^{n|\delta}_i
    & = 
    n^{-1}\sum_{k = 1}^{\mathcal K^{\bm I}}\sum_{ j \in \bm j^{\bm I}_k }\tau^{n|\delta}_{i;j}(k) \cdot \bar{\bm s}_j,
    \\ 
    \bar{\bm S}^n_i
    & = n^{-1}
    \Bigg[
        \bm S_i^{\leqslant, (1,1)}(n\delta)
        +
        \sum_{k = 1}^{\mathcal K^{\bm I}}\sum_{ j \in \bm j^{\bm I}_k }\sum_{m = 1}^{ \tau^{n|\delta}_{i;j}(k)  }
            \bm S^{\leqslant,(k+1,m)}_j\Big(\delta\cdot \tau^{n|\delta}_{i;j}(k)\Big)
    \Bigg].
    \end{aligned}
\end{equation}
Next, define the events
\begin{align}
    A^n_{0}(\delta,\bm I)
    & = 
    \Big\{
        \Big|\Big|
             \underbrace{ n^{-1}\bm S_i^{\leqslant, (1,1)}(n\delta) }_{ \delequal \bm \Delta_{0,i}  }
        \Big|\Big|\leq \tilde \Delta 
    \Big\},
    \label{proof, def, event A n 0, lemma: distance between bar s and hat S given type, cluster size}
\end{align}
and (for each $k \in [\mathcal K^{\bm I}],\ j \in \bm j^{\bm I}_k$), 
\begin{align}
    A^n_{k,j}(\delta,\bm I)
    & =
    \left\{\rule{0cm}{0.9cm}
        \norm{ 
            \underbrace{ 
                \big(\tau^{n|\delta}_{i;j}(k) \big)^{-1}
                \sum_{m = 1}^{ \tau^{n|\delta}_{i;j}(k)  }
            \bigg( \bm S^{\leqslant,(k+1,m)}_j\Big(\delta\cdot \tau^{n|\delta}_{i;j}(k)\Big) - \bar{\bm s}_j \bigg)
            }_{ \delequal \bm \Delta_{k,j} }
        }
        \leq \tilde \Delta 
    \right\}.
    \label{proof, def, event A n k j, lemma: distance between bar s and hat S given type, cluster size}
\end{align}
Also, define the event
\begin{align}
    A^n_{*}(\delta,\bm I) = 
    \big\{
        \bm I^{n|\delta}_i = \bm I,\ 
        \bar R^n_i
        \vee \hat R^{n|\delta}_i 
        > \epsilon
    \big\}
    \cap 
    A_{0}^n(\delta,\bm I)
    \cap
    \Bigg(
        \bigcap_{ k \in [\mathcal K^{\bm I}] }\bigcap_{ j \in \bm j^{\bm I}_k }
        A_{k,j}^n(\delta,\bm I)
    \Bigg).
    \label{proof: def, event A n *, lemma: distance between bar s and hat S given type, cluster size}
\end{align}
Suppose we can show that
\begin{align}
    \text{
         on the event $A^n_{*}(\delta,\bm I)$, it holds that
    }\quad 
    \frac{ \bar R^n_i  }{ \hat R^{n|\delta}_i }\in [1-\Delta, 1+ \Delta]\text{ and }
        \norm{ \bar \Theta^n_i - \hat \Theta^{n|\delta}_i } \leq \Delta.
    \label{proof, goal, lemma: distance between bar s and hat S given type, cluster size}
\end{align}
Then, by the definition in \eqref{proof: def event B n i, lemma: distance between bar s and hat S given type, cluster size},
we have 
$
B_{i}^n(\delta,\bm I,\epsilon,\Delta) \cap A^n_{*}(\delta,\bm I) = \emptyset,
$
and hence
\begin{align*}
    \P\Big(
        B^n_{i}(\delta,\bm I,\epsilon,\Delta)
    \Big)
    & \leq
    \P\Big(
       \big( A_{*}^n(\delta,\bm I)\big)^\complement
    \Big)
    \\ 
    & \leq 
    \P\Big(
        \big( A^n_{0}(\delta,\bm I) \big)^\complement
    \Big)
    +
     \sum_{k = 1}^{\mathcal K^{\bm I}}\sum_{ j \in \bm j^{\bm I}_k }
     \P\Big(
        \big(
            A^n_{k,j}(\delta,\bm I)
        \big)^\complement
     \Big).
\end{align*}
By Lemma~\ref{lemma: tail bound, pruned cluster size S i leq n delta},
it holds for any $\delta > 0$ small enough that 
$
\P\Big(
        \big( A^n_{0}(\delta,\bm I) \big)^\complement
    \Big)
     =
    \P\Big(
        \norm{ \bm S^{\leqslant}_i(n\delta) } > n\tilde\Delta
    \Big)
    =
    \lo ( n^{-\gamma} ).
$
Meanwhile, for each $k \in [\mathcal K^{\bm I}]$ and $j \in \bm j^{\bm I}_k$, 
\begin{align*}
    & \P\Big(
        \big(
            A^n_{k,j}(\delta,\bm I)
        \big)^\complement
     \Big)
     \\
     & \leq 
     \P\Bigg(
        \norm{
            \frac{1}{N}\sum_{m = 1}^N \bm S^{\leqslant, (m)}_j(N\delta) - \bar{\bm s}_j
        }
        > \tilde\Delta
        \text{ for some }N \geq \ceil{n\delta^k} 
     \Bigg)
     \qquad
     \text{due to \eqref{property: lower bound for tau n delta i j k, cluster size}}
     \\ 
     & \leq 
     \sum_{N \geq \ceil{n\delta^k}}
     \underbrace{ \P\Bigg(
        \norm{
            \frac{1}{N}\sum_{m = 1}^N \bm S^{\leqslant, (m)}_j(N\delta) - \bar{\bm s}_j
        }
        > \tilde\Delta
     \Bigg)}_{ = p(N,\delta)}.
\end{align*}
By Lemma~\ref{lemma: concentration ineq for pruned cluster S i},
it holds for any $\delta > 0$ small enough that 
$
p(N,\delta) = \lo ( N^{-\gamma - 2})
$
(as $N \to \infty$).
As a result, given any $\delta > 0$ sufficiently small,
there exists $\bar N(\delta) \in (0,\infty)$ such that 
$
p(N,\delta) \leq N^{-\gamma -2}
$
 $\forall N \geq \bar N(\delta)$.
Then, for any $n$ large enough such that $n\delta^k \geq \bar N(\delta)$, we have
$
\sum_{ N \geq \ceil{n\delta^k} } p(N,\delta) = \bo( n^{ -\gamma - 1  } )
     =
     \lo ( n^{ -\gamma  } ).
$
In summary, 
we have established that
$
\P\big(
        B^n_{i}(\delta,\bm I,\epsilon,\Delta)
    \big)
=
\lo (n^{-\gamma})
$
for any $\delta > 0$ small enough.
This concludes the proof for the case where $\mathcal K^{\bm I} \geq 1$.
Now, it remains to prove Claim~\eqref{proof, goal, lemma: distance between bar s and hat S given type, cluster size},
under the choice of $\tilde \Delta$ in \eqref{proof: constant tilde Delta, 2, lemma: distance between bar s and hat S given type, cluster size}.

\medskip
\noindent
\textbf{Proof of Claim~\eqref{proof, goal, lemma: distance between bar s and hat S given type, cluster size}}.
Using notations in \eqref{proof, def, event A n 0, lemma: distance between bar s and hat S given type, cluster size}, \eqref{proof, def, event A n k j, lemma: distance between bar s and hat S given type, cluster size},
on the event $A^n_{*}(\delta,\bm I)$ we have 
\begin{align}
    \bar{\bm S}^n_i
    & = 
    \bm \Delta_{0,i}
    +
    \sum_{k = 1}^{\mathcal K^{\bm I}}\sum_{ j \in \bm j^{\bm I}_k }
    n^{-1}\tau^{n|\delta}_{i;j}(k) \cdot \big( \bar{\bm s}_j + \bm \Delta_{k,j}\big),
    \label{proof: expression for bar S n i, lemma: distance between bar s and hat S given type, cluster size}
\end{align}
with $\norm{\bm \Delta_{k,j}} \leq \tilde \Delta $
for each $k = 0,1,\ldots,\mathcal K^{\bm I}$ and $j \in [d]$.
First, we show that on the event $A_{*}^n(\delta,\bm I)$,
\begin{enumerate}[(i)]
    \item 
        $
           \bar R^n_i \wedge  \hat R^{n|\delta}_i > \epsilon \rho/4;
        $
        
    \item
        $
        \norm{
            \bar{\bm S}^{n}_i - \hat{\bm S}^{n|\delta}_i
        }
        \big/ 
        \norm{ \hat{\bm S}^{n|\delta}_i  } \leq \Delta/2;
        $

    \item 
        $
        \norm{ \bar\Theta^n_i - \hat \Theta^{n|\delta}_i  } \leq \Delta. 
        $
\end{enumerate}
To prove (i),
note that under the $L_1$ norm $\norm{\cdot}$, we have
\begin{align}
    \hat{R}^{n|\delta}_i = 
    \norm{ 
        \hat{\bm S}^{n|\delta}_i
    }
    =
    \sum_{k = 1}^{\mathcal K^{\bm I}}\sum_{ j \in \bm j^{\bm I}_k }
        n^{-1}\tau^{n|\delta}_{i;j}(k) \cdot \norm{ \bar{\bm s}_j }.
    \label{proof: expression for hat R, lemma: distance between bar s and hat S given type, cluster size}
\end{align}
Likewise, in \eqref{proof: def hat S n delta i, lemma: distance between bar s and hat S given type, cluster size}, the coordinates of each $\bm S^{\leqslant,(k,m)}_i(n\delta)$ are non-negative by definition,
which implies
\begin{align}
    \bar R^n_i 
    =
    \norm{
        \bar{\bm S}^n_i
    }
    =
    \norm{
        \bm \Delta_{0,i}
    }
    +
     \sum_{k = 1}^{\mathcal K^{\bm I}}\sum_{ j \in \bm j^{\bm I}_k }
        n^{-1}\tau^{n|\delta}_{i;j}(k) \cdot \norm{ \bar{\bm s}_j + \bm \Delta_{k,j} }.
    \label{proof: expression for bar R, lemma: distance between bar s and hat S given type, cluster size}
\end{align}
By definitions in \eqref{proof: def, event A n *, lemma: distance between bar s and hat S given type, cluster size},
on the event $A^n_{*}(\delta,\bm I)$ we have 
$\bar R^n_i > \epsilon$ or $\hat R^{n|\delta}_i > \epsilon$.
We first consider the case of $\bar R^n_i > \epsilon$.
By \eqref{proof: expression for bar R, lemma: distance between bar s and hat S given type, cluster size},
on the event $A^n_{*}(\delta,\bm I)$ we have
$
\epsilon < \tilde \Delta + 
\sum_{k = 1}^{\mathcal K^{\bm I}}\sum_{ j \in \bm j^{\bm I}_k }
        n^{-1}\tau^{n|\delta}_{i;j}(k) \cdot  \Big( \max_{j \in [d]}\norm{ \bar{\bm s}_j} + \tilde \Delta\Big).
$
Under the choice of $\tilde \Delta$ in \eqref{proof: constant tilde Delta, 2, lemma: distance between bar s and hat S given type, cluster size}, it then holds on the event $A^n_{*}(\delta,\bm I)$
that
$
\sum_{k = 1}^{\mathcal K^{\bm I}}\sum_{ j \in \bm j^{\bm I}_k }
        n^{-1}\tau^{n|\delta}_{i;j}(k)
        >
        \frac{
            \epsilon
        }{
            4\max_{j \in [d]}\norm{\bar{\bm s}_j}
        }.
$
Together with \eqref{proof: expression for hat R, lemma: distance between bar s and hat S given type, cluster size},
we confirm that (on the event $A^n_{*}(\delta,\bm I)$)
\begin{align*}
    \hat R^{n|\delta}
    & \geq
    \min_{j \in [d]}\norm{\bar{\bm s}_j}\cdot
    \sum_{k = 1}^{\mathcal K^{\bm I}}\sum_{ j \in \bm j^{\bm I}_k }
    n^{-1}\tau^{n|\delta}_{i;j}(k)
    >
    \frac{\epsilon}{4} \cdot \frac{\min_{j \in [d]}\norm{\bar{\bm s}_j}}{\max_{j \in [d]}\norm{\bar{\bm s}_j}}
    = \epsilon \rho/4;
    \quad \text{see \eqref{proof: constant tilde Delta, 1, lemma: distance between bar s and hat S given type, cluster size}.}
\end{align*}
Similarly, if $\hat{R}^{n|\delta}_i > \epsilon$,
then by \eqref{proof: expression for hat R, lemma: distance between bar s and hat S given type, cluster size},
we get
$
 \sum_{k = 1}^{\mathcal K^{\bm I}}\sum_{ j \in \bm j^{\bm I}_k }
        n^{-1}\tau^{n|\delta}_{i;j}(k)
    > 
    \frac{\epsilon}{
        \max_{j \in [d]}\norm{\bar{\bm s}_j}
    },
$
and hence
\begin{align*}
    \bar R^n_i
    \geq 
    \sum_{k = 1}^{\mathcal K^{\bm I}}\sum_{ j \in \bm j^{\bm I}_k }
        n^{-1}\tau^{n|\delta}_{i;j}(k)
        \cdot \bigg( \min_{j \in [d]}\norm{ \bar{\bm s}_j } - \tilde{\Delta}\bigg)
    & > 
    \epsilon \rho/2
    \qquad
    \text{ by \eqref{proof: expression for bar R, lemma: distance between bar s and hat S given type, cluster size}.}
\end{align*}
This concludes the proof of Claim (i).
Next, 
it follows
from \eqref{proof: def hat S n delta i, lemma: distance between bar s and hat S given type, cluster size} and \eqref{proof: expression for bar S n i, lemma: distance between bar s and hat S given type, cluster size}
that 
$
\norm{
            \bar{\bm S}^{n}_i - \hat{\bm S}^{n|\delta}_i
        }
\leq \tilde \Delta + 
\sum_{k = 1}^{\mathcal K^{\bm I}}
\sum_{ j \in \bm j^{\bm I}_k }
n^{-1}\tau^{n|\delta}_{i;j}(k) \cdot \tilde \Delta.
$
This leads to
\begin{align*}
    \frac{
        \norm{
            \bar{\bm S}^{n}_i - \hat{\bm S}^{n|\delta}_i
        }
    }{
        \norm{ \hat{\bm S}^{n|\delta}_i }
    }
    & \leq 
    \frac{
        \tilde \Delta 
    }{
        \norm{ \hat{\bm S}^{n|\delta}_i }        
    }
    +
    \frac{
        \sum_{k = 1}^{\mathcal K^{\bm I}}\sum_{ j \in \bm j^{\bm I}_k }
        n^{-1}\tau^{n|\delta}_{i;j}(k) \cdot \tilde \Delta
    }{
        \norm{ \hat{\bm S}^{n|\delta}_i }        
    }
    \\ 
    & \leq 
    \frac{
        \tilde \Delta 
    }{
        \epsilon \cdot \rho/4
    }
    +
    \frac{
        \tilde \Delta \cdot 
        \sum_{k = 1}^{\mathcal K^{\bm I}}\sum_{ j \in \bm j^{\bm I}_k }n^{-1}\tau^{n|\delta}_{i;j}(k)
    }{
        \min_{j \in [d]}\norm{\bar{\bm s}_j} \cdot 
        \sum_{k = 1}^{\mathcal K^{\bm I}}\sum_{ j \in \bm j^{\bm I}_k }n^{-1}\tau^{n|\delta}_{i;j}(k)
    }
    \ \ 
    \text{by Claim (i) and \eqref{proof: expression for hat R, lemma: distance between bar s and hat S given type, cluster size}}
    \\
    & \leq 
    \frac{\Delta}{4} + \frac{\Delta}{4} = \frac{\Delta}{2}
    \qquad
    \text{ by our choice of $\tilde \Delta$ in \eqref{proof: constant tilde Delta, 2, lemma: distance between bar s and hat S given type, cluster size}.}
\end{align*}
This verifies Claim (ii).
For Claim (iii), 
note again that $\bar R^n_i \wedge  \hat R^{n|\delta}_i > 0$  on the event $A^n_{*}(\delta,\bm I)$
by Claim (i).
Then, by the definition in \eqref{def: polar coordinates, bar S and hat S, cluster size},
\begin{align*}
    \norm{ \bar\Theta^n_i - \hat \Theta^{n|\delta}_i  }
    & = 
    \norm{
    \frac{
        \bar{\bm S}^n_i
    }{
        \norm{\bar{\bm S}^n_i}
    }
    -
    \frac{
        \hat{\bm S}^{n|\delta}_i
    }{
        \norm{ \hat{\bm S}^{n|\delta}_i }
    }
    }
    \leq 
    \norm{
        \frac{
          \hat{\bm S}^{n|\delta}_i - \bar{\bm S}^n_i
        }{
            \norm{ \hat{\bm S}^{n|\delta}_i}
        }
    }
    +
    \norm{
        \bar{\bm S}^n_i
        \cdot 
        \bigg(
            \frac{1}{
                 \norm{ \hat{\bm S}^{n|\delta}_i}
            }
            -
            \frac{1}{
                \norm{\bar{\bm S}^n_i}
            }
        \bigg)
    }
    \\ 
    & = 
    \frac{
        \norm{\hat{\bm S}^{n|\delta}_i - \bar{\bm S}^n_i}
    }{
         \norm{ \hat{\bm S}^{n|\delta}_i}
    }
    +
    \norm{\bar{\bm S}^n_i}
    \cdot 
    \frac{
    \Big|
        \norm{\bar{\bm S}^n_i} - \norm{ \hat{\bm S}^{n|\delta}_i}
    \Big|
    }{
        \norm{\bar{\bm S}^n_i} \norm{ \hat{\bm S}^{n|\delta}_i}
    }
    \leq 
     \frac{
        \norm{\hat{\bm S}^{n|\delta}_i - \bar{\bm S}^n_i}
    }{
         \norm{ \hat{\bm S}^{n|\delta}_i}
    }
    +
     \frac{
        \norm{\hat{\bm S}^{n|\delta}_i - \bar{\bm S}^n_i}
    }{
         \norm{ \hat{\bm S}^{n|\delta}_i}
    }
    \\ 
    & \leq 
    2 \cdot \frac{\Delta}{2} = \Delta
    \qquad
    \text{by Claim (ii)}.
\end{align*}
This verifies Claim (iii).
Lastly, Claims (ii) and (iii) imply that, on the event $A_{*}^n(\delta,\bm I)$, we have
$
{ \bar R^n_i  }\big/{ \hat R^{n|\delta}_i } \in [1-\Delta, 1+ \Delta]
$
and
$
\norm{ \bar \Theta^n_i - \hat \Theta^{n|\delta}_i } \leq \Delta.
$
This concludes the proof of Claim~\eqref{proof, goal, lemma: distance between bar s and hat S given type, cluster size}.

\medskip
\noindent
\textbf{Proof of the case with $\mathcal K^{\bm I} = 0$}.
If $\mathcal K^{\bm I} = 0$ (i.e., $I_{k,j} \equiv 0\ \forall k \geq 1, j \in [d]$),
it holds on the event $\{\bm I^{n|\delta}_i  = \bm I\}$ that
$\hat{\bm S}^{n|\delta}_i = \bm 0$ and $\bar{\bm S}^n_i = n^{-1}\bm S_i^{\leqslant, (1,1)}(n\delta)$.
Therefore, 
on the event 
$
\big\{
        \bm I^{n|\delta}_i = \bm I,\ 
        \bar R^n_i \vee \hat R^{n|\delta}_i > \epsilon
    \big\},
$
we have $\norm{  n^{-1}\bm S_i^{\leqslant, (1,1)}(n\delta) } > \epsilon$.
Applying Lemma~\ref{lemma: tail bound, pruned cluster size S i leq n delta} again,
we get
$
    \P\Big(
        \norm{\bm S_i^{\leqslant,(1,1)}(n\delta)} > n\epsilon
    \Big) = \lo (n^{-\gamma})
$
for any $\delta > 0$ small enough.
\end{proof}

We are now ready to state the proof of Proposition~\ref{proposition: asymptotic equivalence, tail asymptotics for cluster size}. 

\begin{proof}[Proof of Proposition~\ref{proposition: asymptotic equivalence, tail asymptotics for cluster size}]
\linksinpf{proposition: asymptotic equivalence, tail asymptotics for cluster size}


Define  the event
\begin{align}
    E^n_\delta(\Delta) \delequal
        \big\{
            \bar R^n_i \vee \hat R^{n|\delta}_i > \Delta
        \big\}
        \cap
        \bigg\{
            \frac{ \bar R^n_i }{ \hat R^{n|\delta}_i } \notin [1-\Delta,1+\Delta]
            \text{ or }
            \norm{ \bar\Theta^n_i - \hat \Theta^{n|\delta}_i } > \Delta
        \bigg\}.
        \nonumber
\end{align}
First, 
due to the arbitrariness of $\Delta > 0$ in Claim \eqref{claim, proposition: asymptotic equivalence, tail asymptotics for cluster size}
and the simple fact (for any $c \in (0,1)$ and $r,r^\prime \geq 0$)
\begin{align*}
    r \vee r^\prime > 0,\ \frac{r}{r^\prime} \notin [1 -c, 1+ c]
    \quad 
    \Longrightarrow
    \quad
    r \vee r^\prime > 0,\ \frac{r^\prime}{r} \notin [1/(1+c), 1/(1- c)],
\end{align*}
it is equivalent to show that, given $\Delta \in (0,1)$,
it holds for any $\delta > 0$ small enough that
$
\P\big( E^n_\delta(\Delta)\big) = \lo (n^{-\gamma}).
$
Next,
recall the definitions of $\tilde{\alpha}(\cdot)$ in \eqref{def: tilde c I, cost for generalized type, cluster size}
and $\alpha(\cdot)$ in \eqref{def: cost function, cone, cluster}.
Due to
$
\P\big(E^n_\delta(\Delta)\big)
    = 
    \P\Big(E^n_\delta(\Delta) \cap \big\{ \tilde{\alpha}\big(\bm I^{n|\delta}_i\big) > \alpha(\bm j)  \big\} \Big)
    +
    \P\Big(E^n_\delta(\Delta)\cap \big\{ \tilde{\alpha}\big(\bm I^{n|\delta}_i\big) \leq \alpha(\bm j)  \big\} \Big),
$
it suffices to show that (for any $\delta > 0$ small enough)
\begin{align}
    \lim_{n \to \infty} \big(\lambda_{\bm j}(n)\big)^{-1}
    \P\Big( \tilde{\alpha}\big(\bm I^{n|\delta}_i\big) > \alpha(\bm j) \Big)
    & = 0,
    \label{proof: goal 1, asymptotic equivalence, tail asymptotics for cluster size}
    \\ 
    \lim_{n \to \infty} \big(\lambda_{\bm j}(n)\big)^{-1}
    \P\Big(E^n_\delta(\Delta) \cap \big\{ \tilde{\alpha}\big(\bm I^{n|\delta}_i\big) \leq \alpha(\bm j)  \big\} \Big)
    & = 0.
    \label{proof: goal 2, asymptotic equivalence, tail asymptotics for cluster size}
\end{align}

\medskip
\noindent
\textbf{Proof of Claim~\eqref{proof: goal 1, asymptotic equivalence, tail asymptotics for cluster size}}.
We first define the set (where the partial ordering for generalized types is defined in \eqref{proof: def ordering of generalized types, cluster size})
\begin{align*}
    \partial \mathscr I(\bm j) \delequal
    \Big\{
        \bm I \in \widetilde{\mathscr I}:\ 
        \tilde{\alpha}(\bm I) > \alpha(\bm j);\ 
        \tilde{\alpha}(\hat{\bm I}) \leq \alpha(\bm j)
        \text{ for any $\hat{\bm I} \in \widetilde{\mathscr I}$ with }
        \hat{\bm I} \subseteq \bm I,\ \hat{\bm I} \neq \bm I 
    \Big\},
\end{align*}
and stress the following: 
for any $\bm I^\prime \in \widetilde{\mathscr I}$ with $\tilde{\alpha}(\bm I^\prime) > \alpha(\bm j)$, there must be some $\bm I \subseteq \bm I^\prime$ such that $\bm I \in \partial \mathscr I(\bm j).$
To see why, note that 
 the function $\tilde{\alpha}(\cdot)$ in \eqref{def: tilde c I, cost for generalized type, cluster size} is linear w.r.t.\ the $I_{k,j}$'s, and the coefficients $(\alpha^*(j) - 1)_{j \in [d]}$ are strictly positive under Assumption~\ref{assumption: heavy tails in B i j}.
Then, given any $\bm I^\prime = (I^\prime_{k,j})_{k \geq 1, j \in [d]} \in \widetilde{\mathscr I}$ with $\tilde{\alpha}(\bm I^\prime) > \alpha(\bm j)$,
by identifying the smallest $K \in [\mathcal K^{\bm I^\prime}]$ satisfying
\begin{align*}
    1 + \sum_{k = 1}^K\sum_{j \in [d]}I^\prime_{k,j}\big(\alpha^*(j) - 1\big) > \alpha(\bm j),
\end{align*}
and setting $\bm I = (I_{k,j})_{k \geq 1, j \in [d]}$ with $I_{k,j} = I^{\prime}_{k,j}\ \forall j$ if $k \leq K$ and $I_{k,j} \equiv 0$ if $k > K$,
we have 
$\bm I \subseteq \bm I^\prime$ and $\bm I \in \partial \mathscr I(\bm j).$
Also, due to $\alpha^*(j) - 1> 0$ for each $j \in [d]$,
the set $\partial \mathscr I(\bm j)$ contains only finitely many elements.
In summary, we get
$
\big\{ \tilde{\alpha}\big(\bm I^{n|\delta}_i\big) > \alpha(\bm j)  \big\}
    \subseteq
    \bigcup_{ \bm I \in  \partial \mathscr I(\bm j)}\big\{ \bm I \subseteq \bm I^{n|\delta}_i \big\},
$
and it suffices to show that, given $\bm I \in \partial \mathscr I(\bm j)$, it holds for all $\delta > 0$ small enough that 
$
\P\big( \bm I \subseteq \bm I^{n|\delta}_i\big) = \lo \big(\lambda_{\bm j}(n)\big)
$
as $n \to \infty$.
To proceed, let
\begin{align}
    \tilde \lambda^{\bm I}(n) \delequal 
n^{-1}
        \prod_{k = 1}^{ \mathcal K^{\bm I} }\prod_{ j \in \bm j^{\bm I}_k }
        n\P\big(B_{j \leftarrow l^*(j)} > n\delta\big),
\qquad
\forall \bm I \in \widetilde{\mathscr I},
\label{proof, def tilde lambda generalized type I}
\end{align}
and fix some $\bm I \in  \partial \mathscr I(\bm j)$.
First, due to $\tilde \alpha(\bm I) > \alpha(\bm j)$, we have $\mathcal K^{\bm I} \geq 1$ (otherwise, we get $\tilde \alpha(\bm I) = 0$ by definition).
Then,
by Lemma~\ref{lemma: crude estimate, type of cluster},
it holds for any $\delta > 0$ small enough that 
 $
 \P\big( \bm I \subseteq \bm I^{n|\delta}_i\big)
 = \bo \big(\tilde \lambda^{\bm I}(n) \big).
 $
Due to $\lambda_{\bm j}(n) \in \RV_{ -\alpha(\bm j)  }(n)$,
$
\tilde \lambda^{\bm I}(n) \in \RV_{ -\tilde\alpha(\bm I)  }(n),
$
and 
$
\tilde{\alpha}(\bm I) > \alpha(\bm j),
$
we get 
$
\P\big( \bm I \subseteq \bm I^{n|\delta}_i\big)  = \lo \big(\lambda_{\bm j}(n)\big).
$
This concludes the proof of Claim~\eqref{proof: goal 1, asymptotic equivalence, tail asymptotics for cluster size}.

\medskip
\noindent
\textbf{Proof of Claim~\eqref{proof: goal 2, asymptotic equivalence, tail asymptotics for cluster size}}.
Due to $\alpha^*(j) > 1$ for each $j \in [d]$ (see \eqref{def: cluster size, alpha * l * j} and Assumption~\ref{assumption: heavy tails in B i j}),
there are only finitely many generalized types $\bm I \in \widetilde{\mathscr I}$ such that $\tilde{\alpha}(\bm I) \leq \alpha(\bm j)$.
Therefore, it suffices to fix one of such $\bm I$ and show that
$
\P\big( \{ \bm I^{n|\delta}_i = \bm I  \} \cap  E^n_\delta(\Delta) \big) = \lo\big( \lambda_{\bm j}(n)  \big) 
$
(as $n \to \infty$) for any $\delta > 0$ small enough.
Applying Lemma~\ref{lemma: distance between bar s and hat S given type, cluster size}, we conclude the proof of
Claim~\eqref{proof: goal 2, asymptotic equivalence, tail asymptotics for cluster size}.
\end{proof}

\subsection{Proof of Proposition~\ref{proposition, M convergence for hat S, tail asymptotics for cluster size}}
\label{subsec: proof, proposition, M convergence for hat S, tail asymptotics for cluster size}

We first prepare a few technical lemmas.
Lemma~\ref{lemma: type and generalized type} states useful properties of 
types (Definition~\ref{def: cluster, type}) and
generalized types (Definition~\ref{def: cluster, generalized type}).

\begin{lemma}\label{lemma: type and generalized type}
\linksinthm{lemma: type and generalized type}
Let Assumption~\ref{assumption: heavy tails in B i j} hold.
For any $\bm j \subseteq [d]$ that is non-empty, let
\begin{align}
    \mathscr I(\bm j)
    \delequal
    \{\bm I \in \mathscr I:\ \bm j^{\bm I} = \bm j  \},
    \quad  
    \widetilde{\mathscr I}(\bm j)\delequal
    \{
    \bm I \in \widetilde{\mathscr I}:\ 
     \bm j^{\bm I} = \bm j,\ \tilde{\alpha}(\bm I) = \alpha(\bm j)
    \},
    \label{def: mathscr I bm j, lemma: type and generalized type}
\end{align}
where $\mathscr I$ is the set of all types, $\widetilde{\mathscr I}$ is the set of all generalized types, 
$\bm j^{\bm I}$ is the set of active indices of $\bm I$ (see Definition~\ref{def: cluster, generalized type}),
and $\tilde{\alpha}(\cdot)$, $\alpha(\cdot)$ are defined in \eqref{def: tilde c I, cost for generalized type, cluster size} and \eqref{def: cost function, cone, cluster}, respectively.
The following claims hold for any  non-empty $\bm j \subseteq [d]$:
\begin{enumerate}[(i)]
    \item 
        $\mathscr I(\bm j) \subseteq \widetilde{\mathscr I}(\bm j)$;

    \item 
        It holds for any $\bm I \in \widetilde{\mathscr I}(\bm j) \setminus \mathscr I(\bm j)$ that $|\bm j^{\bm I}_1| \geq 2$.
\end{enumerate}
\end{lemma}

Lemma~\ref{lemma: choice of bar epsilon bar delta, cluster size} studies geometric properties
of sets that are bounded away from $\C^d_\leqslant(\bm j)$.

\begin{lemma}\label{lemma: choice of bar epsilon bar delta, cluster size}
\linksinthm{lemma: choice of bar epsilon bar delta, cluster size}
Let Assumption~\ref{assumption: heavy tails in B i j} hold.
Let $\bm j \subseteq [d]$ be non-empty, and $B \subseteq [0,\infty) \times \mathfrak N^d_+$ be  a Borel set that is bounded away from $\C^d_\leqslant(\bm j)$ (see \eqref{def: cone C d leq j, cluster size}) under $\bm d_\textbf{U}$ (see \eqref{def: metric for polar coordinates}).
\begin{enumerate}[(a)]
    \item 
        There exist $\bar\epsilon > 0$ and $\bar\delta > 0$ such that the following claims hold:
        given $\bm x = \sum_{i \in \bm j}w_i\bar{\bm s}_i$ with $w_i \geq 0\ \forall i \in \bm j$, 
        if $\Phi(\bm x) \in B$ (with $\Phi(\cdot)$ defined in \eqref{def: Phi, polar transform}),
        then
        \begin{itemize}
            \item
                $\min_{j \in \bm j}w_j \geq \bar\epsilon$,\smallskip
                
            \item 
                ${ w_j }/{ w_i  } \geq \bar\delta$ for any $j \in \bm j$ and $i \in \bm j$;\smallskip
        \end{itemize}

    \item 
        $
        \mathbf C^{\bm I} \circ \Phi^{-1}(B) < \infty
        $
        for any $\bm I \in \mathscr I$ with $\bm j^{\bm I} = \bm j$,
        where
        $\bm j^{\bm I}$ is the set of active indices of type $\bm I$ in Definition~\ref{def: cluster, type},
        $\mathbf C^{\bm I}$ is defined in \eqref{def: measure C I, cluster},
        and $\mu \circ \Phi^{-1}(B) = \mu\big( \Phi^{-1}(B) \big)$.
\end{enumerate}
\end{lemma}

These two results follow directly from the definitions of types, the functions $\tilde\alpha(\cdot)$, $\alpha(\cdot)$, and the measure $\mathbf C^{\bm I}$.
We collect the proofs of Lemmas~\ref{lemma: type and generalized type} and \ref{lemma: choice of bar epsilon bar delta, cluster size} in Section~\ref{subsec: proof, technical lemmas, cluster size} of the Appendix.
Next, we prepare results for the asymptotic analysis of $\big(n^{-1}\tau^{n|\delta}_{i;j}(k)\big)_{k \geq 1, j \in [d]}$.
For each $\bm I \in \mathscr I$ and $M,c > 0$, we define the set
\begin{equation}\label{proof: def, set B type I M c, cluster size}
    \begin{aligned}
        & \notationdef{notation-set-hat-B-type-I-M-c}{E^{\bm I}(M,c)} \delequal 
    \Bigg\{
        (w_{k,j})_{k \geq 1, j \in [d]} \in [0,\infty)^{ \infty \times d }:
        \\ 
        &\qquad
        w_{k,j} > M\text{ and } \min_{ k^\prime \in [\mathcal K^{\bm I}], j^\prime \in \bm j^{\bm I}_{k^\prime} }\frac{ w_{k,j}  }{ w_{k^\prime, j^\prime} } \geq c
        \ \forall k \in [\mathcal K^{\bm I}],\ j \in \bm j^{\bm I}_{k};\ 
        w_{k,j} = 0 \ \forall k \geq 1,\ j \notin \bm j^{\bm I}_k
    \Bigg\},
    \end{aligned}
\end{equation}
where $\mathcal K^{\bm I}$ is the depth of type $\bm I$, and $\bm j^{\bm I}_k$ is the set of active indices of type $\bm I$ at depth $k$; see Definition~\ref{def: cluster, type}.
In particular, for any $k > \mathcal K^{\bm I}$ we have $\bm j^{\bm I}_k = \emptyset$ by definition.
Meanwhile, recall the definitions of $\tau^{n|\delta}_{i;j}(k)$ in \eqref{def: layer zero, from pruned cluster to full cluster} and \eqref{def: from pruned cluster to full cluster, tau and S at step k + 1}.
In this section, 
we write
\begin{align}
    \notationdef{notation-bold-tau-n-delta-i}{\bm \tau^{n|\delta}_{i}} \delequal
    \Big( \tau^{n|\delta}_{i;j}(k) \Big)_{ k \geq 1, j \in [d] }.
    \label{proof: def bm tau n delta i, cluster size }
\end{align}
Lemma~\ref{lemma: prob of type I, tail prob of tau, cluster size} bounds
 the probability that $n^{-1}\bm \tau^{n|\delta}_{i}$ lies outside of a bounded set.

\begin{lemma}\label{lemma: prob of type I, tail prob of tau, cluster size}
\linksinthm{lemma: prob of type I, tail prob of tau, cluster size}
Let Assumptions~\ref{assumption: subcriticality}--\ref{assumption: regularity condition 2, cluster size, July 2024} hold.
Let $i \in [d]$, $c\in (0,1)$, and let the type $\bm I = (I_{k,j})_{ k \geq 1, j \in [d] } \in \mathscr I$ be such that $\mathcal K^{\bm I} \geq 1$ (i.e., $I_{k,j} \geq 1$ for some $k,j$).
There exists some function $C^{\bm I}_i(\cdot,c)$ with $\lim_{M \to \infty}C^{\bm I}_i(M,c) = 0$ such that,
for each $M > 0$,
\begin{align}
    \limsup_{ n \to \infty }
    \Big(\lambda_{ \bm j^{\bm I} }(n)\Big)^{-1}
    \P\Big(
        n^{-1}\bm \tau^{n|\delta}_i \in E^{\bm I}(M,c)
    \Big) \leq C^{\bm I}_i(M,c),
    \quad
    \forall \delta > 0\text{ small enough},
    \nonumber
\end{align}
where $\bm j^{\bm I}$ is the set of active indices of type $\bm I$ (see Definition~\ref{def: cluster, type}),
$\lambda_{\bm j}(n)$ is defined in \eqref{def: rate function lambda j n, cluster size},
${\bm I^{n|\delta}_i} = \big(I^{n|\delta}_{i;j}(k)\big)_{k \geq 1,\ j \in [d]}$ is defined in \eqref{def, I k M j for M type, cluster size},
and $E^{\bm I}(M,c)$ is defined in \eqref{proof: def, set B type I M c, cluster size}.
\end{lemma}

Next, for each type $\bm I = (I_{k,j})_{k \geq 1, j \in [d]} \in \mathscr I$ with $\mathcal K^{\bm I} \geq 1$ and each $i \in [d]$, define
\begin{align}
    \notationdef{notation-hat-mathbf-C-i-type-I-cluster-size}{\widehat{\mathbf C}^{\bm I}(\cdot)}
        & \delequal
    \int \mathbbm{I}\bigg\{
        (w_{k,j})_{ k \in [ \mathcal K^{\bm I}] ,\ j \in \bm j_k^{\bm I} } \in \ \bcdot\ 
    \bigg\}
    \bigg(
    \prod_{ k = 1 }^{ \mathcal K^{\bm I} - 1}
        g_{ \bm j_k^{\bm I} \leftarrow \bm j_{k+1}^{\bm I} }(\bm w_k)
    \bigg)
    \nu^{\bm I}(d \bm w),
    \quad 
    {\widehat{\mathbf C}^{\bm I}_i} \delequal \bar s_{i,l^*(j^{\bm I}_1)}\cdot {\widehat{\mathbf C}^{\bm I}_i}
    \label{def, measure hat C i type I, cluster size}
\end{align}
where
${\bm j^{\bm I}_k}$ is the set of active indices at depth $k$ of type $\bm I$,
the mapping
$g_{ \mathcal I \leftarrow \mathcal J }(\bm w)$ is defined in \eqref{def: function g mathcal I mathcal J, for measure C i bm I, cluster size},
$\nu^{\bm I}(d\bm w)$ is defined in \eqref{def, measure nu type I, cluster size},
and, as noted in Remark~\ref{remark: def of type}, $j_1^{\bm I}$ is the unique index in $\{1,2,\ldots,d\}$ such that $\bm j^{\bm I}_1 = \{j^{\bm I}_1\}$.
Also, for any type $\bm I \in \mathscr I$ with $\mathcal K^{\bm I} \geq 1$, let 
\begin{align}
    h^{\bm I}(\bm w) \delequal \sum_{ k \in [\mathcal K^{\bm I}] }\sum_{ j \in \bm j^{\bm I}_k }w_{k,j}\bar{\bm s}_j,
    \qquad
    \forall \bm w = (w_{k,j})_{k \in [\mathcal K^{\bm I}], j \in \bm j^{\bm I}_k}.
    \label{def: mapping h type I, cluster size proof}
\end{align}
By the definition of $\mathbf C^{\bm I}_i$ in \eqref{def: measure C i I, cluster}, we have
\begin{align}
    \mathbf C^{\bm I}_i(B)
    =
    {\widehat{\mathbf C}^{\bm I}}_i\Big( (h^{\bm I})^{-1}(B)  \Big),
    \qquad
    \forall \text{ Borel measurable }B \subseteq \R^d_+.
    \label{property: connection between C I and hat C I measures, cluster size}
\end{align}
Lemma~\ref{lemma: limit theorem for n tau to hat C I, cluster size}
studies the (asymptotic) law of
$n^{-1}\bm \tau^{ n|\delta }_{i}$ when restricted on compact sets.

\begin{lemma}\label{lemma: limit theorem for n tau to hat C I, cluster size}
    \linksinthm{lemma: limit theorem for n tau to hat C I, cluster size}
Let Assumptions~\ref{assumption: subcriticality}--\ref{assumption: regularity condition 2, cluster size, July 2024} hold.
Let $i \in [d]$ and $0 < c < C < \infty$.
Let $\bm j \subseteq [d]$ be non-empty, 
and  $\bm I \in \mathscr I$ be such that $\bm j^{\bm I} = \bm j$.
There exists $\delta_0 > 0$ such that
for any $\delta \in (0,\delta_0)$
and any
$\bm x = (x_{k,j})_{ k \in [\mathcal K^{\bm I}],\ j \in \bm j^{\bm I}_k }$, 
$\bm y = (y_{k,j})_{ k \in [\mathcal K^{\bm I}],\ j \in \bm j^{\bm I}_k }$
with $c \leq x_{k,j} < y_{k,j}\leq C\ \forall k \in [\mathcal K^{\bm I}], j \in \bm j^{\bm I}_k$,
\begin{align}
    \lim_{n \to\infty}
    \big(\lambda_{\bm j}(n)\big)^{-1}
    \P\Big(
        n^{-1}\bm \tau^{n|\delta}_{i} \in 
        A^{\bm I}(\bm x, \bm y)
    \Big)
    =
    \widehat{\mathbf C}^{\bm I}_i
    \Bigg(
        \bigtimes_{ k \in [\mathcal K^{\bm I}] }\bigtimes_{ j \in \bm j^{\bm I}_k } (x_{k,j},y_{k,j}]
    \Bigg),
    \label{claim, lemma: limit theorem for n tau to hat C I, cluster size}
\end{align}
where
$\bm \tau^{n|\delta}_{i}$ is defined in \eqref{proof: def bm tau n delta i, cluster size },
$\widehat{\mathbf C}^{\bm I}_i$ is defined in \eqref{def, measure hat C i type I, cluster size},
$\bar s_{i,j} = \E S_{i,j}$ (see \eqref{def: bar s i, ray, expectation of cluster size}),
$j^{\bm I}_1$ is the unique index $j \in [d]$ such that $I_{j,1} = 1$ (see Remark~\ref{remark: def of type}),
and
\begin{equation}\label{def: set A x y type I, lemma: limit theorem for n tau to hat C I, cluster size}
    \begin{aligned}
        & A^{\bm I}(\bm x, \bm y)
    \delequal
    \Big\{
        (w_{k,j})_{ k \geq 1,j \in [d] } \in [0,\infty)^{\infty \times d }:
        \\
        &\qquad\qquad
        w_{k,j} \in (x_{k,j}, y_{k,j}]\ \forall k \in [\mathcal K^{\bm I}],\ j \in \bm j^{\bm I}_k;\ 
        w_{k,j} = 0\ \forall k \geq 1,\ j \notin \bm j^{\bm I}_k
    \Big\}.
    \end{aligned}
\end{equation}
\end{lemma}

In essence, Lemmas~\ref{lemma: prob of type I, tail prob of tau, cluster size} and \ref{lemma: limit theorem for n tau to hat C I, cluster size} refine the asymptotics in Lemma~\ref{lemma: crude estimate, type of cluster}.
Their proofs follow the same spirit as Lemma~\ref{lemma: crude estimate, type of cluster}
and proceed by
combining the asymptotics in Lemma~\ref{lemma: cluster size, asymptotics, N i | n delta, cdot j, crude estimate} with the Markov property \eqref{proof, cluster size, markov property in pruned clusters}.
The key difference is that, this time, we apply \eqref{claim 1, part iii, lemma: cluster size, asymptotics, N i | n delta, cdot j, refined estimates} and \eqref{claim 2, part iii, lemma: cluster size, asymptotics, N i | n delta, cdot j, refined estimates}
to characterize the asymptotic law of $n^{-1}\bm \tau^{n|\delta}_i$,
rather than relying on the other cruder estimates in Lemma~\ref{lemma: cluster size, asymptotics, N i | n delta, cdot j, crude estimate} and only obtaining the asymptotics of
$
\bm I^{n|\delta}_i
$
defined in \eqref{def, I k M j for M type, cluster size},
which simply indicates the positivity of the coordinates in $n^{-1}\bm \tau^{n|\delta}_i$.
To avoid repetition, 
we defer the proofs of Lemmas~\ref{lemma: prob of type I, tail prob of tau, cluster size} and \ref{lemma: limit theorem for n tau to hat C I, cluster size} to Section~\ref{subsec: proof, technical lemmas, cluster size}  of the Appendix.

Now, we state the proof of Proposition~\ref{proposition, M convergence for hat S, tail asymptotics for cluster size} using the technical tools introduced above.
We first prove part (iii) of Proposition~\ref{proposition, M convergence for hat S, tail asymptotics for cluster size},
and then move onto parts (i) and (ii).

\begin{proof}[Proof of Proposition~\ref{proposition, M convergence for hat S, tail asymptotics for cluster size}, Part (iii)]
Follows from
part (b) of Lemma~\ref{lemma: choice of bar epsilon bar delta, cluster size},
as well as
$
 \mathbf C^{\bm I}_i \circ \Phi^{-1} = \bar s_{i,l^*(j_1)} \cdot \mathbf C^{\bm I} \circ \Phi^{-1};
$
see \eqref{def: measure C i I, cluster}.
\end{proof}

\begin{proof}[Proof of Proposition~\ref{proposition, M convergence for hat S, tail asymptotics for cluster size}, Part (i)]
Since there are only finitely many elements in $\mathscr I(\bm j)$,
it suffices to fix some $\Delta > 0$ and type $\bm I \in \mathscr I(\bm j)$, and then verify \eqref{claim, proposition, M convergence for hat S, tail asymptotics for cluster size} for any $\delta > 0$ small enough.
Also, since the set $B$ (and hence its closure $B^-$) is bounded away from $\C_\leqslant^d(\bm j)$ under $\bm d_\text{U}$,
by part (a) of Lemma~\ref{lemma: choice of bar epsilon bar delta, cluster size}
there exists $\bar\epsilon \in (0,1)$ such that
for any $\bm x = \sum_{j \in \bm j}w_j\bar{\bm s}_j$ with $w_j \geq 0\ \forall j \in \bm j$,
\begin{align}
    \Phi(\bm x) \in B^-
    \qquad
    \Longrightarrow
    \qquad
    \min_{j \in \bm j}w_j > \bar\epsilon,\text{ and }
    \min_{j,j^\prime \in \bm j}
    \frac{
        w_j
    }{
        w_{j^\prime}
    } > \bar\epsilon.
    \label{proof: choice of bar epsilon and bar delta, proposition, M convergence for hat S, tail asymptotics for cluster size}
\end{align}
Define
$
\bar h^{ \bm I }: [0,\infty)^{ \infty \times d } \to \R^d_+
$
by $\bar h^{\bm I}(\bm w) \delequal \sum_{ k \in [\mathcal K^{\bm I}] }\sum_{ j \in \bm j^{\bm I}_k }w_{k,j}\bar{\bm s}_j$
for any $\bm w = (w_{k,j})_{k \geq 1,j \in [d]}$.
That is, $\bar h^{\bm I}$ trivially extends the domain of $h^{\bm I}$ 
to $[0,\infty)^{ \infty \times d  }$.
By the definition of $h^{\bm I}$ in \eqref{def: mapping h type I, cluster size proof} and
that ${(\hat R^{n|\delta}_j, \hat \Theta^{n|\delta}_j)}
    =
    \Phi(\hat{\bm S}^{n|\delta}_j)$
    (see \eqref{def: hat S n delta i, cluster size}--\eqref{def: polar coordinates, bar S and hat S, cluster size}),
    on the event $\{ \bm I^{n|\delta}_i = \bm I \}$
    we have
    $
     (\hat R^{n|\delta}_j, \hat \Theta^{n|\delta}_j)
    =
    \Phi\Big( \bar h^{\bm I}\big( n^{-1}\bm \tau^{n|\delta}_i \big)  \Big).
    $
Meanwhile, regarding the type $\bm I = (I_{k,j})_{k \geq 1, j \in [d]} \in \mathscr I(\bm j)$ fixed at the beginning of the proof,
by the third bullet point in Definition~\ref{def: cluster, type}, 
for each $j \in \bm j$ there uniquely exists some $k(j) \geq 1$ such that $I_{k(j),j} = 1$.
Then,
by \eqref{proof: choice of bar epsilon and bar delta, proposition, M convergence for hat S, tail asymptotics for cluster size}
and the definition of ${\bm I^{n|\delta}_i} = \big(I^{n|\delta}_{i;j}(k)\big)_{k \geq 1,\ j \in [d]}$ in \eqref{def, I k M j for M type, cluster size},
we have
\begin{align}
         \Phi\Big( \bar h^{\bm I}\big( n^{-1}\bm \tau^{n|\delta}_i \big)  \Big) \in B,\ 
          \bm I^{n|\delta}_i = \bm I 
    \quad \Longleftrightarrow \quad
     \Phi\Big( \bar h^{\bm I}\big( n^{-1}\bm \tau^{n|\delta}_i \big)  \Big) \in B,\ 
     n^{-1}\bm \tau^{n|\delta}_i \in E^{\bm I}(\bar\epsilon,\bar\epsilon),
     \nonumber
\end{align}
where the set $E^{\bm I}(M,c)$ is defined in \eqref{proof: def, set B type I M c, cluster size}.
Therefore, for any $M > 0$,
\begin{align*}
     & \P\Big(
                    (\hat R^{n|\delta}_i, \hat \Theta^{n|\delta}_i) \in B,\ 
                    \bm I^{n|\delta}_i = \bm I
    \Big)
    \\
    & = 
    \P\bigg(
         \Phi\Big( \bar h^{\bm I}\big( n^{-1}\bm \tau^{n|\delta}_i \big)  \Big) \in B,\ 
          n^{-1}\bm \tau^{n|\delta}_i \in E^{\bm I}(\bar\epsilon,\bar\epsilon)
    \bigg)
    \\ 
    & \leq 
    \P\Big(
         n^{-1}\bm \tau^{n|\delta}_i \in E^{\bm I}(M,\bar\epsilon)
    \Big)
    +
    \P\bigg(
         \Phi\Big( \bar h^{\bm I}\big( n^{-1}\bm \tau^{n|\delta}_i \big)  \Big) \in B,\ 
          n^{-1}\bm \tau^{n|\delta}_i \in E^{\bm I}(\bar\epsilon,\bar\epsilon) \setminus E^{\bm I}(M,\bar\epsilon)
    \bigg).
\end{align*}
By defining 
\begin{align}
    & E^{\bm I}_{\leqslant}(M)
    \delequal
     E^{\bm I}(\bar\epsilon,\bar\epsilon) \setminus E^{\bm I}(M,\bar\epsilon),
    \label{proof, def set B I leq M, proposition, M convergence for hat S, tail asymptotics for cluster size, part i}
\end{align}
we obtain the upper bound (for each $M > 0$)
\begin{align}
    & \P\Big(
                    (\hat R^{n|\delta}_i, \hat \Theta^{n|\delta}_i) \in B,\ 
                    \bm I^{n|\delta}_i = \bm I
    \Big)
    \label{proof, target upper bound, proposition, M convergence for hat S, tail asymptotics for cluster size, part i}
    \\ 
    & \leq 
    \P\Big(
         n^{-1}\bm \tau^{n|\delta}_i \in E^{\bm I}(M,\bar\epsilon)
    \Big)
    +
    \P\bigg(
         \Phi\Big( \bar h^{\bm I}\big( n^{-1}\bm \tau^{n|\delta}_i \big)  \Big) \in B,\ 
          n^{-1}\bm \tau^{n|\delta}_i \in E^{\bm I}_\leqslant(M)
    \bigg).
    \nonumber
\end{align}
Likewise, we get the lower bound
\begin{align}
    & \P\Big(
                    (\hat R^{n|\delta}_i, \hat \Theta^{n|\delta}_i) \in B,\ 
                    \bm I^{n|\delta}_i = \bm I
    \Big)
    \geq 
     \P\bigg(
         \Phi\Big( \bar h^{\bm I}\big( n^{-1}\bm \tau^{n|\delta}_i \big)  \Big) \in B,\ 
          n^{-1}\bm \tau^{n|\delta}_i \in E^{\bm I}_\leqslant(M)
    \bigg).
    \label{proof, target lower bound, proposition, M convergence for hat S, tail asymptotics for cluster size, part i}
\end{align}
Recall that we arbitrarily picked some $\Delta > 0$ at the beginning.
Suppose there exists some $\bar M = \bar M(\Delta) > 0$ such that, given $M > \bar M$,
it holds for any $\delta > 0$ small enough that 
\begin{align}
    \limsup_{n \to \infty}\big( \lambda_{\bm j}(n) \big)^{-1}\P\Big(
         n^{-1}\bm \tau^{n|\delta}_i \in E^{\bm I}(M,\bar\epsilon)
    \Big) < \Delta, 
    \label{proof, goal 1, proposition, M convergence for hat S, tail asymptotics for cluster size, part i}
\end{align}
and 
\begin{equation}
    \begin{aligned}
        \limsup_{n \to \infty}
        \big( \lambda_{\bm j}(n) \big)^{-1}
        \P\bigg(
         \Phi\Big( \bar h^{\bm I}\big( n^{-1}\bm \tau^{n|\delta}_i \big)  \Big) \in B,\ 
          n^{-1}\bm \tau^{n|\delta}_i \in E^{\bm I}_\leqslant(M)
    \bigg)
    & \leq \mathbf C^{\bm I}_i \circ \Phi^{-1}(B^\Delta),
    \\
    \liminf_{n \to \infty}
        \big( \lambda_{\bm j}(n) \big)^{-1}
        \P\bigg(
         \Phi\Big( \bar h^{\bm I}\big( n^{-1}\bm \tau^{n|\delta}_i \big)  \Big) \in B,\ 
          n^{-1}\bm \tau^{n|\delta}_i \in E^{\bm I}_\leqslant(M)
    \bigg)
    & \geq \mathbf C^{\bm I}_i \circ \Phi^{-1}(B_\Delta) - \Delta.
    \label{proof, goal 2, proposition, M convergence for hat S, tail asymptotics for cluster size, part i}
    \end{aligned}
\end{equation}
Then, by plugging these claims into the upper and lower bounds \eqref{proof, target upper bound, proposition, M convergence for hat S, tail asymptotics for cluster size, part i}--\eqref{proof, target lower bound, proposition, M convergence for hat S, tail asymptotics for cluster size, part i},
we conclude the proof for part (i) of Proposition~\ref{proposition, M convergence for hat S, tail asymptotics for cluster size}.
Now, it remains to verify Claims \eqref{proof, goal 1, proposition, M convergence for hat S, tail asymptotics for cluster size, part i}--\eqref{proof, goal 2, proposition, M convergence for hat S, tail asymptotics for cluster size, part i}.

\medskip
\noindent
\textbf{Proof of Claim \eqref{proof, goal 1, proposition, M convergence for hat S, tail asymptotics for cluster size, part i}}.
This is exactly the content of Lemma~\ref{lemma: prob of type I, tail prob of tau, cluster size}.
In particular, it suffices to prick $\bar M$ large enough such that, in Lemma~\ref{lemma: prob of type I, tail prob of tau, cluster size}, $C_i^{\bm I}(M,\bar\epsilon) < \Delta$ for any $M > \bar M$.

\medskip
\noindent
\textbf{Proof of Claim \eqref{proof, goal 2, proposition, M convergence for hat S, tail asymptotics for cluster size, part i}}.
For any $\bm w = (w_{k,j})_{ k \geq 1, j \in [d] } \in [0,\infty)^{\infty \times d}$,
we define $\psi^{\bm I}(\bm w) \delequal (w_{k,j})_{ k \in [\mathcal K^{\bm I}], j \in \bm j^{\bm I}_k }$.
That is, 
$\psi^{\bm I}$ is the projection mapping from $[0,\infty)^{ \infty \times d }$ onto the coordinates corresponding to the active indices of $\bm I$.
Also, recall the definition of the set $E^{\bm I}_\leqslant(M)$ in \eqref{proof, def set B I leq M, proposition, M convergence for hat S, tail asymptotics for cluster size, part i}.
Given $M > 0$,
Lemma~\ref{lemma: limit theorem for n tau to hat C I, cluster size} shows that
for any $\delta > 0$ small enough,
\begin{align}
    \big(\lambda_{\bm j}(n)\big)^{-1}
    \P\bigg(
        n^{-1}\bm \tau^{n|\delta}_{i} \in  E^{\bm I}_{\leqslant}(M);\ 
        \psi^{\bm I}\Big(n^{-1}\bm \tau^{n|\delta}_{i}\Big) \in \ \cdot\ 
    \bigg)
    \Rightarrow
    \widehat{\mathbf C}^{\bm I}_i\bigg( \ \cdot \ \cap  \psi^{\bm I}\Big(E^{\bm I}_{\leqslant}(M)\Big)\bigg)
    \label{proof: goal, weak convergence, proposition, M convergence for hat S, tail asymptotics for cluster size}
\end{align}
(as $n \to \infty$)
in terms of weak convergence of finite measures.
To see why, it suffices to note the following.
\begin{itemize}
    \item 
        By Lemma~\ref{lemma: crude estimate, type of cluster},
        we confirm that for any $\delta > 0$ small enough,
        \begin{align*}
        \sup_{n \geq 1}\big(\lambda_{\bm j}(n)\big)^{-1}
            \P\Big(
                n^{-1}\bm \tau^{n|\delta}_{i} \in E^{\bm I}_{\leqslant}(M)
            \Big) 
            \leq 
        \sup_{n \geq 1}\big(\lambda_{\bm j}(n)\big)^{-1}
            \P\big(
                \bm I \subseteq \bm I^{n|\delta}_i
            \big) 
            < \infty;
        \end{align*}
        In other words, 
        the LHS of \eqref{proof: goal, weak convergence, proposition, M convergence for hat S, tail asymptotics for cluster size} is a sequence of finite measures with a uniform upper bound on their masses.

    \item 
        Fix some $\bm w = (w_{k,j})_{k \geq 1, j \in [d]} \in E^{\bm I}_\leqslant(M)$.
        By the definitions of $E^{\bm I}(M,c)$ in \eqref{proof: def, set B type I M c, cluster size}
        and $E^{\bm I}_\leqslant(M)$ in \eqref{proof, def set B I leq M, proposition, M convergence for hat S, tail asymptotics for cluster size, part i},
        there exists some $k \in [\mathcal K^{\bm I}]$ and $j \in \bm j^{\bm I}_k$ such that 
        $
        w_{k,j} \in (\bar\epsilon,M].
        $
        Then, by the condition that
        $
        w_{k,j}/w_{k^\prime,j^\prime} \geq \bar\epsilon
        $
        for any $k,k^\prime \in [\mathcal K^{\bm I}]$ and $j \in \bm j^{\bm I}_k$, $j^\prime \in \bm j^{\bm I}_{k^\prime}$ (see \eqref{proof: def, set B type I M c, cluster size}),
        we must have $w_{k,j} \in (\bar\epsilon,M/\bar\epsilon]$ for each $k \in [\mathcal K^{\bm I}]$, $j \in \bm j^{\bm I}_k$.
        
    \item 
        Furthermore,
        the sets of the form
        $
        \bigtimes_{i \in m}(x_i,y_i]
        $
        studied in Lemma~\ref{lemma: limit theorem for n tau to hat C I, cluster size}
        constitute a convergent determining class for the weak convergence of finite measures on $(\bar\epsilon, M/\bar\epsilon ]^{m}$.
        This allows us to apply Lemma~\ref{lemma: limit theorem for n tau to hat C I, cluster size} and verify \eqref{proof: goal, weak convergence, proposition, M convergence for hat S, tail asymptotics for cluster size} for any $\delta > 0$ small enough.

\end{itemize}
To apply the weak convergence in \eqref{proof: goal, weak convergence, proposition, M convergence for hat S, tail asymptotics for cluster size},
we make a few observations.
First, 
 by definitions of $\bar h^{\bm I}$ and $\psi^{\bm I}$, we have
$
\bar h^{\bm I}(\bm w) = h^{\bm I}\big( \psi^{\bm I}(\bm w) \big)
$
for any $\bm w \in [0,\infty)^{\infty \times d}$,
which implies
\begin{equation} \label{proof, equiavlence for projection psi and extension bar h, proposition, M convergence for hat S, tail asymptotics for cluster size, part i}
    \begin{aligned}
        & \Big\{
     \Phi\Big( \bar h^{\bm I}\big( n^{-1}\bm \tau^{n|\delta}_i \big)  \Big) \in B,\ 
          n^{-1}\bm \tau^{n|\delta}_i \in E^{\bm I}_\leqslant(M)
    \Big\}
    \\ 
    & 
    =
    \Big\{
     \psi^{\bm I}\big( n^{-1}\bm \tau^{n|\delta}_i \big) \in (h^{\bm I})^{-1}\big( \Phi^{-1}(B)\big),\ 
          n^{-1}\bm \tau^{n|\delta}_i \in E^{\bm I}_\leqslant(M)
    \Big\}.
    \end{aligned}
\end{equation}
Next, note that $\Phi \circ h^{\bm I}(\cdot)$ is continuous at any $(w_{k,j})_{k \in [\mathcal K^{\bm I}], j \in  \bm j^{\bm I}_k }$
with $w_{k,j} > 0\ \forall k,j$,
and that 
$
B^- \subseteq \Phi \circ h^{\bm I}
\Big(\Big\{
    (w_{k,j})_{k \in [\mathcal K^{\bm I}], j \in  \bm j^{\bm I}_k }:
    \ w_{k,j} > 0\ \forall k,j 
\Big\}\Big);
$
see \eqref{proof: choice of bar epsilon and bar delta, proposition, M convergence for hat S, tail asymptotics for cluster size}.
This implies that
$
(h^{\bm I})^{-1}\big( \Phi^{-1}(B^-)\big)
$
is closed and 
$
(h^{\bm I})^{-1}\big( \Phi^{-1}(B^\circ)\big)
$
is open, and hence
\begin{align}
    \Big( (h^{\bm I})^{-1}\big( \Phi^{-1}(B)\big) \Big)^- \subseteq (h^{\bm I})^{-1}\big( \Phi^{-1}(B^-)\big),
    \
    \Big( (h^{\bm I})^{-1}\big( \Phi^{-1}(B)\big) \Big)^\circ \supseteq (h^{\bm I})^{-1}\big( \Phi^{-1}(B^\circ)\big).
    \label{proof, closure and interior under Phi h composition, proposition, M convergence for hat S, tail asymptotics for cluster size, part i}
\end{align}
As a result, for any $\delta > 0$ small enough,
\begin{align*}
     & \limsup_{n \to \infty}
        \big( \lambda_{\bm j}(n) \big)^{-1}
        \P\bigg(
         \Phi\Big( \bar h^{\bm I}\big( n^{-1}\bm \tau^{n|\delta}_i \big)  \Big) \in B,\ 
          n^{-1}\bm \tau^{n|\delta}_i \in E^{\bm I}_\leqslant(M)
    \bigg)
    \\ 
    & =
    \limsup_{n \to \infty}
        \big( \lambda_{\bm j}(n) \big)^{-1}
        \P\bigg(
          \psi^{\bm I}\big( n^{-1}\bm \tau^{n|\delta}_i \big) \in (h^{\bm I})^{-1}\big( \Phi^{-1}(B)\big),\ 
          n^{-1}\bm \tau^{n|\delta}_i \in E^{\bm I}_\leqslant(M)
    \bigg)
    \ 
    \text{by \eqref{proof, equiavlence for projection psi and extension bar h, proposition, M convergence for hat S, tail asymptotics for cluster size, part i}}
    \\ 
    & \leq 
     \widehat{\mathbf C}^{\bm I}_i\bigg( 
        \Big( (h^{\bm I})^{-1}\big( \Phi^{-1}(B^-)\big) \Big)
        \cap  \psi^{\bm I}\Big(E^{\bm I}_{\leqslant}(M)\Big)\bigg)
    \quad 
    \text{by \eqref{proof: goal, weak convergence, proposition, M convergence for hat S, tail asymptotics for cluster size} and \eqref{proof, closure and interior under Phi h composition, proposition, M convergence for hat S, tail asymptotics for cluster size, part i}}
    \\ 
    & \leq 
    \widehat{\mathbf C}^{\bm I}_i\Big(
        (h^{\bm I})^{-1}\big( \Phi^{-1}(B^-)\big)
        \Big)
    = 
    \mathbf C^{\bm I}_i\Big( \Phi^{-1}(B^-) \Big)
    = 
    \mathbf C_i^{\bm I} \circ \Phi^{-1}(B^-)
    \quad
    \text{by \eqref{property: connection between C I and hat C I measures, cluster size}.}
\end{align*}
This verifies the upper bound in Claim \eqref{proof, goal 2, proposition, M convergence for hat S, tail asymptotics for cluster size, part i} for any $M > 0$.
Likewise, 
using \eqref{proof, equiavlence for projection psi and extension bar h, proposition, M convergence for hat S, tail asymptotics for cluster size, part i}, \eqref{proof, closure and interior under Phi h composition, proposition, M convergence for hat S, tail asymptotics for cluster size, part i}, and the weak convergence in \eqref{proof: goal, weak convergence, proposition, M convergence for hat S, tail asymptotics for cluster size},
given $M > 0$
we obtain the lower bound
\begin{equation} \label{proof, lower bound display, claim 2, proposition, M convergence for hat S, tail asymptotics for cluster size, part i}
    \begin{aligned}
         & 
    \liminf_{n \to \infty}
        \big( \lambda_{\bm j}(n) \big)^{-1}
        \P\bigg(
         \Phi\Big( \bar h^{\bm I}\big( n^{-1}\bm \tau^{n|\delta}_i \big)  \Big) \in B,\ 
          n^{-1}\bm \tau^{n|\delta}_i \in E^{\bm I}_\leqslant(M)
    \bigg)
    \\ 
    & \geq 
     \widehat{\mathbf C}^{\bm I}_i\bigg( 
        \Big( (h^{\bm I})^{-1}\big( \Phi^{-1}(B^\circ)\big) \Big)
        \cap  \psi^{\bm I}\Big(E^{\bm I}_{\leqslant}(M)\Big)\bigg)
    \end{aligned}
\end{equation}
for any $\delta > 0$ small enough.
To further bound the RHS of \eqref{proof, lower bound display, claim 2, proposition, M convergence for hat S, tail asymptotics for cluster size, part i},
we make a few observations.
First,
\eqref{proof: choice of bar epsilon and bar delta, proposition, M convergence for hat S, tail asymptotics for cluster size} implies that for any
 $\bm w = (w_{k,j})_{k \in [\mathcal K^{\bm I}], j \in  \bm j^{\bm I}_k }$
with $w_{k,j} \geq 0\ \forall k,j$
and
$\Phi\big( h^{\bm I}(\bm w)  \big) \in B$,
we must have
$
\bm w \in 
\psi^{\bm I}\big( E^{\bm I}(\bar\epsilon,\bar\epsilon) \big).
$
As a result,
\begin{align*}
    \widehat{\mathbf C}^{\bm I}_i\bigg( 
     \Big( (h^{\bm I})^{-1}\big( \Phi^{-1}(B^\circ)\big) \Big)
        \cap  \psi^{\bm I}\big(E^{\bm I}(\bar\epsilon,\bar\epsilon)\big)
    \bigg)
    =
    \widehat{\mathbf C}^{\bm I}_i \Big( (h^{\bm I})^{-1}\big( \Phi^{-1}(B^\circ)\big) \Big).
\end{align*}
Next,
the sequence of sets $E^{\bm I}_{\leqslant}(M)$ in \eqref{proof, def set B I leq M, proposition, M convergence for hat S, tail asymptotics for cluster size, part i}
is monotone increasing w.r.t.\ $M$,
with
$
\bigcup_{M > 0}E^{\bm I}_{\leqslant}(M) = E^{\bm I}(\bar\epsilon,\bar\epsilon).
$
On the other hand, 
by \eqref{property: connection between C I and hat C I measures, cluster size} we get 
$
\widehat{\mathbf C}^{\bm I}_i \big( (h^{\bm I})^{-1}( \Phi^{-1}(B^\circ)) \big) = \mathbf C_i^{\bm I} \circ \Phi^{-1}(B^\circ).
$
Also,
part (iii) of Proposition~\ref{proposition, M convergence for hat S, tail asymptotics for cluster size},
which we established earlier,
confirms that $\mathbf C_i^{\bm I} \circ \Phi^{-1}(B^\circ) < \infty$.
Then, by continuity of measures, there exists $\bar M = \bar M(\Delta)$ such that
\begin{align*}
    \widehat{\mathbf C}^{\bm I}_i\bigg( 
        \Big( (h^{\bm I})^{-1}\big( \Phi^{-1}(B^\circ)\big) \Big)
        \cap  \psi^{\bm I}\Big(E^{\bm I}_{\leqslant}(M)\Big)\bigg)
    > 
    \mathbf C_i^{\bm I} \circ \Phi^{-1}(B^\circ)  - \Delta,
    \quad \forall M > \bar M.
\end{align*}
Plugging this into \eqref{proof, lower bound display, claim 2, proposition, M convergence for hat S, tail asymptotics for cluster size, part i}, we conclude the proof for the lower bound in Claim~\eqref{proof, goal 2, proposition, M convergence for hat S, tail asymptotics for cluster size, part i}.
\end{proof}

\begin{proof}[Proof of Proposition~\ref{proposition, M convergence for hat S, tail asymptotics for cluster size}, Part (ii)]
Recall the definitions of $\mathscr I(\bm j)$ and $\widetilde{\mathscr I}(\bm j)$ in  \eqref{def: mathscr I bm j, lemma: type and generalized type},
and that we have fixed some non-empty $\bm j \subseteq \{1,2,\ldots,d\}$ in the statement of this proposition.
Note that if, for some generalized type $\bm I \in \widetilde{\mathscr I}(\bm j)$, we have $\bm I \notin \mathscr I(\bm j)$,
then, there are only three possibilities: 
(1) $\tilde\alpha(\bm I) > \alpha(\bm j)$;
(2) $\bm I \in \widetilde{\mathscr I}(\bm j) \setminus \mathscr I(\bm j)$;
or 
(3) $\tilde\alpha(\bm I) \leq \alpha(\bm j),\ \bm I \notin \widetilde{\mathscr I}(\bm j)$.
Therefore, to prove part (ii), it suffices to show that (for any $\delta > 0$ small enough)
\begin{align}
    \lim_{n \to \infty}
    \big(\lambda_{\bm j}(n)\big)^{-1}
    \P\Big( \tilde{\alpha}\big(\bm I^{n|\delta}_i\big) > \alpha(\bm j)  \Big)
    & = 0,
    \label{goal 1, proposition, M convergence for hat S, tail asymptotics for cluster size}
    \\
    \lim_{n \to \infty}
    \big(\lambda_{\bm j}(n)\big)^{-1}
    \P\Big( 
        (\hat R^{n|\delta}_i, \hat \Theta^{n|\delta}_i) \in B,\ 
            \bm I^{n|\delta}_i \in \widetilde{\mathscr I}(\bm j) \setminus {\mathscr I}(\bm j)
    \Big)
    & = 0,
    \label{goal 2, alt version, proposition, M convergence for hat S, tail asymptotics for cluster size}
    \\
        \Big\{
            (\hat R^{n|\delta}_i, \hat \Theta^{n|\delta}_i) \in B,\ 
            \tilde\alpha(\bm I) \leq \alpha(\bm j),\
            \bm I^{n|\delta}_i \notin \widetilde{\mathscr I}(\bm j)
        \Big\}
    & = \emptyset.
    \label{goal 2, proposition, M convergence for hat S, tail asymptotics for cluster size}
\end{align}

\medskip
\noindent
\textbf{Proof of Claim \eqref{goal 1, proposition, M convergence for hat S, tail asymptotics for cluster size}}.
This is verified
in the proof of Proposition~\ref{proposition: asymptotic equivalence, tail asymptotics for cluster size}; see Claim~\eqref{proof: goal 1, asymptotic equivalence, tail asymptotics for cluster size}.

\medskip
\noindent
\textbf{Proof of Claim \eqref{goal 2, alt version, proposition, M convergence for hat S, tail asymptotics for cluster size}}.
Due to $\alpha^*(j) \in (1,\infty)\ \forall j \in [d]$ (see Assumption~\ref{assumption: heavy tails in B i j} and \eqref{def: cluster size, alpha * l * j}),
we must have $|\widetilde{\mathscr I}(\bm j) \setminus {\mathscr I}(\bm j)| < \infty$,
so
it suffices to fix some $\bm I \in \widetilde{\mathscr I}(\bm j) \setminus {\mathscr I}(\bm j)$ and show that
\begin{align}
    \lim_{n \to \infty}
    \big(\lambda_{\bm j}(n)\big)^{-1}
        \P\big(
            \bm I^{n|\delta}_i = \bm I
        \big)
    \leq 
    \lim_{n \to \infty}
    \big(\lambda_{\bm j}(n)\big)^{-1}
        \P\big(
            \bm I \subseteq \bm I^{n|\delta}_i
        \big)
    & = 0,
    \quad
    \forall \delta > 0 
    \text{ small enough,}
    \nonumber
\end{align}
where the partial ordering $\bm I \subseteq \bm I^\prime$ is defined in \eqref{proof: def ordering of generalized types, cluster size}.
By part (ii) of Lemma~\ref{lemma: type and generalized type}, $\bm I \in \widetilde{\mathscr I}(\bm j) \setminus {\mathscr I}(\bm j)$ implies that $|\bm j^{\bm I}_1 |\geq 2$.
This allows us to apply the Claim~\eqref{claim: lemma: crude estimate, |j 1 type I| geq 2, type of cluster} of Lemma~\ref{lemma: crude estimate, type of cluster}
and get
$
\P\big( \bm I \subseteq \bm I^{n|\delta}_i\big)
=
\lo\big( \tilde \lambda^{\bm I}(n)  \big)
$
for any $\delta > 0$ small enough,
with
$\tilde \lambda^{\bm I}(n)$ defined in \eqref{proof, def tilde lambda generalized type I}.
Also,
by the definition of $\tilde{\mathscr I}(\bm j)$, we have $\alpha(\bm j) = \tilde \alpha(\bm I)$ and $\bm j = \bm j^{\bm I}$.
Then, by property \eqref{property: cost of generalized type when alpha agrees with tilde alpha}
and the definition of $\lambda_{\bm j}(n)$ in \eqref{def: rate function lambda j n, cluster size},
we must have 
$
\lambda_{\bm j}(n) = \tilde \lambda^{\bm I}(n).
$
This verifies
$
\P\big(
            \bm I \subseteq \bm I^{n|\delta}_i
        \big)
= \lo \big(\lambda_{\bm j}(n)\big)
$
for any $\delta > 0$ small enough and
concludes the proof of Claim~\eqref{goal 2, alt version, proposition, M convergence for hat S, tail asymptotics for cluster size}.

\medskip
\noindent
\textbf{Proof of Claim \eqref{goal 2, proposition, M convergence for hat S, tail asymptotics for cluster size}}.
We arbitrarily pick some generalized type $\bm I \in \widetilde{\mathscr I}$
such that $\tilde\alpha(\bm I) \leq \alpha(\bm j)$ and $\bm I \notin \widetilde{\mathscr I}(\bm j)$.
By the definition of $\widetilde{\mathscr I}(\bm j)$ in \eqref{def: mathscr I bm j, lemma: type and generalized type},
for such $\bm I$ we either have $\tilde\alpha(\bm I) < \alpha(\bm j)$,
or $\tilde \alpha(\bm I) = \alpha(\bm j),\ \bm j^{\bm I} \neq \bm j$,
where $\bm j^{\bm I}$ is the set of active indices in $\bm I$ (see Definition~\ref{def: cluster, generalized type}),
and $\bm j \subseteq \{1,2,\ldots,d\}$ is the non-empty set prescribed in the statement of this proposition.
In both cases,
due to \eqref{property: cost of j I and type I, 1, cluster size}, we must have
\begin{align}
    \bm j^{\bm I} \neq \bm j,\qquad {\alpha}(\bm j^{\bm I}) \leq \alpha(\bm j).
    \label{proof: goal 2, active index set for generalized type I, proposition, M convergence for hat S, tail asymptotics for cluster size}
\end{align}
Meanwhile,
it holds on the event $\{ \bm I^{n|\delta}_i = \bm I \}$ that
$
\hat{\bm S}^{n|\delta}_i
    = 
    \sum_{k = 1}^{\mathcal K^{\bm I}}\sum_{ j \in \bm j^{\bm I}_k }
        n^{-1}\tau^{n|\delta}_{i;j}(k) \cdot \bar{\bm s}_j;
$
see \eqref{def: hat S n delta i, cluster size}.
Then,
by \eqref{proof: goal 2, active index set for generalized type I, proposition, M convergence for hat S, tail asymptotics for cluster size}
and the definition of $\R^d_\leqslant(\bm j) = \bigcup_{ \bm j^\prime \neq \bm j,\ \alpha(\bm j^\prime) \leq \alpha(\bm j)   }\R^d(\bm j^\prime)$ in \eqref{def: cone R d i basis S index alpha},
on the event $\{ \bm I^{n|\delta}_i = \bm I  \}$
we must have $\hat{\bm S}^{n|\delta}_i \in \R^d_{\leqslant}(\bm j)$,
and hence 
$
\Phi(\hat{\bm S}^{n|\delta}_i) = \big(\hat R^{n|\delta}_i, \hat \Theta^{n|\delta}_i\big) \in \mathbb C^d_\leqslant(\bm j);
$
see \eqref{def: cone C d leq j, cluster size}.
Since $B$ is bounded away from $\mathbb C^d_\leqslant(\bm j)$, we have just confirmed that
$
\big\{  \big(\hat R^{n|\delta}_i, \hat \Theta^{n|\delta}_i\big) \in B,\ \bm I^{n|\delta}_i = \bm I  \big\} = \emptyset.
$
Repeating this argument for each generalized type $\bm I$
satisfying $\tilde\alpha(\bm I) \leq \alpha(\bm j)$ and $\bm I \notin \widetilde{\mathscr I}(\bm j)$,
we conclude the proof of Claim~\eqref{goal 2, proposition, M convergence for hat S, tail asymptotics for cluster size}.
\end{proof}

\subsection{Proof of Lemma~\ref{lemma: cluster size, asymptotics, N i | n delta, cdot j, crude estimate}}
\label{subsec: proof, lemma: cluster size, asymptotics, N i | n delta, cdot j, crude estimate}

Recall 
the definitions of $W^{>}_{i;j\leftarrow l}(M)$, $\ W^{>}_{i;j}(M)$, $N^{>}_{i;j\leftarrow l}(M)$, and $N^{>}_{i;j}(M)$
in \eqref{def: W i M j, pruned cluster, 1, cluster size}--\eqref{def: cluster size, N i | M cdot j}, and that $\bar s_{i,j} = \E S_{i,j}$.
To prove Lemma~\ref{lemma: cluster size, asymptotics, N i | n delta, cdot j, crude estimate}, we prepare the following result.

\begin{lemma}\label{lemma: cluster size, asymptotics, N i | n delta, l j}
\linksinthm{lemma: cluster size, asymptotics, N i | n delta, l j}
Let Assumptions~\ref{assumption: subcriticality}--\ref{assumption: regularity condition 2, cluster size, July 2024} hold.
\begin{enumerate}[$(i)$]
    \item Let $\delta > 0$. As $n \to \infty$,
        \begin{align}
    \P\Big(N^{>}_{i;j \leftarrow l}(n\delta) = 1\Big) & \sim \bar s_{i,l} \P(B_{j \leftarrow l} > n\delta),
        \label{claim 1, lemma: cluster size, asymptotics, N i | n delta, l j}
    \\
    \P\Big(N^{>}_{i;j \leftarrow l}(n\delta) \geq 2\Big) & = \lo \big(\P(B_{j \leftarrow l} > n\delta)\big).
        \label{claim 2, lemma: cluster size, asymptotics, N i | n delta, l j}
\end{align}

    \item 
        There exists $\delta_0 > 0$ such that
        for any $\mathcal T \subseteq [d]^2$ with $|\mathcal T| \geq 2$ and any $\delta \in (0,\delta_0)$,
        \begin{align*}
            \P\Big(
               N^{>}_{i;j \leftarrow l}(n\delta) \geq 1
                \ \forall (l,j) \in \mathcal T
            \Big)
            =
            \lo 
            \bigg(
                n^{|\mathcal T| - 1} \cdot \prod_{(l,j) \in \mathcal T}\P(B_{j \leftarrow l} > n\delta)
            \bigg),
            \quad
            \text{ as }n\to\infty.
        \end{align*}
        
\end{enumerate}

\end{lemma}

\begin{proof}\linksinpf{lemma: cluster size, asymptotics, N i | n delta, l j}
$(i)$ Suppose that we can verify (as $n \to \infty$)
\begin{align}
    \E \big[ N^{>}_{i;j \leftarrow l}(n\delta)\big] \sim \bar s_{i,l} \P(B_{j \leftarrow l} > n\delta),
    \label{claim 3, lemma: cluster size, asymptotics, N i | n delta, l j}
    \\
    \P\big(N^{>}_{i;j \leftarrow l}(n\delta)\geq 1\big)\sim \bar s_{i,l} \P(B_{j \leftarrow l} > n\delta).
    \label{claim 3, 2, lemma: cluster size, asymptotics, N i | n delta, l j}
\end{align}
Then, note that for a sequence of random variables $Z_n$ taking non-negative integer values, 
by the elementary bound $\E Z_n = \sum_{k \geq 1}\P(Z_n \geq k) \geq \P(Z_n \geq 2) + \P(Z_n \geq 1)$,
we get
\begin{align*}
    \lim_{n \to \infty}
    \frac{ \E Z_n  }{a_n} = c,
    \ 
    \lim_{n \to \infty}
    \frac{ \P(Z_n \geq 1)  }{a_n} = c
    \quad
    \Longrightarrow\quad
    \lim_{n \to \infty}
    \frac{ \P(Z_n = 1) }{a_n} = c,
    \ 
    \lim_{n \to \infty}
    \frac{ \P(Z_n \geq 2) }{a_n} = 0
\end{align*}
for any $c > 0$ and any sequence of strictly positive real numbers $a_n$.
Therefore, the asymptotics stated in \eqref{claim 1, lemma: cluster size, asymptotics, N i | n delta, l j} and \eqref{claim 2, lemma: cluster size, asymptotics, N i | n delta, l j} follow from Claims~\eqref{claim 3, lemma: cluster size, asymptotics, N i | n delta, l j} and \eqref{claim 3, 2, lemma: cluster size, asymptotics, N i | n delta, l j}.
Next, we prove these two claims.

\medskip
\noindent
\textbf{Proof of Claim~\eqref{claim 3, lemma: cluster size, asymptotics, N i | n delta, l j}}.
By definitions in \eqref{def: branching process X n leq M}, \eqref{def: pruned tree S i leq M}, \eqref{def, N i | M l j, cluster size}, \eqref{def: cluster size, N i | M cdot j},
$N^{>}_{i;j \leftarrow l}(M)$ counts the number of type-$l$ nodes with pruned type-$j$ children under threshold $M$ 
in the branching process $\bm X^\leqslant_j(t;M)$.
Therefore, the $N^{>}_{i;j \leftarrow l}(M)$'s solve the fixed-point equations
\begin{align*}
     N^{>}_{i;j \leftarrow l}(M)
     \distequal
        \mathbbm{I}\{
            i = l,\ B_{j \leftarrow l} > M
        \}
    +
        \sum_{k \in [d]} \sum_{m = 1}^{ B_{k \leftarrow i}\mathbbm{I}\{ B_{k \leftarrow i} \leq M\} }
        N^{>,(m)}_{k;j \leftarrow l}(M),
    \qquad 
    i,j,l \in [d],
\end{align*}
where the $N^{>,(m)}_{k;j \leftarrow l}(M)$'s are independent copies of $N^{>}_{k;j \leftarrow l}(M)$.
Let $\bar{\textbf B}^{\leqslant M} = (\bar b^{\leqslant M}_{j \leftarrow i})_{i,j \in [d]}$:
that is, the element on the $i^\text{th}$ row and $j^\text{th}$ column is
$\bar b^{\leqslant M}_{j\leftarrow i} = \E B_{j \leftarrow i}\mathbbm{I}\{B_{j \leftarrow i} \leq M\}$.
Provided that the spectral radius of $\bar{\textbf B}^{\leqslant M}$ is strictly less than 1,
we can
apply Proposition~1 of \cite{Asmussen_Foss_2018} and get
$
\bar{\bm n}^{j \leftarrow l} = \bar{\bm q}^{j \leftarrow l} +\bar{\textbf B}^{\leqslant M} \bar{\bm n}^{j \leftarrow l}, 
$
where the vectors $\bar{\bm n}^{j \leftarrow l} = (\bar n^{j \leftarrow l}_1,\ldots,\bar n^{j \leftarrow l}_d)^\top,\ \bar{\bm q}^{j \leftarrow l} = (\bar q^{j \leftarrow l}_1,\ldots,\bar q^{j \leftarrow l}_d)^\top$ are defined by
$
{\bar n}^{j \leftarrow l}_k = 
        \E\big[ N^{>}_{k;j \leftarrow l}(M) \big]
$
and
$
\bar q^{j \leftarrow l}_k = \mathbbm{I}\{k = l\}\P(B_{j \leftarrow l} > M).
$
This implies
$
\bar{\bm n}^{j\leftarrow l} = (\textbf I - \bar{\textbf B}^{\leqslant M})^{-1}\bar{\bm q}^{j\leftarrow l},
$
and hence
\begin{align}
    \E\big[  N^{>}_{i;j \leftarrow l}(M) \big] 
    = \E\big[ S^{\leqslant}_{i,l}(M)\big]\P(B_{j \leftarrow l} > M),\qquad \forall i,l,j \in [d].
    \label{proof, part i, claim 3, expectation of N i l to j M}
\end{align}
In particular,
due to Assumption~\ref{assumption: subcriticality} and monotone convergence,
it holds for any $M$ large enough that $\bar{\textbf B}^{\leqslant M}$ has a spectral radius less than $1$. 
Besides,
applying monotone convergence
to $\sum_{t \geq 0}\bm X^\leqslant_i(t;M) \distequal \bm S^\leqslant_i(M)$
and
$
\sum_{t \geq 0}\bm X_i(t) \distequal \bm S_i
$
(see \eqref{def, proof strategy, collection of B i j copies}--\eqref{def: pruned tree S i leq M}),
we get
$
\lim_{M \to \infty}\E S^{\leqslant}_{i,l}(M) = \E S_{i,l} = \bar s_{i,l}.
$
By setting $M = n\delta$ in \eqref{proof, part i, claim 3, expectation of N i l to j M} and sending $n \to \infty$,
we conclude the proof of Claim~\eqref{claim 3, lemma: cluster size, asymptotics, N i | n delta, l j},
where we must have $\bar s_{i,l} > 0$ under Assumption~\ref{assumption: regularity condition 1, cluster size, July 2024}.

\medskip
\noindent
\textbf{Proof of Claim~\eqref{claim 3, 2, lemma: cluster size, asymptotics, N i | n delta, l j}}.
Combining \eqref{claim 3, lemma: cluster size, asymptotics, N i | n delta, l j} with Markov inequality, we are able to obtain the upper bound
$
\limsup_{n \to \infty}
\P\big(N^{>}_{i;j \leftarrow l}(n\delta)\geq 1\big)
\big/ \P(B_{j \leftarrow l} > n\delta) \leq \bar s_{i,l}.
$
Now, we focus on establishing
\begin{align}
    \liminf_{n \to \infty}
    \P\big(N^{>}_{i;j \leftarrow l}(n\delta)\geq 1\big)
    \big/ \P(B_{j \leftarrow l} > n\delta) \geq \bar s_{i,l}.
    \label{proof, part i, claim 3, 2, goal LB, lemma: cluster size, asymptotics, N i | n delta, l j}
\end{align}
First, 
the definition of $N^{>}_{i;j \leftarrow l}(M)$
in \eqref{def, N i | M l j, cluster size}
is equivalent to
\begin{align}
    N^{>}_{i;j \leftarrow l}(M)
    & = 
    \#
    \Big\{
        (t,m) \in \mathbb N^2:\ 
        m \leq X^{\leqslant}_{i,l}(t - 1;M),\ B^{(t,m)}_{j \leftarrow l} > M
    \Big\}.
    \label{proof, equivalent def for N > i j l, lemma: cluster size, asymptotics, N i | n delta, l j}
\end{align}
Next, given $M^\prime > M > 0$,
the stochastic comparison in \eqref{property: stochastic comparison, general, glaton watson trees}, \eqref{property: stochastic comparison, S and pruned S} implies
\begin{align}
    N^{>}_{i;j \leftarrow l}(M^\prime)
    \geq 
    \underbrace{ \#
    \Big\{
        (t,m) \in \mathbb N^2:\ 
        m \leq X^{\leqslant}_{i,l}(t - 1;M),\ B^{(t,m)}_{j \leftarrow l} > M^\prime
    \Big\}
    }_{ \delequal \hat N_{i;j \leftarrow l}(M,M^\prime)  }.
    \label{proof, part i, claim 3, 2, ineq 1, lemma: cluster size, asymptotics, N i | n delta, l j}
\end{align}
Furthermore, the branching process $(\bm X^\leqslant_i(t;M))_{t \geq 0}$ is independent 
from the actual value of any $B_{j \leftarrow l}^{(t,m)}$ if $B_{j \leftarrow l}^{(t,m)} \in \{0\} \cup (M,\infty)$:
indeed, the pruning mechanism in \eqref{def: branching process X n leq M} would always result in 
$
B_{j \leftarrow l}^{(t,m)}\mathbbm{I}\{B_{j \leftarrow l}^{(t,m)} \leq M\} = 0
$
in such cases.
This leads to a coupling between $(\bm X^\leqslant_i(t;M))_{t \geq 0}$ and the $B^{(t,m)}_{j \leftarrow l}$'s,
where we first generate the branching process $(\bm X^\leqslant_i(t;M))_{t \geq 0}$ under offspring counts
$\hat B_{j \leftarrow l}^{(t,m)}(M) \distequal B_{j \leftarrow l}^{(t,m)}\mathbbm{I}\{B_{j \leftarrow l}^{(t,m)} \leq M\}$,
and then, independently for each $(t,m,j,l)$, recover $B_{j \leftarrow l}^{(t,m)}$ based on the value of $\hat B_{j \leftarrow l}^{(t,m)}$.
More specifically,
given $M^\prime > M > 0$,
the term
$\hat N^{>}_{i;j \leftarrow l}(M,M^\prime)$ in \eqref{proof, part i, claim 3, 2, ineq 1, lemma: cluster size, asymptotics, N i | n delta, l j} can be generated as follows:
\begin{enumerate}[$(1)$]
    \item 
         first, we generate $(\bm X^\leqslant_i(t;M))_{t \geq 0}$ as a branching process under offspring counts $\hat B_{j \leftarrow l}^{(t,m)}$, which are independent copies of $B_{j \leftarrow l}\mathbbm{I}\{B_{j \leftarrow l} \leq M\}$;

    \item 
        next, independently for any $(t,m) \in \mathbb N^{2}$ with $m \leq X^{\leqslant}_{i,l}(t - 1;M)$ and $\hat B_{j \leftarrow l}^{(t,m)}(M) = 0$ 
        (that is, the $m^\text{th}$ type-$l$ node in the $(t-1)^\text{th}$ generation of the branching process $(\bm X^\leqslant_i(t;M))_{t \geq 0}$ did not give birth to any type-$j$ child in the $t^\text{th}$ generation),
        we sample $B^{(t,m)}_{j \leftarrow l}$ under the conditional law
        $
        \P\big(B_{j \leftarrow l} \in \cdot\ \big|\ B_{j \leftarrow l} \in \{0\} \cup (M,\infty)\big);
        $
    
    \item 
        lastly, we count the number of pairs $(t,m)$ in step $(2)$ with $B^{(t,m)}_{j \leftarrow l} > M^\prime$.
\end{enumerate}
In particular, by setting
\begin{align}
    Z_M \delequal 
    \#\Big\{
        (t,m) \in \mathbb N^2:\ 
        m \leq X^{\leqslant}_{i,l}(t - 1;M),\ \hat B^{(t,m)}_{j \leftarrow l}(M) = 0
    \Big\},
    \nonumber
\end{align}
the coupling described above and
\eqref{proof, part i, claim 3, 2, ineq 1, lemma: cluster size, asymptotics, N i | n delta, l j} and imply that (for any $M,\delta$, and any $n$ large enough with $n\delta > M$)
\begin{align*}
    & \P\big( N^{>}_{i;j \leftarrow l}(n\delta) \geq 1 \big) 
    \\ 
    &
    \geq 
    \P\big( \hat  N^{>}_{i;j \leftarrow l}(M,n\delta) \geq 1  \big)
    \geq 
    \P\big( \hat  N^{>}_{i;j \leftarrow l}(M,n\delta) = 1  \big)
    \\ 
    & = 
    \sum_{k \geq 1}\P\big( Z_M = k \big)
    \cdot
     \binom{k}{1} \cdot 
            \frac{\P(B_{j\leftarrow l} > n\delta)}{\P\big(
                B_{j\leftarrow l} \in \{0\} \cup (M,\infty) 
                \big)}
        \cdot 
            \Bigg( \frac{
                \P\big( B_{j\leftarrow l} \in \{0\} \cup (M,n\delta]  \big)
                }{\P\big(
                    B_{j\leftarrow l} \in \{0\} \cup (M,\infty) 
                \big)}\Bigg)^{k - 1},
\end{align*}
and hence
\begin{equation} \label{proof: eq 1, claim 2, part i, lemma: cluster size, asymptotics, N i | n delta, l j}
    \begin{aligned}
        \frac{ \P\big( N^{>}_{i;j\leftarrow l}(n\delta) \geq 1 \big) }{ \P(B_{j\leftarrow l} > n\delta)  }
        &
        \geq
        \sum_{k \geq 1}  \frac{ k
            \P(Z_M = k)
        }{
            \P\big(B_{j\leftarrow l} \in \{0\} \cup (M,\infty) \big)
        }
            \Bigg( \frac{
                \P\big( B_{j\leftarrow l} \in \{0\} \cup (M,n\delta]  \big)
                }{\P\big(B_{j\leftarrow l} \in \{0\} \cup (M,\infty) \big)}\Bigg)^{k-1}.
    \end{aligned}
\end{equation}
By the regularly varying conditions in Assumption~\ref{assumption: heavy tails in B i j},
we have $\P(B_{j\leftarrow l} > M) > 0$ for any $M > 0$.
Also, we obviously have $\lim_{n \to \infty}\P(B_{j \leftarrow l} > n\delta) = 0$.
Consequently, 
given $\rho \in (0,1)$ and $\delta, M > 0$, 
in \eqref{proof: eq 1, claim 2, part i, lemma: cluster size, asymptotics, N i | n delta, l j}
it holds for any $n$ large enough that
\begin{align}
    \frac{ \P\big( N^{>}_{i;j\leftarrow l}(n\delta) \geq 1 \big) }{ \P(B_{j\leftarrow l} > n\delta)  }
        &
    \geq 
    \sum_{k \geq 1}
    \frac{
            k\P(Z_M
        = k            
        )
        }{
            \P\big( B_{j\leftarrow l}  \in \{0\} \cup (M,\infty) \big)
        } \cdot \rho^{k-1}
    =
    \frac{
        \E\big[ Z_M\rho^{  Z_M - 1 } \big]
    }{
        \P\big( B_{j\leftarrow l}  \in \{0\} \cup (M,\infty) \big)
    }.
    \label{proof, LB, claim 1, lemma: cluster size, asymptotics, N i | n delta, l j}
\end{align}
Note that $Z_M\rho^{Z_M - 1} \leq Z_M$ for any $\rho \in (0,1)$.
By monotone convergence, we get
\begin{align}
   \liminf_{n \to \infty}
   \frac{ \P\big( N^{>}_{i;j\leftarrow l}(n\delta) \geq 1 \big) }{ \P(B_{l,j} > n\delta)  }
        &
    \geq 
    \lim_{\rho \uparrow 1}
    \frac{
        \E\big[ Z_M\rho^{  Z_M - 1 } \big]
    }{
        \P\big( B_{j\leftarrow l}  \in \{0\} \cup (M,\infty) \big)
    }
    =
    \frac{\E Z_M}{\P\big( B_{j\leftarrow l}  \in \{0\} \cup (M,\infty) \big)}
    \label{proof, LB, intermediate, claim 1, lemma: cluster size, asymptotics, N i | n delta, l j}
\end{align}
for any $M,\delta > 0$.
Moreover, by repeating the arguments in \eqref{proof, part i, claim 3, expectation of N i l to j M} based on Proposition~1 of \cite{Asmussen_Foss_2018},
we get
$
\E Z_M = \E\big[ S^\leqslant_{i,l}(M) \big]\cdot \P\big(B_{j \leftarrow l} \in \{0\}\cup (M,\infty)\big). 
$
Then, in \eqref{proof, LB, intermediate, claim 1, lemma: cluster size, asymptotics, N i | n delta, l j}, we have
$
\liminf_{n \to \infty}
   \frac{ \P( N^{>}_{i;j\leftarrow l}(n\delta) \geq 1) }{ \P(B_{j \leftarrow l} > n\delta)  }
   \geq 
   \E\big[ S^\leqslant_{i,l}(M) \big]
$
for any $\delta, M > 0$.
Lastly, we have established earlier that
$
\lim_{M \to \infty}\E[S^{\leqslant}_{i,l}(M)] = \bar s_{i,l}.
$
Sending $M \to \infty$, we verify \eqref{proof, part i, claim 3, 2, goal LB, lemma: cluster size, asymptotics, N i | n delta, l j}.

\medskip
$(ii)$ 
Since there are only finitely many possible choices for such $\mathcal T$,
it suffices to fix some $\mathcal T \subseteq [d]^2$
with $|\mathcal T| \geq 2$ and prove the claim.
For clarity of the proof, we focus on the case where $|\mathcal T| = 2$.
That is, we fix some $(l,j)\neq(l^\prime,j^\prime)$ and show that, for all $\delta > 0$ small enough,
\begin{align*}
            \P\Big(
                N^{>}_{i;j \leftarrow l}(n\delta) \geq 1,\ 
                N^{>}_{i;j^\prime \leftarrow l^\prime}(n\delta) \geq 1\Big)
            =
            \lo 
                \Big(
                    n\cdot\P(B_{j \leftarrow l} > n\delta)\P(B_{j^\prime \leftarrow l^\prime} > n\delta)
                \Big)
\end{align*}
as $n \to \infty$.
However, we stress that this approach can be easily applied to more general cases,
at the cost of more involved notations.
Also, since each $B_{j \leftarrow l}$ is a non-negative integer-valued random variable, 
the essential lower bound
$
\underline{b}_{j \leftarrow l} \delequal \min\big\{ k \geq 0:\ \P(B_{j \leftarrow l} = k)> 0  \big\}
$
is  well-defined for each pair $(l,j)$.
We first consider the case where $\underline b_{j \leftarrow l} = 0$ and $\underline b_{j^\prime \leftarrow l^\prime} = 0$, i.e.,
\begin{align}
    \P(B_{j \leftarrow l} = 0) > 0,\qquad
    \P(B_{j^\prime \leftarrow l^\prime} = 0) > 0.
    \label{proof, part ii, condition on essential lower bound b i j, lemma: cluster size, asymptotics, N i | n delta, l j}
\end{align}
Towards the end of this proof, we address the cases where \eqref{proof, part ii, condition on essential lower bound b i j, lemma: cluster size, asymptotics, N i | n delta, l j} does not hold.

Let
\begin{align}
    Z_n(\delta)
    & \delequal
    \#\Big\{
    (t,m) \in \mathbb N^2:\
        t \geq 1,\
            m \leq X^{\leqslant}_{i,l}(t - 1;n\delta),\ 
            B^{(t,m)}_{j \leftarrow l} \in \{0\}\cup(n\delta,\infty)
                    \Big\}, \nonumber
    \\ 
    Z^\prime_n(\delta)
    & \delequal
    \#\Big\{
    (t,m) \in \mathbb N^2:\
        t \geq 1,\
            m \leq X^{\leqslant}_{i,l^\prime}(t - 1;n\delta),\ 
            B^{(t,m)}_{j^\prime \leftarrow l^\prime} \in \{0\}\cup(n\delta,\infty)
                    \Big\}. 
    \nonumber
\end{align}
Take $\Delta > 0$.
Using the coupling constructed in the proof of Claim~\eqref{claim 3, 2, lemma: cluster size, asymptotics, N i | n delta, l j} in part $(i)$,
we have
\begin{align}
    &\P\big(
                N^{>}_{i;j \leftarrow l}(n\delta) \geq 1,\ 
                N^{>}_{i;j^\prime \leftarrow l^\prime}(n\delta) \geq 1\big)
    \label{proof, part ii, math display 1, lemma: cluster size, asymptotics, N i | n delta, l j}
    \\
    = &  \sum_{k \geq 1}\sum_{k^\prime \geq 1}\sum_{s \geq 1} \P\Big(
                    Z_n(\delta) = k,\ 
        Z^\prime_n(\delta) = k^\prime,\ 
        \norm{\bm S^{\leqslant}_i(n\delta)} = s
        \Big)
    \nonumber
   \\ 
    & \cdot 
        \sum_{p = 1}^k
        \binom{k}{p} \cdot 
            \Bigg(\frac{\P(B_{j \leftarrow l} > n\delta)}{\P\big(B_{j \leftarrow l} \in \{0\} \cup (n\delta,\infty)  \big)}\Bigg)^{p}
        \cdot 
            \Bigg( \frac{\P(B_{j \leftarrow l} = 0)}{\P\big(B_{j \leftarrow l} \in \{0\} \cup (n\delta,\infty)  \big)}\Bigg)^{k - p}
    \nonumber
    \\ 
    &\cdot 
                \sum_{p^\prime = 1}^{k^\prime}
        \binom{k^\prime}{p^\prime} \cdot 
            \Bigg(\frac{\P(B_{j^\prime \leftarrow l^\prime } > n\delta)}{
                \P\big(B_{j^\prime \leftarrow l^\prime} \in \{0\} \cup (n\delta,\infty)  \big)
            }\Bigg)^{p^\prime}
        \cdot 
            \Bigg( \frac{\P(B_{j^\prime \leftarrow l^\prime} = 0)}{
                 \P\big(B_{j^\prime \leftarrow l^\prime} \in \{0\} \cup (n\delta,\infty)  \big)
                }\Bigg)^{k^\prime - p^\prime}
    \nonumber
    \\
\leq &
        \P\Big(\norm{\bm S^{\leqslant}_i(n\delta)} > \floor{n\Delta}\Big)
    \nonumber
    \\ 
+ &
    \sum_{k \geq 1}\sum_{k^\prime \geq 1}\sum_{s \leq \floor{n\Delta} } \P\Big(
                    Z_n(\delta) = k,\ 
        Z^\prime_n(\delta) = k^\prime,\ 
        \norm{\bm S^{\leqslant}_i(n\delta)} = s
        \Big)
    \nonumber
    \\ 
    & \cdot 
        \P\Bigg(
        \text{Binomial}\bigg( k,\ 
        \frac{\P(B_{j \leftarrow l} > n\delta)}{\P\big(B_{j \leftarrow l} \in \{0\} \cup (n\delta,\infty)  \big)}
        \bigg) \geq 1
        \Bigg)
    \nonumber
    \\ & 
    \cdot 
        \P\Bigg(
        \text{Binomial}\bigg( k^\prime,\ 
        \frac{\P(B_{j^\prime \leftarrow l^\prime } > n\delta)}{\P\big(B_{ j^\prime \leftarrow l^\prime } \in \{0\} \cup (n\delta,\infty)  \big)}
        \bigg) \geq 1
        \Bigg)
    \nonumber
    \\ 
    \leq &
        \P\Big(\norm{\bm S^{\leqslant}_i(n\delta)} > \floor{ n\Delta } \Big)
    \nonumber
    \\ 
    + &
    \sum_{k \geq 1}\sum_{k^\prime \geq 1}\sum_{s \leq \floor{ n\Delta } } \P\Big(
                    Z_n(\delta)
        = k,\ 
        Z^\prime_n(\delta)
        = k^\prime,\ 
        \norm{\bm S^{\leqslant}_i(n\delta)} = s
        \Big)
    \nonumber
    \\ 
    &\cdot 
        k \cdot 
        \frac{\P(B_{j\leftarrow l} > n\delta)}{\P\big(B_{j\leftarrow l} \in \{0\} \cup (n\delta,\infty) \big)}
        \cdot 
        k^\prime \cdot 
        \frac{\P(B_{j^\prime \leftarrow l^\prime} > n\delta)}{\P\big(B_{j^\prime\leftarrow l^\prime} \in \{0\} \cup (n\delta,\infty) \big)}
    \nonumber
    \\ 
    &\qquad\qquad
    \text{ by the preliminary bound }
    \P\big(\text{Binomial}(k,p) \geq 1\big) \leq \E\big[\text{Binomial}(k,p)\big] = kp
    \nonumber
    \\ 
     \leq &
        \P\Big(\norm{\bm S^{\leqslant}_i(n\delta)} > n\Delta\Big)
    \nonumber
    \\
    + &
    \underbrace{
        \E\bigg[ \norm{\bm S^{\leqslant}_i(n\delta)}^2 \mathbbm{I}\Big\{ \norm{\bm S^{\leqslant}_i(n\delta)} \leq n\Delta  \Big\} \bigg]
    \cdot 
    \prod_{ (p,q) = (l,j) \text{ or }(l^\prime,j^\prime) }
     \frac{\P(B_{q\leftarrow p} > n\delta)}{\P\big(B_{q\leftarrow p} \in \{0\} \cup (n\delta,\infty) \big)}
    }_{\delequal I(n,\Delta,\delta)}.
    \nonumber
\end{align}
The last inequality follows from $Z_n(\delta) \leq \norm{\bm S_i^\leqslant(n\delta)}$ and 
$Z_n^\prime(\delta) \leq \norm{\bm S_i^\leqslant(n\delta)}$.
Applying Lemma~\ref{lemma: tail bound, pruned cluster size S i leq n delta},
we fix some $\delta_0 = \delta_0(\Delta) > 0$ such that for any $\delta \in (0,\delta_0)$,
\begin{align}
    \P\Big(\norm{\bm S^{\leqslant}_i(n\delta)} > n\Delta\Big) = 
    \lo 
                \Big(
                    n\cdot\P(B_{j \leftarrow l} > n\delta)\P(B_{j^\prime \leftarrow l^\prime} > n\delta)
                \Big).
    \label{proof, bound, part ii, lemma: cluster size, asymptotics, N i | n delta, l j}
\end{align}
Meanwhile,
by our running assumption \eqref{proof, part ii, condition on essential lower bound b i j, lemma: cluster size, asymptotics, N i | n delta, l j},
there exists $C \in (0,\infty)$ such that for any $\delta > 0$ and any $n \geq 1$,
\begin{align}
    I(n,\Delta,\delta) 
    & \leq 
        C  \cdot \E\bigg[ \norm{\bm S^{\leqslant}_i(n\delta)}^2 \mathbbm{I}\Big\{ \norm{\bm S^{\leqslant}_i(n\delta)} \leq n\Delta  \Big\} \bigg]
        \cdot 
        \P(B_{j \leftarrow l} > n\delta)
        \P(B_{j^\prime \leftarrow l^\prime } > n\delta).
    \label{proof, part ii, ineq for I n Delta delta, lemma: cluster size, asymptotics, N i | n delta, l j}
\end{align}
Let $\alpha^* = \min\{\alpha_{q \leftarrow p}:\ p,q \in [d]\}$.
Under Assumption~\ref{assumption: heavy tails in B i j}, we have $\alpha^* > 1$.
Then, by Theorem~2 of \cite{Asmussen_Foss_2018},
we get $\P(\norm{\bm S_i} > x) \in \RV_{-\alpha^*}(x)$ as $x \to \infty$.
Now, we consider two different cases.
If $\alpha^* > 2$, then
\begin{align*}
    \E\bigg[ \norm{\bm S^{\leqslant}_i(n\delta)}^2 \mathbbm{I}\Big\{ \norm{\bm S^{\leqslant}_i(n\delta)} \leq n\Delta  \Big\} \bigg]
    & \leq 
    \E[ \norm{\bm S_i}^2]
    < \infty \quad
    \text{due to \eqref{property: stochastic comparison, S and pruned S} and $\alpha^* > 2$}.
\end{align*}
Plugging this bound into \eqref{proof, part ii, ineq for I n Delta delta, lemma: cluster size, asymptotics, N i | n delta, l j}, we verify that
\begin{align}
    I(n,\Delta,\delta)
    =
    \lo 
                \Big(
                    n\cdot\P(B_{j \leftarrow l} > n\delta)\P(B_{j^\prime \leftarrow l^\prime} > n\delta)
                \Big)
    \qquad\text{when }\alpha^* > 2.
    \label{proof, bound term I n Delta delta, 1, lemma: cluster size, asymptotics, N i | n delta, l j}
\end{align}
If $\alpha^* \in (1,2]$,
we obtain
$
\E\Big[ \norm{\bm S^{\leqslant}_i(n\delta)}^2 \mathbbm{I}\Big\{ \norm{\bm S^{\leqslant}_i(n\delta)} \leq n\Delta  \Big\} \Big]
     \leq 
    \int_0^{n\Delta} 2x\P(\norm{\bm S_i} >x)dx \in \RV_{2 - \alpha^*}(n)
$
using \eqref{property: stochastic comparison, S and pruned S} and Karamata's Theorem
(see, e.g., Theorem 2.1 of \cite{resnick2007heavy}).
Due to $\alpha^* > 1$, 
any $\RV_{2 - \alpha^*}(n)$ function is of order $\lo(n)$.
Plugging this into  \eqref{proof, part ii, ineq for I n Delta delta, lemma: cluster size, asymptotics, N i | n delta, l j}, we get
\begin{align}
    I(n,\Delta,\delta)
    =
    \lo 
                \Big(
                    n\cdot\P(B_{j \leftarrow l} > n\delta)\P(B_{j^\prime \leftarrow l^\prime} > n\delta)
                \Big)
    \qquad\text{when }\alpha^* \in (1,2].
    \label{proof, bound term I n Delta delta, 2, lemma: cluster size, asymptotics, N i | n delta, l j}
\end{align}
Plugging \eqref{proof, bound, part ii, lemma: cluster size, asymptotics, N i | n delta, l j}, \eqref{proof, bound term I n Delta delta, 1, lemma: cluster size, asymptotics, N i | n delta, l j}, and \eqref{proof, bound term I n Delta delta, 2, lemma: cluster size, asymptotics, N i | n delta, l j} into \eqref{proof, part ii, math display 1, lemma: cluster size, asymptotics, N i | n delta, l j}, we conclude the proof of part $(ii)$ under condition \eqref{proof, part ii, condition on essential lower bound b i j, lemma: cluster size, asymptotics, N i | n delta, l j}.

Lastly, we explain how to extend the proof to the cases where the condition \eqref{proof, part ii, condition on essential lower bound b i j, lemma: cluster size, asymptotics, N i | n delta, l j} does not hold.
Recall the definition of the essential lower bounds 
$
\underline{b}_{j \leftarrow l} = \min\big\{ k \geq 0:\ \P(B_{j \leftarrow l} = k)> 0  \big\},
$
and consider the following branching process 
\begin{align}
    \tilde{\bm X}_i(t;M)
    = \sum_{j \in [d]} \sum_{m = 1}^{ \tilde X_{i,j}(t - 1;M) }\tilde{\bm B}^{(t,m)}_{\bcdot \leftarrow j}(M),
    \qquad \forall  t \geq 1,
    \nonumber
\end{align}
under initial values $\tilde{\bm X}_{i}(0) = \bm e_i$, where
\begin{align}
    \tilde B_{l \leftarrow j}^{(t,m)}
    =
     \underline b_{l \leftarrow j} \vee \big( B^{(t,m)}_{l \leftarrow j}\mathbbm{I}\{  B^{(t,m)}_{l \leftarrow j} \leq M \} \big),
     \quad 
     \tilde{\bm B}_{\bcdot \leftarrow j}^{(t,m)}
     =
     ( \tilde B_{1 \leftarrow j}^{(t,m)},\ldots, \tilde B_{d \leftarrow j}^{(t,m)})^\top.
    \nonumber
\end{align}
That is, $\tilde{\bm X}_i(t;M)$ modifies the process $\bm X^\leqslant_i(t;M)$ defined in \eqref{def: branching process X n leq M} by pruning down to the essential lower bound of each $B_{l\leftarrow j}$ instead of $0$.
Obviously, $\bm X^\leqslant_i(t;n\delta) \leq \tilde{\bm X}_i(t;n\delta)$ for each $t,n$.
Then, from the definition of $N_{i;j \leftarrow l}^{>}(n\delta)$ in \eqref{proof, equivalent def for N > i j l, lemma: cluster size, asymptotics, N i | n delta, l j}, we get
\begin{align*}
    N_{i;j \leftarrow l}^{>}(n\delta)
    \leq 
    \#\Big\{
        (t,m) \in \mathbb N^2:\ 
        t \geq 1,\
            m \leq \tilde X_{i,l}(t - 1;n\delta),\ 
            B^{(t,m)}_{j \leftarrow l}(M) > n\delta
    \Big\}.
\end{align*}
Using the coupling constructed when proving Claim~\eqref{claim 3, 2, lemma: cluster size, asymptotics, N i | n delta, l j} in part $(i)$,
we arrive at upper bounds analogous to those in the display~\eqref{proof, part ii, math display 1, lemma: cluster size, asymptotics, N i | n delta, l j},
with the key difference being that the terms  $\P(B_{j \leftarrow l} \in \{0\} \cup (n\delta,\infty))$ and $\P(B_{j^\prime \leftarrow l^\prime} \in \{0\} \cup (n\delta,\infty))$ in the denominators
are substituted by 
$\P(B_{j \leftarrow l} \in \{ \underline b_{j \leftarrow l} \} \cup (n\delta,\infty))$ and $\P(B_{j^\prime \leftarrow l^\prime} \in \{ \underline b_{j^\prime \leftarrow l^\prime} \} \cup (n\delta,\infty))$.
In particular, by the definition of the essential lower bounds, we must have $\P(B_{j\leftarrow l} = \underline b_{j\leftarrow l}) > 0$ and 
$\P(B_{j^\prime\leftarrow l^\prime} = \underline b_{j^\prime\leftarrow l^\prime}) > 0$,
so an upper bound of the form \eqref{proof, part ii, ineq for I n Delta delta, lemma: cluster size, asymptotics, N i | n delta, l j} would still hold, and the subsequent calculations would follow.
We omit the details here to avoid repetition.
\end{proof}

Utilizing Lemma~\ref{lemma: cluster size, asymptotics, N i | n delta, l j},
we provide the proof of Lemma~\ref{lemma: cluster size, asymptotics, N i | n delta, cdot j, crude estimate}.

\begin{proof}[Proof of Lemma~\ref{lemma: cluster size, asymptotics, N i | n delta, cdot j, crude estimate}]
\linksinpf{lemma: cluster size, asymptotics, N i | n delta, cdot j, crude estimate}
(i)
By the definition of $N^{>}_{i;j}(n\delta) = \sum_{l \in [d]}N^{>}_{i;j \leftarrow l}(n\delta)$  in \eqref{def: cluster size, N i | M cdot j},
we have
$
 \P\big(N^{>}_{i;j}(n\delta) = 1\big)
 \leq 
 \sum_{l \in [d]}\P\big(N^{>}_{i;j \leftarrow l}(n\delta) = 1\big).
$
By \eqref{claim 1, lemma: cluster size, asymptotics, N i | n delta, l j} in part (i) of Lemma~\ref{lemma: cluster size, asymptotics, N i | n delta, l j},
\begin{align}
    \P\big(N^{>}_{i;j \leftarrow l}(n\delta) = 1\big)
    \sim 
    \bar s_{i,l}\cdot \P(B_{j \leftarrow l} > n\delta)
    \in 
    \RV_{ -\alpha_{j \leftarrow l} }(n),
    \qquad
    \forall l \in [d],
    \label{proof: bound 1, part i, lemma: cluster size, asymptotics, N i | n delta, cdot j, crude estimate}
\end{align}
 where $\bar s_{i,l} > 0\ \forall l \in [d]$; see Assumptions~\ref{assumption: heavy tails in B i j} and \ref{assumption: regularity condition 1, cluster size, July 2024}.
Also, 
under Assumption~\ref{assumption: regularity condition 2, cluster size, July 2024},
the argument minimum $l^*(j)$ in \eqref{def: cluster size, alpha * l * j} is uniquely defined for each $j \in [d]$, and we have
 $
\alpha^*(j) = \alpha_{j \leftarrow l^*(j)},
 $
 $\alpha^*(j)< \alpha_{j \leftarrow l}$ for any $l \neq l^*(j)$.
This leads to
\begin{align}
    \limsup_{n \to \infty}
    {\P\big(N^{>}_{i;j}(n\delta) = 1\big)}\Big/{\Big( \bar s_{i,l^*(j)}\cdot \P(B_{j \leftarrow l^*(j) } > n\delta) \Big)} \leq 1.
    \label{proof: intermediate results 1, part i, lemma: cluster size, asymptotics, N i | n delta, cdot j, crude estimate}
\end{align}
On the other hand, observe the lower bound
\begin{align*}
  & \P\big(N^{>}_{i;j}(n\delta) = 1\big)
    \geq 
    \P\Big(
    N^{>}_{i;j \leftarrow l^*(j) }(n\delta) = 1,\ 
    N^{>}_{i;j \leftarrow l}(n\delta) = 0\ \forall l \neq l^*(j)
    \Big)
    \\ 
    & \geq 
    \underbrace{\P\big(
      N^{>}_{i;j \leftarrow l^*(j) }(n\delta) = 1
     \big)
     }_{\delequal p_1(n)}
     -
    \underbrace{ \sum_{l \in [d]:\ l \neq l^*(j)}\P\big(
    N^{>}_{i;j \leftarrow l^*(j)}(n\delta) = 1,\ N^{>}_{i;j \leftarrow l}(n\delta) \geq 1
     \big)
     }_{
     \delequal p_2(n)
     }.
\end{align*}
For the term $p_1(n)$,
it follows from
\eqref{proof: bound 1, part i, lemma: cluster size, asymptotics, N i | n delta, cdot j, crude estimate} that
$p_1(n) \sim \bar s_{i,l^*(j)}\P(B_{j \leftarrow l^*(j) } > n\delta)$
as $n \to \infty$.
As for the term $p_2(n)$,
we apply part (ii) of Lemma~\ref{lemma: cluster size, asymptotics, N i | n delta, l j} and get
(for any $\delta > 0$ small enough)
\begin{align}
    p_2(n) & = \sum_{l \in [d]:\ l \neq l^*(j)}\lo \Big( n\P(B_{j \leftarrow l^*(j)} > n\delta)\P(B_{j \leftarrow l} > n\delta) \Big)
= \lo \Big(\P(B_{j \leftarrow l^*(j)} > n\delta)\Big).
\label{proof: intermediate results 2, part i, lemma: cluster size, asymptotics, N i | n delta, cdot j, crude estimate}
\end{align}
The last equality follows from $\alpha_{j \leftarrow l} > 1\ \forall l \in [d]$; see Assumption~\ref{assumption: heavy tails in B i j}.
In summary, we have
\begin{align}
    \liminf_{n \to \infty}
    {\P\big(N^{>}_{i;j}(n\delta) = 1\big)}\Big/{\Big(  \bar s_{i,l^*(j)}\cdot\P(B_{j \leftarrow l^*(j)} > n\delta) \Big)} \geq 1.
    \label{proof: intermediate results 3, part i, lemma: cluster size, asymptotics, N i | n delta, cdot j, crude estimate}
\end{align}
Combining \eqref{proof: intermediate results 1, part i, lemma: cluster size, asymptotics, N i | n delta, cdot j, crude estimate} and \eqref{proof: intermediate results 3, part i, lemma: cluster size, asymptotics, N i | n delta, cdot j, crude estimate},
we conclude the proof of Claim~\eqref{claim 1, lemma: cluster size, asymptotics, N i | n delta, cdot j, crude estimate}.
Next, observe that
\begin{equation}
    \begin{aligned}
        \P\big(N^{>}_{i;j}(n\delta) \geq 2\big)
    & \leq 
    \sum_{l \in [d]}\P\big(N^{>}_{i;j \leftarrow l}(n\delta) \geq 2\big)
    \\ 
    &
    +
    \sum_{l,l^\prime \in [d]:\ l \neq l^\prime}
    \P\big(N^{>}_{i;j \leftarrow l}(n\delta) \geq 1,\ N^{>}_{i;j \leftarrow l^\prime}(n\delta) \geq 1\big).
    \label{proof: intermediate results 4, part i, lemma: cluster size, asymptotics, N i | n delta, cdot j, crude estimate}
    \end{aligned}
\end{equation}
Claim~\eqref{claim 2, lemma: cluster size, asymptotics, N i | n delta, cdot j, crude estimate} then follows from part (i), Claim \eqref{claim 2, lemma: cluster size, asymptotics, N i | n delta, l j} and part (ii) of Lemma~\ref{lemma: cluster size, asymptotics, N i | n delta, l j}.

To prove Claims~\eqref{claim 1, part iii, lemma: cluster size, asymptotics, N i | n delta, cdot j, refined estimates} and \eqref{claim 2, part iii, lemma: cluster size, asymptotics, N i | n delta, cdot j, refined estimates},
we
define the event
$
A(n,\delta) \delequal \big\{
                N^{>}_{i;j\leftarrow l^*(j)}(n\delta) = 1;\ 
                N^{>}_{i;j\leftarrow l}(n\delta) = 0\ \forall l \neq l^*(j)\big\}.
$
By  \eqref{def: W i M j, pruned cluster, 1, cluster size}--\eqref{def: cluster size, N i | M cdot j},
the law of $W^{>}_{i;j}(n\delta)$ conditioned on the event $A(n,\delta)$
is the same as
$
\P\big( B_{ j \leftarrow l^*(j) }\in \ \cdot\ \big| B_{ j \leftarrow l^*(j) } > n\delta  \big).
$
As a result,
\begin{equation}\label{proof: equality for conditional law of W, part i, lemma: cluster size, asymptotics, N i | n delta, cdot j, refined estimate}
    \begin{aligned}
        & \P\Big(
               W^{>}_{i;j}(n\delta) > nx\ 
                \Big|\
                A(n,\delta)
            \Big)
        =\frac{\P(B_{j \leftarrow l^*(j)} > nx)}{\P(B_{j \leftarrow l^*(j)} > n\delta)},
        \qquad\forall x \geq \delta.
    \end{aligned}
\end{equation}
Next, given $\delta \in (0,c)$ and $x \geq \delta$,
by conditioning on $A(n,\delta)$ or $\big(A(n,\delta)\big)^\complement$,
we get
\begin{align*}
    & \lim_{n \to \infty}
            \sup_{ x \in [c,C] }
            \Bigg| 
                \frac{ 
                    \P\big(
                W^{>}_{i;j}(n\delta) > nx\ 
                \big|\ 
                N^{>}_{i;j}(n\delta)  \geq 1
            \big)
                }{
                    ({\delta}/{x})^{\alpha^*(j)}
                }
                - 1
            \Bigg|
    \\ 
    & \leq 
    \lim_{n \to \infty}
            \sup_{ x \in [c,C] }
            \Bigg| 
                \frac{ 
                    \P(B_{ j \leftarrow  l^*(j)} > nx) \big/ \P(B_{ j \leftarrow  l^*(j)} > n\delta)
                }{
                    ({\delta}/{x})^{\alpha^*(j)}
                }
            \cdot 
            \P\big(
                A(n,\delta)\ 
                \big|\ 
                N^{>}_{i;j}(n\delta)  \geq 1
            \big)
                - 1
            \Bigg|
    \\ 
    & \quad  + 
    \lim_{n \to \infty} \bigg( \frac{C}{\delta} \bigg)^{ \alpha^*(j) } 
        \cdot 
    \P\Big( \big(A(n,\delta)\big)^\complement \ \Big|\ N^{>}_{i;j}(n\delta)  \geq 1 \Big)
    \qquad
    \text{due to \eqref{proof: equality for conditional law of W, part i, lemma: cluster size, asymptotics, N i | n delta, cdot j, refined estimate}}.
\end{align*}
Suppose that 
Claim~\eqref{claim 1, part iii, lemma: cluster size, asymptotics, N i | n delta, cdot j, refined estimates} holds for any $\delta > 0$ small enough:
that is, $\P\big( A(n,\delta) \big| N_{i;j}(n\delta) \geq 1\big) \to 1$ as $n \to \infty$.
Then, 
by applying uniform convergence theorem (e.g., Proposition~2.4 of \cite{resnick2007heavy}) 
to $\P(B_{j \leftarrow l^*(j)}> x) \in \RV_{ -\alpha^*(j) }(x)$
in the display above, 
we verify  Claim~\eqref{claim 2, part iii, lemma: cluster size, asymptotics, N i | n delta, cdot j, refined estimates} for any $\delta > 0$.
Now, it only remains to prove Claim~\eqref{claim 1, part iii, lemma: cluster size, asymptotics, N i | n delta, cdot j, refined estimates}.
In particular,
note that
\begin{align*}
    \P\big( N^{>}_{i;j}(n\delta)  \geq 1 \big)
    & \geq 
    \underbrace{ 
        \P\big(
            N^{>}_{i;j\leftarrow l^*(j)}(n\delta) = 1;\ 
                N^{>}_{i;j\leftarrow l}(n\delta) = 0\ \forall l \neq l^*(j)
            \big)
    }_{ \delequal \underline{p}(n,\delta) }
    ,
    \\ 
    \P\big( 
            N^{>}_{i;j}(n\delta)  \geq 1
        \big)
    & \leq 
    \underbrace{
        \P\big(  N^{>}_{i;j\leftarrow l^*(j)}(n\delta) \geq 1\big)
    }_{ \delequal \bar p_*(n,\delta) }
    +
    \sum_{ l \in [d]:\ l\neq l^*(j)  }
        \underbrace{
            \P\big( N^{>}_{i;j\leftarrow l}(n\delta)\geq 1\big)
        }_{ \delequal \bar p_l (n,\delta) }.
\end{align*}
Repeating the calculations in \eqref{proof: intermediate results 1, part i, lemma: cluster size, asymptotics, N i | n delta, cdot j, crude estimate}--\eqref{proof: intermediate results 4, part i, lemma: cluster size, asymptotics, N i | n delta, cdot j, crude estimate},
we can show that
$
 \lim_{n \to \infty} \underline{p}(n,\delta)\big/ \bar p_*(n,\delta)  = 1
$
and
$
\lim_{n \to \infty} \bar p_l(n,\delta)/\bar p_*(n,\delta) = 0
$
(for each $l \in [d],\ l \neq l^*(j)$)
under any $\delta > 0$ small enough.
This concludes the proof of Claim \eqref{claim 1, part iii, lemma: cluster size, asymptotics, N i | n delta, cdot j, refined estimates}.

\smallskip
(ii) 
By the definition of $N^{>}_{i;j}(n\delta) = \sum_{l \in [d]}N^{>}_{i;j \leftarrow l}(n\delta)$,
$$
 \P\big(
                N^{>}_{i;j}(n\delta) \geq 1\ \forall j \in \mathcal J
            \big)
    \leq 
    \sum_{
        l_j \in [d]\ \forall j \in \mathcal J
    }
    \P\big(
                N^{>}_{i;j \leftarrow l_j }(n\delta) \geq 1\ \forall j \in \mathcal J
            \big).
$$
Applying part (ii) of Lemma~\ref{lemma: cluster size, asymptotics, N i | n delta, l j},
for each $(l_j)_{j \in \mathcal J} \in [d]^{|\mathcal J|}$
we have
$$
\P\big(
                N^{>}_{i;j \leftarrow l_j }(n\delta) \geq 1\ \forall  j \in \mathcal J
        \big)
    = 
    \lo 
    \bigg(
        n^{|\mathcal J| - 1}
        \prod_{j \in \mathcal J}\P(B_{j \leftarrow l_j} > n\delta)
    \bigg)\text{ as }n \to \infty,
$$
under any $\delta > 0$ small enough.
Lastly, by Assumption~\ref{assumption: regularity condition 2, cluster size, July 2024} and the definitions in   \eqref{def: cluster size, alpha * l * j},
we have 
$
\lo 
    \big(
        n^{|\mathcal J| - 1}
        \prod_{j \in \mathcal J}\P(B_{j \leftarrow l_j} > n\delta)
    \big)
=
\lo 
    \big(
        n^{|\mathcal J| - 1}
        \prod_{j \in \mathcal J}\P(B_{j \leftarrow l^*(j)} > n\delta)
    \big).
$
This establishes part (ii).

\smallskip
(iii)
Note that it suffices to prove the claim for the case of
$|\mathcal I| = 1$, i.e., $\mathcal I = \{i\}$ for some $i \in [d]$.
Specifically,
let $\delta_0 > 0$ be characterized as in part (ii).
It suffices to show that
\begin{align}
    \limsup_{ n \to \infty}
    \sup_{ 
        T \geq nc
    }
    \frac{
        \P\big(
            \sum_{m = 1}^{ T }N^{>,(m)}_{i;j}(\delta T)  \geq 1\text{ iff }j \in \mathcal J
        \big)
    }{
    \prod_{j \in \mathcal J}n\P( B_{j \leftarrow l^*(j)} > n\delta)
    }
    < \infty,
    \quad
    \forall \delta \in (0,\delta_0),\ c > 0.
    \label{proof: goal, part 3, lemma: cluster size, asymptotics, N i | n delta, cdot j, crude estimate}
\end{align}
To see why \eqref{proof: goal, part 3, lemma: cluster size, asymptotics, N i | n delta, cdot j, crude estimate} implies \eqref{claim, part iii, lemma: cluster size, asymptotics, N i | n delta, cdot j, crude estimate},
we use $\mathfrak T_{\mathcal I \leftarrow \mathcal J}$
to denote the set of all assignment from $\mathcal J$ to $\mathcal I$, \emph{allowing for replacements}:
that is,
$\mathfrak T_{\mathcal I \leftarrow \mathcal J}$
contains all $\{ \mathcal J(i) \subseteq \mathcal J:\ i \in \mathcal I \}$
satisfying $\bigcup_{i \in \mathcal I}\mathcal J(i) = \mathcal J$.
Observe that
\begin{align}
    & \P\Big(
            {N^{>|\delta}_{\bm t(\mathcal I);j}} \geq 1\text{ iff }j \in \mathcal J
        \Big)
        \label{proof: decomp of events, N geq 1, part ii, lemma: cluster size, asymptotics, N i | n delta, cdot j, crude estimate}
    = 
    \P\Bigg(
        \sum_{i \in \mathcal I} \sum_{m = 1}^{ t_i }N^{>,(m)}_{i;j}(\delta t_i) \geq 1
        \text{ iff }j \in \mathcal J
    \Bigg)
    \quad
    \text{by \eqref{def, proof cluster size, N W mathcal I mathcal J bcdot j}}
    \\
    & = 
    \sum_{ 
        \{ \mathcal J(i):\ i \in \mathcal I \} \in \mathfrak T_{ \mathcal I \leftarrow \mathcal J }
    }
    \prod_{ i \in \mathcal I}
    \P\Bigg(
            \sum_{m = 1}^{ t_i }N^{>,(m)}_{i;j}(\delta t_i)  \geq 1\text{ iff }j \in \mathcal J(i)
        \Bigg).
        \nonumber
\end{align}
The last equality follows from the independence of the random vectors
$
\big\{  \big( N^{>,(m)}_{i;j}(M) \big)_{j \in [d]}:\ m \geq 1  \big\}
$
across $i \in [d]$; see \eqref{proof: def copies of W > and N > vectors}.
Applying \eqref{proof: goal, part 3, lemma: cluster size, asymptotics, N i | n delta, cdot j, crude estimate} to each term
$
\P\big(
            \sum_{m = 1}^{ t_i }N^{>,(m)}_{i;j}(\delta t_i)  \geq 1\text{ iff }j \in \mathcal J(i)
        \big)
$
in \eqref{proof: decomp of events, N geq 1, part ii, lemma: cluster size, asymptotics, N i | n delta, cdot j, crude estimate},
we verify Claim \eqref{claim, part iii, lemma: cluster size, asymptotics, N i | n delta, cdot j, crude estimate} for any $|\mathcal I| \geq 2$.

Now, it only remains to prove \eqref{proof: goal, part 3, lemma: cluster size, asymptotics, N i | n delta, cdot j, crude estimate} (i.e., part (iii) with $\mathcal I = \{i\}$).
To proceed, 
we say that $\mathscr J =  \{\mathcal J_1,\ldots,\mathcal J_k\}$ is a partition of $\mathcal J$ if: 
       (i) $\emptyset \neq \mathcal J_l \subseteq \mathcal J$ for each $l \in [k]$, and
        $\cup_{l \in [k]}\mathcal J_l = \mathcal J$; (ii)
        $\mathcal J_p \cap \mathcal J_q = \empty$ for any $p \neq q$ (that is,  $\mathcal J_l$'s are disjoint).
Let ${\mathbb J}$ be the set of all partitions of $\mathcal J$,
and note that $|\mathbb J| < \infty$.
Given partition $\mathscr J = \{ \mathcal J_1,\ldots,\mathcal J_k \}$
and some
$T \in \mathbb N$,
define the event
\begin{align}
    A^\mathscr{J}_n(T,\delta)
    &\delequal 
    \Big\{
    \exists \{m_1,\cdots,m_k\} \subseteq [T]
    \text{ such that }
    N^{>,(m_l)}_{i;j}(\delta T) \geq 1\ \forall l \in [k],\  j \in \mathcal{J}_l
    \Big\}.
    \label{proof: def event A T partition J, part iii, lemma: cluster size, asymptotics, N i | n delta, cdot j, crude estimate}
\end{align}
First, note that
for any $T \in \mathbb N$ and $\delta > 0$,
\begin{align}
    \Bigg\{
        \sum_{m = 1}^{ T }N^{>,(m)}_{i;j}(\delta T)  \geq 1\text{ iff }j \in \mathcal J
    \Bigg\}
    \subseteq
    \bigcup_{ \mathscr J \in \mathbb J }A_n^\mathscr{J}(T,\delta).
    \label{proof, part i, decomp, upper bound, lemma: cluster size, asymptotics, N i | n delta, cdot j, multiple ancsetor}
\end{align}
Next, given $T \in \mathbb N$ and some partition $\mathscr J = \{ \mathcal J_1,\ldots,\mathcal J_k \}$, note that
(in the display below we write $p(i,M,\mathcal T)
 \delequal 
 \P\big(
N^{>}_{i;j}(M) \geq 1\ \forall j \in \mathcal{T}
    \big)$)
\begin{align}
\P\Big(A_n^\mathscr{J}(T,\delta)\Big)
& \leq 
 \prod_{l \in [k]}
 \P\Big(
    \text{Binomial}\big(T,p(i,\delta T,\mathcal J_l) \big) \geq 1
 \Big)
    \nonumber
\\
& \leq 
\prod_{l \in [k]}
\E\Big[  \text{Binomial}\big(T,p(i,\delta T,\mathcal J_l) \big)  \Big]
%
=
    \prod_{l \in [k]} T \cdot p(i,\delta T,\mathcal J_l).
    \label{proof, part i, event A i mathcal T upper bound, lemma: cluster size, asymptotics, N i | n delta, cdot j, multiple ancsetor}
\end{align}
Furthermore, 
by applying either part (i), Claims \eqref{claim 1, lemma: cluster size, asymptotics, N i | n delta, cdot j, crude estimate}--\eqref{claim 2, lemma: cluster size, asymptotics, N i | n delta, cdot j, crude estimate} (if $|\mathcal J_l| = 1$) or part (ii) (if $|\mathcal J_l| \geq 2$)
of Lemma~\ref{lemma: cluster size, asymptotics, N i | n delta, cdot j, crude estimate}
for each $l \in [k]$,
we identify some $\delta_0 > 0$ such that, given any $\delta \in (0,\delta_0)$,
there exists $\bar n = \bar n(\delta) \in (0,\infty)$ such that 
\begin{align}
    p(i,n\delta,\mathcal J_l)
    \leq 
    \underbrace{ \max_{ q,q^\prime \in [d] }2\bar s_{q,q^\prime} }_{ \delequal \bar C  }
    \cdot 
    n^{ |\mathcal J_l| - 1 }
    \prod_{  j \in \mathcal J_l }\P(B_{j \leftarrow l^*(j)} > n\delta),
    \quad
    \forall n \geq \bar n,\ l \in [k].
    \label{proof: part iii, bound for term p i, lemma: cluster size, asymptotics, N i | n delta, cdot j, crude estimate}
\end{align}
Recall that $c > 0$ is the constant fixed in \eqref{proof: goal, part 3, lemma: cluster size, asymptotics, N i | n delta, cdot j, crude estimate}.
Given $\delta \in (0,\delta_0)$ and any $n$ with $nc > \bar n(\delta)$,
by \eqref{proof, part i, event A i mathcal T upper bound, lemma: cluster size, asymptotics, N i | n delta, cdot j, multiple ancsetor} and \eqref{proof: part iii, bound for term p i, lemma: cluster size, asymptotics, N i | n delta, cdot j, crude estimate},
it holds for each $T \geq nc$ that 
\begin{align}
  \P\Big(A_n^\mathscr{J}(T,\delta)\Big) 
  & \leq 
  \bar C^{ k }\prod_{l \in [k]} T^{ |\mathcal J_l| }\prod_{j \in \mathcal J_l}\P(B_{j \leftarrow l^*(j)} > T\delta)
    \nonumber
  \\ 
  & = 
  \bar C^k T^{ |\mathcal J| }\prod_{l \in [k]}\prod_{j \in \mathcal J_l}\P(B_{j \leftarrow l^*(j)} > T\delta)
  =
  \bar C^k \prod_{ j \in \mathcal J } T \cdot \P(B_{j \leftarrow l^*(j)} > T\delta).
  \label{proof: part iii, bound for event A T J n delta, lemma: cluster size, asymptotics, N i | n delta, cdot j, crude estimate}
\end{align}
The last line follows from
the definition of the partition
$\mathscr J = \{ \mathcal J_1,\ldots \mathcal J_k \}$.
Furthermore, for each $j \in [d]$,
note that
$$
f_{j,\delta}(x) \delequal x \cdot \P(B_{j \leftarrow l^*(j)} >  x\delta) \in \RV_{ -(\alpha^*(j) - 1) }(x)
$$
with $\alpha^*(j) > 1$.
Using Potter's bound, we have (by picking a larger $\bar n = \bar n(\delta)$ if necessary)
$
 f_{j,\delta}(y) \leq 2 f_{j,\delta}(x)
$
for any $y \geq x \geq \bar n(\delta),\ j \in [d].$
Then, in \eqref{proof: part iii, bound for event A T J n delta, lemma: cluster size, asymptotics, N i | n delta, cdot j, crude estimate},
it holds for any $\delta \in (0,\delta_0)$ and any $n$ with $nc \geq \bar n(\delta)$ that
$
\sup_{T \geq nc}\P\big(A_n^\mathscr{J}(T,\delta)\big) 
=
     \bo 
     \big(
        \prod_{j \in \mathcal J}
        n \cdot \P(B_{j \leftarrow l^*(j)} > n\delta)
     \big).
$
Applying this bound for any partition $\mathscr J \in \mathbb J$
in \eqref{proof, part i, decomp, upper bound, lemma: cluster size, asymptotics, N i | n delta, cdot j, multiple ancsetor},
we conclude the proof of Claim  \eqref{proof: goal, part 3, lemma: cluster size, asymptotics, N i | n delta, cdot j, crude estimate}.
\end{proof}

\begin{appendix}
\section{Additional Auxiliary Results}
\label{sec: appendix, lemmas}

For completeness, we collect in this section the proofs of several useful results.
The first lemma provides concentration inequalities for truncated regularly varying random vectors,
and the proof is similar to that of Lemma~3.1 in \cite{wang2023large}.
Recall that, throughout this paper, we consider the $L_1$ norm $\norm{\bm x} = \sum_{i = 1}^k|x_i|$ for any vector $\bm x \in \R^k$. 
For any $c > 0$ and $x \in \R$, let $\phi_c(x) \delequal x \wedge c$,
and 
$\psi_c(x) = (x \wedge c) \vee (-c)$.
That is, $\psi_c$ is the projection mapping onto the interval $[-c,c]$,
and $\phi_c(x)$ truncates $x$ under threshold $c$.
For any $\bm x = (x_1,\ldots,x_k) \in \R^k$, let 
\begin{align*}
    \phi^{(k)}_c(\bm x) \delequal \big( \phi_c(x_1),\ldots,\phi_c(x_k)\big),
    \quad 
    \psi^{(k)}_c(\bm x) \delequal \big( \psi_c(x_1),\ldots,\psi_c(x_k)\big).
\end{align*}
Under any $c > 0$, note that
\begin{align}
    \psi^{(k)}_c(\bm x) = \phi^{(k)}_c(\bm x),
    \qquad
    \forall \bm x \in [0,\infty)^k.
    \label{property, appendix, psi to phi, truncation mapping}
\end{align}


\begin{lemma}
\label{lemma: concentration ineq, truncated heavy tailed RV in Rd}
\linksinthm{lemma: concentration ineq, truncated heavy tailed RV in Rd}
Let $\bm Z_i$'s be independent copies of a random vector $\bm Z$ in $\R^k$.
Suppose that 
$\P(\norm{\bm Z} > x) \in \RV_{-\alpha}(x)$ as $x \to \infty$ for some $\alpha > 1$.
Given any $\epsilon,\ \gamma \in (0,\infty)$, there exists $\delta_0 = \delta_0(\epsilon,\gamma) > 0$ such that
for all $\delta \in (0,\delta_0)$,
\begin{align}
    \lim_{n \to \infty}n^{\gamma}\cdot
    \P\Bigg(\max_{t \leq n}
            \norm{
                \frac{1}{n}\sum_{i = 1}^t \bm Z_i\mathbbm{I}\{\norm{\bm Z_i} \leq n\delta \} 
                - \E\bm Z
            }
        > \epsilon
    \Bigg)
    & = 0,
    \label{claim 1, lemma: concentration ineq, truncated heavy tailed RV in Rd}
    \\ 
    \lim_{n \to \infty}n^{\gamma}\cdot
    \P\Bigg(\max_{t \leq n}
            \norm{
                \frac{1}{n}\sum_{i = 1}^t \psi^{(k)}_{n\delta}(\bm Z_i)
                - \E\bm Z
            }
        > \epsilon
    \Bigg)
    & = 0.
    \label{claim 2, lemma: concentration ineq, truncated heavy tailed RV in Rd}
\end{align}
\end{lemma}

\begin{proof}\linksinpf{lemma: concentration ineq, truncated heavy tailed RV in Rd}
Without loss of generality, we take $\E\bm Z = \bm 0$.
Also, the proof of Claim~\eqref{claim 1, lemma: concentration ineq, truncated heavy tailed RV in Rd}
is a rather straightforward adaptation of the proof of Lemma~3.1 in \cite{wang2023large},
and is almost identical to the proof of Claim~\eqref{claim 2, lemma: concentration ineq, truncated heavy tailed RV in Rd} given below.
To avoid repetition, in this proof we focus on establishing Claim~\eqref{claim 2, lemma: concentration ineq, truncated heavy tailed RV in Rd}.

Take $\beta$ such that $\frac{1}{2 \wedge \alpha} < \beta < 1$. Let
\begin{align*}
    \bm Z^{(1)}_i \delequal \psi^{(k)}_{n\delta}(\bm Z_i) \mathbbm{I}\{ \norm{\bm Z_i} \leq n^\beta\},
    \quad
    \hat{\bm Z}^{(1)}_i \delequal \bm Z^{(1)}_i - \E \bm Z^{(1)}_i,
    \quad
    \bm Z^{(2)}_i \delequal 
        \psi^{(k)}_{n\delta}(\bm Z_i) \mathbbm{I}\{ \norm{\bm Z_i} > n^\beta \}.
\end{align*}
Due to
\begin{align*}
  \norm{
       \frac{1}{n} \sum_{i = 1}^t \psi^{(k)}_{n\delta}(\bm Z_i)
            }
 \leq 
 \norm{ \frac{1}{n} \sum_{i = 1}^t \hat{\bm Z}^{(1)}_i }
 +
 \norm{\frac{1}{n}  \sum_{i = 1}^t {\bm Z}^{(2)}_i }
 +
  \frac{t}{n}\norm{\E \bm Z^{(1)}_i},
\end{align*}
it suffices to find $\delta_0 > 0$ such that for all $\delta \in (0,\delta_0)$,
\begin{align}
    \lim_{n \to \infty} \norm{\E \bm Z^{(1)}_i} & < \frac{\epsilon}{3},
    \label{proof: goal 1, lemma: concentration ineq, truncated heavy tailed RV in Rd}
    \\ 
    \lim_{n \to \infty} n^\gamma \cdot 
        \P\Bigg(\max_{t\leq n}
            \norm{ \frac{1}{n}\sum_{i = 1}^t \hat{\bm Z}^{(1)}_i } > \frac{\epsilon}{3}
        \Bigg) 
        & = 0,
    \label{proof: goal 2, lemma: concentration ineq, truncated heavy tailed RV in Rd}
    \\ 
    \lim_{n \to \infty} n^\gamma \cdot 
        \P\Bigg(\max_{t\leq n}
            \norm{ \frac{1}{n}\sum_{i = 1}^t {\bm Z}^{(2)}_i } > \frac{\epsilon}{3}
        \Bigg)
        & = 0.
    \label{proof: goal 3, lemma: concentration ineq, truncated heavy tailed RV in Rd}
\end{align}

We show that Claim \eqref{proof: goal 1, lemma: concentration ineq, truncated heavy tailed RV in Rd} holds for any $\delta > 0$.
To this end, we make a few observations.
First, given $\delta > 0$, it holds for any $n$ large enough such that $n\delta > n^\beta$ due to our choice of $\beta < 1$.
For such $n$, note that any vector $\bm x = (x_1,\ldots,x_k)$, $\norm{\bm x} \leq n^{\beta}$ implies that
$
|x_j| \leq n^\beta < n\delta
$
for each $j \in [k]$. 
Therefore,
\begin{align}
    \bm Z^{(1)}_i
    =
    \psi^{(k)}_{n\delta}(\bm Z_i) \mathbbm{I}\{ \norm{\bm Z_i} \leq n^\beta\}
    =
    \bm Z_i \mathbbm{I}\{ \norm{\bm Z_i} \leq n^\beta\},
    \quad
    \text{whenever }n^\beta < n\delta.
    \label{proof, property, term Z 1, lemma: concentration ineq, truncated heavy tailed RV in Rd}
\end{align}
Then, for such large $n$,
\begin{align}
    \norm{\E \bm Z^{(1)}_i}  
    & = 
    \norm{
         \E\big[ \bm Z_i \mathbbm{I}\{ \norm{\bm Z_i} \leq n^\beta\} \big]
    }
    \nonumber
    \\ 
    & = \norm{
         \E\big[ \bm Z_i \mathbbm{I}\{ \norm{\bm Z_i} > n^\beta\} \big]
    }
    \leq \E\Big[\norm{\bm Z_i}\mathbbm{I}\{ \norm{\bm Z_i} > n^\beta \}\Big]
    \quad\text{due to }\E\bm Z = \bm 0
        \nonumber
    \\ 
    & = \int^{\infty}_{n^\beta}\P(\norm{\bm Z_i} > x)dx + n^\beta\cdot\P( \norm{\bm Z_i}> n^\beta ) \in \RV_{-(\alpha - 1)\beta}(n).
    \label{proof: term E Z 1, lemma: concentration ineq, truncated heavy tailed RV in Rd}
\end{align}
The last inequality follows from $\P(\norm{\bm Z} > x) \in \RV_{-\alpha}(x)$ and Karamata's Theorem.
Due to $\alpha > 1$, we have $(\alpha-1)\beta > 0$ in \eqref{proof: term E Z 1, lemma: concentration ineq, truncated heavy tailed RV in Rd}, which verifies Claim~\eqref{proof: goal 1, lemma: concentration ineq, truncated heavy tailed RV in Rd}.
Also, by \eqref{proof, property, term Z 1, lemma: concentration ineq, truncated heavy tailed RV in Rd},
under any $n$ sufficiently large
we must have $\norm{\bm Z^{(1)}_i} < n^\beta$.
As a result, given $\delta > 0$,
it holds for all $n$ large enough that $\norm{\hat{\bm Z}^{(1)}_i} < 2n^\beta$.
Henceforth in this proof, we only consider such large $n$.

Next, we show that Claim \eqref{proof: goal 2, lemma: concentration ineq, truncated heavy tailed RV in Rd} holds for any $\delta > 0$.
Fix some $p$ such that
\begin{align}
    p\geq 1,\quad \ p > \frac{2\gamma}{\beta},\quad \ p > \frac{2\gamma}{1 - \beta},\quad \ p > \frac{2\gamma}{(\alpha - 1)\beta} > \frac{2\gamma}{(2\alpha - 1)\beta}. \label{proof: choose p, lemma: concentration ineq, truncated heavy tailed RV in Rd}
\end{align}
We write $\hat{\bm Z}^{(1)}_i = (\hat{Z}^{(1)}_{i,j})_{j \in [k]}$,
and note that
under $L_1$ norm, we have
$
\norm{ \frac{1}{n}\sum_{i = 1}^n \hat{\bm Z}^{(1)}_i } 
=
\sum_{j = 1}^k|\frac{1}{n}\sum_{i = 1}^n \hat{Z}^{(1)}_{i,j}|.
$
Furthermore, for each $j \in [k]$, $n \geq 1$, and $y \geq 1$,
\begin{align}
    & \P\Bigg( \bigg|\sum_{i = 1}^{ n }\frac{1}{n} \hat{Z}^{(1)}_{i,j} \bigg|^p > \frac{y}{ n^{2\gamma} } \Bigg)
    \nonumber
    \\
    &
    = 
    \P\Bigg( \bigg|\sum_{i = 1}^{ n }\frac{1}{n} \hat{Z}^{(1)}_{i,j} \bigg| > \frac{y^{1/p}}{ n^{2\gamma/p}} \Bigg)
    \nonumber
    \\
    &
    \leq
    2\exp\Bigg( - \frac{ 
        \frac{1}{2}y^{2/p}\cdot n^{-4\gamma/p}
        }{ 
            \frac{2}{3}{y^{1/p}} \cdot { n^{ - (1  - \beta + 2\gamma/p )  } }
                + 
            n \cdot \frac{1}{n^2} \cdot \E\big[|\hat{Z}^{(1)}_{i,j}|^2\big]  
        } \Bigg)
        \quad 
        \text{by Bernstein's inequality and }\norm{\hat{\bm Z}^{(1)}_i} < 2n^\beta
    \nonumber
    \\ 
    & \leq 
    2\exp\Bigg( - \frac{ 
        \frac{1}{2}y^{2/p}\cdot n^{-4\gamma/p}
        }{ 
            \frac{2}{3}{y^{1/p}} \cdot { n^{ - (1  - \beta + 2\gamma/p )  } }
                + 
            \frac{1}{n} \cdot \E\Big[\norm{\hat{\bm Z}^{(1)}_{i}}^2\Big]  
        } \Bigg)
        \qquad 
        \text{due to }|\hat{Z}^{(1)}_{i,j}| \leq \norm{\hat{\bm Z}^{(1)}_{i}}.
    \label{proof: applying berstein ineq, lemma: concentration ineq, truncated heavy tailed RV in Rd}
\end{align}
Our next goal is to show that 
$
\frac{1}{n} \cdot \E\Big[\norm{\hat{\bm Z}^{(1)}_{i}}^2\Big]
< \frac{1}{3}{ n^{ - (1  - \beta + 2\gamma/p )  } }
$
for all $n$ large enough.
First, due to $(a+b)^2 \leq 2a^2 + 2b^2$,
\begin{align*}
\E\Big[\norm{\hat{\bm Z}^{(1)}_{i}}^2\Big]
& = 
\E\Big[\norm{\bm Z^{(1)}_i - \E \bm Z^{(1)}_i}^2\Big]
\leq 
2\E\Big[\norm{\bm Z^{(1)}_i}^2\Big] + 2\norm{\E\bm Z^{(1)}_i}^2.
\end{align*}
Also, as established in \eqref{proof: term E Z 1, lemma: concentration ineq, truncated heavy tailed RV in Rd},
$\norm{\E\bm Z^{(1)}_i}^2$ is upper bounded by some $\RV_{-2(\alpha - 1)\beta}(n)$ function.
By the choice of $p$ in \eqref{proof: choose p, lemma: concentration ineq, truncated heavy tailed RV in Rd} that $p > \frac{2\gamma}{(2\alpha - 1)\beta}$,
we have 
$
1 + 2(\alpha - 1)\beta > 1 - \beta + \frac{2\gamma}{p},
$
and hence
$$
\frac{2}{n}\norm{\E\bm Z^{(1)}_i}^2 < \frac{1}{6}n^{-(1 - \beta + \frac{2\gamma}{p})},
\qquad 
\text{ for any $n$ large enough}.
$$ 
Next, using \eqref{proof, property, term Z 1, lemma: concentration ineq, truncated heavy tailed RV in Rd},
for any $n$ large enough we have
$$
\E\Big[\norm{\bm Z^{(1)}_i}^2\Big]= \int^{\infty}_0 2x \P\Big(\norm{\bm Z^{(1)}_i} > x\Big)dx \leq
\int^{n^\beta}_0 2x  \P(\norm{\bm Z_i} > x)dx.
$$ 
If $\alpha \in (1,2]$, Karamata's theorem gives
$\int^{n^\beta}_0 2x  \P(\norm{\bm Z_i} > x)dx \in \RV_{ (2-\alpha)\beta }(n)$.
In \eqref{proof: choose p, lemma: concentration ineq, truncated heavy tailed RV in Rd},
we have chosen $p$ large enough such that
$
p > \frac{2\gamma}{(\alpha-1)\beta},
$
and hence
$1 - (2-\alpha)\beta > 1 - \beta + \frac{2\gamma}{p}$.
As a result, 
for all $n$ large enough
we have
$\frac{2}{n}\E\Big[\norm{\bm Z^{(1)}_i}^2\Big] < \frac{1}{6}n^{-(1-\beta + \frac{2\gamma}{p})}$.
If $\alpha > 2$, we have
$\lim_{n \to \infty}\int^{n^\beta}_0 2x  \P(\norm{\bm Z_i} > x)dx
=  
\int^{\infty}_0 2x  \P(\norm{\bm Z_i} > x)dx < \infty$.
Also,  \eqref{proof: choose p, lemma: concentration ineq, truncated heavy tailed RV in Rd} implies that $1 - \beta + \frac{2\gamma}{p} < 1$.
Again, for any $n$ large enough
we have
$\frac{2}{n}\E\Big[\norm{\bm Z^{(1)}_i}^2\Big] < \frac{1}{6}n^{-(1-\beta + \frac{2\gamma}{p})}$.
In summary, we have shown that
\begin{align}
    \frac{1}{n} \cdot \E\Big[\norm{\hat{\bm Z}^{(1)}_{i}}^2\Big]
< \frac{1}{3}{ n^{ - (1  - \beta + 2\gamma/p )  } },
\qquad 
\text{ for all $n$ large enough}.
    \label{term second order moment hat Z 1, lemma LDP, small jump perturbation}
\end{align}
Along with \eqref{proof: applying berstein ineq, lemma: concentration ineq, truncated heavy tailed RV in Rd},
we obtain that for all $n$ large enough,
\begin{align*}
     \P\Bigg( \bigg|\sum_{i = 1}^{ n }\frac{1}{n} \hat{Z}^{(1)}_{i,j} \bigg|^p > \frac{y}{ n^{2\gamma} } \Bigg)
     \leq 
     2\exp\Bigg( 
        { -\frac{1}{2}{y^{1/p}} } \cdot { n^{ 1 - \beta - \frac{2\gamma}{p} } } 
        \Bigg)
     \leq 
     2\exp\Big( -\frac{1}{2}{y^{1/p}} \Big),
     \qquad \forall y \geq 1,\ j \in [k].
\end{align*}
Here, the last inequality follows from our choice of $p$ in \eqref{proof: choose p, lemma: concentration ineq, truncated heavy tailed RV in Rd}
with $p > \frac{2\gamma}{1 - \beta}$,
and hence $1 - \beta - \frac{2\gamma}{p} > 0$.
Moreover, since $C^{(1)}_p \delequal \int_0^\infty\exp( -\frac{1}{2}{y^{1/p}})dy < \infty$,
the display above implies
\begin{align*}
    \max_{j \in [k]} n^{2\gamma}\cdot \E\Bigg[\bigg|\sum_{i = 1}^{ n }\frac{1}{n} \hat{Z}^{(1)}_{i,j} \bigg|^p \Bigg] 
    < C^{(1)}_p < \infty,
    \qquad 
    \text{ for all $n$ large enough}.
\end{align*}
Therefore, for such large $n$,
\begin{align*}
    \P\Bigg(\max_{t\leq n}
            \norm{ \frac{1}{n}\sum_{i = 1}^t \hat{\bm Z}^{(1)}_i } > \frac{\epsilon}{3}
        \Bigg)
    & \leq 
    \sum_{j \in [k]}
    \P\Bigg( \max_{t\leq n} \bigg|\sum_{i = 1}^{ t}\frac{1}{n} \hat{Z}^{(1)}_{i,j} \bigg| > \frac{\epsilon}{3k} \Bigg)
    \\ 
    & \leq 
    \sum_{j \in [k]}
    \frac{
        \E\Bigg[\bigg|\sum_{i = 1}^{ n }\frac{1}{n} \hat{Z}^{(1)}_{i,j} \bigg|^p \Bigg] 
    }{
    (\epsilon/3k)^p
    }
    \qquad \text{by Doob's Inequality}
    \\ 
    & \leq \frac{k}{(\epsilon/3k)^p}C^{(1)}_p \cdot \frac{1}{n^{2\gamma}}.
\end{align*}
This concludes the proof of Claim \eqref{proof: goal 2, lemma: concentration ineq, truncated heavy tailed RV in Rd} (under any $\delta > 0$).

Finally, for Claim \eqref{proof: goal 3, lemma: concentration ineq, truncated heavy tailed RV in Rd}, recall that we have chosen $\beta$ in such a way that $\alpha\beta - 1 > 0$.
Fix a constant $J = \ceil{\frac{\gamma}{\alpha\beta - 1}} + 1$,
and define
$I(n) =\#\big\{ i \leq n:\ \bm Z^{(2)}_i \neq \bm 0 \big\}$.
Besides, fix $\delta_0 = \frac{\epsilon}{3Jk}$.
For any $\delta \in (0,\delta_0)$,
by the definition of the projection mapping $\psi_c^{(k)}$,
we have
\begin{align*}
    \norm{\bm Z^{(2)}_i} \leq nk\delta < n \cdot \frac{\epsilon}{3J}.
\end{align*}
Then, for any $\delta \in (0,\delta_0)$,
on the event $\{I(n) < J\}$,
we  have
$\max_{t\leq n}
            \norm{ \frac{1}{n}\sum_{i = 1}^t {\bm Z}^{(2)}_i }
< \frac{1}{n} \cdot J \cdot \frac{\epsilon}{3J} < \epsilon/3$.
On the other hand, (let $H(x) = \P(\norm{\bm Z} > x)$)
\begin{align*}
    \P\big( I(n) \geq J\big)
    \leq 
    \binom{n}{J}\cdot \Big(H(n^\beta)\Big)^J
    \leq 
    n^J \cdot \Big(H(n^\beta)\Big)^J \in \RV_{ - J(\alpha\beta - 1) }(n)\text{ as }n \to\infty.
\end{align*}
Our choice of $J = \ceil{\frac{\gamma}{\alpha\beta - 1}} + 1$ guarantees that $J(\alpha\beta - 1) > \gamma$,
and hence,
$$
\lim_{n \to \infty} n^\gamma \cdot 
        \P\Bigg(\max_{t\leq n}
            \norm{ \frac{1}{n}\sum_{i = 1}^t {\bm Z}^{(2)}_i } > \frac{\epsilon}{3}
        \Bigg)
\leq 
\lim_{n \to \infty}{ n^\gamma \cdot 
\P( I(n) \geq J )  } = 0.
$$
This concludes the proof.
\end{proof}

The next lemma verifies equality  \eqref{proof, equality, from S i truncated to S i} regarding $\bm S^{\leqslant}_j(M)$ and $\bm S_j$.

\begin{lemma}\label{lemma: fixed point equation for pruned clusters}
Let $W^>_{j;i}(M)$ be defined as in \eqref{def: W i M j, pruned cluster, 1, cluster size}--\eqref{def: W i M j, pruned cluster, 2, cluster size}
on the probability space supporting the collection of independent random vectors $\bm B^{(t,m)}_{\bcdot \leftarrow j}$ in \eqref{def, proof strategy, collection of B i j copies},
and let $\bm S_j^\leqslant(M)$ be defined as in \eqref{def: pruned tree S i leq M}.
Under Assumption~\ref{assumption: subcriticality},
it holds for each $j \in [d]$ and $M > 0$ that
\begin{align*}
    \bm S_j \distequal \bm S^{\leqslant}_j(M) + \sum_{i \in [d]}\sum_{m = 1}^{ W^>_{j;i}(M) }\bm S^{(m)}_i,
\end{align*}
where, for each $i \in [d]$, the $\bm S^{(m)}_i$'s are i.i.d.\ copies of $\bm{S}_i$ and are independent from the random vector  $\big(\bm S^{\leqslant}_j(M), W^>_{j;1}(M),\ldots,W^>_{j;d}(M)\big)$.
\end{lemma}

\begin{proof}
Throughout this proof, we fix some $M \in (0,\infty)$,
and lighten the notations by writing 
$
\bm S^\leqslant_j = \bm S^\leqslant_j(M),
$
$
W^{>}_{j;i} = W^{>}_{j;i}(M),
$
$
W^{>}_{j;i \leftarrow l} = W^{>}_{j;i \leftarrow l}(M),
$
and $\bm X^\leqslant_j(t) = \bm X^\leqslant_j(t;M)$,
$
X^\leqslant_{j,i}(t)=X^\leqslant_{j,i}(t;M);
$
see \eqref{def: branching process X n leq M} and \eqref{def: W i M j, pruned cluster, 1, cluster size}--\eqref{def: W i M j, pruned cluster, 2, cluster size}.
Besides, henceforth in this proof,
notations
$B^{(t,m,k)}_{l\leftarrow i}$
are saved for i.i.d.\ copies of $B_{l \leftarrow i}$
that are also independent from the $B^{(t,m)}_{i\leftarrow j}$'s,
and notations
$\bm S^{(t,m,k)}_i$, $\tilde{\bm S}_i^{(t,m,k)}$, and $\bm S_i^{(t,m,k,k^\prime)}$, $\tilde{\bm S}_i^{(t,m,k,k^\prime)}$ are for i.i.d.\ copies of $\bm S_i$ whose law is independent from that of the $B^{(t,m)}_{i \leftarrow j}$'s and $B^{(t,m,k)}_{i \leftarrow j}$'s.
This is made rigorous through proper augmentation of the underlying probability space.
In particular, we note that:
(i)
the vector
$\big(\bm S^{\leqslant}_j, W^>_{j;1},\ldots,W^>_{j;d}\big)$
and the variables $X^\leqslant_{j,i}(t)$ are measurable w.r.t.\ the $\sigma$-algebra generated by the 
$B^{(t,m)}_{l \leftarrow i}$'s in \eqref{def, proof strategy, collection of B i j copies};
and
(ii)
since the
$B^{(t,m,k)}_{l \leftarrow i}$'s
and
$\bm S^{(t,m,k)}_i$, $\tilde{\bm S}_i^{(t,m,k)}$, $\bm S_i^{(t,m,k,k^\prime)}$, $\tilde{\bm S}_i^{(t,m,k,k^\prime)}$
are independent from the $B^{(t,m)}_{l \leftarrow i}$'s,
they are
also independent from 
the vector
$\big(\bm S^{\leqslant}_j, W^>_{j;1},\ldots,W^>_{j;d}\big)$
and $X^\leqslant_{j,i}(t)$.

\smallskip
By \eqref{def: fixed point equation for cluster S i},
\begin{align*}
    \bm S_j
    & \distequal \bm e_j + \sum_{i \in [d]}\sum_{ m = 1 }^{ B^{ (1,1) }_{i \leftarrow j}  }\bm S_i^{(1,1,m)}
    \\ 
    & \distequal 
     \bm e_j 
     + \sum_{i \in [d]}\sum_{ m = 1 }^{ B^{ (1,1) }_{i \leftarrow j}\mathbbm{I}\{ B^{ (1,1) }_{i \leftarrow j} \leq M  \} }\tilde{\bm S}_i^{ (1,1,m) }
     +
      \sum_{i \in [d]}\sum_{ m = 1 }^{ B^{ (1,1) }_{i \leftarrow j}\mathbbm{I}\{ B^{ (1,1) }_{i \leftarrow j} > M  \} }{\bm S}_i^{ (1,1,m) }.
\end{align*}
Furthermore, for each $T \geq 1$ we define
\begin{align*}
    I_1(T) 
    & \delequal
    \sum_{t = 0}^{T - 1}\bm X^\leqslant_j(t),
    \\ 
    I_2(T)
    & \delequal 
    \sum_{i \in [d]}
    \sum_{ m = 1 }^{ X^\leqslant_{j,i}(T-1)  }
    \sum_{l \in [d]}
    \sum_{k = 1}^{ B^{ (T,m)  }_{ l \leftarrow i  }\mathbbm{I}\{  B^{ (T,m)  }_{ l \leftarrow i  } \leq M   \}   }
    \tilde{\bm S}_l^{ (T,m,k)  },
    \\ 
    I_3(T) & \delequal 
    \sum_{t = 1}^T
    \sum_{i \in [d]}
    \sum_{ m = 1 }^{ X^\leqslant_{j,i}(t-1)  }
    \sum_{l \in [d]}
    \sum_{k = 1}^{ B^{ (t,m)  }_{ l \leftarrow i  }\mathbbm{I}\{  B^{ (t,m)  }_{ l \leftarrow i  } > M   \}   }
    \bm S_l^{ (t,m,k) }.
\end{align*}
Due to $\bm X^\leqslant_{j}(0) = \bm e_j$,
we have
$I_1(1) = \bm X^\leqslant_j(0) = \bm e_j$,
$
I_2(1) = \sum_{i \in [d]}\sum_{ k = 1 }^{ B^{ (1,1) }_{i \leftarrow j}\mathbbm{I}\{ B^{ (1,1) }_{i \leftarrow j} \leq M  \} }\tilde{\bm S}_i^{ (1,1,k) },
$
and 
$
I_3(1) = 
 \sum_{i \in [d]}\sum_{ k  = 1 }^{ B^{ (1,1) }_{i \leftarrow j}\mathbbm{I}\{ B^{ (1,1) }_{i \leftarrow j} > M  \} }{\bm S}_i^{ (1,1,k) }.
$
This confirms that 
$
\bm S_j \distequal I_1(1) + I_2(1) + I_3(1).
$
Next, we consider an inductive argument, and
suppose that 
$
\bm S_j \distequal I_1(T) + I_2(T) + I_3(T)
$
for some positive integer $T$.
Then, 
using \eqref{def: fixed point equation for cluster S i} again,
we get
\begin{align*}
    \bm S_j & \distequal
    I_1(T) + I_3(T)
    + 
    \sum_{i \in [d]}
    \sum_{ m = 1 }^{ X^\leqslant_{j,i}(T-1)  }
    \sum_{l \in [d]}
    \sum_{k = 1}^{ B^{ (T,m)  }_{ l \leftarrow i  }\mathbbm{I}\{  B^{ (T,m)  }_{ l \leftarrow i  } \leq M   \}   }
    \Bigg(
        \bm e_l + \sum_{ l^\prime \in [d]  }\sum_{ k^\prime = 1  }^{ 
            B^{ (T+1,m,k)  }_{ l^\prime \leftarrow l  }
        }
        \tilde{\bm S}_{l^\prime}^{ (T+1,m,k,k^\prime)  }
    \Bigg)
    \\ 
    & = 
    I_1(T) + I_3(T) +  
    \sum_{i \in [d]}
    \sum_{ m = 1 }^{ X^\leqslant_{j,i}(T-1)  }
    \sum_{l \in [d]}
    \sum_{k = 1}^{ B^{ (T,m)  }_{ l \leftarrow i  }\mathbbm{I}\{  B^{ (T,m)  }_{ l \leftarrow i  } \leq M   \}   }
    \bm e_l
    \\
    & \quad + 
    \sum_{i \in [d]}
    \sum_{ m = 1 }^{ X^\leqslant_{j,i}(T-1)  }
    \sum_{l \in [d]}
    \sum_{k = 1}^{ B^{ (T,m)  }_{ l \leftarrow i  }\mathbbm{I}\{  B^{ (T,m)  }_{ l \leftarrow i  } \leq M   \}   }
    \sum_{ l^\prime \in [d]  }\sum_{ k^\prime = 1  }^{ 
            B^{ (T+1,m,k)  }_{ l^\prime \leftarrow l  }
            }
    \tilde{\bm S}_{l^\prime}^{ (T+1,m,k,k^\prime)  }.
\end{align*}
By \eqref{def: branching process X n leq M},
we have 
$
\sum_{i \in [d]}
    \sum_{ m = 1 }^{ X^\leqslant_{j,i}(T-1)  }
    \sum_{l \in [d]}
    \sum_{k = 1}^{ B^{ (T,m)  }_{ l \leftarrow i  }\mathbbm{I}\{  B^{ (T,m)  }_{ l \leftarrow i  } \leq M   \}   }
    \bm e_l
    = \bm X^\leqslant_j(T),
$
Also, 
in  the display above,
note that:
(i) the $\tilde{\bm S}_{l^\prime}^{ (T+1,m,k,k^\prime)  }$'s are independent from the $\bm S^{(t,m,k)}_l$'s
and the variables $X^\leqslant_{j,i}(t)$ and $B^{(t,m)}_{l \leftarrow i}$;
(ii) the sequence $\big(B^{(T+1,m)}_{l \leftarrow i}\big)_{l,i,m}$ is independent from $I_1(T)$ and $I_3(T)$.
Therefore,
\begin{align*}
    \bm S_j & \distequal 
    I_3(T) + I_1(T) + \bm X^\leqslant_j(T)
    + 
    \sum_{i \in [d]}\sum_{m = 1}^{ X^\leqslant_{j,i}(T)  }
    \sum_{l \in [d]}
    \sum_{k = 1}^{ B^{ (T+1,m)  }_{ l \leftarrow i  }\mathbbm{I}\{  B^{ (T+1,m)  }_{ l \leftarrow i  } \leq M   \}   }
    \tilde{\bm S}_{l}^{ (T+1,m,k)  }
    \\ 
    & = I_3(T) + I_1(T+1) + 
    \sum_{i \in [d]}\sum_{m = 1}^{ X^\leqslant_{j,i}(T)  }
    \sum_{l \in [d]}
    \sum_{k = 1}^{ B^{ (T+1,m)  }_{ l \leftarrow i  }  }
    \tilde{\bm S}_{l}^{ (T+1,m,k)  }
    \\ 
    & \distequal
    I_1(T+1) + 
    I_3(T) + 
    \sum_{i \in [d]}\sum_{m = 1}^{ X^\leqslant_{j,i}(T)  }
    \sum_{l \in [d]}
    \sum_{k = 1}^{ B^{ (T+1,m)  }_{ l \leftarrow i  }\mathbbm{I}\{ B^{ (T+1,m)  }_{ l \leftarrow i  } > M \}  }
    {\bm S}_{l}^{ (T+1,m,k)  }
    \\ 
    & 
    \qquad 
    + 
    \sum_{i \in [d]}\sum_{m = 1}^{ X^\leqslant_{j,i}(T)  }
    \sum_{l \in [d]}
    \sum_{k = 1}^{ B^{ (T+1,m)  }_{ l \leftarrow i  }\mathbbm{I}\{ B^{ (T+1,m)  }_{ l \leftarrow i  } \leqslant M \}  }
    \tilde{\bm S}_{l}^{ (T+1,m,k)  }
    \\ 
    & = I_1(T+1) + I_3(T+1) + I_2(T+1).
\end{align*}
Proceeding inductively, we conclude that 
$
\bm S_j \distequal I_1(T) + I_2(T) + I_3(T)
$
hold for any $T \geq 1$.
Now, it suffices to show that 
$
I_2(T) \Rightarrow \bm 0
$
and 
$
I_1(T) + I_3(T) \Rightarrow \bm S^{\leqslant}_j + \sum_{i \in [d]}\sum_{m = 1}^{ W^>_{j;i} }\bm S^{(m)}_i
$
as $T \to \infty$

\smallskip
\noindent
\textbf{Proof of $I_2(T) \Rightarrow \bm 0$}.
We  prove the claim in terms of convergence in probability. 
The sub-criticality condition in Assumption~\ref{assumption: subcriticality} implies $\bm S_j<\infty$ almost surely.
Then, due to $\bm X_j^\leqslant(t) \leq \bm X_j(t)$ for each $t$ (see \eqref{def: branching process X n}--\eqref{def: branching process X n leq M}),
almost surely we have $\bm X^\leqslant_j(T) = \bm 0$ eventually for any $T$ large enough, and hence
\begin{align*}
    \lim_{T \to \infty}
    \P\big( I_2(T) = \bm 0  \big)
    \geq 
    \lim_{T \to \infty}
    \P\big( \bm X^\leqslant_j(T-1) = \bm 0  \big) = 1.
\end{align*}

\smallskip
\noindent
\textbf{Proof of $
I_1(T) + I_3(T) \Rightarrow \bm S^{\leqslant}_j + \sum_{i \in [d]}\sum_{m = 1}^{ W^>_{j;i} }\bm S^{(m)}_i.
$}
Applying monotone convergence theorem along each of the $d$ dimensions and by the definition in \eqref{def: pruned tree S i leq M},
we get 
\begin{align*}
    I_1(T) + I_3(T)
    \rightarrow
    \bm S_j^\leqslant
    + 
    \sum_{t \geq 1}
    \sum_{i \in [d]}
    \sum_{ m = 1 }^{ X^\leqslant_{j,i}(t-1)  }
    \sum_{l \in [d]}
    \sum_{k = 1}^{ B^{ (t,m)  }_{ l \leftarrow i  }\mathbbm{I}\{  B^{ (t,m)  }_{ l \leftarrow i  } > M   \}   }
    \bm S_l^{ (t,m,k) },
    \qquad
    \text{as }T \to \infty
\end{align*}
almost surely.
By the definitions in \eqref{def: W i M j, pruned cluster, 1, cluster size}--\eqref{def: W i M j, pruned cluster, 2, cluster size},
it holds for each $l \in [d]$ that
$$
W^>_{j;l}
=
\sum_{t \geq 1}
    \sum_{i \in [d]}\sum_{m = 1}^{ X^{\leqslant}_{j,i}(t-1) }
            B^{(t,m)}_{l \leftarrow i}\mathbbm{I}\big\{
                B^{(t,m)}_{l \leftarrow i} > M
            \big\}.
$$
Then, since the $\bm S^{(t,m,k)}_l$'s are independent from the $B^{(t,m)}_{l \leftarrow i}$'s 
(and hence the $\bm S^\leqslant_j$ and  $W^>_{j;i}$'s),
we get
\begin{align*}
    I_1(T) + I_3(T)
    & \rightarrow
    \bm S_j^\leqslant
    + 
    \sum_{t \geq 1}
    \sum_{i \in [d]}
    \sum_{ m = 1 }^{ X^\leqslant_{j,i}(t-1)  }
    \sum_{l \in [d]}
    \sum_{k = 1}^{ B^{ (t,m)  }_{ l \leftarrow i  }\mathbbm{I}\{  B^{ (t,m)  }_{ l \leftarrow i  } > M   \}   }
    \bm S_l^{ (t,m,k) }
    \\ 
    & \distequal 
    \bm S_j^\leqslant + 
    \sum_{l \in [d]}\sum_{ m = 1  }^{ W^>_{j;l}  }\bm S^{(m)}_l.
\end{align*}
This concludes the proof.
\end{proof}

\section{Counterexample}
\label{sec: counterexamples}
This section presents an example to illustrate that in Theorem~\ref{theorem: main result, cluster size}, 
it is not trivial to uplift
the condition of $A$ being bounded away from $\bar\R^d_\leqslant(\bm j,\epsilon)$ for some $\epsilon > 0$ (i.e., $\mathbb M$-convergence under polar transform, as shown in Lemma~\ref{lemma: M convergence for MRV}) to $A$ being bounded away from $\R^d_\leqslant(\bm j)$ (i.e., $\mathbb M$-convergence under Cartesian coordinates).

Specifically, we assume $d=2$ and
impose Assumptions~\ref{assumption: subcriticality}--\ref{assumption: regularity condition 2, cluster size, July 2024}.
Also, for clarity of the presentation, we consider a strict power-law version of Assumption~\ref{assumption: heavy tails in B i j}:
\begin{align}
    \lim_{x \to \infty}\P(B_{j \leftarrow i} > x)\cdot x^{ \alpha_{j \leftarrow i} } = c_{i,j} \in (0,\infty),
    \qquad
    \text{for each index pair $(i,j)$,}
    \label{example, power law assumtpion, tail asymptotics without polar transform}
\end{align}
and assume that
        $\alpha_{1 \leftarrow 2} > \alpha_{1 \leftarrow 1} > 2$,
        $\alpha_{2 \leftarrow 1} \wedge \alpha_{2 \leftarrow 2} > 2\alpha_{1 \leftarrow 1}$,
and
        $\P(B_{2 \leftarrow 1} = 0) > 0$.
By the definitions in \eqref{def: cluster size, alpha * l * j}, we have
\begin{align}
    \alpha^*(1) = \alpha_{1 \leftarrow 1} > 2,\quad \alpha^*(2) > 2\alpha^*(1).
    \label{example, index alpha *, tail asymptotics without polar transform}
\end{align}
We are interested in the asymptotics of $\P(n^{-1}\bm S_1 \in A)$,
where $A = A(1)$ with
\begin{align}
    A(r) \delequal \big\{
        (x_1,x_2)^\top \in \R^2_+:\ 
        \exists w \geq 0\text{ s.t. }x_1 = w\bar s_{1,1},\ |x_2 - w\bar s_{1,2}| > r
    \big\},
    \label{example, def, target set A, tail asymptotics without polar transform}
\end{align}
with $\bar{\bm s}_i = (\bar s_{i,1},\bar s_{i,2})^\top = \E\bm S_i$.
That is, the set $A(r)$ is the tube around the ray $\R^2(\{1\}) =  \{ w \bar{\bm s}_1:\ w \geq 0 \}$ with a (vertical) radius $r$,
restricted in $\R^2_+$.
We stress that this is almost equivalent to considering
\begin{align*}
    \breve A(r) \delequal \bigg\{ \bm x \in \R^2_+: \inf_{ \bm y \in \R^2(\{1\})   }\norm{\bm x - \bm y} > r  \bigg\},
    \quad r > 0.
\end{align*}
In particular,
given any $r > 0$, one can find $r_1,r_2 > 0$ such that
$
A(r_1) \subseteq \breve A(r) \subseteq A(r_2).
$
This will allow us to apply the subsequent analysis onto $\breve A(r)$.

For clarity, we focus on the case with $r = 1$ in \eqref{example, def, target set A, tail asymptotics without polar transform} (i.e., with $A = A(1)$).
Under Assumption~\ref{assumption: subcriticality}, it is easy to verify that $\bar{\bm s}_1$ and $\bar{\bm s}_2$ are linearly independent.
By the definition in \eqref{example, def, target set A, tail asymptotics without polar transform}, we must have $A \cap \R^2(\{2\}) \neq \emptyset$, where $\R^2(\{i\}) = \{ w \bar{\bm s}_i:\ w \geq 0 \}$.
Also, by \eqref{example, index alpha *, tail asymptotics without polar transform}, we have $\R^2_\leqslant(\{2\}) = \R^2(\{1\})$, which is bounded away from $A$.
Therefore, suppose that the asymptotics \eqref{claim, theorem: main result, cluster size} stated in Theorem~\ref{theorem: main result, cluster size}  hold for sets bounded away from $\R^d(\bm j)$ (instead of $\bar\R^d(\bm j,\epsilon)$), then we are led to believe that
\begin{align}
    \P(n^{-1}\bm S_1 \in A) \sim n^{  -\alpha^*(2) },
    \quad \text{ as }n \to \infty.
    \label{example, the wrong statement, tail asymptotics without polar transform}
\end{align}
However, our analysis below disproves \eqref{example, the wrong statement, tail asymptotics without polar transform}, indicating that it is  non-trivial  to relax in Theorem~\ref{theorem: main result, cluster size} the condition that $A$ needs to be bounded away from $\bar\R^d(\bm j,\epsilon)$ for some $\epsilon > 0$.

For the type-$1$ ancestor of $\bm S_1$, we use $B_{j \leftarrow 1}$ to denote the count of its type-$j$ children.
Conditioned on the event
\begin{align*}
    E \delequal \big\{
        B_{1 \leftarrow 1} \in (n^2,2n^2];\ 
        B_{2 \leftarrow 1} = 0
    \big\},
\end{align*}
$\bm S_1$ admits the law of
\begin{align}
    \begin{pmatrix}
        1 \\ 0
    \end{pmatrix}
    +
    \sum_{k = 1}^{ B_{1\leftarrow 1} }\bm S^{(k)}_1,
    \label{example, conditional law of population, tail asymptotics without polar transform}
\end{align}
where the $\bm S^{(k)}_1$'s are i.i.d.\ copies of $\bm S_1$.
Furthermore, 
let $\hat A \delequal \big\{
        \bm x \in \R^2_+:\ \norm{ \bm x - w \bar{\bm s}_1 } > 2\ \forall w \geq 0
    \big\}$.
Obviously, for each $n \geq 1$ we have
 $n^{-1}\norm{ (1,0)^\top  } = 1/n \leq 1$.
Then, on the event 
\begin{align*}
    E \cap 
    \bigg\{
        n^{-1}\sum_{k = 1}^{ B_{1\leftarrow 1} }\bm S^{(k)}_1 \in \hat A
    \bigg\} 
    ,
\end{align*}
by \eqref{example, conditional law of population, tail asymptotics without polar transform} we must have 
\begin{align*}
    n^{-1}\Bigg(
    (1,0)^\top
    +
    \sum_{k = 1}^{ B_{1\leftarrow 1} }\bm S^{(k)}_1
    \Bigg)
    \in 
    \big\{
        \bm x \in \R^2_+:\ \norm{ \bm x - w \bar{\bm s}_1 } > 1\ \forall w \geq 0
    \big\},
\end{align*}
and hence
$
n^{-1}\bm S_1\in A.
$
In summary,
\begin{align}
    \P(n^{-1}\bm S_1 \in A)
    \geq 
    \P\big(
         B_{1 \leftarrow 1} \in (n^2,2n^2];\ 
        B_{2 \leftarrow 1} = 0
    \big)
    \cdot 
    \P\Bigg(
        n^{-1}\sum_{k = 1}^{ B_{1 \leftarrow 1} }\bm S^{(k)}_1 \in \hat A
        \ \bigg|\ B_{1 \leftarrow 1} \in (n^2,2n^2]
    \Bigg).
    \label{example, lower bound 1, tail asymptotics without polar transform}
\end{align}
To proceed, we make a few observations.
First,
\begin{align*}
    & \P\big(
         B_{1 \leftarrow 1} \in (n^2,2n^2];\ 
        B_{2 \leftarrow 1} = 0
    \big)
    \\
    & = 
    \P\big(
         B_{1 \leftarrow 1} \in (n^2,2n^2]\big) \cdot 
    \P(B_{2 \leftarrow 1} = 0
    )
    \qquad
    \text{by Assumption~\ref{assumption: heavy tails in B i j}}
    \\ 
    & =
    c \cdot 
    \P\big(
         B_{1 \leftarrow 1} \in (n^2,2n^2]
    \big)
    \qquad
    \text{for some $c > 0$ due to }\P(B_{2 \leftarrow 1 } = 0) > 0,
\end{align*}
which implies
\begin{align*}
        & 
    \lim_{n \to \infty}
    n^{ 2\alpha^*(1)  }\cdot 
    {
         \P\big(
         B_{1 \leftarrow 1} \in (n^2,2n^2];\ 
        B_{2 \leftarrow 1} = 0
    \big)
    }
    \\ 
    & = c \cdot 
    \lim_{n \to \infty}
     n^{ 2\alpha^*(1)  }\cdot \P\big(
         B_{1 \leftarrow 1} \in (n^2,2n^2]
    \big)
     = c  \cdot
    \Bigg(
        c_{1 \leftarrow 1} - \frac{ c_{1 \leftarrow 1} }{ 2^{ \alpha^*(1) } }
    \Bigg) > 0
    \qquad
    \text{by \eqref{example, power law assumtpion, tail asymptotics without polar transform} and \eqref{example, index alpha *, tail asymptotics without polar transform}.}
\end{align*}
Second, under the tail indices specified in \eqref{example, index alpha *, tail asymptotics without polar transform}, 
Theorem~2 of \cite{Asmussen_Foss_2018} confirms that
$
\P(\norm{\bm S_1} > n) = \bo (n^{ -\alpha^*(1)  }).
$
Since $\alpha^*(1) > 2$, we have $\E\norm{\bm S_1}^2 < \infty$, hence the covariance matrix for the random vector $\bm S_1$ is a well-defined symmetric and positive semi-definite matrix, which we denote by $\bm \Sigma$.
Obviously, our heavy-tailed assumption \eqref{example, power law assumtpion, tail asymptotics without polar transform} prevents the trivial case of $\bm\Sigma = \textbf{0}$.
Now,
let
\begin{align*}
    A^* \delequal
    \big\{
        \bm y \in \R^2_+:\ \bm y + \bar{\bm s}_1 \in \hat A
    \big\}.
\end{align*}
Note that $A^*$ is open and non-empty.
Furthermore, 
we write $\bm x + E = \{ \bm x + \bm y:\ \bm y \in E  \}$ for any set $E \subseteq \R^2$ and vector $\bm x$, and note the following:
due to $\bm x + \R^2_+ \subseteq \bm y + \R^2_+$ for any $\bm x \leq \bm y$ (i.e., $x_1\leq y_1$ and $x_2 \leq y_2$),
we have $ \bm y \in A^* \Longrightarrow \bm y + w\bar{\bm s}_1 \in \hat A\ \forall w \geq 1$.
As a result,
\begin{align}
    \big\{ 
        B_{1 \leftarrow 1} \in (n^2,2n^2]
    \big\} \cap 
    \bigg\{n^{-1}\sum_{ k = 1  }^{ \floor{m}  }\big( \bm S^{(k)}_1 - \bar{\bm s}_1\big) \in A^*\ \forall m \in (n^2,2n^2]\bigg\}
    \subseteq
    \bigg\{
        n^{-1}\sum_{k = 1}^{ B_{1 \leftarrow 1} }\bm S^{(k)}_1 \in \hat A
    \bigg\}.
    \nonumber
\end{align}
Therefore,
\begin{align*}
    & \liminf_{n \to \infty}\P\Bigg(
        n^{-1}\sum_{k = 1}^{ B_{1 \leftarrow 1} }\bm S^{(k)}_1 \in \hat A
        \ \bigg|\ B_{1 \leftarrow 1} \in (n^2,2n^2]
    \Bigg)
    \\ 
    & \geq 
    \liminf_{n \to \infty}\P\Bigg(
        n^{-1}\sum_{ k = 1  }^{ \floor{m}  }\big( \bm S^{(k)}_1 - \bar{\bm s}_1\big) \in A^*\ \forall m \in (n^2,2n^2]
    \Bigg)
    \\ 
    & \geq 
    \P\Big(
        \bm B(t)\bm \Sigma^{1/2}  \in A^*\ \forall t \in [1,2]
    \Big)
    \ \ 
    \text{by multivariate Donsker's theorem; see, e.g., Theorem 4.3.5 of \cite{whitt2002stochastic}}
    \\ 
    & > 0
    \quad 
    \text{since $\bm \Sigma^{1/2} \neq \textbf{0}$ and $A^*$ is non-empty and open}.
\end{align*}
In summary, from \eqref{example, lower bound 1, tail asymptotics without polar transform} we get
\begin{align*}
    \liminf_{n \to \infty}n^{ 2\alpha^*(1) } \cdot \P(n^{-1}\bm S_1 \in A) > 0.
\end{align*}
In light of the condition $2\alpha^*(1) < \alpha^*(2)$ in \eqref{example, index alpha *, tail asymptotics without polar transform},
we arrive at a contradiction to Claim \eqref{example, the wrong statement, tail asymptotics without polar transform}.
This concludes the example and confirms that the asymptotics \eqref{claim, theorem: main result, cluster size} in Theorem~\ref{theorem: main result, cluster size} generally fails when relaxing the bounded-away condition.

\section{Proofs for $\mathbb M$-Convergence and Asymptotic Equivalence}
\label{subsec: proof, M convergence and asymptotic equivalence}

This section collects the proof of Lemmas~\ref{lemma: asymptotic equivalence, MRV in Rd} and \ref{lemma: M convergence for MRV}.

\begin{proof}[Proof of Lemma~\ref{lemma: asymptotic equivalence, MRV in Rd}]
\linksinpf{lemma: asymptotic equivalence, MRV in Rd}
Throughout this proof, we write $\mathbb S = [0,\infty) \times \mathfrak N^d_+$.
We arbitrarily pick some Borel measurable $B \subseteq \mathbb S$ that is bounded away from $\mathbb C$ under $\bm d_\textbf{U}$.
This allows us to fix some $\bar\epsilon \in (0,1)$
such that 
$
\bm d_\textbf{U}(B,\mathbb C) > \bar\epsilon.
$
Let 
\begin{align}
B_\theta =
    \big\{
        \bm w \in \mathfrak N^d_+:\ (r,\bm w) \in B\text{ for some } r > 0
    \big\},\quad
    \mathbb C_\theta =
    \big\{
        \bm w \in \mathfrak N^d_+:\ (r,\bm w) \in \mathbb C\text{ for some } r > 0
    \big\}.
    \label{proof: projection of B and C, lemma: asymptotic equivalence, MRV in Rd}
\end{align}
We must have 
\begin{align}
    \inf_{ \bm w \in B_\theta,\ \bm w^\prime \in \mathbb C_\theta }\norm{\bm w - \bm w^\prime} \geq \bar\epsilon.
    \label{proof: distance between projections of B and C, lemma: asymptotic equivalence, MRV in Rd}
\end{align}
Otherwise, there exist $(r, \bm w) \in B$ and $(r^\prime, \bm w^\prime) \in \mathbb C$ such that 
$
r, r^\prime > 0
$
yet
$
\norm{\bm w - \bm w^\prime} < \bar\epsilon.
$
By condition \eqref{condition: C is a cone, lemma: asymptotic equivalence, MRV in Rd},
we also have $(r, \bm w^\prime) \in \mathbb C$, and hence
$
\bm d_\textbf{U}\big( (r,\bm w),\  (r, \bm w^\prime) \big)
    =
    \norm{\bm w - \bm w^\prime} < \bar\epsilon,
$
which contradicts 
$
\bm d_\textbf{U}(B,\mathbb C) > \bar\epsilon.
$
Also, since $(0,\bm w) \in \mathbb C$ for any $\bm w \in \mathfrak N^d_+$,
by 
$
\bm d_\textbf{U}(B,\mathbb C) > \bar\epsilon
$
we have
\begin{align}
    (r,\bm w) \in B
    \quad \Longrightarrow\quad 
    r > \bar\epsilon.
    \label{proof, lower bound for r in B, lemma: asymptotic equivalence, MRV in Rd}
\end{align}

For any $M \in (0,\infty)$, let
$
B(M) = 
\big\{  (r,\bm w) \in B:\ r \leq M \big\}.
$
For any $\Delta,M,n,\delta > 0$, observe that 
\begin{align*}
    \Big\{ (R_n,\Theta_n) \in B\Big\}
    & \supseteq
     \Big\{  (R_n,\Theta_n) \in B(M) \Big\}
     \\ 
     & \supseteq
     \Big\{ (R_n,\Theta_n) \in B(M);\ 
        \hat R^\delta_n \in \big[(1 - \Delta)R_n, (1 + \Delta)R_n\big],\ 
        \norm{\hat \Theta^\delta_n - \Theta_n} \leq \Delta 
     \Big\}.
\end{align*}
Furthermore, for any $\bar\Delta > 0$, and any $\Delta > 0$, $M \geq 1$ satisfying $M\Delta < \bar\Delta$,
\begin{align}
    {\hat r}/{r} \in [1-\Delta, 1+ \Delta],\ r \in [0,M],\ \norm{\bm w - \hat{\bm w} } \leq \Delta 
    \
    \Longrightarrow
    \
    |r - \hat r| \vee \norm{\bm w - \hat{\bm w}} \leq M\Delta < \bar\Delta.
    \label{proof: result 1, lemma: asymptotic equivalence, MRV in Rd}
\end{align}
Therefore, for any $\bar\Delta \in (0,\bar\epsilon)$, and any $\Delta \in (0,1)$, $M \geq 1$ such that $M\Delta < \bar\Delta$,
\begin{align*}
        & \P\Big( (R_n,\Theta_n) \in B \Big)
    \\ 
    & \geq 
    \P\Big( (R_n,\Theta_n) \in B(M);\ 
        \hat R^\delta_n \in \big[(1 - \Delta)R_n, (1 + \Delta)R_n\big],\ 
        \norm{\hat \Theta^\delta_n - \Theta_n} \leq \Delta
     \Big)
     \\ 
     & = 
     \P\Big( (R_n,\Theta_n) \in B(M),\ R_n \leq M/(1-\Delta);\ 
     \hat R^\delta_n \in \big[(1 - \Delta)R_n, (1 + \Delta)R_n\big],\ 
        \norm{\hat \Theta^\delta_n - \Theta_n} \leq \Delta
    \Big)
    \\ 
    &\qquad\qquad\qquad\qquad\qquad\qquad\qquad\qquad\qquad\qquad\text{by the definition of }B(M)
     \\ 
    & = 
     \P\Big( (R_n,\Theta_n) \in B(M),\ R_n \leq M/(1-\Delta);\ 
     \\
     &\qquad\quad 
        \hat R^\delta_n \in \big[(1 - \Delta)R_n, (1 + \Delta)R_n\big],\ 
        \norm{\hat \Theta^\delta_n - \Theta_n} \leq \Delta,\ 
        \bm d_\textbf{U}\big( (R_n,\Theta_n),\ (\hat R^\delta_n,\hat \Theta^\delta_n)  \big) \leq \bar\Delta
     \Big)
     \quad 
     \text{by \eqref{proof: result 1, lemma: asymptotic equivalence, MRV in Rd}}.
\end{align*}
Also, recall that for any metric space $(\mathbb S, \bm d)$ and $r > 0$, we use
${E^r} =
\{ y \in \mathbb{S}:\ \bm{d}(E,y)\leq r\}$ to denote the $r$-enlargement of the set $E$,
and
$
{E_{r}} =
((E^c)^r)^\complement
=
\{
y \in \mathbb S:\ \bm d(E^c,y) > r
\}
$
for the $r$-shrinkage
of $E$.
Given any $\bar\Delta \in (0,\bar\epsilon)$, and any $\Delta \in (0,1)$, $M \geq 1$ such that $M\Delta/(1 - \Delta) < \bar\Delta$, we then have
\begin{align*}
    & \P\Big( (R_n,\Theta_n) \in B \Big)
    \\ 
     & \geq 
     \P\Bigg(
        (\hat R^\delta_n,\hat \Theta^\delta_n) \in \big( B(M)\big)_{ \bar\Delta },\ R_n \leq \frac{M}{1 - \Delta};
        \\ 
        &\qquad\qquad
        \hat R^\delta_n \in \big[(1 - \Delta)R_n, (1 + \Delta)R_n\big],\ 
        \norm{\hat \Theta^\delta_n - \Theta_n} \leq \Delta,\ 
        \bm d_\textbf{U}\big( (R_n,\Theta_n),\ (\hat R^\delta_n,\hat \Theta^\delta_n)  \big) \leq \bar\Delta
     \Bigg)
     \\ 
     & =
     \P\Bigg(
        (\hat R^\delta_n,\hat \Theta^\delta_n) \in \big( B(M)\big)_{ \bar\Delta };\ R_n \leq \frac{M}{1 - \Delta},\ 
        \hat R^\delta_n \in \big[(1 - \Delta)R_n, (1 + \Delta)R_n\big],\ 
        \norm{\hat \Theta^\delta_n - \Theta_n} \leq \Delta
     \Bigg)
     \\ 
     &
     \qquad\qquad\qquad\qquad\qquad\qquad\qquad\qquad\qquad\qquad\qquad
     \text{by \eqref{proof: result 1, lemma: asymptotic equivalence, MRV in Rd} and }M\Delta/(1-\Delta) < \bar\Delta
     \\ 
     & =
     \P\Big(
        (\hat R^\delta_n,\hat \Theta^\delta_n) \in \big( B(M)\big)_{ \bar\Delta }
        \Big)
    \\ 
    &\quad
    - 
    \P\Bigg(
        (\hat R^\delta_n,\hat \Theta^\delta_n) \in \big( B(M)\big)_{ \bar\Delta };\ 
        R_n > \frac{M}{1 - \Delta} \text{ or }
        \hat R^\delta_n \notin \big[(1 - \Delta)R_n, (1 + \Delta)R_n\big]\text{ or }
        \norm{\hat \Theta^\delta_n - \Theta_n} > \Delta
     \Bigg)
     \\ 
     & \stackrel{(*)}{\geq}  
     \P\Big(
        (\hat R^\delta_n,\hat \Theta^\delta_n) \in \big( B(M)\big)_{ \bar\Delta }
        \Big)
    \\ 
    & \quad 
    - 
    \P\Bigg(
        \hat R_n^\delta \in (\Delta,M];\ 
         R_n > \frac{M}{1 - \Delta}\text{ or }
        \hat R^\delta_n \notin \big[(1 - \Delta)R_n, (1 + \Delta)R_n\big]\text{ or }
        \norm{\hat \Theta^\delta_n - \Theta_n} > \Delta
     \Bigg)
     \\ 
     & \stackrel{(\dagger)}{=} 
     \P\Big(
        (\hat R^\delta_n,\hat \Theta^\delta_n) \in \big( B(M)\big)_{ \bar\Delta }
        \Big)
    - 
    \P\Big(
        \hat R_n^\delta \in (\Delta,M];\ 
        \hat R^\delta_n \notin \big[(1 - \Delta)R_n, (1 + \Delta)R_n\big]\text{ or }
        \norm{\hat \Theta^\delta_n - \Theta_n} > \Delta
     \Big).
\end{align*}
Here, the step $(*)$ follows from $\text{\eqref{proof, lower bound for r in B, lemma: asymptotic equivalence, MRV in Rd} and }\Delta < \bar\epsilon$,
and the step $(\dagger)$ follows from
$$
\big\{
     \hat R_n^\delta \in (\Delta,M],\ R_n > \frac{M}{1 - \Delta}
\big\}
\subseteq 
\big\{
    \hat R_n^\delta \in (\Delta,M],\ \hat R^\delta_n \notin \big[(1 - \Delta)R_n, (1 + \Delta)R_n\big]
\big\}.
$$
Then, by condition (i), for any $M \geq 1$ and $\bar\Delta > 0$,
\begin{align}
    \liminf_{n \to\infty}
    \epsilon^{-1}_n
    \P\Big( (R_n,\Theta_n) \in B \Big)
    \geq 
    \liminf_{n \to\infty}
    \epsilon^{-1}_n
    \P\Big(
        (\hat R^\delta_n,\hat \Theta^\delta_n) \in \big( B(M)\big)_{ \bar\Delta }
        \Big),
    \qquad
    \forall \delta > 0\text{ small enough}.
    \nonumber
\end{align}
By condition (ii),  given $M \geq 1$ and $\bar\Delta > 0$ it holds for any $\delta > 0$ small enough that 
\begin{align}
    \liminf_{n \to\infty}
    \epsilon^{-1}_n
    \P\Big( (R_n,\Theta_n) \in B \Big)
    \geq 
    \liminf_{n \to\infty}
    \epsilon^{-1}_n
    \P\Big(
        (\hat R^\delta_n,\hat \Theta^\delta_n) \in \big( B(M)\big)_{ \bar\Delta }
        \Big)
    \geq 
    -|\mathcal V|\bar\Delta + 
     \sum_{v \in \mathcal V}\mu_v\Big( \big( B(M)\big)_{2 \bar\Delta } \Big).
    \label{proof: ineq for lower bound, lemma: asymptotic equivalence, MRV in Rd}
\end{align}
Furthermore, note that
$
\bigcup_{ M > 0 }B(M) = B
$
and $|\mathcal V| < \infty$.
By sending $M \to \infty$ and then $\bar\Delta \to 0$, we get
\begin{align}
    \liminf_{n \to \infty}\epsilon^{-1}_n
    \P\Big( (R_n,\Theta_n) \in B \Big)
    \geq 
    \sum_{v \in \mathcal V}\mu_v(B^\circ).
    \label{proof: lower bound, lemma: asymptotic equivalence, MRV in Rd}
\end{align}
Meanwhile, for any $\Delta \in (0, \bar\epsilon)$, we have the upper bound
\begin{align*}
    & \P\Big( (R_n,\Theta_n) \in B \Big)
    \\ 
    & = 
    \P\Big( (R_n,\Theta_n) \in B;\ 
        \hat R^\delta_n \in \big[(1 - \Delta)R_n, (1 + \Delta)R_n\big],\ 
        \norm{\hat \Theta^\delta_n - \Theta_n} \leq \Delta
     \Big)
     \\ 
     & \qquad 
     + 
     \P\Big( (R_n,\Theta_n) \in B;\ 
        \hat R^\delta_n \notin \big[(1 - \Delta)R_n, (1 + \Delta)R_n\big]\text{ or }
        \norm{\hat \Theta^\delta_n - \Theta_n} > \Delta
     \Big)
     \\ 
     & \leq 
     \P\Big( (R_n,\Theta_n) \in B;\ 
        \hat R^\delta_n \in \big[(1 - \Delta)R_n, (1 + \Delta)R_n\big],\ 
        \norm{\hat \Theta^\delta_n - \Theta_n} \leq \Delta
     \Big)
     \\ 
     &\qquad
     +
     \P\Big( R_n > \Delta;\ 
        \hat R^\delta_n \notin \big[(1 - \Delta)R_n, (1 + \Delta)R_n\big]\text{ or }
        \norm{\hat \Theta^\delta_n - \Theta_n} > \Delta
     \Big)
     \quad
     \text{by \eqref{proof, lower bound for r in B, lemma: asymptotic equivalence, MRV in Rd} and }\Delta < \bar\epsilon.
\end{align*}
By condition (i), 
\begin{align*}
    & \limsup_{n \to \infty}
    \epsilon^{-1}_n \P\Big( (R_n,\Theta_n) \in B \Big)
    \\ 
    & \leq  
    \limsup_{n \to \infty}
    \epsilon^{-1}_n
    \P\Big( (R_n,\Theta_n) \in B;\ 
        \hat R^\delta_n \in \big[(1 - \Delta)R_n, (1 + \Delta)R_n\big],\ 
        \norm{\hat \Theta^\delta_n - \Theta_n} \leq \Delta
     \Big).
\end{align*}
On the other hand, 
recall the definition of $B_\theta$ in \eqref{proof: projection of B and C, lemma: asymptotic equivalence, MRV in Rd},
and
let
\begin{align}
    \breve B(M,\delta) = 
     \Big\{
        (r,\bm w) \in [0,\infty) \times \mathfrak N^d_+:\ 
        r \geq M,\ \norm{\bm w - \bm w^\prime} \leq \delta \text{ for some }\bm w^\prime \in B_\theta
     \Big\}.
    \nonumber
\end{align}
Also, recall that we picked $\bar\epsilon > 0$ such that $
\bm d_\textbf{U}(B,\mathbb C) > \bar\epsilon.
$
For any $\bar\Delta \in (0, \frac{\bar\epsilon}{2}\wedge\frac{1}{2} )$ and all $M \geq 1, \Delta > 0$ with $M\Delta < \bar\Delta$, note that
\begin{align*}
   & \P\Big( (R_n,\Theta_n) \in B;\ 
        \hat R^\delta_n \in \big[(1 - \Delta)R_n, (1 + \Delta)R_n\big],\ 
        \norm{\hat \Theta^\delta_n - \Theta_n} \leq \Delta
     \Big) 
    \\ 
    & = 
    \P\Big( (R_n,\Theta_n) \in B(M);\ 
        \hat R^\delta_n \in \big[(1 - \Delta)R_n, (1 + \Delta)R_n\big],\ 
        \norm{\hat \Theta^\delta_n - \Theta_n} \leq \Delta
     \Big)
     \\ 
     & \qquad +
     \P\Big( (R_n,\Theta_n) \in B \setminus B(M);\ 
        \hat R^\delta_n \in \big[(1 - \Delta)R_n, (1 + \Delta)R_n\big],\ 
        \norm{\hat \Theta^\delta_n - \Theta_n} \leq \Delta
     \Big)
     \\ 
     & \stackrel{(\diamond)}{\leq} 
     \P\bigg( (\hat R^\delta_n, \hat \Theta^\delta_n) \in \big( B(M) \big)^{ \bar\Delta } \bigg)
     +
     \P\bigg( \hat R^\delta_n \geq (1-\Delta)M,\ \norm{\hat \Theta^\delta_n - \Theta_n} \leq \Delta  \bigg)
     \\ 
     & \leq
    \P\bigg(  \underbrace{ (\hat R^\delta_n, \hat \Theta^\delta_n) \in \big( B(M) \big)^{ \bar\Delta }  }_{ = \text{(I)} }\bigg)
     +
     \P\bigg(\underbrace{  (\hat R^\delta_n, \hat \Theta^\delta_n) \in \breve B\big( (1 - \bar\Delta)M,\bar\Delta\big)  }_{ = \text{(II)} }\bigg).
\end{align*}
Here, the step $(\diamond)$ follows from \eqref{proof: result 1, lemma: asymptotic equivalence, MRV in Rd} and the definition of $B(M)$.
For the event (I), 
by our choice of $\bar\Delta < \bar\epsilon/2$,
it follows from
$
\bm d_\textbf{U}(B,\mathbb C) > \bar\epsilon
$
that $\big(B(M)\big)^{2\bar\Delta} \subseteq B^{2\bar\Delta}$ is still bounded away from $\mathbb C$ under $\bm d_\textbf{U}$;
then by condition (ii), it holds for any $\delta > 0$ small enough that 
\begin{align}
    \limsup_{ n \to \infty }\epsilon^{-1}_n\P\big(\text{(I)}\big) 
    \leq |\mathcal V|\bar\Delta +  \sum_{v \in \mathcal V}\mu_v\bigg( \big(B(M)\big)^{2\bar\Delta} \bigg).
    \nonumber
\end{align}
Analogously, for the event (II),
note that \eqref{proof: distance between projections of B and C, lemma: asymptotic equivalence, MRV in Rd} implies that 
$
\breve B(M,\Delta)
$
is bounded away from $\mathbb C$ under $\bm d_\textbf{U}$ for any $M > 0$ and $\Delta < \bar\epsilon$.
Then by condition (ii), it holds for any $\delta > 0$ small enough that
\begin{align}
    \limsup_{ n \to \infty }\epsilon^{-1}_n\P\big(\text{(II)}\big) 
    \leq |\mathcal V|\bar\Delta +  \sum_{v \in \mathcal V}\mu_v\bigg( \Big( \breve B\big(  (1-\bar\Delta)M,\ \bar\Delta   \big)\Big)^{2\bar\Delta} \bigg).
    \nonumber
\end{align}
Note that 
$
\bigcap_{M > 0}\breve B(M,\bar\Delta) = \emptyset
$
and $|\mathcal V| < \infty$.
By sending $M \to \infty$ and then $\bar\Delta \to 0$,
we get
\begin{align}
    \limsup_{n \to \infty}\epsilon^{-1}_n
    \P\Big( (R_n,\Theta_n) \in B \Big)
    \leq
    \sum_{v \in \mathcal V}\mu_v(B^-).
    \label{proof: upper bound, lemma: asymptotic equivalence, MRV in Rd}
\end{align}
In light of Theorem~\ref{portmanteau theorem M convergence}---the Portmanteau theorem for $\mathbb M$-convergence---and the arbitrariness in our choice of $B$,
we combine \eqref{proof: lower bound, lemma: asymptotic equivalence, MRV in Rd} and \eqref{proof: upper bound, lemma: asymptotic equivalence, MRV in Rd},
concluding the proof.
\end{proof}

Next, to prove Lemma~\ref{lemma: M convergence for MRV}, we recall the definition of $\Phi$ in \eqref{def: Phi, polar transform}.
In particular, given $A\subseteq \R^d_+$ that does not contain the origin, note that
\begin{align}
    \bm x \in A\qquad \Longleftrightarrow\qquad \Phi(\bm x) \in \Phi(A).
    \label{property: polar transform, 1}
\end{align}
In addition, the following properties follow from the fact that the polar transform is a homeomorphism between $\R^d_+\setminus\{\bm 0\}$ and $(0,\infty) \times \mathfrak N^d_+$:
given $A \subseteq \R^d_+$ that is bounded away from $\bm 0$ (i.e., $\inf_{\bm x \in A}\norm{\bm x} > 0$),
\begin{align}
    \text{$A$ is open }\iff\text{ $\Phi(A)$ is open},
    \qquad
    \text{$A$ is closed }\iff\text{ $\Phi(A)$ is closed}.
    \label{property: polar transform, 2}
\end{align}
We prepare the following lemma.

\begin{lemma}\label{lemma: equivalence for bounded away condition, polar coordinates}
\linksinthm{lemma: equivalence for bounded away condition, polar coordinates}
    Let $\mathbb C$ be a closed cone in $\R^d_+$.
    Let $
\mathbb C_\Phi \delequal
    \big\{
        (r,\bm \theta) \in [0,\infty) \times \mathfrak N^d_+ :\ r\bm \theta \in \mathbb C
    \big\},
$
and let $ \bar{\mathbb C}(\epsilon)$ be defined as in \eqref{def: cone C enlarged by angles}.
For any Borel set $B \subseteq \R^d_+$,
the following two conditions are equivalent:
\begin{enumerate}[(i)]
    \item 
        $B$ is bounded away from $\bar{\mathbb C}(\epsilon)$ for some (and hence all) $\epsilon > 0$ small enough;

    \item 
        $\Phi(B)$ is bounded away from $\mathbb C_\Phi$ under $\bm d_\textbf{U}$.
\end{enumerate}
\end{lemma}

\begin{proof}
\linksinpf{lemma: equivalence for bounded away condition, polar coordinates}
\textbf{Proof of $(i) \Rightarrow (ii)$}.
Fix some $\epsilon,\Delta > 0$ such that $\inf\{ \norm{\bm x - \bm y}:\ \bm x \in B,\ \bm y \in \bar{\mathbb C}(\epsilon)   \} > \Delta$.
Since $\mathbb C$ is a cone, we have $\bm 0 \in \mathbb C \subseteq \bar{\mathbb C}(\epsilon)$, and hence $\inf_{\bm x \in B}\norm{\bm x} > \Delta$.
Next, we consider a proof by contradiction.
Suppose there are sequences $(r^{\bm x}_n,\theta^{\bm x}_n) \in \Phi(B)$ and $(r^{\bm y}_n,\theta^{\bm y}_n) \in \mathbb C_\Phi$ such that 
\begin{align}
    \bm d_\textbf{U}\big( (r^{\bm x}_n,\theta^{\bm x}_n),  (r^{\bm y}_n,\theta^{\bm y}_n)  \big)
    =
    \norm{ r^{\bm x}_n - r^{\bm y}_n   } \vee \norm{ \theta^{\bm x}_n - \theta^{\bm y}_n   }
    \to 0
    \quad \text{as }n \to \infty.
    \label{proof, property 1, lemma: equivalence for bounded away condition, polar coordinates}
\end{align}
By property \eqref{property: polar transform, 1}, there exists a sequence $\bm x_n \in B$ such that $(r^{\bm x}_n,\theta^{\bm x}_n) = \Phi(\bm x_n)$ for each $n \geq 1$,
and hence $\bm x_n = r^{\bm x}_n \theta^{\bm x}_n$ due to $\norm{\bm x_n} > \Delta$.
By \eqref{proof, property 1, lemma: equivalence for bounded away condition, polar coordinates}, for any $n$ large enough we have 
$
\norm{\theta^{\bm x}_n - \theta^{\bm y}_n } < \epsilon.
$
Since $\mathbb C$ is a cone, by the definition in \eqref{def: cone C enlarged by angles} we arrive at the contradiction $\bm x_n =  r^{\bm x}_n \theta^{\bm x}_n \in \bar{\mathbb C}(\epsilon)$ for all $n$ large enough.
This concludes the proof of $(i) \Rightarrow (ii)$.

\medskip
\noindent
\textbf{Proof of $(ii) \Rightarrow (i)$}.
Fix some $\Delta > 0$ such that 
\begin{align}
    \inf\big\{
        \norm{ r^{\bm x} - r^{\bm y} } \vee \norm{ \theta^{\bm x} - \theta^{\bm y}   }:\ 
        (r^{\bm x},\theta^{\bm x}) \in \Phi(B),\ (r^{\bm y},\theta^{\bm y}) \in \mathbb C_\Phi
    \big\} > \Delta.
    \label{proof, property 2, lemma: equivalence for bounded away condition, polar coordinates}
\end{align}
First, note that $r^{\bm x} > \Delta$ for any $(r^{\bm x},\theta^{\bm x}) \in \Phi(B)$.
To see why, simply note that $\bm 0 \in \mathbb C$, and hence $(0,\theta) \in \mathbb C_\Phi$ for any $r \in \mathfrak N^d_+$.
As a result, we have $\inf_{\bm x\in B}\norm{\bm x} > \Delta$.
Furthermore, note that
\begin{align}
    \inf\big\{
       \norm{ \theta^{\bm x} - \theta^{\bm y}   }:\ 
        (r^{\bm x},\theta^{\bm x}) \in \Phi(B),\ (r^{\bm y},\theta^{\bm y}) \in \mathbb C_\Phi,\ r^{\bm y} > 0
    \big\} > \Delta.
    \label{proof, property 3, lemma: equivalence for bounded away condition, polar coordinates}
\end{align}
To see why, note that for any
$(r^{\bm x},\theta^{\bm x}) \in \Phi(B)$ and $(r^{\bm y},\theta^{\bm y}) \in \mathbb C_\Phi$ with $r^{\bm y} > 0$,
we have $(r^{\bm x},\theta^{\bm y}) \in \mathbb C_\Phi$ since $\mathbb C$ is a cone.
Claim~\eqref{proof, property 3, lemma: equivalence for bounded away condition, polar coordinates} then follows from \eqref{proof, property 2, lemma: equivalence for bounded away condition, polar coordinates}.
On the other hand, by the definition of $\bar{\mathbb C}(\epsilon)$, for any $\epsilon,\delta > 0$ and any $B \subseteq \R^d_+$ with $\inf_{\bm x \in B}\norm{\bm x} > \delta$,
the claim $B \cap \bar{\mathbb C}(\epsilon) = \emptyset$ would imply that $B$ is bounded away from $\bar{\mathbb C}(\epsilon/2)$.
Indeed, $\bar{\mathbb C}(\epsilon/2) \cap \{ \bm x \in \R^d_+:\ \norm{\bm x} < \delta / 2 \}$ is clearly bounded away from $B$ due to 
$\inf_{\bm x \in B}\norm{\bm x} > \delta$;
as for $\bar{\mathbb C}(\epsilon/2) \cap \{ \bm x \in \R^d_+:\ \norm{\bm x} \geq \delta / 2 \}$,
one only needs to note that this set is bounded away from $\big(\bar{\mathbb C}(\epsilon)\big)^\complement$.
In summary, it suffices to find some $\epsilon > 0$ such that 
\begin{align}
    B \cap \bar{\mathbb C}(\epsilon) = \emptyset.
    \nonumber
\end{align}
To this end, we fix some $\epsilon\in (0,\Delta)$.
Since $\inf_{\bm x \in B}\norm{\bm x} > \Delta$, it suffices to consider some $\bm y \in \bar{\mathbb C}(\epsilon)$ with $\bm y \neq \bm 0$.
Let $(r,\theta^\prime) = \Phi(\bm y)$. Note that $r > 0$ due to $\bm y \neq \bm 0$.
Besides,
by the definition of $\bar{\mathbb C}(\epsilon)$, there exists some $\theta \in \mathfrak N^d_+$ such that $\norm{\theta - \theta^\prime} \leq \epsilon < \Delta$
and $(r,\theta) \in \C_\Phi$.
Then, by the property~\eqref{proof, property 3, lemma: equivalence for bounded away condition, polar coordinates} and our choice of $\epsilon \in (0,\Delta)$, we must have $\bm y \notin B$.
By the arbitrariness of $\bm y \in \bar{\mathbb C}(\epsilon) \setminus \{\bm 0\}$, 
we yield $B \cap \bar{\mathbb C}(\epsilon) = \emptyset$ and conclude the proof of $(ii) \Rightarrow (i)$.
\end{proof}

Next, we state the proof of Lemma~\ref{lemma: M convergence for MRV}.

\begin{proof}[Proof of Lemma~\ref{lemma: M convergence for MRV}]
\linksinpf{lemma: M convergence for MRV}
To prove $(i)\Rightarrow (ii)$,
we fix some
closed $F \subset \R^d_+$ and open $O \subset \R^d_+$ such that $F$ and $O$ are both bounded away from $\bar{\mathbb C}(\epsilon)$ for some $\epsilon > 0$.
Due to $\bm 0\in \mathbb C$, we must have that $\bm 0$ is bounded away from both $F$ and $O$.
Furthermore, by Lemma~\ref{lemma: equivalence for bounded away condition, polar coordinates}, we get
\begin{align}
    \bm d_\textbf{U}\big( \Phi(F), \C_\Phi\big) > 0,
    \qquad
    \bm d_\textbf{U}\big( \Phi(O), \C_\Phi\big) > 0.
    \label{proof: bounded away condition under d U, lemma: M convergence for MRV}
\end{align}
Now, observe that
\begin{align*}
    \P(X_n \in F)
    & = 
    \P\big( (R_n,\Theta_n) \in \Phi(F)\big)
    \qquad
    \text{by \eqref{property: polar transform, 1}},
    \\ 
    \Longrightarrow
    \limsup_{n \to \infty}\epsilon^{-1}_n
    {
        \P( X_n \in F)
    }
    & \leq 
    \mu \circ \Phi^{-1}\Big((\Phi(F))^-\Big)
    \qquad
    \text{
        by \eqref{condition: M convergence for polar coordinates, lemma: M convergence for MRV} and \eqref{proof: bounded away condition under d U, lemma: M convergence for MRV}
    }
    \\ 
    & = 
    \mu \circ \Phi^{-1}\big(\Phi(F)\big)
    \qquad
    \text{by \eqref{property: polar transform, 2}}
    \\ 
    & =  \mu(F)
    \qquad
    \text{by the definitions in \eqref{def: mu composition Phi inverse measure}}.
\end{align*}
Furthermore, condition~\eqref{condition: M convergence for polar coordinates, lemma: M convergence for MRV} implies that $\mu \circ \Phi^{-1} \in \M\big( [0,\infty) \times \mathfrak N^d_+ \setminus \mathbb C_\Phi\big)$,
and hence $\mu\circ \Phi^{-1}(E) < \infty$ for any Borel set $E \subseteq [0,\infty) \times \mathfrak N^d_+$ that is bounded away from $\mathbb C_\Phi$.
Since  $\Phi(F)$ is bounded away from $\mathbb C_\Phi$, we verify that
$
 \mu \circ \Phi^{-1}\big(\Phi(F)\big) = \mu(F) < \infty.
$
Analogously, one can show that
$
\liminf_{n \to \infty}\epsilon^{-1}_n \P(X_n \in O) \geq \mu(O).
$
To conclude the proof of  $(i)\Rightarrow (ii)$, we pick $O = A^\circ$ and $F = A^-$ in \eqref{claim, lemma: M convergence for MRV}.
Lastly, we note that the proof of $(ii)\Rightarrow (i)$ is almost identical and follows from a reverse applicaton of Lemma~\ref{lemma: equivalence for bounded away condition, polar coordinates}.
We omit the details here to avoid repetition.
\end{proof}

\section{Proofs of Technical Lemmas}
\label{subsec: proof, technical lemmas, cluster size}

\subsection{Proofs of Lemmas~\ref{lemma: tail bound, pruned cluster size S i leq n delta} and \ref{lemma: concentration ineq for pruned cluster S i}}
\label{subsubsec, proof, concentration inequalities for S leq M}

Recall the definition of ${\bar b_{j\leftarrow i}} = \E B_{j\leftarrow i}$,
as well as the mean offspring matrix ${\bar{\textbf B}} = (\bar b_{j\leftarrow i})_{j,i\in [d]}$.
We adopt the operator norm
$
\norm{\textbf A} = \sup_{ \norm{\bm x} = 1  }\norm{\textbf A\bm x}
$
for any $d \times d$ real-valued matrix under the $L_1$ norm for vectors in $\R^d$.
We first provide the proofs of
Lemmas~\ref{lemma: tail bound, pruned cluster size S i leq n delta} and \ref{lemma: concentration ineq for pruned cluster S i} under the condition that $\norm{\bar{\textbf B}} < 1$.
Then, inspired by the approach in \cite{KEVEI2021109067} based on Gelfand's formula, we extend the proof to general cases.

\begin{proof}[Proof of Lemma~\ref{lemma: tail bound, pruned cluster size S i leq n delta} ($\norm{\bar{\textbf B}} < 1$)]
\linksinpf{lemma: tail bound, pruned cluster size S i leq n delta}
By considering the transform $N = n\Delta$ (and hence $n\delta = N \frac{\delta}{\Delta}$), it suffices to prove the claim for $\Delta = 1$.
Besides, since the index $i$ takes finitely many possible values from $[d] = \{1,2,\ldots,d\}$, we only need to fix some $i \in [d]$ in this proof and and show the existence of some $\delta_0 = \delta_0(\gamma) > 0$ such that
$
\P\Big(
        \norm{ \bm S^{\leqslant}_i(n\delta) } > n
    \Big) = \lo(n^{-\gamma})
$
for any $\delta \in (0,\delta_0)$.
Also, recall that we work with the condition that
$\rho \delequal \norm{\bar{\textbf B}} < 1$.
We fix some $\epsilon > 0$ small enough such that
\begin{align}
    2d\epsilon + \rho(1+2d\epsilon) + d\epsilon(1 + 2d\epsilon) < 1.
    \label{proof: pick epsilon, lemma: tail bound, pruned cluster size S i leq n delta}
\end{align}
Henceforth in the proof, we only consider $n$ large enough such that $n\epsilon > 1$.
Now,
we are able to fix some integer $K_\epsilon$ and a collection of vectors 
$\{ \bm z(k) = (z_{1}(k),\ldots,z_{d}(k))^\top:\ k \in [K_\epsilon] \}$
such that the following claims hold:
$(i)$ for each $k \in [K_\epsilon]$, we have $z_{j}(k) \geq 0\ \forall j \in [d]$ and $\sum_{j = 1}^d z_{j}(k) = 1$;
$(ii)$ given any $\bm z = (z_1,\cdots,z_d)^\top \in [0,\infty)^d$ with $\sum_{j = 1}^d z_j = 1$,
there exists some $k \in [K_\epsilon]$ such that
\begin{align}
    | z_j - z_j(k) | < \epsilon,\qquad \forall j \in [d].
    \label{proof: property of e k, lemma: tail bound, pruned cluster size S i leq n delta}
\end{align}
The vectors $\big(\bm z(k)\big)_{k \in [K_\epsilon]}$ provide a finite covering of
\begin{align}
    \mathcal Z \delequal \Bigg\{ (z_1,\ldots,z_d)^\top \in [0,\infty)^d:\ \sum_{j =1 }^d z_j = 1 \Bigg\}
    \label{proof: def line segment E, lemma: tail bound, pruned cluster size S i leq n delta}
\end{align}
with resolution $\epsilon$.

For each $j \in [d]$, let $\{ {\bm B}^{(m)}_{\bcdot \leftarrow j}:\  m \geq 1 \}$
be i.i.d.\ copies of
$
\bm B_{\bcdot \leftarrow j},
$
which will be interpreted as the offspring count of the $m^\text{th}$ type-$j$ individual in the branching tree of $\bm S_i$.
More precisely, 
in this proof we order nodes in a multi-type branching tree using a standard rule:
given $j \in [d]$, type-$j$ nodes are numbered left to right, starting from generation 0, then continuing similarly in each subsequent generation.
 For instance, (i) in the branching tree for $\bm S_i$, the first type-$i$ node will always be the type-$i$ root node in the $0^\text{th}$ generation; 
 and
 (ii) if there are $n$ type-$j$ nodes in the first $k$ generations, the numbering in the $(k+1)^\text{th}$ generation starts from $n+1$.
In doing so, the underlying branching processes (and hence the total progeny $\bm S_i$) are measurable functions of $({\bm B}^{(m)}_{\bcdot \leftarrow j})_{j \in [d], m \geq 1}$.
Next, we set 
\begin{align}
\bm B^{\leqslant,(m)}_{ \bcdot \leftarrow j  }(M)
=
\big( B^{\leqslant,(m)}_{i \leftarrow j}(M)  \big)_{i \in [d]},
\quad\text{where }
B^{\leqslant,(m)}_{i \leftarrow j}(M) = B^{(m)}_{i \leftarrow j}\mathbbm{I}\{ B^{(m)}_{i \leftarrow j} \leq M \}.
    \nonumber
\end{align}
For each $M > 0$, we consider a similar coupling between $({\bm B}^{\leqslant,(m)}_{\bcdot \leftarrow j}(M))_{j \in [d], m \geq 1}$ and the branching tree for $\bm S^\leqslant_i(M)$,
such that ${\bm B}^{\leqslant,(m)}_{\bcdot \leftarrow j}(M)$ is the offspring count for the $m^\text{th}$ type-$j$ node in the branching tree for $\bm S^\leqslant_i(M)$.
Now, observe the following
on the event $\Big\{ \norm{ \bm S^{\leqslant}_i(n\delta) } > n \Big\}$:
by considering the first $n$ nodes in the tree\footnote{
The exact counting of the first $n$ nodes, across the $d$ types, can be made precise by assuming the following: 
within each generation, type-1 nodes reproduce first, followed by type-2, and so on; similarly, each node gives birth in order, first to type-1 children, then type-2, and so forth.
}
as well as their children,
we can find some $(n_1,\ldots,n_d)^\top \in \mathbb Z_+^d$ with $\sum_{j = 1}^d n_j =n$
such that
$
 n_j \leq \mathbbm{I}\{j = i\} + \sum_{l \in [d]} \sum_{m = 1}^{n_l}B^{\leqslant,(m)}_{ j \leftarrow l }(n\delta)
$
holds for each $j \in [d]$.
%
%
Also, we fix the $\bm z \in \mathcal Z$ (see \eqref{proof: def line segment E, lemma: tail bound, pruned cluster size S i leq n delta}) such that $(n_1,\ldots,n_d)^\top = n \bm z$,
and
recall that we only consider $n$ with $n\epsilon > 1$.
By our choice of $\bm z(k)$'s in \eqref{proof: property of e k, lemma: tail bound, pruned cluster size S i leq n delta}, there exists some $k \in [K_\epsilon]$ such that
\begin{align*}
    n z_j(k) - n\epsilon \leq n\epsilon + \sum_{l \in [d]} \sum_{m = 1}^{ \ceil{ n z_l(k) + n\epsilon } }B^{\leqslant,(m)}_{j \leftarrow l}(n\delta),
    \qquad
    \forall j \in [d].
\end{align*}
In summary, we obtain
\begin{align}
   \P\Big(
        \norm{ \bm S^{\leqslant}_i(n\delta) } > n
    \Big)
    & \leq 
    \sum_{ k \in [K_\epsilon] }
    \P\Bigg(
        nz_j(k) \leq 2n\epsilon + \sum_{l \in [d]} \sum_{m = 1}^{ \ceil{ nz_l(k) + n\epsilon } }B^{\leqslant,(m)}_{ j \leftarrow l }(n\delta)\ \forall j \in [d]
    \Bigg).
    \label{proof: ineq S i leq n delta, lemma: tail bound, pruned cluster size S i leq n delta}
\end{align}
Furthermore, 
suppose that for each $\bm z = (z_1,\ldots,z_d)^\top \in \mathcal Z$, we have
(for any $\delta > 0$ small enough)
\begin{align}
    \P\Bigg(
        \underbrace{ n z_j \leq 2n\epsilon + \sum_{l \in [d]} \sum_{m = 1}^{ \ceil{ nz_l + n\epsilon } }B^{\leqslant,(m)}_{j \leftarrow l}(n\delta)\ \forall j \in [d]
        }_{ \delequal A(n,\delta,\bm z) }
    \Bigg) = \lo(n^{-\gamma}),
    \quad\text{as }n \to \infty.
    \label{proof: goal 1, lemma: tail bound, pruned cluster size S i leq n delta}
\end{align}
Then, by applying \eqref{proof: goal 1, lemma: tail bound, pruned cluster size S i leq n delta} for the finitely many $\bm z(k)$'s identified in \eqref{proof: property of e k, lemma: tail bound, pruned cluster size S i leq n delta},
we can find some $\delta_0 > 0$---depending only on $\epsilon$ and $\gamma$---such that in \eqref{proof: ineq S i leq n delta, lemma: tail bound, pruned cluster size S i leq n delta},
we have
$
 \P\Big(
        \norm{ \bm S^{\leqslant}_i(n\delta) } > n
    \Big)
    \leq K_\epsilon \cdot o(n^{-\gamma}) = \lo (n^{-\gamma})
$
for any $\delta \in (0,\delta_0)$.
Now, it only remains to prove Claim~\eqref{proof: goal 1, lemma: tail bound, pruned cluster size S i leq n delta}.

\medskip
\noindent
\textbf{Proof of Claim~\eqref{proof: goal 1, lemma: tail bound, pruned cluster size S i leq n delta}}.
 Note that $\bm z \in \mathcal Z$ implies $\sum_{j \in [d]} z_j = 1$.
Also, recall that $\bar b_{j \leftarrow l} = \E B_{j \leftarrow l} \geq \E B^{\leqslant,(m)}_{j \leftarrow l}(n\delta)$.
Define the event
\begin{align*}
F(n,\delta,\bm z) \delequal
    \bigcap_{l \in [d],\ j \in [d]}
    \Bigg\{
        \sum_{m = 1}^{ \ceil{ n(z_l + \epsilon) } }B^{\leqslant,(m)}_{j \leftarrow l}(n\delta) 
        \leq 
        \ceil{ n(z_l + \epsilon) }\cdot (\bar b_{j \leftarrow l} + \epsilon)
    \Bigg\}.
\end{align*}
We first show that on the event $F(n,\delta,\bm z)$,
we have 
\begin{align*}
    \norm{n\bm z} > \norm{ \bigg( 
        2n\epsilon +   \sum_{l \in [d]} \sum_{m = 1}^{ \ceil{ nz_l + n\epsilon } }B^{\leqslant, (m)}_{1 \leftarrow l}( n\delta),
        \ldots, 
        2n\epsilon +   \sum_{l \in [d]} \sum_{m = 1}^{ \ceil{ nz_l + n\epsilon } }B^{\leqslant, (m)}_{d \leftarrow l}( n\delta)
    \bigg)^\top  },
\end{align*}
and hence $F(n,\delta,\bm z) \cap A(n,\delta,\bm z) = \emptyset$.
To see why, note that
on $F(n,\delta,\bm z)$, we have
\begin{align}
    2n\epsilon + \sum_{l \in [d]} \sum_{m = 1}^{ \ceil{ nz_l + n\epsilon } }B^{\leqslant,(m)}_{j \leftarrow l}(n\delta)
    \leq 
    2n\epsilon + \sum_{l \in [d]} \ceil{ n(z_l + \epsilon) }\cdot (\bar b_{j \leftarrow l} + \epsilon),
    \quad \forall j \in [d].
    \label{proof, implication of event F, lemma: tail bound, pruned cluster size S i leq n delta}
\end{align}
To describe the implications of \eqref{proof, implication of event F, lemma: tail bound, pruned cluster size S i leq n delta},
we first recall the notational conventions
$(x_1,\ldots,x_d)^\top \leq (y_1,\ldots,y_d)^\top$ if $x_j \leq y_j\ \forall j \in [d]$,
$\bm 1 = (1,\ldots,1)^\top$,
and
$\textbf X + c = (x_{l,j}+c)_{l,j}$ for a matrix $\textbf X = (x_{l,j})_{l,j}$.
For any $n$ large enough such that $n\epsilon > 1$ (with $\epsilon$ specified in \eqref{proof: pick epsilon, lemma: tail bound, pruned cluster size S i leq n delta}),
the vectorized version of the RHS of
Claim~\eqref{proof, implication of event F, lemma: tail bound, pruned cluster size S i leq n delta} is upper bounded by
\begin{align}
    & 2n\epsilon \cdot \bm 1 +  
     (\bar{\textbf B} + \epsilon)\big(\ceil{ n(z_1 + \epsilon) },\ldots,\ceil{ n(z_d + \epsilon) } \big)^\top
    \label{proof, ineq 1, goal 1, lemma: tail bound, pruned cluster size S i leq n delta}
    \\
    \leq &
        2n\epsilon \cdot \bm 1
            +
        (\bar{\textbf B} + \epsilon)\big( n(z_1 + 2\epsilon),\ldots, n(z_d + 2\epsilon)\big)^\top
    \quad 
    \text{ due to $n\epsilon > 1$}
    \nonumber
    \\ 
    =  & 
    2n\epsilon \cdot \bm 1
    +
    \bar{\textbf B}\big( n(z_1 + 2\epsilon),\ldots, n(z_d + 2\epsilon)\big)^\top
    +
    n\epsilon(1 + 2d\epsilon) \cdot \bm 1
    \
    \text{ since }\sum_{j \in [d]}z_j = 1.
    \nonumber
\end{align}
Due to $\rho = \norm{\bar{\textbf B}} < 1$
and  $\sum_{j \in [d]} z_j = 1$,
we get
$
 \norm{  \bar{\textbf B}\big( n(z_1 + 2\epsilon),\ldots, n(z_d + 2\epsilon)\big)^\top }
   \leq 
   \rho n \cdot \sum_{j \in [d]}(z_j + 2\epsilon)
   =
   \rho n (1 + 2d\epsilon).
$
Combining this bound with \eqref{proof, ineq 1, goal 1, lemma: tail bound, pruned cluster size S i leq n delta}, we get
\begin{align*}
   & 
   \norm{ 
    2n\epsilon \cdot \bm 1 + 
     (\bar{\textbf B} + \epsilon)\big(\ceil{ n(z_1 + \epsilon) },\ldots,\ceil{ n(z_d + \epsilon) } \big)^\top
   }
   \\ 
   \leq &
   n \cdot \Big( 2d\epsilon + \rho(1+2d\epsilon) + d\epsilon(1 + 2d\epsilon) \Big)
   < n = \norm{ n \bm z}
   \qquad
   \text{by \eqref{proof: pick epsilon, lemma: tail bound, pruned cluster size S i leq n delta}.}
\end{align*}
In summary, $F(n,\delta,\bm z) \cap A(n,\delta,\bm z) = \emptyset$ holds for any $n$ large enough.
This implies
\begin{align}
    \label{proof, ineq 2, goal 1, lemma: tail bound, pruned cluster size S i leq n delta}
    \P\big( A(n,\delta,\bm z)\big)
    \leq 
    \sum_{ l \in [d],j \in [d] }
    \P\Bigg(
        \frac{1}{\ceil{ n(z_l + \epsilon) }}
        \sum_{m = 1}^{ \ceil{ n(z_l + \epsilon) } }B^{\leqslant,(m)}_{j \leftarrow l}(n\delta) 
        >
        (1 + \epsilon) \bar b_{j \leftarrow l}
    \Bigg).
\end{align}
Recall the $B^{\leqslant,(m)}_{j \leftarrow l}(n\delta) $'s are i.i.d.\ copies of $B_{j \leftarrow l}\mathbbm{I}\{B_{j \leftarrow l} \leq n\delta\}$.
By Assumption~\ref{assumption: heavy tails in B i j},
we have $\P(B_{j \leftarrow l} > x) \in \RV_{-\alpha_{j \leftarrow l}}(x)$ with $\alpha_{j \leftarrow l} > 1$.
Applying Lemma~\ref{lemma: concentration ineq, truncated heavy tailed RV in Rd},
we confirm that for any $\delta > 0$ small enough,
the RHS of \eqref{proof, ineq 2, goal 1, lemma: tail bound, pruned cluster size S i leq n delta}
is upper bounded by an $\lo(n^{-\gamma})$ term.
This concludes the proof of Claim~\eqref{proof: goal 1, lemma: tail bound, pruned cluster size S i leq n delta} for the case of $\norm{\bar{\textbf B}} < 1$.
\end{proof}

\begin{proof}[Proof of Lemma~\ref{lemma: concentration ineq for pruned cluster S i}]
\linksinpf{lemma: concentration ineq for pruned cluster S i}
We first note that
this proof does not explicitly require the condition $\norm{\bar{\textbf B}} < 1$.
That is, once we establish  Lemma~\ref{lemma: tail bound, pruned cluster size S i leq n delta} for the case of $\norm{\bar{\textbf B}} \geq 1$,
the same proof below will follow, so there is no need to distinguish these two cases for the proof of Lemma~\ref{lemma: concentration ineq for pruned cluster S i}.
In addition,
it suffices to fix some $i,j \in [d]$ and $\epsilon,\gamma > 0$, and then prove the existence of $\delta_0 = \delta_0(\epsilon,\gamma) > 0$ such that the claims
\begin{align}
 \lim_{n \to \infty} n^\gamma\cdot 
    \P\Bigg(
        \frac{1}{n}\sum_{m = 1}^n S^{\leqslant, (m)}_{i,j}(n\delta) <  \bar s_{i,j} - \epsilon
    \Bigg) & = 0,
    \label{proof: goal lower bound, lemma: concentration ineq for pruned cluster S i}
    \\
    \lim_{n \to \infty} n^\gamma\cdot
    \P\Bigg(
        \frac{1}{n}\sum_{m = 1}^n S^{\leqslant, (m)}_{i,j}(n\delta) >  \bar s_{i,j} + \epsilon
    \Bigg) & = 0
    \label{proof: goal upper bound, lemma: concentration ineq for pruned cluster S i}
\end{align}
hold for any $\delta \in (0,\delta_0)$,
where we write $\bm S^{\leqslant,(m)}_i(M) = \big( S^{\leqslant,(m)}_{i,j}(M)\big)_{j \in [d]}$.

\medskip
\noindent\textbf{Proof of Claim \eqref{proof: goal lower bound, lemma: concentration ineq for pruned cluster S i}}.
Take any $\delta > 0$.
Monotone convergence implies $\lim_{M \to \infty}\E \bm S^{\leqslant}_{i}(M) = \E \bm S_{i} = \bar{\bm s}_{i}$,
thus allowing us to fix $M > 0$ such that
$
\bar{\bm s}_{i} - \epsilon \bm 1 < \E \bm S^{\leqslant}_{i}(M).
$
Furthermore, monotone convergence implies that for any $M^\prime$ large enough, we have
$
\bar{\bm s}_{i} - \epsilon \bm 1 < \E\Big[ \bm S^{\leqslant}_{i}(M)\mathbbm{I}\Big\{ \norm{\bm S^{\leqslant}_{i}(M)} \leq M^\prime \Big\}\Big].
$
The stochastic comparison property  \eqref{property: stochastic comparison, S and pruned S} then implies
$
\bm S^{\leqslant}_{i}(M)\mathbbm{I}\Big\{ \norm{\bm S^{\leqslant}_{i}(M)} \leq M^\prime \Big\}
\leq 
\bm S^{\leqslant}_{i}(M) \stleq \bm S^{\leqslant}_i(n\delta)
$
for any $n$ large enough such that $n\delta \geq M$.
Therefore, it suffices to prove
\begin{align}
    \P\bigg(
        \frac{1}{n}\sum_{m = 1}^n S^{\leqslant, (m)}_{i,j}(M)\mathbbm{I}\Big\{
         \norm{\bm S^{\leqslant, (m)}_{i}(M)} \leq M^\prime
        \Big\}
        < \bar s_{i,j} - \epsilon
    \bigg) = \lo(n^{-\gamma}).
    \nonumber
\end{align}
In particular, note that the i.i.d.\ copies $S^{\leqslant, (m)}_{i,j}(M)\mathbbm{I}\Big\{
         \norm{\bm S^{\leqslant, (m)}_{i}(M)} \leq M^\prime
        \Big\}$
have finite moment generating functions due to the truncation under $M^\prime$.
This allows us to apply Cram\`er's Theorem to conclude the proof of Claim~\eqref{proof: goal lower bound, lemma: concentration ineq for pruned cluster S i}.

\medskip
\noindent\textbf{Proof of Claim \eqref{proof: goal upper bound, lemma: concentration ineq for pruned cluster S i}}.
Take any $\Delta > 0$.
For any $x,c \in \R$, let $\phi_c(x) = x \wedge c$.
Observe that
\begin{equation} \label{proof, decomp of events, goal upper bound, lemma: concentration ineq for pruned cluster S i}
    \begin{aligned}
       &  \bigg\{
    \frac{1}{n}\sum_{m = 1}^n S^{\leqslant, (m)}_{i,j}(n\delta) >  \bar s_{i,j} +  \epsilon
  \bigg\}  
  \\
    & \subseteq
  \Big\{
  \norm{ \bm S^{\leqslant, (m)}_i(n\delta) } > n\Delta\text{ for some }m \in [n]
  \Big\}
  \cup 
  \bigg\{
    \frac{1}{n}\sum_{m = 1}^n \phi_{n\Delta}\big(S^{\leqslant, (m)}_{i,j}(n\delta)\big) >  \bar s_{i,j} +  \epsilon
  \bigg\}.
    \end{aligned}
\end{equation}
On the one hand, given any $\Delta > 0$, there exists $\delta_0(\Delta,\gamma) > 0$ such that 
  $\forall \delta \in (0,\delta_0)$,
\begin{align}
   \P\Big(
  \norm{ \bm S^{\leqslant, (m)}_i(n\delta) } > n\Delta\text{ for some }m \in [n]
  \Big) 
  \leq n \cdot \P\Big(
  \norm{ \bm S^{\leqslant}_i(n\delta) } > n\Delta
  \Big)
  = \lo(n^{-\gamma})
  \label{proof: intermediate upper bound, lemma: concentration ineq for pruned cluster S i}
\end{align}
cf.\ Lemma~\ref{lemma: tail bound, pruned cluster size S i leq n delta}.
On the other hand,
by the stochastic comparison in \eqref{property: stochastic comparison, S and pruned S},
\begin{align*}
\P\bigg(
    \frac{1}{n}\sum_{m = 1}^n \phi_{n\Delta}\big(S^{\leqslant, (m)}_{i,j}(n\delta)\big) >  \bar s_{i,j} +  \epsilon
  \bigg)
& \leq 
\underbrace{ \P\bigg(
    \frac{1}{n}\sum_{m = 1}^n \phi_{n\Delta}(S_{i,j}^{(m)}) >  \bar s_{i,j} +  \epsilon
  \bigg)
  }_{ = p(n,\Delta) },
\end{align*}
with the $S^{(m)}_{i,j}$'s being i.i.d.\ copies of 
$S_{i,j}$.
Suppose we can show that
    $\P(S_{i,j} > x) \in \RV_{-\alpha}(x)$ for some $\alpha > 1$.
Then, by Claim~\eqref{claim 2, lemma: concentration ineq, truncated heavy tailed RV in Rd} in Lemma~\ref{lemma: concentration ineq, truncated heavy tailed RV in Rd} and property \eqref{property, appendix, psi to phi, truncation mapping},
we fix some $\Delta > 0$ small enough such that
$
p(n,\Delta) = \lo(n^{-\gamma})
$
as $n \to \infty$.
Plugging this bound and \eqref{proof: intermediate upper bound, lemma: concentration ineq for pruned cluster S i} into \eqref{proof, decomp of events, goal upper bound, lemma: concentration ineq for pruned cluster S i}, we conclude the proof of Claim~\eqref{proof: goal upper bound, lemma: concentration ineq for pruned cluster S i}.
Now, it only remains to verify the regular variation of $S_{i,j}$.
By Assumptions~\ref{assumption: heavy tails in B i j} and \ref{assumption: regularity condition 2, cluster size, July 2024},
there uniquely exists a pair $(l^*,k^*) \in [d]^2$ such that 
$
\alpha_{l^* \leftarrow k^*} = \alpha^* \delequal \min_{l.k \in [d]}\alpha_{l \leftarrow k},
$
and $\alpha^* > 1$, $\P(B_{l^*\leftarrow k^*}>x)\in\RV_{-\alpha^*}(x)$.
By Theorem~2 of \cite{Asmussen_Foss_2018} (under the choice of $Q(k) = \mathbbm{I}\{k = j\}$ in Equation (6) of \cite{Asmussen_Foss_2018}), 
there exists a constant $c^*_{i,j} > 0$ such that 
$
\P( S_{i,j} > x) \sim c^*_{i,j} \P(B_{l^* \leftarrow k^*} > x)
$
as $x \to \infty$.
In particular, 
translating our Assumption~\ref{assumption: regularity condition 1, cluster size, July 2024} into the context of \cite{Asmussen_Foss_2018},
we have $m_{ik} > 0$ for any $i,k$ in Equation (15) of \cite{Asmussen_Foss_2018},
thus implying $d_i > 0$ for any $i$ in Equation (15) of \cite{Asmussen_Foss_2018}.
Equivalently, this confirms $c^*_{i,j} > 0$.
\end{proof}

\subsection{Proof of Lemma~\ref{lemma: tail bound, pruned cluster size S i leq n delta}: General Case}
Recall the definition of ${\bar b_{j\leftarrow i}} = \E B_{j\leftarrow i}$,
 the mean offspring matrix ${\bar{\textbf B}} = (\bar b_{j\leftarrow i})_{j,i \in [d]}$,
and the operator norm
$
\norm{\textbf A} = \sup_{ \norm{\bm x} = 1  }\norm{\textbf A\bm x}
$
for matrix $\textbf A \in \R^{d \times d}$ under the $L_1$ norm for vectors in $\R^d$.
We provide the proof of Lemma~\ref{lemma: tail bound, pruned cluster size S i leq n delta} without the additional assumption that $\norm{\bar{\textbf B}} < 1$.
We first prepare the following lemma.

\begin{lemma} \label{lemma: tail asymptotics for truncated GW, at generation t}
\linksinthm{lemma: tail asymptotics for truncated GW, at generation t}
    Let Assumption \ref{assumption: heavy tails in B i j} hold.
    Let $\bm X^\leqslant_j(t;M)$ be defined as in \eqref{def: branching process X n leq M}.
    Given $t \geq 1$, $j \in [d]$, $\Delta > 0$, and $\gamma > 0$,
    \begin{align}
        \lim_{n \to \infty}n^\gamma\cdot \P\bigg(\norm{ \bm X^\leqslant_j(t;n\delta) } > n\Delta  \bigg) = 0,
        \qquad
        \forall \delta > 0\text{ sufficiently small.}
        \label{claim, lemma: tail asymptotics for truncated GW, at generation t}
    \end{align}
\end{lemma}

\begin{proof}\linksinpf{lemma: tail asymptotics for truncated GW, at generation t}
We first consider the case of $t = 1$.
By definitions in \eqref{def: branching process X n leq M} and that $\bm X_j^\leqslant(0;n\delta) = \bm e_j$,
we have
$
\bm X^\leqslant_j(1;n\delta) = \big( B^{(1,1)}_{ i \leftarrow j }\mathbbm{I}\{ B^{(1,1)}_{ i \leftarrow j }  \leq n\delta \} \big)_{i \in [d]}.
$
By picking $\delta \in (0,\Delta/d)$, we must have $\norm{\bm X^\leqslant_j(1;n\delta)} \leq d \cdot n\delta < n\Delta$.
Next, we proceed inductively.
Specifically, we fix some $\gamma > 0$, $j \in [d]$,
and suppose that there exists some positive integer $T$ such that Claim~\eqref{claim, lemma: tail asymptotics for truncated GW, at generation t} holds for any $t \in [T]$ and $\Delta > 0$.
Then, given $\Delta,\Delta^\prime > 0$, by the definitions in \eqref{def: branching process X n leq M} we have
\begin{align*}
    & \bigg\{
        \norm{ \bm X^\leqslant_j(T+1;n\delta) } > n\Delta
    \bigg\}
    \\ 
    & \subseteq
    \underbrace{ \bigg\{
        \norm{ \bm X^\leqslant_j(T;n\delta) } > n\Delta^\prime
    \bigg\}
    }_{ \delequal \text{(I)}  }
    \cup 
    \underbrace{ \Bigg\{
        \norm{
             \sum_{i \in [d]} \sum_{m = 1}^{ \floor{n\Delta^\prime} }{\bm B}^{\leqslant,(T+1,m)}_{ \bcdot \leftarrow i}(n\delta)
        } > n\Delta
    \Bigg\} }_{ \delequal \text{(II)}  }.
\end{align*}
In particular, 
recall that $\bar b_{l \leftarrow i} = \E B_{l \leftarrow i}$.
Given $\Delta > 0$, we pick $\Delta^\prime > 0$ small enough such that
\begin{align}
    \Delta^\prime \cdot \max_{ l \in [d],i \in [d] }\bar b_{l \leftarrow i} < \Delta/d^2.
    \label{proof, choice of constant Delta prime, lemma: tail asymptotics for truncated GW, at generation t}
\end{align}
On the one hand, by our assumption for the inductive argument, we have
$
\P\big( \text{(I)} \big) = \lo(n^{-\gamma})
$
under any $\delta > 0$ small enough.
On the other hand, using $B^{(m)}_{ l \leftarrow i  }$ to denote generic i.i.d.\ copies of $B_{ l \leftarrow i  }$,
we have
\begin{align*}
    \P\big(\text{(II)}\big)
    & \leq 
    \sum_{i \in [d]}\sum_{l \in [d]}
    \P\Bigg(
        \Bigg|
            \sum_{m = 1}^{ \floor{n\Delta^\prime} }{B}^{(m)}_{ l \leftarrow i}\mathbbm{I}\Big\{ {B}^{(m)}_{ l \leftarrow i} \leq n\delta  \Big\}
        \Bigg| > \frac{n\Delta}{d^2}
    \Bigg).
\end{align*}
Assumption~\ref{assumption: heavy tails in B i j} dictates that 
$
\P(B_{l \leftarrow i} > x) \in \RV_{-\alpha_{l \leftarrow i}}(x) 
$
with $\alpha_{l \leftarrow i} > 1$.
With $\Delta^\prime$ fixed in \eqref{proof, choice of constant Delta prime, lemma: tail asymptotics for truncated GW, at generation t}, we apply Claim~\eqref{claim 1, lemma: concentration ineq, truncated heavy tailed RV in Rd} in Lemma~\ref{lemma: concentration ineq, truncated heavy tailed RV in Rd}
for each pair $(l,i) \in [d]^2$ to obtain $\P\big(\text{(II)}\big) = \lo(n^{-\gamma})$ under any $\delta > 0$ small enough.
This confirms that, given $\Delta > 0$, the claim
$
\P\Big(\norm{ \bm X^\leqslant_j(T+1;n\delta) } > n\Delta \Big) = \lo(n^{-\gamma})
$
holds
for any $\delta >0$ small enough.
By proceeding inductively, we conclude the proof.
\end{proof}

Our proof of Lemma~\ref{lemma: tail bound, pruned cluster size S i leq n delta} (in the general case) is inspired by the strategy in 
 \cite{KEVEI2021109067}.
 In particular, we show that, for some positive integer $r$, results analogous to Lemma~\ref{lemma: tail bound, pruned cluster size S i leq n delta} hold for the $r$-step sub-sampled verison $\bm X^\leqslant_i(t;n\delta)$,
 and we apply the bounds for each sub-tree.
To this end, we first precisely define the total progeny of the sub-sampled branching process (for every $r$ generations):
\begin{align}
    \bm S^{[r],\leqslant}_j(M)
    \delequal
    \sum_{k \geq 0}\bm X_j^\leqslant(kr;M),
    \qquad j \in [d],\ M > 0,\ r \in \mathbb N,
    \label{proof, def, law of r step sub tree size}
\end{align}
with the multi-type branching process $\bm X_j^\leqslant(t;M)$ defined in \eqref{def: branching process X n leq M}.
That is, we only inspect the original branching process for every $r$ generations,
and use $\bm S^{[r],\leqslant}_j(M)$ to denote the total progeny of this $r$-step sub-sampled branching process.
Furthermore, let the random vectors $\bm B^{ [r],\leqslant  }_{ \bcdot \leftarrow j  }(M) = \big( B^{ [r],\leqslant  }_{ i \leftarrow j  }(M)\big)_{i \in [d]}$ have law
\begin{align}
    \mathscr L\Big(
        \bm B^{ [r],\leqslant  }_{ \bcdot \leftarrow j  }(M)
    \Big)
    = 
    \mathscr L\Big(  \bm X_j^\leqslant(r;M)  \Big),
    \qquad j \in [d],
    \label{def, law of offspring in r step sub sampled tree}
\end{align}
and note that (with the $ \bm S^{[r],\leqslant;(k)}_{i}(M)$'s being i.i.d.\ copies of $ \bm S^{[r],\leqslant}_{i}(M)$)
\begin{align}
    \bm S^{[r],\leqslant}_j(M)
    \distequal 
    \bm e_j + 
    \sum_{i \in [d]}\sum^{ B^{[r],\leqslant}_{i \leftarrow j}(M)   }_{k = 1}
    \bm S^{[r],\leqslant;(k)}_{i}(M),
    \qquad j \in [d].
    \nonumber
\end{align}
In other words, $\bm S^{[r],\leqslant}_j(M)$ also represents the total progeny of a branching process, whose offspring distribution admits the law in \eqref{def, law of offspring in r step sub sampled tree} and coincides with the $r^\text{th}$ generation offspring from a type-$j$ ancestor in the branching process $\big(\bm X_j^\leqslant(t;M)\big)_{ t \geq 0}$.

We use $\textbf{A}^k$ to denote the $k$-fold product of $\textbf{A}$ under matrix multiplication.
The next result establishes claims analogous to those in Lemma~\ref{lemma: tail bound, pruned cluster size S i leq n delta}, but for the sub-sampled $\bm S^{[r],\leqslant}_i(n\delta)$.

\begin{lemma} \label{lemma, tail bound, pruned cluster sub sampled size S i r leq n delta}
\linksinthm{lemma, tail bound, pruned cluster sub sampled size S i r leq n delta}
    Let Assumptions~\ref{assumption: heavy tails in B i j}--\ref{assumption: regularity condition 2, cluster size, July 2024} hold,
    and suppose that $\norm{\bar{\textbf{B}}^r} < 1$ holds for some positive integer $r$.
    Given any $\Delta,\ \gamma \in (0,\infty)$,
    there exists $\delta_0 = \delta_0(\Delta,\gamma,r) > 0$ such that
    \begin{align*}
    \lim_{n \to \infty}
    n^\gamma \cdot 
    \P\Big(
        \norm{ \bm S^{[r],\leqslant}_i(n\delta) } > n\Delta
    \Big) = 0,
    \qquad 
    \forall \delta \in (0,\delta_0),\ i \in [d].
\end{align*}
\end{lemma}

\begin{proof}\linksinpf{lemma, tail bound, pruned cluster sub sampled size S i r leq n delta}
Repeating the arguments in the proof of Lemma~\ref{lemma: tail bound, pruned cluster size S i leq n delta} under the additional condition $\norm{\bar{\textbf B}} < 1$
(in particular, the derivation of the bound \eqref{proof, ineq 2, goal 1, lemma: tail bound, pruned cluster size S i leq n delta})
in Section~\ref{subsubsec, proof, concentration inequalities for S leq M},
it suffices to show that given $\epsilon > 0$ and a vector $\bm z = (z_1,\ldots,z_d)^\top \in [0,\infty)^d$ with $\sum_{j \in [d]} z_j = 1$,
the claim
\begin{align}
    \sum_{j \in [d],\ l \in [d]}
    \P\Bigg(
        \frac{1}{\ceil{ n(z_l + \epsilon) }}
        \sum_{m = 1}^{ \ceil{ n(z_l + \epsilon) } }B^{[r],\leqslant;(m)}_{j \leftarrow l}(n\delta) 
        >
        (1 + \epsilon) \bar b^{[r]}_{j \leftarrow l}
    \Bigg)
    =
    \lo(n^{-\gamma})
    \nonumber
\end{align}
holds for any $\delta > 0$ small enough.
Here, 
$
\bm B^{[r],\leqslant;(m)}_{\bcdot \leftarrow l}
=
\big(
    B^{[r],\leqslant;(m)}_{j \leftarrow l}
\big)_{j \in [d]}
$
are i.i.d.\ copies of $\bm B^{[r],\leqslant}_{\bcdot \leftarrow l}$  under the law stated in \eqref{def, law of offspring in r step sub sampled tree},
and $\bar b^{[r]}_{j \leftarrow l}$ is the element on the $j^\text{th}$ row and $l^\text{th}$ column of the matrix $\bar{\textbf B}^r$,
meaning that $\bar b^{[r]}_{j \leftarrow l} = \E\big[ X_{l,j}(r) \big]$
for the branching process $\big(\bm X_j(t)\big)_{t \geq 0}$ defined in \eqref{def: branching process X n}.
Now, let $\phi_{c}(x) = x \wedge c$, and note that 
\begin{align*}
    & \Bigg\{
        \frac{1}{\ceil{ n(z_l + \epsilon) }}
        \sum_{m = 1}^{ \ceil{ n(z_l + \epsilon) } }B^{[r],\leqslant;(m)}_{j \leftarrow l}(n\delta) 
        >
        (1 + \epsilon) \bar b^{[r]}_{j \leftarrow l}
    \Bigg\}
    \\ 
    & \subseteq 
    \underbrace{ \bigg\{
        \Big|
            B^{[r],\leqslant;(m)}_{j \leftarrow l}(n\delta)
        \Big| > n\Delta
        \text{ for some }
        m \leq  \ceil{ n(z_l + \epsilon) }
    \bigg\} }_{ \delequal \text{(I)}  }
    \\ 
    & \qquad
    \cup 
    \underbrace{ \Bigg\{
        \frac{1}{\ceil{ n(z_l + \epsilon) }}
        \sum_{m = 1}^{ \ceil{ n(z_l + \epsilon) } }
        \phi_{n\Delta}\Big( B^{[r],\leqslant;(m)}_{j \leftarrow l}(n\delta)  \Big)
        >
        (1 + \epsilon) \bar b^{[r]}_{j \leftarrow l}
    \Bigg\} }_{ \delequal \text{(II)} }.
\end{align*}
As a result, it suffices to fix a pair $(l,j) \in [d]^2$ and find some $\Delta > 0$ such that $\P\big(\text{(I)}\big) = \lo(n^{-\gamma})$ and 
$\P\big(\text{(II)}\big) = \lo(n^{-\gamma})$
hold
under any $\delta > 0$ small enough.

\medskip
\noindent
\textbf{Proof of $\P\big(\text{(I)}\big) = \lo(n^{-\gamma})$}.
This claim holds for any $\Delta > 0$, due to
the law stated in \eqref{def, law of offspring in r step sub sampled tree} and Lemma~\ref{lemma: tail asymptotics for truncated GW, at generation t}.

\medskip
\noindent
\textbf{Proof of $\P\big(\text{(II)}\big) = \lo(n^{-\gamma})$}.
Suppose that we can find some random variable $\tilde B$ such that $B^{[r],\leqslant}_{j \leftarrow l}(M) \stleq \tilde B$ for any $M > 0$,
$\P(\tilde B > x) \in \RV_{-\alpha}(x)$ for some $\alpha > 1$, and 
$
\E\tilde B < (1+\epsilon) \bar b^{[r]}_{j \leftarrow l}.
$
Then, by combining
$
\sum_{m = 1}^k
 \phi_{n\Delta}\big( B^{[r],\leqslant;(m)}_{j \leftarrow l}(n\delta)  \big)
 \stleq 
\sum_{m = 1}^k
\phi_{n\Delta}\big( \tilde B^{(m)}  \big)
$
(with the $\tilde B^{(m)}$'s being independent copies of $\tilde B$)
with Claim~\eqref{claim 2, lemma: concentration ineq, truncated heavy tailed RV in Rd} in Lemma~\ref{lemma: concentration ineq, truncated heavy tailed RV in Rd} (applied onto $\phi_{n\Delta}\big( \tilde B^{(m)}  \big)$) and property \eqref{property, appendix, psi to phi, truncation mapping},
we get  $\P\big(\text{(II)}\big) = \lo(n^{-\gamma})$ for any $\Delta > 0$ small enough.

Now, it only remains to construct such $\tilde B$.
By \eqref{def, law of offspring in r step sub sampled tree} and the stochastic comparison stated in \eqref{property: stochastic comparison, general, glaton watson trees},
we have 
$
B^{[r],\leqslant}_{j \leftarrow l}(M)
\distequal 
X_{l,j}^\leqslant(r;M)
\stleq
X_{l,j}(r)
$
for each $M > 0$.
Also, we obviously have $X_{l,j}(r) \stleq S_{l,j}$ (since $\sum_{t \geq 0} \bm X_l(t) \distequal \bm S_l$).
By Theorem~2 of \cite{Asmussen_Foss_2018},
we have
$\P(S_{l,j} > x) \in \RV_{-\alpha^*}(x)$ for some $\alpha^* > 1$
(in fact, this has already been established at the end of the proof of Lemma~\ref{lemma: concentration ineq for pruned cluster S i}).
To proceed, let $\bar F(x) \delequal \P(X_{l,j}(r) > x)$, and pick some $\alpha \in (1,\alpha^*)$.
We consider some random variable $\tilde B$ with tail cdf $\tilde F(x) \delequal \P(\tilde B > x)$
with some parameter $L > 0$:
\begin{align}
    \tilde F(x) =
    \begin{cases}
        \bar F(x)\qquad &\text{ if }x \leq L
        \\ 
        \bar F(x) \vee
        \frac{L^\alpha \bar F(L)}{x^\alpha} \qquad &\text{ if }x > L
    \end{cases}.
    \label{proof, def law of tilde B, lemma, tail bound, pruned cluster sub sampled size S i r leq n delta}
\end{align}
To conclude the proof, we only need to note the following:
(i) by definition, we have $\tilde F(x) \geq \bar F(x)$ for any $x \in \R$, which implies $X_{l,j}(r) \stleq \tilde B$;
(ii) by Assumption~\ref{assumption: heavy tails in B i j}, the support of $X_{l,j}(r)$ is unbounded, so $\bar F(L) > 0$ for any $L \in (0,\infty)$;
(iii) due to $X_{l,j}(r) \stleq S_{l,j}$, $\P(S_{l,j} > x) \in \RV_{-\alpha^*}(x)$, and our choice of $\alpha \in (1,\alpha^*)$,
it follows from Potter's bound (see, e.g., Proposition~2.6 of \cite{resnick2007heavy}) that
$
\P(X_{l,j}(r) > x) = \bar F(x) <  \frac{L^\alpha \bar F(L)}{x^\alpha}
$
eventually for any $x$ large enough, 
meaning that under the law specified in \eqref{proof, def law of tilde B, lemma, tail bound, pruned cluster sub sampled size S i r leq n delta},
$\tilde B$ has a 
power-law tail with index $\alpha > 1$;
and
(iv) since the expectation of $\tilde B$ converges to $\E X_{l,j}(r) = \bar b^{[r]}_{ j \leftarrow l  }$ as $L \to \infty$,
by picking $L$ large enough we ensure that $\E\tilde B < (1+\epsilon)\bar b^{[r]}_{ j \leftarrow l  }$.
\end{proof}

The next result is in the same spirit of Lemma~\ref{lemma: concentration ineq for pruned cluster S i}, but focuses on $\bm S^{[r],\leqslant}_i(n\delta)$ in \eqref{proof, def, law of r step sub tree size}.

\begin{lemma} \label{lemma: concentration ineq for r step sub tree}
\linksinthm{lemma: concentration ineq for r step sub tree}
    Let Assumptions~\ref{assumption: heavy tails in B i j}--\ref{assumption: regularity condition 2, cluster size, July 2024} hold,
    and suppose that $\norm{\bar{\textbf{B}}^r} < 1$ holds for some positive integer $r$.    
    Given $\gamma > 0$, there exist $\delta_0 > 0$ and $C_0 > 0$ such that
\begin{align}
    \label{claim, lemma: concentration ineq for r step sub tree}
    \lim_{n \to \infty} 
        n^\gamma \cdot 
        \P\Bigg(
        \norm{
            \frac{1}{n}\sum_{m = 1}^n \bm S^{[r],\leqslant; (m)}_i(n\delta)
        }
        > C_0
        \Bigg) = 0,
        \qquad \forall \delta \in (0,\delta_0),\ i \in [d],
\end{align}
where the $\bm S^{[r],\leqslant; (m)}_i(M)$'s are independent copies of $\bm S^{[r],\leqslant}_i(M)$ defined in \eqref{proof, def, law of r step sub tree size}.
\end{lemma}

\begin{proof}\linksinpf{lemma: concentration ineq for r step sub tree}
It suffices to fix some $i \in [d]$ and show the existence of $C_0> 0$, $\delta_0 > 0$
such that \eqref{claim, lemma: concentration ineq for r step sub tree} holds.
The proof is almost identical to that of Claim~\eqref{proof: goal upper bound, lemma: concentration ineq for pruned cluster S i} in the proof of Lemma~\ref{lemma: concentration ineq for pruned cluster S i}.
Specifically, we set
$
C_0 = d + \sum_{j \in [d]}\bar s_{i,j}, 
$
let $\phi_c(x) = x \wedge c$,
and observe that 
\begin{align*}
    & \Bigg\{
        \norm{
            \frac{1}{n}\sum_{m = 1}^n \bm S^{[r],\leqslant; (m)}_i(n\delta)
        }
        > C_0
    \Bigg\}
    \\ 
    & \subseteq 
    \underbrace{ \bigg\{
        \norm{ \bm S^{[r],\leqslant; (m)}_i(n\delta)  } > n\Delta
        \text{ for some }m \in [n]
    \bigg\} }_{ \delequal \text{(I)} }
    \cup 
    \Bigg(
        \bigcup_{ j \in [d] }
        \underbrace{ \Bigg\{
            \frac{1}{n}\sum_{m = 1}^n \phi_{n\Delta}\Big(  S^{[r],\leqslant;(m)}_{i,j}(n\delta)  \Big) > \bar s_{i,j} + 1
        \Bigg\}
        }_{ \delequal \text{(II:$j$)} }
    \Bigg).
\end{align*}
Therefore, it suffices to find some $\Delta > 0$
such that, under any $\delta > 0$ small enough,
the terms
$
\P\big(\text{(I)}\big)
$
and (for each $j \in [d]$)
$
\P\big(\text{(II:$j$)}\big)
$
are of order $\lo(n^{-\gamma})$.

\smallskip
\noindent
\textbf{Proof of $\P\big(\text{(I)}\big) = \lo(n^{-\gamma})$}.
Applying Lemma~\ref{lemma, tail bound, pruned cluster sub sampled size S i r leq n delta},
we know that given any $\Delta > 0$, this claim holds for all $\delta > 0$ sufficiently small.

\smallskip
\noindent
\textbf{Proof of $\P\big(\text{(II:$j$)}\big) = \lo(n^{-\gamma})$ and the choice of $\Delta$}.
By definitions in \eqref{proof, def, law of r step sub tree size} and the stochastic comparison in \eqref{property: stochastic comparison, S and pruned S},
we have 
$
S^{[r],\leqslant}_{i,j}(n\delta) \stleq S^{\leqslant}_{i,j}(n\delta) \stleq S_{i,j}.
$
Using $S^{(m)}_{i,j}$
to denote independent copies of $S_{i,j}$,
it suffices to find some $\Delta > 0$ such that 
\begin{align}
    \P\Bigg(
        \frac{1}{n}\sum_{m = 1}^n \phi_{n\Delta}\Big(  S^{(m)}_{i,j} \Big) > \bar s_{i,j} + 1
    \Bigg)
    = \lo(n^{-\gamma}),
    \qquad
    \forall j \in [d].
    \label{proof, goal 2, lemma: concentration ineq for r step sub tree}
\end{align}
Again, by Theorem~2 of \cite{Asmussen_Foss_2018}, we get
$\P(S_{i,j} > x) \in \RV_{-\alpha^*}(x)$ for some $\alpha^* > 1$.
By Claim~\eqref{claim 2, lemma: concentration ineq, truncated heavy tailed RV in Rd} of Lemma~\ref{lemma: concentration ineq, truncated heavy tailed RV in Rd} and property \eqref{property, appendix, psi to phi, truncation mapping}, we conclude that \eqref{proof, goal 2, lemma: concentration ineq for r step sub tree} holds for any $\Delta > 0$ small enough.
\end{proof}

Now, we are ready to prove Lemma~\ref{lemma: tail bound, pruned cluster size S i leq n delta} for  the general case.

\begin{proof}[Proof of Lemma~\ref{lemma: tail bound, pruned cluster size S i leq n delta} (General Case)]
Under Assumption~\ref{assumption: subcriticality},
we are able to apply Gelfand's formula (see, e.g., p.\ 195 of \cite{lax2014functional}) and identify some positive integer $r$ such that 
 $\norm{\bar{\textbf{B}}^r} < 1$.
Also, in this proof we adopt the same labeling rule considered in Section~\ref{subsubsec, proof, concentration inequalities for S leq M} for multi-type branching trees:
that is, 
given $j \in [d]$, all type-$j$ nodes are numbered left to right, starting from generation 0, then continuing similarly in each subsequent generation.
To proceed, we make a few observations regarding the branching tree for $\bm S^\leqslant_i(n\delta)$.
\begin{enumerate}[(i)]
    \item 
        Given a positive integer $k$ and $q = 0,1,\ldots,r-1$,
        any node in the $(kr + q)^\text{th}$ generation uniquely belongs to the sub-tree rooted at one of the nodes at the $q^\text{th}$ generation.
        This is equivalent to saying that each node in the $(kr + q)^\text{th}$ has exactly one (grand)parent in the 
        $q^\text{th}$ generation.
        As a convention, we also say that any node belongs to the sub-tree rooted at itself.
        
    \item 
       Let $S_{j,l}^{ [r;t,m],\leqslant }(n\delta)$ be the count of type-$l$ nodes at generation $t, r+t, 2r+t, 3r+t,\ldots$ that belong to the sub-tree rooted at the $m^\text{th}$ type-$j$ node in the $t^\text{th}$ generation.
       Let
       $
       \bm S_{j}^{ [r;t,m],\leqslant }(n\delta) = \big( S_{j,l}^{ [r;t,m],\leqslant }(n\delta) \big)_{l \in [d]}.
       $
       A direct consequence of the previous bullet point is that
       \begin{align}
           \bm S^\leqslant_i(n\delta) = \bm S_i^{[r;0,1],\leqslant}(n\delta) +  \sum_{t = 1}^{r-1}\sum_{j \in [d]}\sum_{m = 1}^{ X^\leqslant_{i,j}(t-1;n\delta)   }\bm S^{[r;t,m],\leqslant}_j(n\delta).
           \label{proof, decomp of S leq i using r step sub trees, lemma: tail bound, pruned cluster size S i leq n delta}
       \end{align}
       Also, by definitions in \eqref{proof, def, law of r step sub tree size},
       we have $\bm S_i^{[r;0,1],\leqslant}(n\delta) = \bm S_i^{[r],\leqslant}(n\delta)$.

    \item 
        The next fact follows from the independence of the offspring counts across different nodes:
        for each $t = 1,2,\ldots,r-1$ and $j \in [d]$,
        the sequence $\big(\bm S_{j}^{ [r;t,m],\leqslant }(n\delta)\big)_{m \leq X^{\leqslant}_{i,j}(t-1;n\delta) }$
        are independent copies of $\bm S^{[r],\leqslant}_j(n\delta)$ defined in \eqref{proof, def, law of r step sub tree size}.
        Henceforth in this proof, for each $m > X^{\leqslant}_{i,j}(t-1;n\delta) $ we  independently generate $\bm S_{j}^{ [r;t,m],\leqslant }(n\delta)$ as a generic copy of $\bm S^{[r],\leqslant}_j(n\delta)$,
        so that the infinite sequence $\big(\bm S_{j}^{ [r;t,m],\leqslant }(n\delta)\big)_{m \geq 1}$ is well-defined
        for each $j \in [d]$ and $t = 1,2,\ldots,r-1$.

\end{enumerate}
Now,
take any $\Delta^\prime, C_0 > 0$.
        On the event
        \begin{align}
            & \underbrace{
                \bigg\{ \norm{ \bm X^\leqslant_i(t;n\delta) } \leq n\Delta^\prime\ \forall t = 1,2,\ldots,r-1 \bigg\}
            }_{ \delequal \text{(I)} }
            \nonumber
            \\ 
            &
            \cap 
            \underbrace{ \bigg\{
                \norm{
                    \bm S_i^{[r;0,1],\leqslant}(n\delta)
                } \leq n\Delta^\prime
            \bigg\}
            }_{ \delequal \text{(II)}  }
            \cap \Bigg( 
            \bigcap_{ t \in [r-1],\ j \in [d]  }
            \underbrace{ \Bigg\{
                \norm{
                    \frac{1}{ \floor{n\Delta^\prime} }\sum_{m = 1}^{ \floor{n\Delta^\prime}  }\bm S^{[r;t,m],\leqslant}_j(n\delta)
                } \leq C_0
            \Bigg\} }_{ \delequal (\text{III}:t,j)  }
            \Bigg),
            \nonumber
        \end{align}
        it follows from \eqref{proof, decomp of S leq i using r step sub trees, lemma: tail bound, pruned cluster size S i leq n delta} that 
        \begin{align}
             \norm{\bm S_{j}^{\leqslant }(n\delta)} \leq 
             n\Delta^\prime + (r-1) \cdot d \cdot n\Delta^\prime \cdot C_0
             =
             n\Delta^\prime \cdot \big[
                1 + (r-1) d \cdot C_0
            \big].
            \label{proof, ineq for norm of S leq j n delta after decomp, lemma: tail bound, pruned cluster size S i leq n delta, general cases}
        \end{align}
Therefore,
to prove Claim \eqref{claim, lemma: tail bound, pruned cluster size S i leq n delta} given $i \in [d]$ and $\Delta > 0$,
it suffices to find $C_0, \Delta^\prime > 0$ such that 
\begin{itemize}
    \item 
        $
\Delta^\prime \cdot 
\big[
                1 + (r-1) d \cdot C_0
            \big]
            < \Delta 
$
(so the RHS of \eqref{proof, ineq for norm of S leq j n delta after decomp, lemma: tail bound, pruned cluster size S i leq n delta, general cases} is upper bounded by $n\Delta$);

    \item 
        for any $\delta > 0$ small enough,
the terms
$
\P\big(  (\text{(I)})^\complement  \big),
$
$
\P\big(  (\text{(II)})^\complement  \big),
$
and (for each $t \in [r-1]$, $j \in [d]$)
$
\P\big(  (\text{(III:$t,j$)})^\complement  \big)
$
are of order $\lo(n^{-\gamma})$.
\end{itemize}

\smallskip
\noindent
\textbf{Proof of 
$
\P\big(  (\text{(III:$t,j$)})^\complement  \big) = \lo(n^{-\gamma})
$
and the choice of $C_0$, $\Delta^\prime$}.
Let $C_0$ be characterized as in Lemma~\ref{lemma: concentration ineq for r step sub tree},
based on which fix some $\Delta^\prime > 0$ small enough such that 
$
\Delta^\prime \cdot 
\big[
                1 + (r-1) d \cdot C_0
            \big]
            < \Delta.
$
By Lemma~\ref{lemma: concentration ineq for r step sub tree}
and the observation (iii) above,
we have
$
\P\big(  (\text{(III:$t,j$)})^\complement  \big) = \lo(n^{-\gamma})
$
under any $\delta > 0$ small enough.

\smallskip
\noindent
\textbf{Proof of
$
\P\big(  (\text{(I)})^\complement  \big) = \lo(n^{-\gamma})$}. This follows from Lemma~\ref{lemma: tail asymptotics for truncated GW, at generation t}.

\smallskip
\noindent
\textbf{Proof of
$
\P\big(  (\text{(II)})^\complement  \big) = \lo(n^{-\gamma})$}.
This follows from Lemma~\ref{lemma, tail bound, pruned cluster sub sampled size S i r leq n delta}.
\end{proof}

\subsection{Proofs of Lemmas~\ref{lemma: type and generalized type}--\ref{lemma: limit theorem for n tau to hat C I, cluster size}}

Next, we provide the proofs of Lemmas~\ref{lemma: type and generalized type} and \ref{lemma: choice of bar epsilon bar delta, cluster size}.

\begin{proof}[Proof of Lemma~\ref{lemma: type and generalized type}]
\linksinpf{lemma: type and generalized type}
Part (i) is an immediate consequence of $\mathscr I \subseteq \widetilde{\mathscr I}$ and property \eqref{property: cost of j I and type I, 2, cluster size}.
%
Next, we prove part (ii):
that is, given $\bm I \in \widetilde{\mathscr I}(\bm j) \setminus \mathscr I(\bm j)$,
we must have $|\bm j^{\bm I}_1| \geq 2$.
By \eqref{property: cost of j I and type I, 2, cluster size},
$
\mathscr I(\bm j) = 
    \{\bm I \in \mathscr I:\ \bm j^{\bm I} = \bm j  \}
    =
    \{\bm I \in \mathscr I:\ \bm j^{\bm I} = \bm j,\ \tilde{\alpha}(\bm I) = \alpha(\bm j)  \}.
$
Then, due to $\mathscr I \subseteq \widetilde{\mathscr I}$,
for any $\bm I = (I_{k,j})_{k \geq 1, j \in [d]} \in \widetilde{\mathscr I}(\bm j) \setminus \mathscr I(\bm j)$ we must have $\bm I \in \widetilde{\mathscr I}\setminus \mathscr I$.
Due to $\bm j \neq \emptyset$,
by comparing Definition~\ref{def: cluster, type} with Definition~\ref{def: cluster, generalized type},
at least one of the following two cases must occur:
\begin{enumerate}[(a)]
    \item the set $\{j \in [d]:\ I_{1,j} = 1\}$ contains at least two elements;

    \item there exists $j \in \bm j$ such that $|\{k \geq 1:\ I_{k,j} = 1\}| \geq 2$.
\end{enumerate}
To prove part (ii), it suffices to show that case (b) cannot occur for any $\bm I \in \widetilde{\mathscr I}(\bm j)$.
Specifically,
suppose that $|\{ k \geq 1:\ I_{k,j^*} = 1 \}| \geq 2$ for some $j^* \in \bm j$.
Then, by \eqref{def: tilde c I, cost for generalized type, cluster size},
\begin{align*}
    \tilde{\alpha}(\bm I)
    & =
    1 + \sum_{ j \in [d] }\sum_{ k \geq 1 }I_{k,j}\cdot \big(\alpha^*(j)-1\big)
    \\ 
    & 
    =
    1 + \sum_{j \in \bm j}\big(\alpha^*(j)-1\big) \cdot \big| \{ k \geq 1:\ I_{k,j} = 1\} \big|
    \qquad
    \text{due to }\bm j^{\bm I} = \bm j
    \\ 
    & >
    1 + \sum_{j \in \bm j}\big(\alpha^*(j) - 1\big) 
    \qquad
    \text{due to $|\{ k \geq 1:\ I_{k,j^*} = 1 \}| \geq 2$ and $\alpha^*(j) > 1\ \forall j \in [d]$}
    \\
    & = \alpha(\bm j)
    \qquad
    \text{by definitions in \eqref{def: cost function, cone, cluster}}.
\end{align*}
However, this leads to the contradiction $\bm I \notin \widetilde{\mathscr I}(\bm j)$.
In summary, case (a) must occur for any $\bm I \in \widetilde{\mathscr I}(\bm j) \setminus \mathscr I(\bm j)$, which verifies part (ii) of this lemma.
\end{proof}

\begin{proof}[Proof of Lemma~\ref{lemma: choice of bar epsilon bar delta, cluster size}]
\linksinpf{lemma: choice of bar epsilon bar delta, cluster size}
(a) The claims are equivalent to the following: there exist $\bar\epsilon > 0$ and $\bar\delta > 0$ such that for any 
$\bm x = \sum_{i \in \bm j}w_i\bar{\bm s}_i$ with $w_i \geq 0\ \forall i \in \bm j$ and $\Phi(\bm x) \in B$, we must have
\begin{align}
    \sum_{j \in \bm j}w_j & > \bar\epsilon,
    \label{proof: goal 1, part a, lemma: choice of bar epsilon bar delta, cluster size}
    \\ 
    \min_{j \in \bm j}\frac{w_j}{  \sum_{ i \in \bm j }w_i } & > \bar\delta.
    \label{proof: goal 2, part a, lemma: choice of bar epsilon bar delta, cluster size}
\end{align}

First, since $B$ is bounded away from $\C^d_\leqslant(\bm j)$ under $\bm d_\textbf{U}$,
there exists $r_0 > 0$ such that 
\begin{align}
    (r,\theta) \in B
    \qquad \Longrightarrow \qquad
    r > r_0.
    \label{proof: constant r 0, part a, lemma: choice of bar epsilon bar delta, cluster size}
\end{align}
Next, consider some $\bm x = \sum_{i \in \bm j}w_i\bar{\bm s}_i$ with $w_i \geq 0\ \forall i \in \bm j$ and $\Phi(\bm x) \in B$.
By \eqref{proof: constant r 0, part a, lemma: choice of bar epsilon bar delta, cluster size}, we must have $\norm{\bm x} > r_0$.
On the other hand, for the $L_1$ norm $\norm{\bm x}$, we have 
$
\norm{\bm x}
    \leq 
    \max_{i \in \bm j}\norm{ \bar{\bm s}_i } \cdot \sum_{j \in \bm j} w_j,
$
and hence 
$
    \frac{r_0}{  \max_{ j \in \bm j }\norm{\bar{\bm s}_j}  }  < 
    \sum_{j \in \bm j} w_j.
$
In summary, Claim~\eqref{proof: goal 1, part a, lemma: choice of bar epsilon bar delta, cluster size} holds for any $\bar\epsilon > 0$ small enough such that $\bar\epsilon < \frac{r_0}{  \max_{ j \in \bm j }\norm{\bar{\bm s}_j}  }$.

Next, since $B$ is bounded away from $\C^d_\leqslant(\bm j)$ under $\bm d_\textbf{U}$, there exists some $\Delta > 0$ such that 
\begin{align}
    \bm d_\textbf{U}\big(B, \mathbb C^d_\leqslant(\bm j)\big) > \Delta.
    \label{proof: constant Delta, part a, lemma: choice of bar epsilon bar delta, cluster size}
\end{align}
We show that Claim~\eqref{proof: goal 2, part a, lemma: choice of bar epsilon bar delta, cluster size} holds for any $\bar\delta > 0$ small enough that satisfies
\begin{align}
    \bar\delta \cdot 
    \frac{
        \max_{j \in \bm j}\norm{ \bar{\bm s}_j  }
    }{
        \min_{j \in \bm j}\norm{ \bar{\bm s}_j  }
    }
    < \frac{\Delta}{2}, \qquad
    \bar\delta < 1.
    \label{proof: choice of bar delta, part a, lemma: choice of bar epsilon bar delta, cluster size}
\end{align}
To proceed, 
we consider a proof by contradiction. 
Suppose that for some $\bm x = \sum_{i \in \bm j}w_i\bar{\bm s}_i$ with $w_i \geq 0\ \forall i \in \bm j$ and $\Phi(\bm x) \in B$,
there exists $j^* \in \bm j$ such that Claim~\eqref{proof: goal 2, part a, lemma: choice of bar epsilon bar delta, cluster size} does not hold, i.e.,
\begin{align}
    \frac{w_{j^{*}} }{  \sum_{ i \in \bm j }w_i } \leq \bar\delta.
    \label{proof: assumption for proof by contradiction, part a, lemma: choice of bar epsilon bar delta, cluster size}
\end{align}
We first note  that the set $\bm j \setminus \{j^*\}$ cannot be empty;
otherwise, we have $\bm j = \{j^*\}$ and arrive at the contradiction that $\frac{w_{j^{*}} }{  \sum_{ i \in \bm j }w_i } = 1 > \bar\delta$.
Next, we define
$\bm x^* \delequal \sum_{ j \in \bm j\setminus \{ j^* \}  }w_j \bar{\bm s}_j$,
and note that $\norm{ \bm x^* } > 0$.
Let $(r,\theta) = \Phi(\bm x)$ and $(r^*,\theta^*) = \Phi(\bm x^*)$, and observe that 
\begin{align*}
    \norm{ \theta - \theta^* }
    & = 
    \norm{
        \frac{
            \bm x
        }{
            \norm{ \bm x }
        }
        - \frac{\bm x^*}{ \norm{\bm x^*} }
    }
    \leq 
    \norm{
        \frac{
            \bm x
        }{
            \norm{ \bm x }
        }
        - \frac{\bm x^*}{ \norm{\bm x} }
    }
    +
    \norm{
        \frac{
            \bm x^*
        }{
            \norm{ \bm x }
        }
        - \frac{\bm x^*}{ \norm{\bm x^*} }
    }
   \leq 2 \cdot
    \frac{
        \norm{ \bm x - \bm x^* }
    }{
        \norm{\bm x}
    }
    \\ 
    & = 
    2 \cdot 
    \frac{
        w_{ j^* }\norm{ \bar{ \bm{s} }_{ j^* } }
    }{
        \sum_{j \in \bm j}w_j \norm{ \bar{\bm s}_j  }
    }
    \leq 
    2\cdot 
    \frac{
        w_{j^*}
    }{
        \sum_{j \in \bm j}w_j
    }
    \cdot 
    \frac{
        \max_{j \in \bm j}\norm{ \bar{\bm s}_j  }
    }{
        \min_{j \in \bm j}\norm{ \bar{\bm s}_j  }
    }
   \leq 2 \cdot \frac{\Delta}{2} = \Delta.
\end{align*}
The last inequality in the display above follows from our choice of $\bar\delta$ in \eqref{proof: choice of bar delta, part a, lemma: choice of bar epsilon bar delta, cluster size}
and the condition \eqref{proof: assumption for proof by contradiction, part a, lemma: choice of bar epsilon bar delta, cluster size} for the proof by contraction.
Now, consider $\bm x^* \cdot \frac{ \norm{\bm x}}{ \norm{\bm x^*} }$, i.e., a stretched version of the vector $\bm x^*$ with $L_1$ norm matching $\norm{\bm x}$.
Due to $\bm x^* = \sum_{ j \in \bm j\setminus \{ j^* \}  }w_j \bar{\bm s}_j$, we have $\bm x^* \in \R^d_\leqslant(\bm j)$ 
and $\Phi(\bm x^*) \in \C^d_\leqslant(\bm j)$,
thus implying
$
\Phi\big(\bm x^* \cdot \frac{ \norm{\bm x}}{ \norm{\bm x^*} }\big) \in \C^d_\leqslant(\bm j);
$ see  \eqref{def: cone C d leq j, cluster size}.
However, due to 
$
\Phi(\bm x^* \cdot \frac{ \norm{\bm x}}{ \norm{\bm x^*} }) = (r, \theta^*),
$
we arrive at
$
\bm d_\textbf{U}\big( \Phi(\bm x), \Phi(\bm x^* \cdot \frac{ \norm{\bm x}}{ \norm{\bm x^*} }) \big)
=
\norm{ \theta - \theta^* } \leq \Delta,
$
which contradicts \eqref{proof: constant Delta, part a, lemma: choice of bar epsilon bar delta, cluster size} since $\Phi(\bm x) \in B$.3
This concludes the proof of Claim~\eqref{proof: goal 2, part a, lemma: choice of bar epsilon bar delta, cluster size}.

\medskip
\noindent
(b) We fix some type $\bm I \in \mathscr I$ with active index set $\bm j^{\bm I} = \bm j$.
Due to $\bm j \neq \emptyset$, we have $\mathcal K^{\bm I} \geq 1$; see Definition~\ref{def: cluster, type} and Remark~\ref{remark: def of type}.
Henceforth in this proof, 
we write $\bm w_k = (w_{k,j})_{j \in \bm j^{\bm I}_k}$ and $\bm w = (\bm w_k)_{ k \in [\mathcal K^{\bm I}] }$.
Using results in part (a), one can fix some constants $\bar\epsilon > 0$ and $\bar\delta \in (0,1)$ such that the following holds:
for any $\bm x = \sum_{ k \in [\mathcal K^{\bm I}] }\sum_{j \in \bm j^{\bm I}_k}w_{k,j}\bar{\bm s}_j$ with $w_{k,j} \geq 0$ and $\bm x \in \Phi^{-1}(B)$,
we must have $\bm w \in B(\bar\epsilon,\bar\delta)$ where
\begin{align*}
    B(\bar\epsilon,\bar\delta) \delequal
    \left\{\rule{0cm}{0.9cm}
        \bm w \in [0,\infty)^{ |\bm j| }:\ 
        \min_{ k \in [\mathcal K^{\bm I}],\ j \in \bm j^{\bm I}_k }w_{k,j} \geq \bar\epsilon;\ 
        \min_{ \substack{  
            k \in [\mathcal K^{\bm I}],\ j \in \bm j^{\bm I}_k
            \\ 
            k^\prime \in [\mathcal K^{\bm I}],\ j^\prime \in \bm j^{\bm I}_{k^\prime} 
            } }\frac{w_{k,j}}{w_{k^\prime,j^\prime}} \geq \bar\delta
    \right\}.
\end{align*}
Then, by the definition of $\mathbf C^{\bm I}$ in \eqref{def: measure C I, cluster},
\begin{align*}
    \mathbf C^{\bm I}\circ \Phi^{-1}(B)
    & = 
    \mathbf C^{\bm I}\big(\Phi^{-1}(B)\big)
    \leq 
    \int \mathbbm I\Big\{
        \bm w \in B(\bar\epsilon,\bar\delta)
    \Big\}
    \cdot \Bigg(
    \prod_{ k = 1 }^{ \mathcal K^{\bm I} - 1}
        g_{ \bm j^{\bm I}_k \leftarrow \bm j^{\bm I}_{k+1} }(\bm w_k)
    \Bigg)
    \nu^{\bm I}(d \bm w).
\end{align*}
Note also that there uniquely exists some $j^{\bm I}_1 \in [d]$ such that $\bm j^{\bm I}_1 = \{j^{\bm I}_1\}$;
see Remark~\ref{remark: def of type}.
We fix some $\rho \in (1,\infty)$, and let
\begin{align}
    B_0(\rho,\bar\delta)
    & \delequal
    \left\{\rule{0cm}{0.9cm}
        \bm w \in [0,\infty)^{|\bm j|}:\ 
        w_{1, j^{\bm I}_1} \in [1,\rho),\ 
        \min_{ \substack{  
            k \in [\mathcal K^{\bm I}],\ j \in \bm j^{\bm I}_k
            \\ 
            k^\prime \in [\mathcal K^{\bm I}],\ j^\prime \in \bm j^{\bm I}_{k^\prime} 
            } }\frac{w_{k,j}}{w_{k^\prime,j^\prime}} \geq \bar\delta
    \right\},
    \nonumber
    \\
    B_n(\rho,\delta)
     \delequal \rho^n B_0(\rho,\bar\delta)
    & =
    \left\{\rule{0cm}{0.9cm}
        \bm w \in [0,\infty)^{|\bm j|}:\ 
        w_{1, j^{\bm I}_1} \in [\rho^n, \rho^{n+1}),\ 
        \min_{ \substack{  
            k \in [\mathcal K^{\bm I}],\ j \in \bm j^{\bm I}_k
            \\ 
            k^\prime \in [\mathcal K^{\bm I}],\ j^\prime \in \bm j^{\bm I}_{k^\prime} 
            } }\frac{w_{k,j}}{w_{k^\prime,j^\prime}} \geq \bar\delta
    \right\},
    \quad n \in \mathbb Z.
    \nonumber
\end{align}
For any $N$ large enough we have  $\rho^{-N} < \bar\epsilon$. 
This leads to $B(\bar\epsilon,\bar\delta) \subseteq \bigcup_{ n = -N }^\infty B_n(\rho,\delta)$,
and
\begin{align*}
    \mathbf C^{\bm I} \circ \Phi^{-1}(B) 
    & \leq 
    \sum_{n = -N }^{\infty}
    \underbrace{ \int \mathbbm I\Big\{
        \bm w \in B_n(\rho,\bar\delta)
    \Big\}
    \cdot \Bigg(
    \prod_{ k = 1 }^{ \mathcal K^{\bm I} - 1}
        g_{ \bm j^{\bm I}_k \leftarrow \bm j^{\bm I}_{k+1} }(\bm w_k)
    \Bigg)
    \nu^{\bm I}(d \bm w) 
    }_{ \delequal c_n  }.
\end{align*}
Therefore, it suffices to show that
\begin{align}
    c_0 < \infty,
    \label{proof: goal 1, part b, lemma: choice of bar epsilon bar delta, cluster size}
\end{align}
and that there exists some $\hat \rho \in (1,\infty)$ such that 
\begin{align}
    c_n \leq \hat \rho^{-n} \cdot  c_0,\qquad\forall n \in \mathbb Z.
    \label{proof: goal 2, part b, lemma: choice of bar epsilon bar delta, cluster size}
\end{align}

\smallskip
\noindent
\textbf{Proof of Claim~\eqref{proof: goal 1, part b, lemma: choice of bar epsilon bar delta, cluster size}}.
Note that 
$
B_0(\rho,\bar\delta)
    \subseteq
    \big\{
        \bm w \in [0,\infty)^{|\bm j|}:
        w_{k,j} \in [\bar\delta, \rho/\bar\delta]\ 
        \forall k \in [\mathcal K^{\bm I}], j \in \bm j^{\bm I}_k
    \big\}.
$
Also, by the continuity of $g_{\mathcal I \leftarrow \mathcal J}(\cdot)$ (see \eqref{def: function g mathcal I mathcal J, for measure C i bm I, cluster size}),
we can fix some $M \in (0,\infty)$ such that
$
0 \leq g_{ \bm j^{\bm I}_k \leftarrow \bm j^{\bm I}_{k+1} }(\bm w_k) \leq M
$
for any $k \in [\mathcal K^{\bm I}]$ and $\bm w_k = (w_{k,j})_{ j \in \bm j^{\bm I}_k }$ with $w_{k,j} \in [\bar\delta,\rho/\bar\delta]$ for all $j$.
Therefore,
\begin{align*}
    c_0 & \leq 
    M^{ | \mathcal K^{\bm I} | - 1}
    \int 
    \mathbbm I\Big\{
        \bm w \in [0,\infty)^{|\bm j|}:
        w_{k,j} \in [\bar\delta, \rho/\bar\delta]\ 
        \forall k \in [\mathcal K^{\bm I}], j \in \bm j^{\bm I}_k
    \Big\}
    \nu^{\bm I}(d\bm w)
    \\
    &
 = 
    M^{ | \mathcal K^{\bm I} | - 1}
    \prod_{ k \in [ \mathcal K^{\bm I} ] }
    \prod_{ j \in \bm j^{\bm I}_k  }
    \Big(
        { \bar\delta^{ -\alpha^*(j)  }  } - { (\rho/\bar\delta)^{ -\alpha^*(j)  }  }
    \Big) < \infty.
\end{align*}

\smallskip
\noindent
\textbf{Proof of Claim~\eqref{proof: goal 2, part b, lemma: choice of bar epsilon bar delta, cluster size}}.
By \eqref{def: function g mathcal I mathcal J, for measure C i bm I, cluster size},
for any $a > 0$ we have
$
g_{ \mathcal I \leftarrow \mathcal J}(a \bm w)
    =
    g_{ \mathcal I \leftarrow \mathcal J}(\bm w) \cdot a^{|\mathcal J|}.
$
Therefore,
\begin{align*}
    c_n & = 
    \int_{ \bm w \in B_n(\rho,\bar\delta)  }
    \Bigg(
    \prod_{ k = 1 }^{ \mathcal K^{\bm I} - 1}
        g_{ \bm j^{\bm I}_k \leftarrow \bm j^{\bm I}_{k+1} }(\bm w_k)
    \Bigg)
    \nu^{\bm I}(d \bm w) 
    =
    \int_{ \bm w \in \rho^n B_0(\rho,\bar\delta)  }
    \Bigg(
    \prod_{ k = 1 }^{ \mathcal K^{\bm I} - 1}
        g_{ \bm j^{\bm I}_k \leftarrow \bm j^{\bm I}_{k+1} }(\bm w_k)
    \Bigg)
    \nu^{\bm I}(d \bm w) 
    \\ 
    & = 
    \int_{ \bm x \in B_0(\rho,\bar\delta)  }
    \Bigg(
    \prod_{ k = 1 }^{ \mathcal K^{\bm I} - 1}
        g_{ \bm j^{\bm I}_k \leftarrow \bm j^{\bm I}_{k+1} }( \rho^n  \bm x_k)
    \Bigg)
    \nu^{\bm I}(d \rho^n\bm x)
    \qquad
    \text{by setting }\bm w = \rho^n \bm x
    \\ 
    & = 
    \int_{ \bm x \in B_0(\rho,\bar\delta)  }
    \Bigg(
    \prod_{ k = 1 }^{ \mathcal K^{\bm I} - 1}
    \rho^{ n| \bm j^{\bm I}_{k+1} | }\cdot 
        g_{ \bm j^{\bm I}_k \leftarrow \bm j^{\bm I}_{k+1} }( \bm x_k)
    \Bigg)
    \nu^{\bm I}(d \rho^n\bm x)
    \\ 
    & = 
    \int_{ \bm x \in B_0(\rho,\bar\delta)  }
    \Bigg[
    \bigtimes_{ k = 1 }^{ \mathcal K^{\bm I} - 1  }
    \Bigg(
        \rho^{ n| \bm j^{\bm I}_{k+1} | }\cdot 
        g_{ \bm j^{\bm I}_k \leftarrow \bm j^{\bm I}_{k+1} }( \bm x_k)
        \bigtimes_{  j \in \bm j^{\bm I}_k }
        \frac{
            \alpha^*(j) \rho^n dx_{k,j}
        }{
            (\rho^n x_{k,j})^{ \alpha^*(j) + 1 }
        }
    \Bigg)
    \Bigg] 
    \cdot 
    \Bigg(
        \bigtimes_{ j \in \bm j^{\bm I}_{ \mathcal K^{\bm I} } }
        \frac{
            \alpha^*(j) \rho^n dx_{\mathcal K^{\bm I} ,j}
        }{
            (\rho^n x_{\mathcal K^{\bm I} ,j})^{ \alpha^*(j) + 1 }
        }
    \Bigg)
    \\
    &\qquad\qquad\qquad\qquad\qquad\qquad\qquad\qquad\qquad
    \text{by definitions in \eqref{def, power law measure nu beta} and \eqref{def, measure nu type I, cluster size}}
    \\ 
    & = 
    \Bigg(
        \prod_{ j \in \bm j^{\bm I}_{ 1 } }\rho^{ -n\alpha^*(j)  } 
    \Bigg)
    \cdot
    \Bigg(
        \prod_{  k = 2}^{\mathcal K^{\bm I} }\prod_{ j \in \bm j^{\bm I}_k } \rho^{ -n(\alpha^*(j) - 1)  }
    \Bigg)
    \\
    &\qquad\qquad
    \cdot
    \underbrace{ \int_{ \bm x \in B_0(\rho,\bar\delta)  }
    \Bigg(
    \bigtimes_{ k = 1 }^{ \mathcal K^{\bm I} - 1  }
    g_{ \bm j^{\bm I}_k \leftarrow \bm j^{\bm I}_{k+1} }( \bm x_k)
        \bigtimes_{  j \in \bm j^{\bm I}_k }
        \frac{
            \alpha^*(j) dx_{k,j}
        }{
            x_{k,j}^{ \alpha^*(j) + 1 }
        }
    \Bigg)
    \cdot 
    \Bigg(
        \bigtimes_{ j \in \bm j^{\bm I}_{ \mathcal K^{\bm I} } }
        \frac{
            \alpha^*(j) dx_{\mathcal K^{\bm I} ,j}
        }{
            x_{\mathcal K^{\bm I} ,j}^{ \alpha^*(j) + 1 }
        }
    \Bigg) }_{ = c_0}.
\end{align*}
Therefore, to conclude the proof of Claim~\eqref{proof: goal 2, part b, lemma: choice of bar epsilon bar delta, cluster size}, one only needs to pick
\begin{align*}
    \hat \rho \delequal
    \Bigg(
        \prod_{ j \in \bm j^{\bm I}_{ 1 } }\rho^{ \alpha^*(j)  } 
    \Bigg)
    \cdot
    \Bigg(
        \prod_{  k = 2 }^{\mathcal K^{\bm I} }\prod_{ j \in \bm j^{\bm I}_k } \rho^{ (\alpha^*(j) - 1)  }
    \Bigg).
\end{align*}
In particular, recall that $\alpha^*(j) > 1\ \forall j \in [d]$; see Assumption~\ref{assumption: heavy tails in B i j} and \eqref{def: cluster size, alpha * l * j}.
Then, by our choice of $\rho \in (1,\infty)$, we get  $\hat \rho \in (1,\infty)$.
\end{proof}

For the proofs of 
 Lemmas~\ref{lemma: prob of type I, tail prob of tau, cluster size} and \ref{lemma: limit theorem for n tau to hat C I, cluster size}, we prepare one more result.
Recall the definitions of $N^{>|\delta}_{\bm t(\mathcal I);j}$ and $W^{>|\delta}_{\bm t(\mathcal I);j}$   in \eqref{def, proof cluster size, N W mathcal I mathcal J bcdot j},
which are sums of i.i.d.\ copies of $N^{>}_{i;j}(\cdot)$ and $W^{>}_{i;j}(\cdot)$ defined in \eqref{def: W i M j, pruned cluster, 1, cluster size}--\eqref{def: cluster size, N i | M cdot j}.
Besides, recall the definition of the assignments of $\mathcal J$ to $\mathcal I$ in \eqref{def: assignment from mathcal J to mathcal I}, and that we use ${\mathbb T_{ \mathcal I \leftarrow \mathcal J }}$ to denote the set of all assignments of $\mathcal{J}$ to $\mathcal I$.
For any non-empty $\mathcal I\subseteq [d]$, $\mathcal J \subseteq [d]$, we define
\begin{align}
    \notationdef{notation-C-J-assigned-to-I}{
    C_{ \mathcal I \leftarrow \mathcal J }\big( (t_i)_{ i \in \mathcal I } \big)}
    \delequal 
    \sum_{ 
        \{ \mathcal J(i):\ i \in \mathcal I \} \in \mathbb T_{ \mathcal I \leftarrow \mathcal J }
    }
    \prod_{i \in \mathcal I}\prod_{ j \in \mathcal J(i) }
        \bar s_{ i, l^*(j) } \cdot t_i^{ 1 - \alpha^*(j)  }.
        \label{def: function C J assigned to I, cluster size}
\end{align}
Under Assumption~\ref{assumption: heavy tails in B i j}, we have $\alpha^*(j) > 1$ for each $j \in [d]$ (see \eqref{def: cluster size, alpha * l * j}),
so $C_{ \mathcal I \leftarrow \mathcal J }\big( (t_i)_{ i \in \mathcal I } \big)$ is monotone decreasing w.r.t.\ each $t_i$.
If $\mathcal J = \emptyset$,
we adopt the convention that 
$
C_{ \mathcal I \leftarrow \emptyset }\big( (t_i)_{ i \in \mathcal I } \big) \equiv 1.
$
Likewise, for the function $g_{\mathcal I \leftarrow \mathcal J}$ defined in \eqref{def: function g mathcal I mathcal J, for measure C i bm I, cluster size},
we adopt the convention that 
$
g_{\mathcal I \leftarrow \emptyset}(\bm w)\equiv 1.
$
Also, we use $\mathfrak T_{\mathcal I \leftarrow \mathcal J}$
to denote the set of all assignment from $\mathcal J$ to $\mathcal I$, \emph{allowing for replacements}:
that is,
$\mathfrak T_{\mathcal I \leftarrow \mathcal J}$
contains all $\{ \mathcal J(i) \subseteq \mathcal J:\ i \in \mathcal I \}$
satisfying $\bigcup_{i \in \mathcal I}\mathcal J(i) = \mathcal J$.
Note that $\mathbb T_{\mathcal I \leftarrow \mathcal J} \subset \mathfrak T_{\mathcal I \leftarrow \mathcal J}$
and $|\mathfrak T_{\mathcal I \leftarrow \mathcal J}| < \infty$ given  $\mathcal I, \mathcal J \subseteq [d]$.
In the next result, we write $\bm t(\mathcal I) = (t_i)_{i \in \mathcal I}$.

\begin{lemma}
\label{lemma: cluster size, asymptotics, N i | n delta, cdot j, refined estimate}
\linksinthm{lemma: cluster size, asymptotics, N i | n delta, cdot j, refined estimate}
Let Assumptions~\ref{assumption: subcriticality}--\ref{assumption: regularity condition 2, cluster size, July 2024} hold.
    Let $0 < c < C < \infty$.
    Let
    $\mathcal I \subseteq \{1,2,\ldots,d\}$ be non-empty,
        and let
        $\mathcal J \subseteq \{1,2,\ldots,d\}$.
    There exists some $\delta_0 > 0$ such that
        for any $\delta \in (0,\delta_0)$,
    \begin{align}
        \limsup_{n \to \infty}
        \sup_{\bm t(\mathcal I):\ t_i \geq nc\ \forall i \in \mathcal I  }
        \Bigg|
        \frac{
        \P\big( N^{>|\delta}_{\bm t(\mathcal I);j} \geq 1\text{ iff } j \in \mathcal J \big)
        }{
        C_{\mathcal I \leftarrow \mathcal J}\big( n^{-1}\bm t(\mathcal I)\big)
        \cdot 
        \prod_{j \in \mathcal J}
        n\P(B_{j \leftarrow l^*(j)} > n\delta)
        }
        - 1
        \Bigg|
        & = 0,
        \label{claim weak convergence for N, part (ii), N related, lemma: cluster size, asymptotics, N i | n delta, cdot j, multiple ancsetor, refined estimates}
    \end{align}
        where $C_{\mathcal I \leftarrow \mathcal J}(\bm t)$ is defined in \eqref{def: function C J assigned to I, cluster size},
        and $l^*(j)$ is defined in \eqref{def: cluster size, alpha * l * j}.
        Furthermore, 
        if $\mathcal J \neq \emptyset$, for any $\delta \in (0,\delta_0)$,
        \begin{align}
            \lim_{n \to \infty} 
            \sup_{t_i \in [nc,nC]\ \forall i \in \mathcal I }\ 
            \sup_{{nx_j}/{t_i} \in [c,C]\ \forall i \in \mathcal I,j\in\mathcal J }
            & 
            \Bigg|
            \frac{
                \P\big(
                    W^{>|\delta}_{\bm t(\mathcal I);j}  > nx_j\ \forall j \in \mathcal J\ \big|\ 
                    N^{>|\delta}_{\bm t(\mathcal I);j} \geq 1\text{ iff } j \in \mathcal J
                    \big)
            }{
                \frac{
                    g_{ \mathcal I \leftarrow \mathcal J }( n^{-1}\bm t(\mathcal I))
                        }{
                    C_{\mathcal I \leftarrow \mathcal J}( n^{-1}\bm t(\mathcal I))
                    }
                 \cdot 
                \prod_{ j \in \mathcal J }({\delta}/{x_j})^{\alpha^*(j)}
            }
            \
            -  1\Bigg| 
            \label{claim 2, part (ii), N related, lemma: cluster size, asymptotics, N i | n delta, cdot j, multiple ancsetor, refined estimates}
            \\
            &
            = 0,
        \nonumber
        \\
         \limsup_{n \to \infty} 
            \sup_{t_i \geq nc\ \forall i \in \mathcal I }\ 
            \sup_{{nx_j}/{t_i} \in [c,C]\ \forall i \in \mathcal I,j\in\mathcal J }
            & 
            \frac{
                \P\big(
                    W^{>|\delta}_{\bm t(\mathcal I);j}  > nx_j\ \forall j \in \mathcal J\ \big|\ 
                    N^{>|\delta}_{\bm t(\mathcal I);j} \geq 1\text{ iff } j \in \mathcal J
                    \big)
            }{
                \sum_{ 
                \{ \mathcal J(i):\ i \in \mathcal I \} \in \mathbb T_{ \mathcal I \leftarrow \mathcal J }
                }
                 \prod_{ i \in \mathcal I}
                \prod_{ j \in \mathcal J(i) }({\delta t_i}/{n x_j})^{\alpha^*(j)}
            }
            \leq 1
        \label{claim 2, crude upper bound, part (ii), N related, lemma: cluster size, asymptotics, N i | n delta, cdot j, multiple ancsetor, refined estimates}
        \end{align}
        where $g_{\mathcal I \leftarrow \mathcal J}(\cdot)$ is defined in \eqref{def: function g mathcal I mathcal J, for measure C i bm I, cluster size}.
\end{lemma}

\begin{proof}
\linksinpf{lemma: cluster size, asymptotics, N i | n delta, cdot j, refined estimate}
First, 
we note that for the proof of Claim \eqref{claim weak convergence for N, part (ii), N related, lemma: cluster size, asymptotics, N i | n delta, cdot j, multiple ancsetor, refined estimates},
we only need to consider non-empty $\mathcal J \subseteq \{1,2,\ldots,d\}$.
To see why, note that to prove \eqref{claim weak convergence for N, part (ii), N related, lemma: cluster size, asymptotics, N i | n delta, cdot j, multiple ancsetor, refined estimates} under $\mathcal J = \emptyset$,
it suffices to show that
\begin{align}
    \lim_{n \to \infty}
    \inf_{\bm t(\mathcal I):\ t_i \geq nc\ \forall i \in \mathcal I  }
    \P\Big( N^{>|\delta}_{\bm t(\mathcal I);j}= 0 \ \forall j \in [d]\Big) = 1.
    \label{proof: goal, weak convergence for N when J is empty, part ii, lemma: cluster size, asymptotics, N i | n delta, cdot j, refined estimate}
\end{align}
Suppose that \eqref{claim weak convergence for N, part (ii), N related, lemma: cluster size, asymptotics, N i | n delta, cdot j, multiple ancsetor, refined estimates} holds for any $\delta \in (0,\delta_0)$ and any non-empty $\mathcal J \subseteq [d]$.
Then,
by Assumption~\ref{assumption: heavy tails in B i j} and definitions  in \eqref{def: cluster size, alpha * l * j},
in \eqref{claim weak convergence for N, part (ii), N related, lemma: cluster size, asymptotics, N i | n delta, cdot j, multiple ancsetor, refined estimates} it holds for any $T \geq nc$ that 
\begin{align*}
    n\P(B_{j \leftarrow l^*(j)} > T\delta)
    \leq 
    n\P(B_{j \leftarrow l^*(j)} > n\cdot c\delta)
    \in \RV_{ -(\alpha^*(j) - 1)  }(n),
    \quad
    \text{with }\alpha^*(j) > 1.
\end{align*}
Then, \eqref{claim weak convergence for N, part (ii), N related, lemma: cluster size, asymptotics, N i | n delta, cdot j, multiple ancsetor, refined estimates} implies
\begin{align*}
    \lim_{n \to \infty}\max_{\bm t(\mathcal I):\ t_i \geq nc\ \forall i \in \mathcal I  }
    \P\Big( N^{>|\delta}_{\bm t(\mathcal I);j}\geq 1 \text{ iff } j \in \mathcal J\Big) = 0,
    \qquad
    \forall \delta \in (0,\delta_0),\ \emptyset \neq \mathcal J \subseteq [d].
\end{align*}
Under any $\delta \in (0,\delta_0)$,
the Claim \eqref{proof: goal, weak convergence for N when J is empty, part ii, lemma: cluster size, asymptotics, N i | n delta, cdot j, refined estimate} then follows from the preliminary bound 
\begin{align*}
    \P\Big( N^{>|\delta}_{\bm t(\mathcal I);j}= 0 \ \forall j \in [d]\Big)
    \geq 
    1 - 
    \sum_{ \mathcal J \subseteq [d]:\ \mathcal J \neq \emptyset }
    \P\Big( N^{>|\delta}_{\bm t(\mathcal I);j}\geq 1 \text{ iff } j \in \mathcal J\Big).
\end{align*}

Next, 
we note that it suffices to prove \eqref{claim weak convergence for N, part (ii), N related, lemma: cluster size, asymptotics, N i | n delta, cdot j, multiple ancsetor, refined estimates}%
--\eqref{claim 2, crude upper bound, part (ii), N related, lemma: cluster size, asymptotics, N i | n delta, cdot j, multiple ancsetor, refined estimates}
for  $\mathcal I = \{i\}$ with $i \in [d]$ (i.e., the case of $|\mathcal I| = 1$).
In particular,
it suffices to identify $\delta_0 > 0$ such that for any $\delta \in (0,\delta_0)$,
$i \in [d]$, and non-empty $\mathcal J \subseteq [d]$,
\begin{align}
    \limsup_{n \to \infty}
    \sup_{T \geq nc } & 
    \Bigg|
        \frac{
        \P\big(
            \sum_{m = 1}^{ T }N^{>,(m)}_{i;j}(\delta T)  \geq 1\text{ iff }j \in \mathcal J
        \big)
    }{
        \prod_{j \in \mathcal J}
        (n^{-1}T)^{1 - \alpha^*(j) }\cdot 
        \bar s_{i,l^*(j)}\cdot n\P(B_{j \leftarrow l^*(j)} > n\delta)
    }
    - 1
    \Bigg|
    = 0,
    \label{proof: goal, weak convergence for N, |I| = 1, part ii, N related, lemma: cluster size, asymptotics, N i | n delta, cdot j, multiple ancsetor, refined estimates}
    \\
    \lim_{n \to \infty}
    \sup_{ T \geq nc }\  
    \sup_{ nx_j/T \in [c,C]\ \forall j \in \mathcal J }
    & 
    \Bigg|
        \frac{
        \P\big(
            \sum_{m = 1}^{ T }W^{>,(m)}_{i;j}(\delta T)  > n x_j\ \forall j \in \mathcal J
            \ \big|\ 
            \sum_{m = 1}^{ T }N^{>,(m)}_{i;j}(\delta T)  \geq 1\text{ iff }j \in \mathcal J
        \big)
    }{
        \prod_{j \in \mathcal J}
            \big(
                \frac{\delta T}{n x_j}
            \big)^{ \alpha^*(j) }
    }
    \nonumber
    \\ 
    &\qquad\qquad\qquad\qquad\qquad\qquad\qquad\qquad\qquad\qquad\qquad
    - 1
    \Bigg| = 0.
     \label{proof: goal, weak convergence for W, |I| = 1, part ii, N related, lemma: cluster size, asymptotics, N i | n delta, cdot j, multiple ancsetor, refined estimates}
\end{align}
To see how these claims
lead to 
\eqref{claim weak convergence for N, part (ii), N related, lemma: cluster size, asymptotics, N i | n delta, cdot j, multiple ancsetor, refined estimates}--\eqref{claim 2, crude upper bound, part (ii), N related, lemma: cluster size, asymptotics, N i | n delta, cdot j, multiple ancsetor, refined estimates},
recall that we use $\mathfrak T_{\mathcal I \leftarrow \mathcal J}$
to denote the set containing all $\{ \mathcal J(i) \subseteq \mathcal J:\ i \in \mathcal I \}$
satisfying $\bigcup_{i \in \mathcal I}\mathcal J(i) = \mathcal J$.
Also, 
recall that we use ${\mathbb T_{ \mathcal I \leftarrow \mathcal J }}$ to denote the set of all assignments of $\mathcal{J}$ to $\mathcal I$.
By definitions in \eqref{def: assignment from mathcal J to mathcal I}, 
we have
$\mathbb T_{\mathcal I \leftarrow \mathcal J} \subset \mathfrak T_{\mathcal I \leftarrow \mathcal J}$
and $|\mathfrak T_{\mathcal I \leftarrow \mathcal J}| < \infty$ given  $\mathcal I, \mathcal J \subseteq [d]$.
Next, observe that
\begin{align}
    \label{proof: decomp of events, N geq 1, lemma: cluster size, asymptotics, N i | n delta, cdot j, refined estimate}
    & \P\Big(
            {N^{>|\delta}_{\bm t(\mathcal I);j}} \geq 1\text{ iff }j \in \mathcal J
        \Big)
    = 
    \P\Bigg(
        \sum_{i \in \mathcal I} \sum_{m = 1}^{ t_i }N^{>,(m)}_{i;j}(\delta t_i) \geq 1
        \text{ iff }j \in \mathcal J
    \Bigg)
    \quad
    \text{by \eqref{def, proof cluster size, N W mathcal I mathcal J bcdot j}}
    \\ 
    & = 
    \sum_{ 
        \{ \mathcal J(i):\ i \in \mathcal I \} \in \mathfrak T_{ \mathcal I \leftarrow \mathcal J }
    }
    \prod_{ i \in \mathcal I}
    \P\Bigg(
            \sum_{m = 1}^{ t_i }N^{>,(m)}_{i;j}(\delta t_i)  \geq 1\text{ iff }j \in \mathcal J(i)
        \Bigg)
        \nonumber
    \\ 
    & = 
    \underbrace{ 
        \sum_{ 
        \{ \mathcal J(i):\ i \in \mathcal I \} \in \mathbb T_{ \mathcal I \leftarrow \mathcal J }
    }
    \prod_{ i \in \mathcal I}
    \P\Bigg(
            \sum_{m = 1}^{ t_i }N^{>,(m)}_{i;j}(\delta t_i)  \geq 1\text{ iff }j \in \mathcal J(i)
        \Bigg)
    }_{ \delequal \text{(I)}  }
    \nonumber
    \\ 
    & \quad + 
    \underbrace{ 
        \sum_{ 
        \{ \mathcal J(i):\ i \in \mathcal I \} \in \mathfrak T_{ \mathcal I \leftarrow \mathcal J }\setminus \mathbb T_{ \mathcal I \leftarrow \mathcal J }
    }
    \prod_{ i \in \mathcal I}
    \P\Bigg(
            \sum_{m = 1}^{ t_i }N^{>,(m)}_{i;j}(\delta t_i)  \geq 1\text{ iff }j \in \mathcal J(i)
        \Bigg)
    }_{ \delequal \text{(II)}  }.
    \nonumber
\end{align}
Given $\{ \mathcal J(i):\ i \in \mathcal I \} \in \mathbb T_{ \mathcal I \leftarrow \mathcal J }$,
by the definition of partitions (i.e., the $\mathcal J(i)$'s are mutually disjoint, and $\bigcup_{i \in \mathcal I} \mathcal J(i) = \mathcal J$), we have
\begin{align}
    \label{proof, section D, property of partition}
    & \prod_{i \in \mathcal I}\prod_{j \in \mathcal J(i)}
        t_i^{1 - \alpha^*(j) }\cdot 
        \bar s_{i,l^*(j)}\cdot n\P(B_{j \leftarrow l^*(j)} > n\delta)
    \\ 
    \nonumber
    &
        = 
        \Bigg( \prod_{i \in \mathcal I}\prod_{j \in \mathcal J(i)}
        t_i^{1 - \alpha^*(j) }\cdot 
        \bar s_{i,l^*(j)}
        \Bigg)
        \cdot 
        \Bigg( \prod_{j \in \mathcal J}n\P(B_{j \leftarrow l^*(j)} > n\delta) \Bigg).
\end{align}
By the definition in \eqref{def: function C J assigned to I, cluster size},
\begin{align*}
    & \sum_{ 
        \{ \mathcal J(i):\ i \in \mathcal I \} \in \mathbb T_{ \mathcal I \leftarrow \mathcal J }
    }
    \prod_{i \in \mathcal I}\prod_{j \in \mathcal J(i)}
        t_i^{1 - \alpha^*(j) }\cdot 
        \bar s_{i,l^*(j)}\cdot n\P(B_{j \leftarrow l^*(j)} > n\delta)
        \\ 
        &
     =
     \mathcal C_{\mathcal I \leftarrow \mathcal J}\big( (t_i)_{i \in \mathcal I}\big) \cdot \prod_{j \in \mathcal J}n\P(B_{j \leftarrow l^*(j)} > n\delta).
\end{align*}
Then, applying the uniform convergence \eqref{proof: goal, weak convergence for N, |I| = 1, part ii, N related, lemma: cluster size, asymptotics, N i | n delta, cdot j, multiple ancsetor, refined estimates}
for each  $\P\big(
            \sum_{m = 1}^{ t_i }N^{>,(m)}_{i;j}(\delta t_i)  \geq 1\text{ iff }j \in \mathcal J(i)
        \big)$
in term (I) of the display \eqref{proof: decomp of events, N geq 1, lemma: cluster size, asymptotics, N i | n delta, cdot j, refined estimate},
we get
\begin{align}
    \lim_{n \to \infty}
    \sup_{\bm t(\mathcal I):\ t_i \geq nc\ \forall i \in \mathcal I}
    \Bigg|
        \frac{
            \text{(I)}
        }{
            \mathcal C_{\mathcal I \leftarrow \mathcal J}\big( n^{-1} \bm t(\mathcal I)\big) \cdot \prod_{j \in \mathcal J}n\P(B_{j \leftarrow l^*(j)} > n\delta)
        } 
        - 1
    \Bigg|
    =0.
    \label{proof, bound for term I, decomp of events, N geq 1, lemma: cluster size, asymptotics, N i | n delta, cdot j, refined estimate}
\end{align}
Next, to bound the term (II), we note that for each $\{ \mathcal J(i):\ i \in \mathcal I \} \in \mathfrak T_{ \mathcal I \leftarrow \mathcal J } \setminus \mathbb T_{ \mathcal I \leftarrow \mathcal J }$,
we must have $\mathcal J(i) \cap \mathcal J(i^\prime) \neq \emptyset$ for some $i,i^\prime \in [d]$ with $i\neq i^\prime$:
this is because $\{ \mathcal J(i):\ i \in \mathcal I \}$ is \emph{not} a partition of $\mathcal J$
but still satisfies $\bigcup_{i \in \mathcal I}\mathcal J(i) = \mathcal J$ and $\mathcal J(i) \subseteq \mathcal J\ \forall i \in \mathcal I$.
This has two useful implications.
First, due to  $\mathcal J(i) \cap \mathcal J(i^\prime) \neq \emptyset$ for some $i \neq i^\prime$,
\begin{align}
    \prod_{i \in \mathcal I}\prod_{j \in \mathcal J(i)}
         n\P(B_{j \leftarrow l^*(j)} > n\delta)
    =
    \lo\Bigg( \prod_{j \in \mathcal J}n\P(B_{j \leftarrow l^*(j)} > n\delta) \Bigg),
    \quad\text{as }n \to \infty.
    \label{proof, section D, bound 1}
\end{align}
Second,
for each $\{ \mathcal J(i):\ i \in \mathcal I \} \in \mathfrak T_{ \mathcal I \leftarrow \mathcal J } \setminus \mathbb T_{ \mathcal I \leftarrow \mathcal J }$,
we can find some 
$\{ \hat{\mathcal J}(i):\ i \in \mathcal I \} \in \mathbb T_{ \mathcal I \leftarrow \mathcal J }$
such that $\hat{\mathcal J}(i) \subseteq \mathcal J(i)\ \forall i \in \mathcal I$.
In particular, there exists some $\hat i \in \mathcal I$ and $\hat j \in \mathcal J(\hat i)$ such that $\hat j \notin \hat{\mathcal J}(\hat i)$.
As a result,
for each $n$ and each $\bm t(\mathcal I) = (t_i)_{i \in \mathcal I}$ with $t_i \geq nc\ \forall i \in \mathcal I$,
\begin{align}
    &
    \frac{
        \prod_{i \in \mathcal I}\prod_{j \in  \mathcal J(i)}(n^{-1}t_i)^{ 1 - \alpha^*(j)  }\cdot\bar s_{i,l^*(j)}
    }{
        \mathcal C_{\mathcal I \leftarrow \mathcal J}\big( n^{-1}\bm t(\mathcal I)\big)
    }
     \label{proof, section D, bound 2}
    \\ 
    &
    \leq
    \frac{
        \prod_{i \in \mathcal I}\prod_{j \in  \mathcal J(i)}(n^{-1}t_i)^{ 1 - \alpha^*(j)  }\cdot\bar s_{i,l^*(j)}
    }{
        \prod_{i \in \mathcal I}\prod_{j \in  \hat{\mathcal J}(i)}(n^{-1}t_i)^{ 1 - \alpha^*(j)  }\cdot\bar s_{i,l^*(j)}
    }
    \quad \text{by definitions in \eqref{def: function C J assigned to I, cluster size}}
    \nonumber
    \\ 
    & = 
    \prod_{i \in \mathcal I}\prod_{j \in  \mathcal J(i) \setminus \hat{\mathcal J}(i) }(n^{-1}t_i)^{ 1 - \alpha^*(j)  }\cdot\bar s_{i,l^*(j)}
    \leq 
    \prod_{i \in \mathcal I}\prod_{j \in  \mathcal J(i) \setminus \hat{\mathcal J}(i) }c^{ 1 - \alpha^*(j)  }\cdot\bar s_{i,l^*(j)} < \infty,
    \nonumber
\end{align}
where in the last line we applied $\alpha^*(j) > 1\ \forall j \in [d]$.
Applying the uniform convergence \eqref{proof: goal, weak convergence for N, |I| = 1, part ii, N related, lemma: cluster size, asymptotics, N i | n delta, cdot j, multiple ancsetor, refined estimates} 
for each  $\P\big(
            \sum_{m = 1}^{ t_i }N^{>,(m)}_{i;j}(\delta t_i)  \geq 1\text{ iff }j \in \mathcal J(i)
        \big)$
in term (II) of the display \eqref{proof: decomp of events, N geq 1, lemma: cluster size, asymptotics, N i | n delta, cdot j, refined estimate},
it follows from \eqref{proof, section D, bound 1} and \eqref{proof, section D, bound 2} that
\begin{align}
    \lim_{n \to \infty}
    \sup_{\bm t(\mathcal I):\ t_i \geq nc\ \forall i \in \mathcal I}
        \frac{
            \text{(II)}
        }{
            \mathcal C_{\mathcal I \leftarrow \mathcal J}\big( n^{-1}\bm t(\mathcal I)\big) \cdot \prod_{j \in \mathcal J}n\P(B_{j \leftarrow l^*(j)} > n\delta)
        }  = 0.
    \label{proof, bound for term II, decomp of events, N geq 1, lemma: cluster size, asymptotics, N i | n delta, cdot j, refined estimate}
\end{align}
Combining \eqref{proof, bound for term I, decomp of events, N geq 1, lemma: cluster size, asymptotics, N i | n delta, cdot j, refined estimate} and \eqref{proof, bound for term II, decomp of events, N geq 1, lemma: cluster size, asymptotics, N i | n delta, cdot j, refined estimate},
we establish \eqref{claim weak convergence for N, part (ii), N related, lemma: cluster size, asymptotics, N i | n delta, cdot j, multiple ancsetor, refined estimates}.

To proceed,  we
define the event
\begin{align}
    E_{\mathcal I \leftarrow \mathcal J}\big(n,\delta, \bm t(\mathcal I) \big)
    \delequal
    \bigcup_{ \{\mathcal J(i):\ i \in \mathcal I\} \in \mathbb T_{\mathcal I \leftarrow \mathcal J} }
    \Bigg\{
        \text{for each $i \in \mathcal I$},\ 
            \sum_{m = 1}^{ t_i }N^{>,(m)}_{i;j}(\delta t_i)  \geq 1\text{ iff }j \in \mathcal J(i)
    \Bigg\},
    \nonumber
\end{align}
and
note that our analysis above for terms (I) and (II) in display \eqref{proof: decomp of events, N geq 1, lemma: cluster size, asymptotics, N i | n delta, cdot j, refined estimate}
implies
\begin{align*}
    \lim_{n \to \infty}\inf_{ \bm t(\mathcal I):\  t_i \geq nc\ \forall i \in \mathcal I}
    \P\Big( 
        E_{\mathcal I \leftarrow \mathcal J}\big(n,\delta,\bm t(\mathcal I)\big)
        \ \Big|\
        {N^{>|\delta}_{\bm t(\mathcal I);j}} \geq 1\text{ iff }j \in \mathcal J
    \Big) = 1.
    \nonumber
\end{align*}
Therefore, it is equivalent to prove a modified version of Claims \eqref{claim 2, part (ii), N related, lemma: cluster size, asymptotics, N i | n delta, cdot j, multiple ancsetor, refined estimates} and \eqref{claim 2, crude upper bound, part (ii), N related, lemma: cluster size, asymptotics, N i | n delta, cdot j, multiple ancsetor, refined estimates},
where we condition on the event $E_{\mathcal I \leftarrow \mathcal J}\big(n,\delta,\bm t(\mathcal I)\big)$
instead of $\{{N^{>|\delta}_{\bm t(\mathcal I);j}} \geq 1\text{ iff }j \in \mathcal J\}$.
For Claim~\eqref{claim 2, part (ii), N related, lemma: cluster size, asymptotics, N i | n delta, cdot j, multiple ancsetor, refined estimates},
we have
\begin{align}
    & \P\bigg(
                    W^{>|\delta}_{\bm t(\mathcal I);j}  > nx_j\ \forall j \in \mathcal J\ \bigg|\ 
                    E_{\mathcal I \leftarrow \mathcal J}\big(n,\delta,\bm t(\mathcal I)\big)
                    \bigg)
    \label{proof: equality, decomposition of the W event, part ii, lemma: cluster size, asymptotics, N i | n delta, cdot j, refined estimate}
    \\ 
    & = 
    \sum_{ 
        \{ \mathcal J(i):\ i \in \mathcal I \} \in \mathbb T_{ \mathcal I \leftarrow \mathcal J }
    }
    \P\Bigg(
                    W^{>|\delta}_{\bm t(\mathcal I);j}  > nx_j\ \forall j \in \mathcal J\ \Bigg|\
                    \text{for each $i \in \mathcal I$},\ 
            \sum_{m = 1}^{ t_i }N^{>,(m)}_{i;j}(\delta t_i)  \geq 1\text{ iff }j \in \mathcal J(i)
            \Bigg)
        \nonumber
    \\ 
    &\qquad\qquad\qquad\qquad
            \cdot
    \P\Bigg(
        \text{for each $i \in \mathcal I$},\ 
            \sum_{m = 1}^{ t_i }N^{>,(m)}_{i;j}(\delta t_i)  \geq 1\text{ iff }j \in \mathcal J(i)
            \ \Bigg|\ 
                    E_{\mathcal I \leftarrow \mathcal J}\big(n,\delta,\bm t(\mathcal I)\big)
                    \Bigg)
                \nonumber
    \\ 
    & \stackrel{(*)}{=} 
    \sum_{ 
        \{ \mathcal J(i):\ i \in \mathcal I \} \in \mathbb T_{ \mathcal I \leftarrow \mathcal J }
    }
    \left[\rule{0cm}{0.9cm}
    \prod_{ i \in \mathcal I }
        \P\Bigg(
            \sum_{m = 1}^{ t_i }W^{>,(m)}_{i;j}(\delta t_i)  > n x_j\ \forall j \in \mathcal J(i)
            \ \Bigg|\ 
            \sum_{m = 1}^{ t_i }N^{>,(m)}_{i;j}(\delta t_i)  \geq 1\text{ iff }j \in \mathcal J(i)
        \Bigg)
    \right]
    \nonumber
   \\ 
    &\qquad\qquad\qquad\qquad
            \cdot
    \P\Bigg(
        \text{for each $i \in \mathcal I$},\ 
            \sum_{m = 1}^{ t_i }N^{>,(m)}_{i;j}(\delta t_i)  \geq 1\text{ iff }j \in \mathcal J(i)
            \ \Bigg|\ 
                    E_{\mathcal I \leftarrow \mathcal J}\big(n,\delta,\bm t(\mathcal I)\big)
                    \Bigg)
                \nonumber
\\ 
& = 
\sum_{ 
        \{ \mathcal J(i):\ i \in \mathcal I \} \in \mathbb T_{ \mathcal I \leftarrow \mathcal J }
    }
    \left[\rule{0cm}{0.9cm}
    \prod_{ i \in \mathcal I }
        \P\Bigg(
            \sum_{m = 1}^{ t_i }W^{>,(m)}_{i;j}(\delta t_i)  > n x_j\ \forall j \in \mathcal J(i)
            \ \Bigg|\ 
            \sum_{m = 1}^{ t_i }N^{>,(m)}_{i;j}(\delta t_i)  \geq 1\text{ iff }j \in \mathcal J(i)
        \Bigg)
    \right]
    \nonumber
\\ 
&\qquad\qquad\qquad\qquad\qquad
            \cdot
    \frac{
        \prod_{i \in \mathcal I}
            \P\big(
            \sum_{m = 1}^{ t_i }N^{>,(m)}_{i;j}(\delta t_i)  \geq 1\text{ iff }j \in \mathcal J(i)
        \big)
    }{
        \P\big( E_{\mathcal I \leftarrow \mathcal J}\big(n,\delta,\bm t(\mathcal I)\big) \big)
    }.
    \nonumber
\end{align}
Here, the step $(*)$ follows from the independence of
$
\big\{  \big( N^{>,(m)}_{i;j}(M), W^{>,(m)}_{i;j}(M) \big)_{j \in [d]}:\ m \geq 1  \big\}
$
across $i \in [d]$; see \eqref{proof: def copies of W > and N > vectors}.
Then by applying \eqref{claim weak convergence for N, part (ii), N related, lemma: cluster size, asymptotics, N i | n delta, cdot j, multiple ancsetor, refined estimates}, \eqref{proof: goal, weak convergence for N, |I| = 1, part ii, N related, lemma: cluster size, asymptotics, N i | n delta, cdot j, multiple ancsetor, refined estimates},
and
\eqref{proof: goal, weak convergence for W, |I| = 1, part ii, N related, lemma: cluster size, asymptotics, N i | n delta, cdot j, multiple ancsetor, refined estimates},
under any $\delta > 0$ small enough,
it holds
uniformly over $t_i \in [nc,nC]$ 
and $\frac{nx_j}{t_i} \in [c,C]$---in the sense of \eqref{claim 2, part (ii), N related, lemma: cluster size, asymptotics, N i | n delta, cdot j, multiple ancsetor, refined estimates}---that
\begin{align*}
    & \P\bigg(
                    W^{>|\delta}_{\bm t(\mathcal I);j}  > nx_j\ \forall j \in \mathcal J\ \bigg|\ 
                    E_{\mathcal I \leftarrow \mathcal J}\big(n,\delta,\bm t(\mathcal I)\big)
                    \bigg)
    \\ 
    & \sim 
    \sum_{ 
        \{ \mathcal J(i):\ i \in \mathcal I \} \in \mathbb T_{ \mathcal I \leftarrow \mathcal J }
    }
    \Bigg[
        \prod_{ i \in \mathcal I  }
        \prod_{ j \in \mathcal J(i) }\bigg( \frac{\delta t_i}{n x_j}\bigg)^{ \alpha^*(j) }
    \Bigg]
    \cdot 
    \frac{
        \prod_{i \in \mathcal I}\prod_{ j \in \mathcal J(i)}
            (n^{-1}t_i)^{1 - \alpha^*(j) }\cdot 
        \bar s_{i,l^*(j)}\cdot n\P(B_{ j \leftarrow l^*(j)} > n\delta)
    }{
         C_{\mathcal I \leftarrow \mathcal J}\big( n^{-1}\bm t(\mathcal I)\big)
        \cdot 
        \prod_{j \in \mathcal J}
        n\P(B_{ j \leftarrow l^*(j)} > n\delta)
    }
    \\ 
    & = 
    \Bigg[
        \prod_{ j \in \mathcal J }\bigg( \frac{\delta}{x_j}\bigg)^{ \alpha^*(j) }
    \Bigg]
        \cdot 
    \frac{
        \sum_{ 
        \{ \mathcal J(i):\ i \in \mathcal I \} \in \mathbb T_{ \mathcal I \leftarrow \mathcal J }
        }
        \prod_{i \in \mathcal I}\prod_{ j \in \mathcal J(i)}
        (n^{-1}t_i)\cdot 
        \bar s_{i,l^*(j)}
    }{
         C_{\mathcal I \leftarrow \mathcal J}\big( n^{-1}\bm t(\mathcal I)\big)
    }
    \qquad
    \text{by \eqref{proof, section D, property of partition}}
    \\
    & = 
    \Bigg[
        \prod_{ j \in \mathcal J }\bigg( \frac{\delta}{x_j}\bigg)^{ \alpha^*(j) }
    \Bigg]
        \cdot 
    \frac{
        g_{ \mathcal I \leftarrow \mathcal J }\big( n^{-1}\bm t(\mathcal I)\big)
    }{
         C_{\mathcal I \leftarrow \mathcal J}\big( n^{-1}\bm t(\mathcal I)\big)
    }
    \qquad
    \text{by the definition in \eqref{def: function g mathcal I mathcal J, for measure C i bm I, cluster size}}
\end{align*}
as $n \to \infty$.
This verifies Claim \eqref{claim 2, part (ii), N related, lemma: cluster size, asymptotics, N i | n delta, cdot j, multiple ancsetor, refined estimates}.
Furthermore, from the last line of display \eqref{proof: equality, decomposition of the W event, part ii, lemma: cluster size, asymptotics, N i | n delta, cdot j, refined estimate},
\begin{align*}
    & 
    \P\bigg(
                    W^{>|\delta}_{\bm t(\mathcal I);j}  > nx_j\ \forall j \in \mathcal J\ \bigg|\ 
                    E_{\mathcal I \leftarrow \mathcal J}\big(n,\delta,\bm t(\mathcal I)\big)
                    \bigg)
    \\ 
    & \leq 
    \sum_{ 
        \{ \mathcal J(i):\ i \in \mathcal I \} \in \mathbb T_{ \mathcal I \leftarrow \mathcal J }
    }
    \left[\rule{0cm}{0.9cm}
    \prod_{ i \in \mathcal I }
        \P\Bigg(
            \sum_{m = 1}^{ t_i }W^{>,(m)}_{i;j}(\delta t_i)  > n x_j\ \forall j \in \mathcal J(i)
            \ \Bigg|\ 
            \sum_{m = 1}^{ t_i }N^{>,(m)}_{i;j}(\delta t_i)  \geq 1\text{ iff }j \in \mathcal J(i)
        \Bigg)
    \right].
\end{align*}
Applying \eqref{proof: goal, weak convergence for W, |I| = 1, part ii, N related, lemma: cluster size, asymptotics, N i | n delta, cdot j, multiple ancsetor, refined estimates}, we verify Claim \eqref{claim 2, crude upper bound, part (ii), N related, lemma: cluster size, asymptotics, N i | n delta, cdot j, multiple ancsetor, refined estimates} for any $\delta > 0$ small enough.
In summary, we have shown that
it suffices to prove Claims~\eqref{proof: goal, weak convergence for N, |I| = 1, part ii, N related, lemma: cluster size, asymptotics, N i | n delta, cdot j, multiple ancsetor, refined estimates}
and
\eqref{proof: goal, weak convergence for W, |I| = 1, part ii, N related, lemma: cluster size, asymptotics, N i | n delta, cdot j, multiple ancsetor, refined estimates}.
In the remainder of this proof, we establish the Claims~\eqref{proof: goal, weak convergence for N, |I| = 1, part ii, N related, lemma: cluster size, asymptotics, N i | n delta, cdot j, multiple ancsetor, refined estimates}
and
\eqref{proof: goal, weak convergence for W, |I| = 1, part ii, N related, lemma: cluster size, asymptotics, N i | n delta, cdot j, multiple ancsetor, refined estimates},
i.e., addressing the case where $|\mathcal I| = 1$.

\medskip
\noindent
\textbf{Proof of Claim~\eqref{proof: goal, weak convergence for N, |I| = 1, part ii, N related, lemma: cluster size, asymptotics, N i | n delta, cdot j, multiple ancsetor, refined estimates}}.
%
Let ${\mathbb J}$ be the set of all partitions of the non-empty  $\mathcal J \subseteq [d]$.
Given any partition $\mathscr J = \{ \mathcal J_1,\ldots, \mathcal J_k \} \in \mathbb J$,
let the event $A^\mathscr{J}_n(T,\delta)$ be defined as in \eqref{proof: def event A T partition J, part iii, lemma: cluster size, asymptotics, N i | n delta, cdot j, crude estimate}.
Our proof is based on the decomposition
of events in \eqref{proof, part i, decomp, upper bound, lemma: cluster size, asymptotics, N i | n delta, cdot j, multiple ancsetor}.
We first prove an upper bound.
Let 
$$
p(i,M,\mathcal T)
 \delequal 
 \P\big(
N^{>}_{i;j}(M) \geq 1\ \forall j \in \mathcal{T}
    \big).
$$
Given any $T \in \mathbb N$ and partition $\mathscr J = \{ \mathcal J_1,\ldots,\mathcal J_k \} \in \mathbb J$,
it has been shown in \eqref{proof, part i, event A i mathcal T upper bound, lemma: cluster size, asymptotics, N i | n delta, cdot j, multiple ancsetor} that 
\begin{align}
    \P\Big(A^\mathscr{J}_n(T,\delta)\Big)
    & \leq 
    \prod_{l \in [k]} T \cdot p(i,\delta T,\mathcal J_l).
    \label{proof: upper bound event A T partition J, part ii, lemma: cluster size, asymptotics, N i | n delta, cdot j, refined estimate}
\end{align}
Specifically, consider 
the singleton-partition 
$
\mathscr J_* = \{ \{j\}:\ j \in \mathcal J \}.
$
By Lemma~\ref{lemma: cluster size, asymptotics, N i | n delta, cdot j, crude estimate} (i),
there exists $\delta_0 > 0$ such that
\begin{align*}
    p(i,n\delta,\{j\}) \sim \bar{s}_{i,l^*(j)}\P(B_{j \leftarrow l^*(j)} > n\delta)
    \text{ as }n \to \infty,
    \qquad \forall j \in [d],\ \delta \in (0,\delta_0).
\end{align*}
It then follows from \eqref{proof: upper bound event A T partition J, part ii, lemma: cluster size, asymptotics, N i | n delta, cdot j, refined estimate} that
\begin{align*}
    \limsup_{n \to \infty}
    \sup_{ T \geq nc }
    \frac{
    \P\big(A^{\mathscr{J}_*}_n(T,\delta)\big)
    }{
        (n^{-1}T)^{|\mathcal J|}\prod_{j \in \mathcal J}\bar s_{i,l^*(j)}\cdot n\P(B_{j \leftarrow l^*(j)} > T\delta)
    }
    \leq 1,
    \qquad
    \forall \delta \in (0,\delta_0).
\end{align*}

Next,
we consider some partition $\mathscr J = \{\mathcal J_1,\ldots,\mathcal J_k\} \in \mathbb J \setminus \{ \mathscr{J}_* \}$.
Due to $\mathscr J \neq \mathscr J_*$,
there must be some $l \in [k]$ such that $\mathcal J_l$ contains at least two elements.
By part (ii) of Lemma~\ref{lemma: cluster size, asymptotics, N i | n delta, cdot j, crude estimate},
(and picking a smaller $\delta_0 > 0$ if needed)
\begin{align*}
    p(i,n\delta,\widetilde{\mathcal J})
    =
    \lo 
    \Bigg( n^{ |\widetilde{\mathcal J}| - 1}\prod_{j \in \widetilde{\mathcal J} }\P\big(B_{j \leftarrow l^*(j)} > n\delta\big) \Bigg),
    \quad
    \forall \delta \in (0,\delta_0),\ 
    \widetilde{\mathcal J} \subseteq [d]\text{ with }|\widetilde{\mathcal J}| \geq 2.
\end{align*}
Therefore, for any partition $\mathscr J = \{\mathcal J_1,\ldots,\mathcal J_k\} \in \mathbb J \setminus \{ \mathscr{J}_* \}$,
we have
\begin{align*}
    \prod_{l \in [k]}
    np(i,n\delta,\mathcal J_l)
    =
    \lo \Bigg(
        n^{|\mathcal J|}
        \prod_{j \in \mathcal J}
        \P\big(B_{j \leftarrow l^*(j)} > n\delta\big)
    \Bigg),
    \quad 
    \forall \delta \in (0,\delta_0).
\end{align*}
Then, by \eqref{proof: upper bound event A T partition J, part ii, lemma: cluster size, asymptotics, N i | n delta, cdot j, refined estimate},
\begin{align*}
    \limsup_{n \to \infty}
    \sup_{ T \geq nc }
    \frac{
    \P\big(A_n^{\mathscr{J}}(T,\delta)\big)
    }{
        \prod_{j \in \mathcal J}n\P(B_{ j \leftarrow  l^*(j)} > T\delta)
    }
    =
    0,
    \quad
    \forall \delta \in (0,\delta_0),\ \mathscr J \in \mathbb J \setminus\{ \mathscr J_* \}.
\end{align*}
Using the decomposition of events in \eqref{proof, part i, decomp, upper bound, lemma: cluster size, asymptotics, N i | n delta, cdot j, multiple ancsetor},
we arrive at the upper bound
\begin{align}
    \limsup_{n \to \infty}
    \sup_{ T \geq nc }
    \frac{
        \P\big(
            \sum_{m = 1}^{ T }N^{>,(m)}_{i;j}(\delta T)  \geq 1\text{ iff }j \in \mathcal J
        \big)
    }{
        (n^{-1}T)^{|\mathcal J|}\prod_{j \in \mathcal J}\bar s_{i,l^*(j)}\cdot n\P(B_{j \leftarrow l^*(j)} > T\delta)
    }
    \leq 1,\quad \forall \delta \in (0,\delta_0).
    \label{proof, part (i), upper bound, lemma: cluster size, asymptotics, N i | n delta, cdot j, multiple ancsetor}
\end{align}
We proceed similarly for the derivation of the lower bound.
In particular, note that
$
\big\{ \sum_{m = 1}^{ T }N^{>,(m)}_{i;j}(\delta T)  \geq 1\text{ iff }j \in \mathcal J  \big\}
\supseteq \hat A(T,\delta),
$
where
\begin{align*}
   \hat A(T,\delta)
    &\delequal
    \bigg\{
        \exists \{ m_1, m_2,\ldots,m_{|\mathcal J|} \} \subseteq [T] \text{ such that }
    \\ 
    &\qquad \qquad \qquad \qquad
        N^{>,(m_j)}_{ i;j }(\delta T) = 1,\ \sum_{ l \in [d]:\ l \neq j }N^{>,(m_j)}_{ i;l }(\delta T) = 0\ \forall j = 1,\ldots,|\mathcal J|;
     \\
     & \qquad \qquad \qquad \qquad \qquad\qquad\qquad
     \sum_{l \in [d]}N^{>,(m)}_{ i;l }(\delta T) = 0\ \forall m \in [T]\setminus\{m_j:\ j \in \mathcal{J}\}
    \bigg\}.
\end{align*}
For clarity of the notations in the display below, we write 
$k = |\mathcal J|$, $\mathcal J = \{j_1,\ldots,j_k\}$,
and
\begin{align*}
    \hat p(i,M,j)
 & \delequal
 \P\big(N^>_{ i:j }(M) = 1,\ N^>_{ i:j^\prime }(M) = 0\ \forall j^\prime \neq j\big),
 \\
 \hat{p}_*(i,M) & \delequal
 \P\big(N^>_{ i:j }(M) \geq 1\text{ for some }j\in [d]\big).
\end{align*}
Since the sequence $\big( N^{>,(m)}_{i;j}(\delta T)  \big)_{m \geq 1}$ are i.i.d.\ copies, 
by the law of multinomial distributions,
it holds
for any $T \geq \floor{nc}$ that
\begin{align*}
    & 
     \P\big(
            \sum_{m = 1}^{ T }N^{>,(m)}_{i;j}(\delta T)  \geq 1\text{ iff }j \in \mathcal J
        \big)
        \geq 
    \\
 &
 \geq \P\big(\hat A(T,\delta)\big)
 =
 \frac{
    T!
 }{
    (T - k)!
 }
 \cdot
 \Bigg[\prod_{l \in [k]}
    \hat p(i,\delta T,j_l)\cdot 
    \Bigg]
\cdot 
\Big(
 1 - \hat{p}_*(i,\delta T)
 \Big)^{ T - k }
 \\
 & \geq 
 \Bigg(\frac{\floor{nc} - k}{\floor{nc}}\Bigg)^{k}
 \cdot  T^k \cdot 
\Bigg[\prod_{l \in [k]}
    \hat p(i,\delta T,j_l)\cdot 
    \Bigg]
\cdot 
\Big(
 1 - \hat{p}_*(i,\delta T)
 \Big)^{ T }
 \qquad
 \text{due to }T \geq \floor{nc}
 \\
 & =
 \Bigg(\frac{\floor{nc} - k}{\floor{nc}}\Bigg)^{k}
 \cdot 
 \Big(1 - \hat{p}_*(i,\delta T)\Big)^{ T }
    \cdot 
 (n^{-1}T)^k
 \prod_{l \in [k]}n\hat p(i,\delta T,j_l).
\end{align*}
By part (i) and part (ii) of Lemma~\ref{lemma: cluster size, asymptotics, N i | n delta, cdot j, crude estimate},
there exists $\delta_0 > 0$ such that
\begin{align*}
    \hat p(i,T\delta,j) \sim \bar s_{i,l^*(j)}\P(B_{j \leftarrow l^*(j)} > T\delta)\quad
    \text{as }T \to \infty,
    \qquad \forall j \in [d],\ \delta \in (0,\delta_0).
\end{align*}
Analogously, using part (i) of Lemma~\ref{lemma: cluster size, asymptotics, N i | n delta, cdot j, crude estimate}
(and by picking a smaller $\delta_0 > 0$ if needed),
it holds for any $\delta \in (0,\delta_0)$ that (as $T \to \infty$)
\begin{align}
    \hat p_*(i,\delta T)
    \leq 
    \sum_{ j \in [d] }\P\big( N^>_{i;j}(T\delta) \geq 1 \big)
    =
    \bo \Bigg(
        \sum_{ j \in [d] }\P(B_{ j\leftarrow   l^*(j) } > T\delta)
    \Bigg)
    = \lo \big( T^{ -\tilde \alpha  } \big),
    \quad 
    \forall \tilde\alpha \in \big(1, \min_{ i \in [d],j \in [d] }\alpha_{i \leftarrow j}\big).
    \nonumber
\end{align}
As a result, we get
$
\lim_{n \to \infty}\inf_{ T \geq \floor{nc}  } \big(1 - \hat{p}_*(i,\delta T)\big)^{ T } = 1.
$
We then arrive at the lower bound (under any $\delta \in (0,\delta_0)$)
\begin{align}
    & \liminf_{n \to \infty}
    \inf_{ T \geq  \floor{nc} }
    \frac{
        \P\big(
            \sum_{m = 1}^{ T }N^{>,(m)}_{i;j}(\delta T)  \geq 1\text{ iff }j \in \mathcal J
        \big)
    }{
        (n^{-1}T)^{|\mathcal J|}\prod_{j \in \mathcal J}\bar s_{i,l^*(j)}\cdot n\P(B_{j \leftarrow l^*(j)} > T\delta)
    }
    \geq 1.
     \label{proof, part (i), lower bound, lemma: cluster size, asymptotics, N i | n delta, cdot j, multiple ancsetor}
\end{align}
By the uniform convergence theorem (e.g., Proposition~2.4 of \cite{resnick2007heavy}),
$$
\frac{\P(B_{j \leftarrow l^*(j)} > t \cdot n\delta)}{\P(B_{j\leftarrow l^*(j)} > n\delta)} \to t^{ -\alpha^*(j) } 
\ \text{ as }n \to \infty,
\quad 
\text{uniformly over }t \in [c,\infty).
$$
Plugging such uniform convergence 
into the bounds
\eqref{proof, part (i), upper bound, lemma: cluster size, asymptotics, N i | n delta, cdot j, multiple ancsetor} and \eqref{proof, part (i), lower bound, lemma: cluster size, asymptotics, N i | n delta, cdot j, multiple ancsetor}, we conclude the proof for Claim~\eqref{proof: goal, weak convergence for N, |I| = 1, part ii, N related, lemma: cluster size, asymptotics, N i | n delta, cdot j, multiple ancsetor, refined estimates}.

\medskip
\noindent
\textbf{Proof of Claim~\eqref{proof: goal, weak convergence for W, |I| = 1, part ii, N related, lemma: cluster size, asymptotics, N i | n delta, cdot j, multiple ancsetor, refined estimates}}.
In essence, the proof above for Claim~\eqref{proof: goal, weak convergence for N, |I| = 1, part ii, N related, lemma: cluster size, asymptotics, N i | n delta, cdot j, multiple ancsetor, refined estimates} regarding the event $\hat A(T,\delta)$ has verified that 
$
\lim_{T \to \infty}
    \P\big(
       \hat A(T,\delta)
        \ \big|\ 
        \sum_{m = 1}^{ T }N^{>,(m)}_{i;j}(\delta T)  \geq 1\text{ iff }j \in \mathcal J 
    \big) = 1
$
for any $\delta > 0$ small enough.
Therefore, to prove Claim~\eqref{proof: goal, weak convergence for W, |I| = 1, part ii, N related, lemma: cluster size, asymptotics, N i | n delta, cdot j, multiple ancsetor, refined estimates},
it suffices to show that
\begin{align*}
    \lim_{n \to \infty}
    \sup_{ T \geq nc }\  
    \sup_{ \frac{nx_j}{T} \in [c,C]\ \forall j \in \mathcal J }
    & 
    \Bigg|
        \frac{
        \P\big(
            \sum_{m = 1}^{ T }W^{>,(m)}_{i;j}(\delta T)  > n x_j\ \forall j \in \mathcal J
            \ \big|\ 
            \hat A(T,\delta)
        \big)
    }{
        \prod_{j \in \mathcal J}
            \big(
                \frac{\delta T}{n x_j}
            \big)^{ \alpha^*(j) }
    }
    - 1
    \Bigg| = 0,
    \quad 
    \forall \delta > 0.
\end{align*}
By definitions in \eqref{def: W i M j, pruned cluster, 1, cluster size}--\eqref{def: cluster size, N i | M cdot j}
and the independence of $W^{>,(m)}_{i;j}(T\delta)$ across $m \geq 1$,
when conditioned on the event 
$\hat A(T,\delta)$,
the conditional law of $\big(\sum_{m = 1}^T W^{>,(m)}_{i;j}(T\delta)\big)_{j \in \mathcal J}$ are independent across $j$,
and the conditional law of each $\sum_{m = 1}^T W^{>,(m)}_{i;j}(T\delta)$
is the same as
$
\P\big( B_{ j \leftarrow l^*(j) } \in \ \cdot\ \big| B_{ j \leftarrow l^*(j) } > T\delta  \big).
$
Therefore, 
\begin{align*}
    & \P\Bigg(
            \sum_{m = 1}^{ T }W^{>,(m)}_{i;j}(\delta T)  \geq n x_j\ \forall j \in \mathcal J
            \ \Bigg|\ 
            \hat A(T,\delta)
        \Bigg)
    = 
    \prod_{j \in \mathcal J}
    \frac{\P(B_{j \leftarrow l^*(j)} > \frac{nx_j}{T\delta} \cdot T\delta )}{\P(B_{j \leftarrow l^*(j)} > T\delta)}.
\end{align*}
By the uniform convergence theorem (with $\frac{nx_j}{T\delta} \in [cT/\delta,CT/\delta]$),
we conclude the proof of
Claim~\eqref{proof: goal, weak convergence for W, |I| = 1, part ii, N related, lemma: cluster size, asymptotics, N i | n delta, cdot j, multiple ancsetor, refined estimates}.
\end{proof}

Now, we are ready to provide the proofs of Lemmas~\ref{lemma: prob of type I, tail prob of tau, cluster size} and \ref{lemma: limit theorem for n tau to hat C I, cluster size}.

\begin{proof}[Proof of Lemma~\ref{lemma: prob of type I, tail prob of tau, cluster size}]\linksinpf{lemma: prob of type I, tail prob of tau, cluster size}
Fix $M > 0$ and some $\bm I = (I_{k,j})_{ k \geq 1, j \in [d] } \in \mathscr I$.
By Remark~\ref{remark: def of type},
there uniquely exists $j_1 \in [d]$ such that $I_{1,j_1} = 1$.
Besides, by \eqref{proof: def, set B type I M c, cluster size}, it holds on the event $\{n^{-1}\bm \tau^{n|\delta}_i \in E^{\bm I}(M,c)\}$ that
$
n^{-1}\tau^{n|\delta}_{i;j}(k) > M
$
for any $k \in [\mathcal K^{\bm I}]$, $j \in \bm j^{\bm I}_k$,
and
\begin{align*}
    \frac{
        n^{-1}\tau^{n|\delta}_{i;j}(k)
    }{
        n^{-1}\tau^{n|\delta}_{i;j^\prime}(k^\prime)
    }
    \in [c,1/c],
    \qquad\forall k \in [\mathcal K^{\bm I}],\ j \in \bm j^{\bm I}_k,\ 
    k^\prime \in [\mathcal K^{\bm I}],\ j^\prime \in \bm j^{\bm I}_{k^\prime}.
\end{align*}
On the other hand, 
by the definition of
${\bm I^{n|\delta}_i} = \big(I^{n|\delta}_{i;j}(k)\big)_{k \geq 1,\ j \in [d]}$ in \eqref{def, I k M j for M type, cluster size},
\begin{align*}
\big\{
        n^{-1}\bm \tau^{n|\delta}_i \in E^{\bm I}(M,c)
    \big\}
    =
    \big\{
    \bm I^{n|\delta}_i = \bm I,\ 
        n^{-1}\bm \tau^{n|\delta}_i \in E^{\bm I}(M,c)
    \big\}.
\end{align*}
Therefore, analogous to the derivation of \eqref{proof: applying markov property, decomp of events, lemma: crude estimate, type of cluster}, we get
(henceforth in this proof, we write
$
\bm t(k-1) = (t_j)_{ j \in \bm j^{\bm I}_{k-1} }
$
and
$
\bm x(k) = (x_j)_{ j \in \bm j^{\bm I}_k }
$)
\begin{align*}
    & \P\Big(
        \bm I^{n|\delta}_i = \bm I,\ 
        n^{-1}\bm \tau^{n|\delta}_i \in E^{\bm I}(M,c)
    \Big)
    \\ 
    & \leq 
    \underbrace{ \P\Big(
        W^{>}_{ i;j_1 }(n\delta) > nM 
    \Big) 
    }_{ \delequal p_1(n,M,\delta) }
    \\ 
    & \cdot 
    \prod_{k = 2}^{\mathcal K^{\bm I}}
    \sup_{
        \substack{
            t_l \geq nM\ \forall l \in \bm j^{\bm I}_{k-1}
            \\
            t_l/t_{l^\prime} \in [c,1/c]\ \forall l,l^\prime \in \bm j^{\bm I}_{k-1}
        }
    }
    \underbrace{ 
    \P\Big(
        N^{>|\delta}_{ \bm t(k-1);j } \geq 1\text{ iff }j \in \bm j^{\bm I}_k;\ 
        W^{>|\delta}_{ \bm t(k-1);j } \geq c \cdot \max_{l \in \bm j^{\bm I}_{k-1}}t_l \ \forall j \in \bm j^{\bm I}_k
    \Big)
    }_{
        \delequal 
        p_k(n,M,\delta,\bm t(k-1))
    }.
\end{align*}
First, due to \eqref{property, W and N i M l j when positive, cluster size},
\begin{align*}
    p_1(n,M,\delta) 
    = 
     \P\Big(
        W^{>}_{ i;j_1 }(n\delta) > nM\ \Big|\  N^{>}_{ i;j_1 }(n\delta) \geq 1
    \Big) 
    \cdot 
    \P\Big( N^{>}_{ i;j_1 }(n\delta) \geq 1\Big).
\end{align*}
By part (i) of Lemma~\ref{lemma: cluster size, asymptotics, N i | n delta, cdot j, crude estimate},
there is some $\delta_0 > 0$ such that for all $\delta \in (0,\delta_0)$,
\begin{align}
    \limsup_{ n \to \infty }
    \frac{
        p_1(n,M,\delta)
    }{
        \bar s_{i,l^*(j_1)} \P(B_{ j_1 \leftarrow l^*(j_1)} > n\delta) \cdot (\delta / M)^{ \alpha^*(j_1) }
    }
    \leq 1.
    \nonumber
\end{align}
Furthermore, due to  $\P(B_{ j_1 \leftarrow l^*(j_1)} > x) \in \RV_{ -\alpha^*(j_1) }(x)$ (see Assumption~\ref{assumption: heavy tails in B i j}),
we have 
$
\P(B_{j_1 \leftarrow l^*(j_1)} > n\delta) \cdot (\delta / M)^{ \alpha^*(j_1) } \sim 
\P(B_{j_1 \leftarrow l^*(j_1)} > n) \cdot (1 / M)^{ \alpha^*(j_1) },
$
and hence
\begin{align}
    \limsup_{ n \to \infty }
    \frac{
        p_1(n,M,\delta)
    }{
        \bar s_{i,l^*(j_1)} \P(B_{ j_1 \leftarrow  l^*(j_1)} > n) \cdot (1 / M)^{ \alpha^*(j_1) }
    }
    \leq 1.
    \label{proof: bound for p1, lemma: prob of type I, tail prob of tau, cluster size}
\end{align}
Next, for each $k = 2,3,\ldots,\mathcal K^{\bm I}$,
Claim
\eqref{claim weak convergence for N, part (ii), N related, lemma: cluster size, asymptotics, N i | n delta, cdot j, multiple ancsetor, refined estimates} in Lemma~\ref{lemma: cluster size, asymptotics, N i | n delta, cdot j, refined estimate} gives an upper bound for
$
\P\big( N^{>|\delta}_{ \bm t(k);j } \geq 1\text{ iff }j \in \bm j^{\bm I}_k \big),
$
whereas Claim~\eqref{claim 2, crude upper bound, part (ii), N related, lemma: cluster size, asymptotics, N i | n delta, cdot j, multiple ancsetor, refined estimates} in 
Lemma~\ref{lemma: cluster size, asymptotics, N i | n delta, cdot j, refined estimate} provides an upper bound for 
$
\P\big(
W^{>|\delta}_{ \bm t(k);j } \geq nx\ \forall j \in \bm j^{\bm I}_k\ \big|\ 
N^{>|\delta}_{ \bm t(k);j } \geq 1\text{ iff }j \in \bm j^{\bm I}_k \big),
$
with $x = n^{-1}c \cdot \max_{l \in \bm j^{\bm I}_{k-1}}t_l$.
In particular, under the condition that 
$
t_l/t_{l^\prime} \in [c,1/c]\ \forall l,l^\prime \in \bm j^{\bm I}_{k-1},
$
we have 
$
nx/t_l \in [1,1/c]
$
for each $l \in \bm j^{\bm I}_{k-1}$,
and 
\eqref{claim 2, crude upper bound, part (ii), N related, lemma: cluster size, asymptotics, N i | n delta, cdot j, multiple ancsetor, refined estimates}
provides a bound 
that holds uniformly over  $nx /t_l \in [1,1/c]$ for each $l \in \bm j^{\bm I}_{k - 1}$.
Therefore, by picking a smaller $\delta_0 > 0$ if necessary, it holds for any $\delta \in (0,\delta_0)$ that 
(henceforth in this proof, we use $a_n \lessapprox b_n$ to denote $\limsup_{n \to \infty}a_n/b_n \leq 1$)
\begin{align}
    &
        \sup_{
        \substack{
            t_l \geq nM\ \forall l \in \bm j^{\bm I}_{k-1}
            \\
            t_l/t_{l^\prime} \in [c,1/c]\ \forall l,l^\prime \in \bm j^{\bm I}_{k-1}
        }
    }
    p_k\big(n,M,\delta,\bm t(k-1) \big)
     \nonumber
     \\ 
     & \lessapprox
     \Bigg( \sup_{
        \bm t(k-1):\ t_l \geq nM\ \forall l \in \bm j^{\bm I}_{k-1}
    }
     C_{\bm j^{\bm I}_{k-1} \leftarrow \bm j^{\bm I}_k }\big( n^{-1}\bm t(k-1)\big)
    \Bigg)
    \nonumber
    \\ 
    &\quad 
         \cdot 
        \sup_{
        \substack{
            t_l \geq nM\ \forall l \in \bm j^{\bm I}_{k-1}
            \\
            t_l/t_{l^\prime} \in [c,1/c]\ \forall l,l^\prime \in \bm j^{\bm I}_{k-1}
        }
    }
         \sum_{ 
                \{ \mathcal J(i):\ i \in \bm j^{\bm I}_{k-1} \} \in \mathbb T_{ \bm j^{\bm I}_{k-1} \leftarrow \bm j^{\bm I}_k }
                }
                 \prod_{ i \in \mathcal \bm j^{\bm I}_{k-1}}
                \prod_{ j \in \mathcal J(i) }\bigg(\frac{\delta t_i}{c \cdot \max_{l \in \bm j^{\bm I}_{k-1}}t_l }\bigg)^{\alpha^*(j)}
                \cdot n\P(B_{j \leftarrow l^*(j)} > n\delta)
    \nonumber
    \\
    & \lessapprox
\Bigg( \sup_{
        \bm t(k-1):\ t_l \geq nM\ \forall l \in \bm j^{\bm I}_{k-1}
    }
     C_{\bm j^{\bm I}_{k-1} \leftarrow \bm j^{\bm I}_k }\big( n^{-1}\bm t(k-1)\big)
    \Bigg)
    \nonumber
    \\ 
    &\qquad\quad 
         \cdot 
        \sum_{ 
                \{ \mathcal J(i):\ i \in \bm j^{\bm I}_{k-1} \} \in \mathbb T_{ \bm j^{\bm I}_{k-1} \leftarrow \bm j^{\bm I}_k }
                }
                 \prod_{ i \in \mathcal \bm j^{\bm I}_{k-1}}
                \prod_{ j \in \mathcal J(i) }\bigg(\frac{\delta}{c }\bigg)^{\alpha^*(j)}
                \cdot n\P(B_{j \leftarrow l^*(j)} > n\delta).
    \nonumber
\end{align}
By $\P(B_{j \leftarrow l^*(j)} > x) \in \RV_{ -\alpha^*(j) }(x)$,
we have
$
\delta^{\alpha^*(j)}
                \cdot \P(B_{j \leftarrow l^*(j)} > n\delta)
                \sim \P(B_{j \leftarrow l^*(j)} > n).
$
Next, 
since $ C_{\mathcal I \leftarrow \mathcal J}\big( (t_i)_{i \in \mathcal I}\big)$ defined in \eqref{def: function C J assigned to I, cluster size} 
is monotone decreasing w.r.t.\ each $t_i$,
\begin{align*}
     \sup_{
        t_l \geq nM\ \forall l \in \bm j^{\bm I}_{k-1}
    }
    C_{\bm j^{\bm I}_{k-1} \leftarrow \bm j^{\bm I}_k}\big( n^{-1}\bm t(k-1)\big)
    =
    \underbrace{ \sum_{ 
                \{ \mathcal J(i):\ i \in \bm j^{\bm I}_{k-1} \} \in \mathbb T_{ \bm j^{\bm I}_{k-1} \leftarrow \bm j^{\bm I}_k }
                }
                 \prod_{ i \in \mathcal \bm j^{\bm I}_{k-1}}
                \prod_{ j \in \mathcal J(i) }
          \frac{ \bar s_{ i, l^*(j) } }{M^{\alpha^*(j) - 1} }
    }_{  \delequal C^{(M)}_{ k }  },
\end{align*}
where $C^{(M)}_{ k }$ monotonically tends to $0$ as $M \to \infty$.
In summary, 
by setting the constant
$
\tilde c_k \delequal 
\sum_{ 
                \{ \mathcal J(i):\ i \in \bm j^{\bm I}_{k-1} \} \in \mathbb T_{ \bm j^{\bm I}_{k-1} \leftarrow \bm j^{\bm I}_k }
                }
                 \prod_{ i \in \mathcal \bm j^{\bm I}_{k-1}}
                \prod_{ j \in \mathcal J(i) }
                c^{-\alpha^*(j)} \in (0,\infty),
$
it holds for any $\delta \in (0,\delta_0)$ that
\begin{align}
    \sup_{
        \substack{
            t_l \geq nM\ \forall l \in \bm j^{\bm I}_{k-1}
            \\
            t_l/t_{l^\prime} \in [c,1/c]\ \forall l,l^\prime \in \bm j^{\bm I}_{k-1}
        }
    }
    p_k\big(n,M,\delta,\bm t(k-1), \bm x(k)\big)
    \lessapprox
    C^{(M)}_{ k } \tilde c_k \cdot \prod_{j \in \bm j^{\bm I}_k}n\P(B_{j\leftarrow l^*(j)} > n),
    \ \ 
    \forall k = 2,3,\ldots,\mathcal K^{\bm I}.
    \label{proof: bound for pk, lemma: prob of type I, tail prob of tau, cluster size}
\end{align}
Combining \eqref{proof: bound for p1, lemma: prob of type I, tail prob of tau, cluster size} and \eqref{proof: bound for pk, lemma: prob of type I, tail prob of tau, cluster size},
we obtain (for each $\delta \in (0,\delta_0)$)
\begin{align*}
    \limsup_{n \to \infty}
    \frac{
        \P\big(
        \bm I^{n|\delta}_i = \bm I,\ 
        n^{-1}\bm \tau^{n|\delta}_i \in B^{\bm I}(M,c)
        \big)
    }{
        \big( \bar s_{i,l^*(j_1)}
        \cdot \prod_{k = 2}^{\mathcal K^{\bm I}}
         C^{(M)}_{ k } \tilde c_k\big)
         \cdot n^{-1} \prod_{ k = 1 }^{ \mathcal K^{\bm I} }\prod_{j \in \bm j^{\bm I}_k}n\P(B_{j \leftarrow l^*(j)} > n)
    }
    \leq 1.
\end{align*}
In particular, recall that $\bm j^{\bm I}_1 = \{j_1\}$,
so we have
$
n^{-1}\prod_{j \in \bm j^{\bm I}_1}n\P(B_{j \leftarrow l^*(j)} > n)
=
\P(B_{j_1 \leftarrow l^*(j_1)} > n)
$
in the denominator of the display above.
By \eqref{property: rate function for type I, cluster size} and that 
$
\lim_{M \to \infty}C^{(M)}_{ k } = 0,
$
we conclude the proof by setting
$
C^{\bm I}_i(M,c)
=
 \bar s_{i,l^*(j_1)}
        \cdot \prod_{k = 2}^{\mathcal K^{\bm I}}
         C^{(M)}_{ k }\tilde c_k.
$
\end{proof}

\begin{proof}[Proof of Lemma~\ref{lemma: limit theorem for n tau to hat C I, cluster size}]
\linksinpf{lemma: limit theorem for n tau to hat C I, cluster size}
It suffices to find $\delta_0 > 0$ such that the following holds for all $\delta \in (0,\delta_0)$:
given $\epsilon \in (0,1)$, there exists $\rho = \rho(\epsilon) \in (1,\infty)$ such that the inequalities
\begin{equation}\label{proof: goal, lemma: limit theorem for n tau to hat C I, cluster size}
    \begin{aligned}
        \limsup_{n \to\infty}
    \big(\lambda_{\bm j}(n)\big)^{-1}
    \P\bigg(
        n^{-1}\bm \tau^{n|\delta}_{i;j} \in 
        A^{\bm I}(\bm x, \bm y)
    \bigg)
    & \leq 
    (1 + \epsilon) \cdot 
    \bar s_{i,l^*(j^{\bm I}_1)} \cdot 
    \widehat{\mathbf C}^{\bm I}
    \Bigg(
        \bigtimes_{ k \in [\mathcal K^{\bm I}] }\bigtimes_{ j \in \bm j^{\bm I}_k } (x_{k,j},y_{k,j}]
    \Bigg),
    \\ 
    \liminf_{n \to\infty}
    \big(\lambda_{\bm j}(n)\big)^{-1}
    \P\bigg(
        n^{-1}\bm \tau^{n|\delta}_{i;j} \in 
        A^{\bm I}(\bm x, \bm y)
    \bigg)
    & \geq 
    (1 - \epsilon) \cdot 
    \bar s_{i,l^*(j^{\bm I}_1)} \cdot 
    \widehat{\mathbf C}^{\bm I}
    \Bigg(
        \bigtimes_{ k \in [\mathcal K^{\bm I}] }\bigtimes_{ j \in \bm j^{\bm I}_k } (x_{k,j},y_{k,j}]
    \Bigg)
    \end{aligned}
\end{equation}
hold
under the condition that $c \leq x_{k,j} < y_{k,j} \leq C$ and $y_{k,j}/x_{k,j} < \rho$ for any $k \in [\mathcal K^{\bm I}], j \in \bm j^{\bm I}_k$.
To see why, note that we can always partition
the set $\bigtimes_{ k \in [\mathcal K^{\bm I}] }\bigtimes_{ j \in \bm j^{\bm I}_k } (x_{k,j},y_{k,j}]$
in
\eqref{claim, lemma: limit theorem for n tau to hat C I, cluster size} into a union of finitely many disjoint sets of the form 
$\bigtimes_{ k \in [\mathcal K^{\bm I}] }\bigtimes_{ j \in \bm j^{\bm I}_k } (x^\prime_{k,j},y^\prime_{k,j}]$,
where we have
$c \leq x^\prime_{k,j} < y^\prime_{k,j} \leq C$ and $y^\prime_{k,j}/x^\prime_{k,j} < \rho$ for each $k,j$.
Then, we obtain \eqref{claim, lemma: limit theorem for n tau to hat C I, cluster size} by applying \eqref{proof: goal, lemma: limit theorem for n tau to hat C I, cluster size} onto each of the disjoint subset and sending $\epsilon$ to $0$ in the limit.

To prove \eqref{proof: goal, lemma: limit theorem for n tau to hat C I, cluster size}, we make some observations regarding $\hat{\mathbf C}^{\bm I}$.
Consider $\bigtimes_{ k \in [\mathcal K^{\bm I}] }\bigtimes_{ j \in \bm j^{\bm I}_k } (x_{k,j},y_{k,j}]$
with $0 < x_{k,j} < y_{k,j}$ for each $k,j$.
For clarity of the displays below, we write $\bm w_k = (w_{k,j})_{ j \in \bm j^{\bm I}_k }$.
By the definition of $\hat{\mathbf C}^{\bm I}$ in \eqref{def, measure hat C i type I, cluster size} and the definition of $\nu^{\bm I}$ in \eqref{def, measure nu type I, cluster size},
\begin{align*}
    \widehat{\mathbf C}^{\bm I}
    \Bigg(
        \bigtimes_{k \in [\mathcal K^{\bm I}]}\bigtimes_{j \in \bm j_k}(x_{k,j},y_{k,j}]
    \Bigg)
    & = 
    \left[\rule{0cm}{0cm}
        \prod_{k = 1}^{ \mathcal K^{\bm I} - 1 }
        \underbrace{ \int_{ w_{k,j} \in (x_{k,j},y_{k,j}]\ \forall j \in \bm j_k^{\bm I}  }
        g_{ \bm j_{k}^{\bm I} \leftarrow \bm j^{\bm I}_{k+1} }(\bm w_k)
         \Bigg(\bigtimes_{ j \in \bm j^{\bm I}_k }\nu_{\alpha^*(j)}(dw_{k,j})\Bigg)
        }_{
            \delequal \hat c_k
        }
    \right]
    \\ 
    &\qquad\qquad\qquad\qquad
    \cdot 
    \underbrace{ \prod_{j \in \bm j_{\mathcal K^{\bm I} }^{\bm I} }
    \int_{w_{\mathcal K^{\bm I},j} \in (x_{\mathcal K^{\bm I},j}, y_{\mathcal K^{\bm I},j}]}
    \nu_{\alpha^*(j)}(dw_{\mathcal K^{\bm I},j})}_{
        \delequal \hat c_{\mathcal K^{\bm I}}
    }.
\end{align*}
By the definitions  in \eqref{def, power law measure nu beta} we get
$
 \hat c_{\mathcal K^{\bm I}}
    =
    \prod_{j \in \bm j_{\mathcal K^{\bm I}}^{\bm I}}
        \big(\frac{1}{x_{\mathcal K^{\bm I},j}}\big)^{\alpha^*(j)}
            -
        \big(\frac{1}{y_{\mathcal K^{\bm I},j}}\big)^{\alpha^*(j)}.
$
Next, for each term $\hat c_k$ with $k \in [\mathcal K^{\bm I} - 1]$,
by the intermediate value theorem (in particular, due to the continuity of the mapping $g_{\mathcal I \leftarrow \mathcal J}$ defined in \eqref{def: function g mathcal I mathcal J, for measure C i bm I, cluster size}),
there exists some $\bm z_k = (z_{k,j})_{j \in \bm j^{\bm I}_k}$
with
$
z_{k,j} \in [x_{k,j},y_{k,j}]\ \forall j \in \bm j^{\bm I}_k
$
such that 
\begin{align*}
    \hat c_k
    & = 
    g_{\bm j^{\bm I}_k \leftarrow \bm j^{\bm I}_{k+1}}(\bm z_k)
    \cdot
    \int_{ w_{k,j} \in (x_{k,j},y_{k,j}]\ \forall j \in \bm j^{\bm I}_k  }
        \Bigg( \bigtimes_{ j \in \bm j^{\bm I}_k }\nu_{\alpha^*(j)}(dw_{k,j}) \Bigg)
    \\ 
    & = 
    g_{\bm j^{\bm I}_k \leftarrow \bm j^{\bm I}_{k+1}}(\bm z_k)\cdot
    \prod_{j \in \bm j^{\bm I}_k}
        \bigg(\frac{1}{x_{k,j}}\bigg)^{\alpha^*(j)}
            -
        \bigg(\frac{1}{y_{k,j}}\bigg)^{\alpha^*(j)}.
\end{align*}
On the other hand,
due to $0 < x_{k,j} \leq z_{k,j} \leq y_{k,j}$ and the monotonicity of $g_{\mathcal I \leftarrow \mathcal J}$,
\begin{align*}
    \frac{
        g_{\bm j^{\bm I}_k \leftarrow \bm j^{\bm I}_{k+1}}(\bm x_k)
    }{
        g_{\bm j^{\bm I}_k \leftarrow \bm j^{\bm I}_{k+1}}(\bm y_k)
    }
    \cdot \hat c_k
    \leq 
    \frac{
        g_{\bm j^{\bm I}_k \leftarrow \bm j^{\bm I}_{k+1}}(\bm x_k)
    }{
        g_{\bm j^{\bm I}_k \leftarrow \bm j^{\bm I}_{k+1}}(\bm z_k)
    }
    \cdot \hat c_k
    =
    g_{\bm j^{\bm I}_k \leftarrow \bm j^{\bm I}_{k+1}}(\bm x_k)\cdot
    \prod_{j \in \bm j^{\bm I}_k}
        \bigg(\frac{1}{x_{k,j}}\bigg)^{\alpha^*(j)}
            -
        \bigg(\frac{1}{y_{k,j}}\bigg)^{\alpha^*(j)}.
\end{align*}
In addition, by the definitions in \eqref{def: function g mathcal I mathcal J, for measure C i bm I, cluster size},
\begin{align*}
    \frac{
        g_{\bm j^{\bm I}_k \leftarrow \bm j^{\bm I}_{k+1}}(\bm x_k)
    }{
        g_{\bm j^{\bm I}_k \leftarrow \bm j^{\bm I}_{k+1}}(\bm y_k)
    }
    & \geq 
    \min_{j \in \bm j^{\bm I}_k}\bigg(\frac{x_{k,j}}{y_{k,j}}\bigg)^{|\bm j^{\bm I}_{k+1}|}.
\end{align*}
In summary,
\begin{align}
    & \Bigg(
        \prod_{k \in [\mathcal K^{\bm I} - 1]}
        g_{\bm j^{\bm I}_k \leftarrow \bm j^{\bm I}_{k+1}}(\bm x_k)
    \Bigg)
    \cdot 
    \Bigg[
        \prod_{k \in [\mathcal K^{\bm I}]}\prod_{j \in \bm j^{\bm I}_k}
            \bigg(\frac{1}{x_{k,j}}\bigg)^{\alpha^*(j)}
            -
        \bigg(\frac{1}{y_{k,j}}\bigg)^{\alpha^*(j)}
    \Bigg]
    \label{proof, lower bound for hat C i type I, cluster size}
    \\ 
    & \geq 
    \Bigg[ \prod_{k \in [\mathcal K^{\bm I} - 1]}
        \min_{j \in \bm j^{\bm I}_k}\bigg(\frac{x_{k,j}}{y_{k,j}}\bigg)^{|\bm j^{\bm I}_{k+1}|}
    \Bigg]
    \cdot 
    \prod_{k = 1}^{\mathcal K^{\bm I}}\hat c_k
    = 
    \Bigg[ \prod_{k \in [\mathcal K^{\bm I} - 1]}
        \min_{j \in \bm j^{\bm I}_k}\bigg(\frac{x_{k,j}}{y_{k,j}}\bigg)^{|\bm j^{\bm I}_{k+1}|}
    \Bigg]
    \cdot 
    \widehat{\mathbf C}^{\bm I}
    \Bigg(
        \bigtimes_{k \in [\mathcal K^{\bm I}]}\bigtimes_{j \in \bm j^{\bm I}_k}(x_{k,j},y_{k,j}]
    \Bigg).
    \nonumber
    \end{align}
Similarly, one can obtain the upper bound
\begin{align}
   & \Bigg(
        \prod_{k \in [\mathcal K^{\bm I} - 1]}
        g_{\bm j^{\bm I}_k \leftarrow \bm j^{\bm I}_{k+1}}(\bm y_k)
    \Bigg)
    \cdot 
    \Bigg[
        \prod_{k \in [\mathcal K^{\bm I}]}\prod_{j \in \bm j^{\bm I}_k}
            \bigg(\frac{1}{x_{k,j}}\bigg)^{\alpha^*(j)}
            -
        \bigg(\frac{1}{y_{k,j}}\bigg)^{\alpha^*(j)}
    \Bigg]
    \label{proof, upper bound for hat C i type I, cluster size}
    \\  
    & \leq 
    \Bigg[ \prod_{k \in [\mathcal K^{\bm I} - 1]}
        \max_{j \in \bm j^{\bm I}_k}\bigg(\frac{y_{k,j}}{x_{k,j}}\bigg)^{|\bm j^{\bm I}_{k+1}|}
    \Bigg]
    \cdot 
    \widehat{\mathbf C}^{\bm I}
    \Bigg(
        \bigtimes_{k \in [\mathcal K^{\bm I}]}\bigtimes_{j \in \bm j^{\bm I}_k}(x_{k,j},y_{k,j}]
    \Bigg).
    \nonumber
\end{align}

To proceed,
for any non-empty $\mathcal I \subseteq [d]$ and any $\mathcal J \subseteq [d]$, we define
\begin{align}
    p_{ \mathcal I, \mathcal J }
    \big(\delta,\bm t(\mathcal I),\bm w(\mathcal J)\big)
    \delequal
    \P\bigg(
        W^{>|\delta}_{ \bm t(\mathcal I);j } = w_j\ \forall j \in \mathcal J,\ 
        W^{>|\delta}_{ \bm t(\mathcal I);j } = 0\ \forall j \in [d]\setminus\mathcal J
    \bigg),
    \nonumber
\end{align}
where we write
$\bm w(\mathcal J) = (w_j)_{j \in \mathcal J}$,
and the
$W^{>|\delta}_{ \bm t(\mathcal I);j }$'s are defined in \eqref{def, proof cluster size, N W mathcal I mathcal J bcdot j}.
For $\mathcal J = \emptyset$, we set
$
p_{ \mathcal I, \emptyset }
    \big(\delta,\bm t(\mathcal I)\big)
    \delequal
    \P\big(
        W^{>|\delta}_{ \bm t(\mathcal I);j } = 0\ \forall j \in [d]
    \big).
$
By the Markov property in \eqref{proof, cluster size, markov property in pruned clusters},
we get (recall that $j^{\bm I}_1$ is the unique index $j \in [d]$ such that $I_{j,1} = 1$,
and that we write $\bm w_k = (w_{k,j})_{ j \in \bm j^{\bm I}_k }$)
\begin{align*}
    & \P\bigg(
        \tau^{n|\delta}_{i;j}(k) = w_{k,j}\ \forall k \in [\mathcal K^{\bm I}], j \in \bm j^{\bm I}_k;\ 
        \tau^{n|\delta}_{i;j}(k) = 0\ \forall k \geq 1,\ j \notin \bm j^{\bm I}_k
    \bigg)
    \\ 
    & = \P\Big( 
        \tau^{n|\delta}_{i;j^{\bm I}_1}(1) = w_{1, j^{\bm I}_1};\ 
        \tau^{n|\delta}_{i;j}(1) = 0\ \forall j \neq j^{\bm I}_1
        \Big)
    \\ 
    & \quad 
    \cdot
    \prod_{ k = 1 }^{ \mathcal K^{\bm I} }
    \P\bigg(
        \tau^{n|\delta}_{i;j}(k+1) = w_{k+1,j}\ \forall j \in \bm j^{\bm I}_{k+1};\ 
        \tau^{n|\delta}_{i;j}(k+1) = 0\ \forall j \notin \bm j^{\bm I}_{k+1}
    \\ 
    &\qquad\qquad\qquad\qquad\qquad\qquad\qquad
        \ \bigg|\ 
        \tau^{n|\delta}_{i;j}(k) = w_{k,j}\ \forall j \in \bm j^{\bm I}_{k};\ 
        \tau^{n|\delta}_{i;j}(k) = 0\ \forall j \notin \bm j^{\bm I}_{k}
    \bigg)
    \\ 
    & \stackrel{(*)}{=} \P\Big(
        W^{>}_{ i;j^{\bm I}_1 }(n\delta) = w_{1,j^{\bm I}_1};\ 
        W^{>}_{ i;j }(n\delta) = 0\ \forall j \neq j^{\bm I}_1 
        \Big)
    \\
    & \qquad
    \cdot 
    \prod_{ k = 1 }^{ \mathcal K^{\bm I} }
    \P\bigg(
        W^{>|\delta}_{ \bm w_k; j} = w_{k+1,j}\ \forall j \in \bm j^{\bm I}_{k+1};\ 
        W^{>|\delta}_{ \bm w_k; j} = 0\ \forall j \notin \bm j^{\bm I}_{k+1}
    \bigg)
    \\ 
    & = 
    \P\Big(
        W^{>}_{ i;j^{\bm I}_1 }(n\delta) = w_{1,j^{\bm I}_1};\ 
        W^{>}_{ i;j }(n\delta) = 0\ \forall j \neq j^{\bm I}_1 
        \Big)
    \cdot 
    \Bigg[
    \prod_{ k = 1 }^{ \mathcal K^{\bm I} - 1 }
        p_{ \bm j^{\bm I}_k, \bm j^{\bm I}_{k+1}   }\big( \delta, \bm w_k, \bm w_{k+1} \big)
    \Bigg]
    \cdot 
     p_{ \bm j^{\bm I}_{\mathcal K^{\bm I}}, \emptyset  }\big( \delta, \bm w_{ \mathcal K^{\bm I}  }\big).
\end{align*}
Here, the step $(*)$ in the display above follows from the definition of the $\tau^{n|\delta}_{i;j}(k)$'s in \eqref{def: layer zero, from pruned cluster to full cluster} and \eqref{def: from pruned cluster to full cluster, tau and S at step k + 1}.
Then by \eqref{def: set A x y type I, lemma: limit theorem for n tau to hat C I, cluster size},
we have
(in the displays below, we interpret $\sum_{ w_{k,j} \in (a,b] }$ as the summation over all the integers in $(a,b]$ because $W^{>}_{i;j}(n\delta)$ will only take integer values by definition)
\begin{align}
    & \P\bigg(
        n^{-1}\bm \tau^{n|\delta}_{i;j} \in 
        A^{\bm I}(\bm x, \bm y)
    \bigg)
    \label{proof: applying Markov property, lemma: limit theorem for n tau to hat C I, cluster size}
    \\ 
    & = 
    \sum_{ w_{1, j^{\bm I}_1} \in (nx_{1, j^{\bm I}_1},\  ny_{1, j^{\bm I}_1} ]  }
        \P\Big(
        W^{>}_{ i;j^{\bm I}_1 }(n\delta) = w_{1,j^{\bm I}_1};\ 
        W^{>}_{ i;j }(n\delta) = 0\ \forall j \neq j^{\bm I}_1 
        \Big)
    \nonumber
    \\ 
    & \quad 
    \cdot 
    \sum_{ w_{2,j} \in (nx_{2,j}, ny_{2,j}]\ \forall j \in \bm j^{\bm I}_2    }
        p_{ \bm j^{\bm I}_1, \bm j^{\bm I}_{2}   }\big( \delta, \bm w_1, \bm w_{2} \big)
    \cdot 
    \ldots 
    \cdot 
    \sum_{ w_{ \mathcal K^{\bm I} ,j} \in (nx_{ \mathcal K^{\bm I} ,j},\ ny_{\mathcal K^{\bm I},j}]\ \forall j \in \bm j^{\bm I}_{ \mathcal K^{\bm I} }    }
        p_{ \bm j^{\bm I}_{ \mathcal K^{\bm I} - 1 }, \bm j^{\bm I}_{\mathcal K^{\bm I}}   }\big( \delta, \bm w_{ \mathcal K^{\bm I} - 1 }, \bm w_{\mathcal K^{\bm I}} \big)
    \nonumber
    \\ 
    &\quad\cdot 
    p_{ \bm j^{\bm I}_{\mathcal K^{\bm I}}, \emptyset  }\big( \delta, \bm w_{ \mathcal K^{\bm I}  }\big).
    \nonumber
\end{align}
To characterize the asymptotics of \eqref{proof: applying Markov property, lemma: limit theorem for n tau to hat C I, cluster size},
we first note that
\begin{align*}
    p_{ \bm j^{\bm I}_{\mathcal K^{\bm I}}, \emptyset  }\big( \delta, \bm w_{ \mathcal K^{\bm I}  }\big)
    & = 
    \P\Big(
        W^{>|\delta}_{ \bm w_{ \mathcal K^{\bm I}  };j } = 0\ \forall j \in [d]
    \Big)
    =
    \P\Big(
        N^{>|\delta}_{ \bm w_{ \mathcal K^{\bm I}  };j } = 0\ \forall j \in [d]
    \Big)
    \qquad\text{due to \eqref{property, sum of W and N i M l j when positive, cluster size}}.
\end{align*}
Recall that we have 
$c \leq x_{k,j} < y_{k,j} \leq C$ for each $k \in [\mathcal K^{\bm I}],j \in \bm j^{\bm I}_k$.
By Claim \eqref{claim weak convergence for N, part (ii), N related, lemma: cluster size, asymptotics, N i | n delta, cdot j, multiple ancsetor, refined estimates} of Lemma~\ref{lemma: cluster size, asymptotics, N i | n delta, cdot j, refined estimate}
under the choice of $\mathcal J = \emptyset$
(in which case we have $
C_{ \mathcal I \leftarrow \emptyset }\big( (t_i)_{ i \in \mathcal I } \big) \equiv 1
$ in \eqref{def: function C J assigned to I, cluster size}),
we can identify some $\delta_0 = \delta_0(c) > 0$ such that
\begin{align}
    \lim_{n \to \infty}
    \min_{ w_{ \mathcal K^{\bm I} ,j} \in (nx_{ \mathcal K^{\bm I} ,j},\ ny_{\mathcal K^{\bm I},j}]\ \forall j \in \bm j^{\bm I}_{ \mathcal K^{\bm I} }    }
    \P\Big(
        N^{>|\delta}_{ \bm w_{ \mathcal K^{\bm I}  };j } = 0\ \forall j \in [d]
    \Big)
    = 1,
    \qquad
    \forall \delta \in (0,\delta_0).
    \label{proof: bound 1, lemma: limit theorem for n tau to hat C I, cluster size}
\end{align}
Similarly, for each $k \in [\mathcal K^{\bm I} - 1]$,
\begin{align*}
     & p_{ \bm j^{\bm I}_k, \bm j^{\bm I}_{k+1}   }\big( \delta, \bm w_k, \bm w_{k+1} \big)
    \\
    & = 
    \P\bigg(
        W^{>|\delta}_{ \bm w_k;j } = w_{k+1,j}\ \forall j \in \bm j^{\bm I}_{k+1}
        \ \bigg|\ 
        N^{>|\delta}_{ \bm w_k;j } \geq 1\text{ iff } j \in \bm j^{\bm I}_{k+1}
    \bigg)
    \cdot 
    \P\bigg( N^{>|\delta}_{ \bm w_k;j } \geq 1\text{ iff } j \in \bm j^{\bm I}_{k+1} \bigg)
    \quad\text{by \eqref{property, sum of W and N i M l j when positive, cluster size}},
    \\ 
    \Longrightarrow
    & \sum_{ w_{k+1,j} \in (nx_{k+1,j},ny_{k+1,j}]\ \forall j \in \bm j^{\bm I}_{k+1} }
        p_{ \bm j^{\bm I}_k, \bm j^{\bm I}_{k+1}   }\big( \delta, \bm w_k, \bm w_{k+1} \big)
        \\ 
    & = 
    \P\bigg(
        W^{>|\delta}_{ \bm w_k;j } \in (n{x_{k+1,j}},n{y_{k+1,j}}] \ \forall j \in \bm j^{\bm I}_{k+1}
        \ \bigg|\ 
        N^{>|\delta}_{ \bm w_k;j } \geq 1\text{ iff } j \in \bm j^{\bm I}_{k+1}
    \bigg)
    \cdot 
    \P\bigg( N^{>|\delta}_{ \bm w_k;j } \geq 1\text{ iff } j \in \bm j^{\bm I}_{k+1} \bigg).
\end{align*}
Besides,
under the condition that $w_{l,j} \in (nx_{l,j},ny_{l,j}]$ and $c \leq x_{l,j} < y_{l,j} \leq C$
for each $l,j$, we have 
$$
\frac{n x_{k+1,j} }{ w_{k,j^\prime} } \in \bigg[ \frac{c}{C}, \frac{C}{c} \bigg],\ \ 
\frac{n y_{k+1,j} }{ w_{k,j^\prime} } \in \bigg[ \frac{c}{C}, \frac{C}{c} \bigg],
\qquad
\forall j \in \bm j^{\bm I}_{k+1},\ j^\prime \in \bm j^{\bm I}_k.
$$
This allows us to apply Claim \eqref{claim weak convergence for N, part (ii), N related, lemma: cluster size, asymptotics, N i | n delta, cdot j, multiple ancsetor, refined estimates} and \eqref{claim 2, part (ii), N related, lemma: cluster size, asymptotics, N i | n delta, cdot j, multiple ancsetor, refined estimates} in Lemma~\ref{lemma: cluster size, asymptotics, N i | n delta, cdot j, refined estimate} and obtain that (by picking a smaller $\delta_0 = \delta_0(c,C) > 0$ if needed)
for any $\delta \in (0,\delta_0)$,
\begin{align*}
    & \lim_{n \to \infty}
    \max_{ 
        \substack{
            \bm w_k:\ w_{k,j} \in (nx_{k,j},n y_{k,j}]\ \forall j \in \bm j^{\bm I}_k 
        }
    }
    \\
    &\quad
    \Bigg|
        \frac{
             \sum_{ w_{k+1,j} \in (nx_{k+1,j},ny_{k+1,j}]\ \forall j \in \bm j^{\bm I}_{k+1} }
            p_{ \bm j^{\bm I}_k, \bm j^{\bm I}_{k+1}   }\big( \delta, \bm w_k, \bm w_{k+1} \big)
        }{
            g_{ \bm j^{\bm I}_k \leftarrow \bm j^{\bm I}_{k+1}  }( n^{-1}\bm w_k)
            \prod_{j \in \bm j^{\bm I}_{k+1}}
                n \P(B_{j \leftarrow l^*(j)} > n\delta)
                \cdot 
                \big[
                    \big( \frac{\delta}{ x_{k+1,j} } \big)^{\alpha^*(j)}
                    -
                    \big( \frac{\delta}{ y_{k+1,j} } \big)^{\alpha^*(j)}
                \big]
        }
    - 1
    \Bigg|= 0.
\end{align*}
We stress that the choice of $\delta_0$ only depends on $c$ and $C$,
due to $c \leq x_{l,j} < y_{l,j} \leq C$ for each $l$ and $j$.
Furthermore, due to 
$\P(B_{j \leftarrow l^*(j)} > x) \in \RV_{ -\alpha^*(j) }(x)$ (see Assumption~\ref{assumption: heavy tails in B i j}),
for each $\delta > 0$ we have (as $n\to\infty$)
\begin{align*}
    & \P(B_{j \leftarrow l^*(j)} > n\delta)
                \Bigg[
                    \bigg( \frac{\delta}{ x_{k+1,j} } \bigg)^{\alpha^*(j)}
                    -
                    \bigg( \frac{\delta}{ y_{k+1,j} } \bigg)^{\alpha^*(j)}
                \Bigg]
    \\ & 
    \sim 
    \P(B_{j \leftarrow l^*(j)} > n)
                \Bigg[
                    \bigg( \frac{1}{ x_{k+1,j} } \bigg)^{\alpha^*(j)}
                    -
                    \bigg( \frac{1}{ y_{k+1,j} } \bigg)^{\alpha^*(j)}
                \Bigg].
\end{align*}
Also, the monotonicity of $g_{\mathcal I \leftarrow \mathcal J}$ implies
$
g_{ \bm j^{\bm I}_k \leftarrow \bm j^{\bm I}_{k+1}  }( \bm x_k)
     \leq 
     g_{ \bm j^{\bm I}_k \leftarrow \bm j^{\bm I}_{k+1}  }( n^{-1}\bm w_k)
     \leq 
     g_{ \bm j^{\bm I}_k \leftarrow \bm j^{\bm I}_{k+1}  }( \bm y_k),
$
provided that $w_{k,j} \in (n x_{k,j}, ny_{k,j}]$ for each $j \in \bm j^{\bm I}_k$.
In summary, for each $\delta \in (0,\delta_0)$,
    \begin{align}
      \limsup_{n \to \infty}
    \max_{ 
        \substack{
            w_{k,j} \in (nx_{k,j},n y_{k,j}]\ \forall j \in \bm j^{\bm I}_k 
        }
    }
    &
    \sum_{ w_{k+1,j} \in (nx_{k+1,j},ny_{k+1,j}]\ \forall j \in \bm j^{\bm I}_{k+1} }
        \frac{ 
            p_{ \bm j^{\bm I}_k, \bm j^{\bm I}_{k+1}   }\big( \delta, \bm w_k, \bm w_{k+1} \big)
        }
        {
            \prod_{j \in \bm j^{\bm I}_{k+1}}
                n \P(B_{j \leftarrow l^*(j)} > n)
        }
    \label{proof: bound 2, lemma: limit theorem for n tau to hat C I, cluster size}
    \\
    &\qquad\qquad \leq 
    g_{ \bm j^{\bm I}_k \leftarrow \bm j^{\bm I}_{k+1}  }( \bm y_k)\cdot 
    \prod_{ j \in \bm j^{\bm I}_{k+1} }
    \Bigg[
                    \bigg( \frac{1}{ x_{k+1,j} } \bigg)^{\alpha^*(j)}
                    -
                    \bigg( \frac{1}{ y_{k+1,j} } \bigg)^{\alpha^*(j)}
                \Bigg],
    \nonumber
    \\ 
        \liminf_{n \to \infty}
    \min_{ 
        \substack{
            w_{k,j} \in (nx_{k,j},n y_{k,j}]\ \forall j \in \bm j^{\bm I}_k 
        }
    }
    &
    \sum_{ w_{k+1,j} \in (nx_{k+1,j},ny_{k+1,j}]\ \forall j \in \bm j^{\bm I}_{k+1} }
        \frac{ 
            p_{ \bm j^{\bm I}_k, \bm j^{\bm I}_{k+1}   }\big( \delta, \bm w_k, \bm w_{k+1} \big)
        }
        {
            \prod_{j \in \bm j^{\bm I}_{k+1}}
                n \P(B_{j \leftarrow l^*(j)} > n)
        }
    \nonumber
    \\
    &\qquad\qquad \geq 
    g_{ \bm j^{\bm I}_k \leftarrow \bm j^{\bm I}_{k+1}  }( \bm x_k)\cdot 
    \prod_{ j \in \bm j^{\bm I}_{k+1} }
    \Bigg[
                    \bigg( \frac{1}{ x_{k+1,j} } \bigg)^{\alpha^*(j)}
                    -
                    \bigg( \frac{1}{ y_{k+1,j} } \bigg)^{\alpha^*(j)}
                \Bigg].  \nonumber
    \end{align}
Lastly, for the term
\begin{align*}
    & \sum_{ w_{1, j^{\bm I}_1} \in (nx_{1, j^{\bm I}_1},\  ny_{1, j^{\bm I}_1} ]  }
        \P\Big(
        W^{>}_{ i;j^{\bm I}_1 }(n\delta) = w_{1,j^{\bm I}_1};\ 
        W^{>}_{ i;j }(n\delta) = 0\ \forall j \neq j^{\bm I}_1 
        \Big)
    \\ 
    & 
    = 
    \P\bigg(
        W^{>}_{ i;j^{\bm I}_1 }(n\delta) \in \big(nx_{1, j^{\bm I}_1},\  ny_{1, j^{\bm I}_1} \big] ;\ 
        W^{>}_{ i;j }(n\delta) = 0\ \forall j \neq j^{\bm I}_1 
        \bigg)
\end{align*}
in the display \eqref{proof: applying Markov property, lemma: limit theorem for n tau to hat C I, cluster size},
by part (i) of Lemma~\ref{lemma: cluster size, asymptotics, N i | n delta, cdot j, crude estimate}
(pick a smaller $\delta_0 = \delta_0(c,C) > 0$ if needed),
it holds for any $\delta \in (0,\delta_0)$ that 
\begin{align}
    \lim_{n \to \infty}
    \left|\rule{0cm}{0.9cm}
        \frac{
            \sum_{ w_{1, j^{\bm I}_1} \in (nx_{1, j^{\bm I}_1},\  ny_{1, j^{\bm I}_1} ]  }
        \P\big(
        W^{>}_{ i;j^{\bm I}_1 }(n\delta) = w_{1,j^{\bm I}_1};\ 
        W^{>}_{ i;j }(n\delta) = 0\ \forall j \neq j^{\bm I}_1 
        \big)
        }{
            \bar s_{1, l^*(j^{\bm I}_1) } \cdot \P(B_{ j^{\bm I}_1 \leftarrow l^*(j^{\bm I}_1)} > n)
            \cdot 
            \big[
                \big(\frac{1}{x_{1, j^{\bm I}_1}  }\big)^{\alpha^*( j^{\bm I}_1 )}
                -
                \big(\frac{1}{y_{1, j^{\bm I}_1}  }\big)^{\alpha^*( j^{\bm I}_1 )}
            \big]
        } - 1
     \right|
    = 0. 
    \label{proof: bound 3, lemma: limit theorem for n tau to hat C I, cluster size}
\end{align}
By \eqref{property: rate function for type I, cluster size} and our assumption of $\bm j^{\bm I} = \bm j$,
\begin{align*}
    \lambda_{\bm j}(n)
    & = 
    n^{-1}\prod_{ k = 1 }^{ \mathcal K^{\bm I} }\prod_{ j \in \bm j^{\bm I} }
    n\P(B_{j \leftarrow l^*(j)} > n)
    =
    \P(B_{  j^{\bm I}_1  \leftarrow l^*( j^{\bm I}_1)} > n)
    \cdot 
    \prod_{ k = 2 }^{ \mathcal K^{\bm I} }\prod_{ j \in \bm j^{\bm I} }
    n\P(B_{j \leftarrow l^*(j)} > n).
\end{align*}
Plugging \eqref{proof: bound 1, lemma: limit theorem for n tau to hat C I, cluster size}, \eqref{proof: bound 2, lemma: limit theorem for n tau to hat C I, cluster size}, \eqref{proof: bound 3, lemma: limit theorem for n tau to hat C I, cluster size} into \eqref{proof: applying Markov property, lemma: limit theorem for n tau to hat C I, cluster size}, we obtain (for any $\delta \in (0,\delta_0)$)
\begin{align*}
    & \limsup_{n \to\infty}
    \big(\lambda_{\bm j}(n)\big)^{-1}
    \P\bigg(
        n^{-1}\bm \tau^{n|\delta}_{i;j} \in 
        A^{\bm I}(\bm x, \bm y)
    \bigg)
    \\ 
    & 
    \leq 
    \bar s_{1, l^*(j^{\bm I}_1) }
    \cdot 
    \Bigg(
        \prod_{k \in [\mathcal K^{\bm I} - 1]}
        g_{\bm j^{\bm I}_k \leftarrow \bm j^{\bm I}_{k+1}}(\bm y_k)
    \Bigg)
    \cdot 
    \Bigg[
        \prod_{k \in [\mathcal K^{\bm I}]}\prod_{j \in \bm j^{\bm I}_k}
            \bigg(\frac{1}{x_{k,j}}\bigg)^{\alpha^*(j)}
            -
        \bigg(\frac{1}{y_{k,j}}\bigg)^{\alpha^*(j)}
    \Bigg],
    \\ 
    & \liminf_{n \to\infty}
    \big(\lambda_{\bm j}(n)\big)^{-1}
    \P\bigg(
        n^{-1}\bm \tau^{n|\delta}_{i;j} \in 
        A^{\bm I}(\bm x, \bm y)
    \bigg)
    \\ 
    & \geq 
    \bar s_{1, l^*(j^{\bm I}_1) }
    \cdot 
    \Bigg(
        \prod_{k \in [\mathcal K^{\bm I} - 1]}
        g_{\bm j^{\bm I}_k \leftarrow \bm j^{\bm I}_{k+1}}(\bm x_k)
    \Bigg)
    \cdot 
    \Bigg[
        \prod_{k \in [\mathcal K^{\bm I}]}\prod_{j \in \bm j^{\bm I}_k}
            \bigg(\frac{1}{x_{k,j}}\bigg)^{\alpha^*(j)}
            -
        \bigg(\frac{1}{y_{k,j}}\bigg)^{\alpha^*(j)}
    \Bigg].
\end{align*}
Lastly, to verify Claim \eqref{proof: goal, lemma: limit theorem for n tau to hat C I, cluster size} given $\epsilon > 0$, we observe the following.
By the bounds in \eqref{proof, lower bound for hat C i type I, cluster size} and \eqref{proof, upper bound for hat C i type I, cluster size},
it suffices to pick $\rho > 1$ such that
\begin{align}
    \prod_{k \in [\mathcal K^{\bm I} - 1]}
        1/\rho^{|\bm j^{\bm I}_{k+1}|}
    > 1 - \epsilon,
    \qquad
    \prod_{k \in [\mathcal K^{\bm I} - 1]}
        \rho^{|\bm j^{\bm I}_{k+1}|}
    < 1 + \epsilon,
    \nonumber
\end{align}
In  case that $\mathcal K^{\bm I} = 1$, the display above holds trivially as the product degenerates to $1$.
In  case that $\mathcal K^{\bm I} \geq 2$, the display above holds for any $\rho > 1$ close enough to 1.
\end{proof}

\ifshowtheoremtree
\footnotesize
\newgeometry{left=1cm,right=1cm,top=0.5cm,bottom=1.5cm}

\section{Theorem Tree}\label{sec: appendix, theorem tree}
\textbf{Theorem Tree of Theorem~\ref{theorem: main result, cluster size}}
\begin{thmdependence}[leftmargin=*]
    \thmtreenode{-}{Theorem}{theorem: main result, cluster size}{}
    \begin{thmdependence}
        \thmtreenode{-}{Lemma}{lemma: M convergence for MRV}{}
            \begin{thmdependence}
                \thmtreenode{-}{Lemma}{lemma: equivalence for bounded away condition, polar coordinates}{}
            \end{thmdependence}

        \thmtreenode{-}{Lemma}{lemma: asymptotic equivalence, MRV in Rd}{} 
            \begin{thmdependence}
                \thmtreeref{Theorem}{portmanteau theorem M convergence}
            \end{thmdependence}

        \thmtreenode{-}{Proposition}{proposition: asymptotic equivalence, tail asymptotics for cluster size}{}
            \begin{thmdependence}
                \thmtreenode{-}{Lemma}{lemma: crude estimate, type of cluster}{}
                \begin{thmdependence}
                    \thmtreenode{-}{Lemma}{lemma: cluster size, asymptotics, N i | n delta, cdot j, crude estimate}{}
                \end{thmdependence}

                \thmtreenode{-}{Lemma}{lemma: distance between bar s and hat S given type, cluster size}{}
                \begin{thmdependence}
                    \thmtreenode{-}{Lemma}{lemma: tail bound, pruned cluster size S i leq n delta}{}

                    \thmtreenode{-}{Lemma}{lemma: concentration ineq for pruned cluster S i}{}
                \end{thmdependence}
            \end{thmdependence}

        \thmtreenode{-}{Proposition}{proposition, M convergence for hat S, tail asymptotics for cluster size}{}
            \begin{thmdependence}
                \thmtreenode{-}{Lemma}{lemma: type and generalized type}{}

                \thmtreenode{-}{Lemma}{lemma: choice of bar epsilon bar delta, cluster size}{}

                \thmtreenode{-}{Lemma}{lemma: prob of type I, tail prob of tau, cluster size}{}

                \thmtreenode{-}{Lemma}{lemma: limit theorem for n tau to hat C I, cluster size}{}

                \thmtreeref{Lemma}{lemma: crude estimate, type of cluster}
            \end{thmdependence}

        \thmtreenode{-}{Lemma}{lemma: cluster size, asymptotics, N i | n delta, cdot j, crude estimate}{}
            \begin{thmdependence}
                \thmtreenode{-}{Lemma}{lemma: cluster size, asymptotics, N i | n delta, l j}{}
                \begin{thmdependence}
                    \thmtreeref{Lemma}{lemma: tail bound, pruned cluster size S i leq n delta}
                \end{thmdependence}
            \end{thmdependence}

    \end{thmdependence}

\end{thmdependence}
\bigskip

\noindent
\textbf{Theorem Tree of Technical Lemmas}
\begin{thmdependence}[leftmargin=*]

\thmtreenode{-}{Lemma}{lemma: tail bound, pruned cluster size S i leq n delta}{($\norm{\bar{\textbf B}} < 1$)}
    \begin{thmdependence}
        \thmtreenode{-}{Lemma}{lemma: concentration ineq, truncated heavy tailed RV in Rd}{}
    \end{thmdependence}

    \thmtreenode{-}{Lemma}{lemma: tail bound, pruned cluster size S i leq n delta}{(General Case)}
    \begin{thmdependence}
        \thmtreenode{-}{Lemma}{lemma: tail asymptotics for truncated GW, at generation t}{}
            \begin{thmdependence}
                \thmtreeref{Lemma}{lemma: concentration ineq, truncated heavy tailed RV in Rd}
            \end{thmdependence}

        \thmtreenode{-}{Lemma}{lemma, tail bound, pruned cluster sub sampled size S i r leq n delta}{}
            \begin{thmdependence}
                \thmtreeref{Lemma}{lemma: concentration ineq, truncated heavy tailed RV in Rd}
                \thmtreeref{Lemma}{lemma: tail asymptotics for truncated GW, at generation t}
            \end{thmdependence}

        \thmtreenode{-}{Lemma}{lemma: concentration ineq for r step sub tree}{}
            \begin{thmdependence}
                \thmtreeref{Lemma}{lemma: concentration ineq, truncated heavy tailed RV in Rd}
                \thmtreeref{Lemma}{lemma, tail bound, pruned cluster sub sampled size S i r leq n delta}
            \end{thmdependence}
    \end{thmdependence}

    \thmtreenode{-}{Lemma}{lemma: concentration ineq for pruned cluster S i}{}
        \begin{thmdependence}
            \thmtreeref{Lemma}{lemma: tail bound, pruned cluster size S i leq n delta}
            \thmtreeref{Lemma}{lemma: concentration ineq, truncated heavy tailed RV in Rd}
        \end{thmdependence}

    \thmtreenode{-}{Lemma}{lemma: prob of type I, tail prob of tau, cluster size}{}
        \begin{thmdependence}
            \thmtreenode{-}{Lemma}{lemma: cluster size, asymptotics, N i | n delta, cdot j, refined estimate}{}
            \begin{thmdependence}
                \thmtreeref{Lemma}{lemma: cluster size, asymptotics, N i | n delta, cdot j, crude estimate}
            \end{thmdependence}

            \thmtreeref{Lemma}{lemma: cluster size, asymptotics, N i | n delta, cdot j, crude estimate}
        \end{thmdependence}

    \thmtreenode{-}{Lemma}{lemma: limit theorem for n tau to hat C I, cluster size}{}
        \begin{thmdependence}
            \thmtreeref{Lemma}{lemma: cluster size, asymptotics, N i | n delta, cdot j, refined estimate}
        \end{thmdependence}
\end{thmdependence}

\end{appendix}

\newpage
\bibliographystyle{abbrv} 
\bibliography{aap-bib} 






\end{document}